\documentclass[reqno,11pt]{amsart} %reqno
\usepackage{geometry, eucal, comment}
\usepackage{mathtools,amssymb,amsthm,mathrsfs,color,lineno,paralist,graphicx,float}
\usepackage[colorlinks,
linkcolor=red,
anchorcolor=green,
citecolor=blue,
]{hyperref}

\usepackage[T1]{fontenc}
\usepackage[utf8]{inputenc}
\usepackage[english,french]{babel}
\usepackage{graphicx}
\usepackage{subfigure}
\usepackage{tikz}
\usepackage{enumitem}
\usepackage{calc}
\definecolor{bleu1}{RGB}{0,57,128}
\def\bleu1{\color{bleu1}}
\usepackage{etoolbox}
\patchcmd{\section}{\normalfont}{\normalfont \bleu1}{}{}
\patchcmd{\subsection}{\normalfont}{\normalfont \bleu1}{}{}
\patchcmd{\subsubsection}{\normalfont}{\normalfont \bleu1}{}{}

\newcommand{\R}{{\mathbb R}}

\newcommand{\T}{{\mathbb T}}

\newcommand{\N}{{\mathbb N}}

\newcommand{\Diff}{\textup{Diff}}
\newcommand{\cX}{\EuScript X}
\newcommand{\cT}{\EuScript T}

\newcommand{\cC}{{\mathcal C}}
\newcommand{\cE}{{\mathcal E}}
\newcommand{\cV}{{\mathcal V}}

\newcommand{\eps}{\varepsilon}

\setlist[itemize,1]{label=$\bullet$}

\newtheorem{theorem}{Theorem}[section]
\newtheorem{remark}[theorem]{Remark}%[section]
\newtheorem{lemma}[theorem]{Lemma}%[section]
\newtheorem{addendum}[theorem]{Addendum}%[section]
\newtheorem{definition}[theorem]{Definition}%[section]
\newtheorem{question}[theorem]{Question}
\newtheorem{proposition}[theorem]{Proposition}%[section]
\newtheorem{corollary}[theorem]{Corollary}%[section]
\newtheorem{claim}[theorem]{Claim}%[section]
\newtheorem{theoalph}{Theorem}

\newtheorem{coralph}[theoalph]{Corollary}
\newtheorem{addalph}[theoalph]{Addendum}
\renewcommand{\phi}{\varphi}

\numberwithin{equation}{section}

\begin{document}

\title[Smooth rigidity for $3$-dimensional dissipative Anosov flows]{Smooth rigidity for 3-dimensional dissipative Anosov flows}

\author{Andrey Gogolev}
\thanks{The first author was partially supported by the NSF grant DMS-2247747.}
\address{Department of Mathematics, The Ohio State University\\ Columbus, OH 43210, USA.}
\email{gogolyev.1@osu.edu}

\author{Martin Leguil}
\thanks{The second author was partially supported by the ANR AAPG 2021 PRC CoSyDy (Grant No. ANR-CE40-0014), by the ANR JCJC PADAWAN (Grant No. ANR-21-CE40-0012), and by the LESET Math-AMSUD project.}
\address{\'Ecole polytechnique, CMLS\\
	Route de Saclay, 91128 Palaiseau Cedex, France.}
\email{martin.leguil@polytechnique.edu}

\author{Federico Rodriguez Hertz}
\thanks{The third author was partially supported by the NSF grant DMS-1900778.}
\address{Department of Mathematics, The Pennsylvania State University\\ 
University Park, PA 16802, USA.}
\email{hertz@math.psu.edu}

\maketitle

\selectlanguage{english}
\begin{abstract} 
		We consider two transitive $3$-dimensional Anosov flows which do not preserve volume and which are continuously conjugate to each other. Then, disregarding certain exceptional cases, such as flows with $C^1$ regular stable or unstable distributions, we prove that either the conjugacy is smooth or it sends the positive SRB measure of the first flow to the negative SRB measure of the second flow and vice versa. We give a number of corollaries of this result. In particular, we establish local rigidity on a $C^1$-open $C^\infty$-dense subspace of transitive Anosov flows; we improve the classical de la Llave-Marco-Moriy\'on rigidity theorem for dissipative Anosov diffeomorphisms on the $2$-torus by merely assuming matching of (full) Jacobian data at periodic points; we also exhibit the first evidence that the Teichm\"uller space of smooth conjugacy classes of Anosov diffeomorphisms on the $2$-torus is well-stratified according to regularity.
\end{abstract}
 
\tableofcontents

\section{Overview}\label{ssection_une}

Let us recall that a diffeomorphism $f\colon M\to M$ on some closed smooth Riemannian manifold $M$ is called {\it Anosov} if the tangent bundle admits a $Df$-invariant splitting $TM=E^s\oplus E^u$, where  $E^s$ is uniformly contracting and $E^u$ is uniformly expanding under $Df$.  
Similarly, a flow $X^t\colon M\to M$ is called {\it Anosov} if the tangent bundle admits a $DX^t$-invariant splitting $TM=E^s\oplus \R X\oplus E^u$, where $X$ is the generating vector field of $X^t$, $E^s$ is uniformly contracting and $E^u$ is uniformly expanding under $DX^t$, $t>0$. The stable and unstable bundles $E^s$, $E^u$ integrate uniquely to $X^t$-invariant stable and unstable foliations $\mathcal{W}^s$ and $\mathcal{W}^u$, respectively. 
%We let $\bar \chi_u>\chi_u>0$ and $\bar \chi_s>\chi_s>0$ such that  we have
%\begin{equation}\label{growth subble}
%e^{-\bar \chi_s}<\|DX^1|_{E^s}\|<e^{-\chi_s},\quad e^{-\bar \chi_u}<\|DX^{-1}|_{E^u}\|<e^{-\chi_u}.
%\end{equation}
%for some choice of the Riemannian metric.

%and set $\chi_0:=\min (\xi_u,\xi_s)$, $\chi_*:=\max(\bar\xi_u,\bar\xi_s)$. 

The well-known examples of Anosov flows are geodesic flows on negatively curved Riemannian manifolds and suspension flows of Anosov diffeomorphisms. Many more examples of $3$-dimensional Anosov flows can be constructed by various surgery techniques, especially in dimension $3$ (see~\cite{Bar} for a survey).

Let $X^t\colon M \to M$ be an Anosov flow, and let $Y^t$ be a $C^1$-small perturbation of $X^t$. By Anosov's structural stability, these two flows are orbit-equivalent, that is, there exists a homeomorphism $\Phi \colon M\to M$ which sends orbits of $X^t$ to orbits of $Y^t$. It is well-known that such orbit equivalence usually cannot be improved to a conjugacy since the difference of periods of corresponding periodic orbits provide obstructions. It is a well-known corollary of the Livshits Theorem~\cite{livshits} that matching of all periods for a pair of transitive Anosov flows is a necessary and sufficient condition for the existence of a continuous  conjugacy.

In a different direction, one can ask if the orbit equivalence can be improved to be smooth. Similarly, this cannot be expected and obstructions are given by the eigenvalues of Poincar\'e return maps at corresponding periodic orbits. In fact, matching of eigenvalue obstructions implies that the orbit equivalence can be chosen to be smooth (we include a proof, which follows the strategy of de la Llave~\cite{dllSRB} and  Pollicott~\cite{Pol} of this fact in Appendix~\ref{appa}).

Both the periods and eigenvalues obstructions discussed above are very natural moduli. In this paper we explore a less obvious connection between periods and regularity of the conjugacy. 

Given a transitive Anosov flow $X^t\colon M\to M$ on a compact manifold $M$, we say that $X^t$ is \emph{conservative}, or \emph{volume preserving}, if it leaves invariant a smooth measure. Otherwise, we will say that $X^t$ is \emph{dissipative}. 

In the setting of volume preserving $3$-dimensional Anosov flows the first and the last authors proved that a continuous conjugacy is necessarily smooth unless both flows are constant roof suspensions~\cite{GRH}. Thus, in this paper, we focus on $3$-dimensional dissipative Anosov flows and carry out a similar program as well as explore some applications.

\begin{theoalph}\label{strong dllave flow}
	Let $X^t\colon M\to M$ and  $Y^t\colon N\to N$ be two transitive dissipative $C^\infty$ Anosov flows on $3$-manifolds $M$ and $N$. Assume that they are $C^0$-conjugate by a homeomorphism $\Phi\colon M \to N$, $\Phi \circ X^t=Y^t \circ \Phi$. 
	Assume that for any periodic  point $p=X^T(p)$ Jacobians match, i.e.,
	\begin{equation*}
		\det DX^T(p)= \det DY^T(\Phi(p)).
	\end{equation*} %Assume that they are dissipative, i.e., they do not preserve a volume form.  
	Then, %at least one of the following statements holds:
	%\begin{itemize}
		%\item $X^t$ and $Y^t$ are constant roof suspension flows;
		%\item 
		$X^t$ and $Y^t$ are $C^{\infty}$-conjugate. 
	%\end{itemize}
\end{theoalph}

\begin{coralph}\label{strong dllave}
	Let $f,g\colon \T^2 \to\T^2$ be two dissipative $C^\infty$ Anosov diffeomorphisms that are $C^0$-conjugate by a homeomorphism $h\colon \T^2 \to \T^2$, $h\circ f=g \circ h$. Assume that for any periodic  point $p=f^n(p)$ Jacobians match, i.e.,
	\begin{equation*}
		\det Df^n(p)= \det Dg^n(h(p)).
	\end{equation*}
	Then, $f$ and $g$ are $C^\infty$-conjugate. 
\end{coralph}

\begin{remark}\label{remarque_reg_basse}
Theorem~\ref{strong dllave flow} was stated in the $C^\infty$ category, but it also works for $C^r$ Anosov flows with $r \geq 4$, which are $(r-1)$-pinched in the sense of Definition~\ref{defi d pinched} below. In that case, the conjugacy $\Phi$ is going to be $C^{r_*}$, where%\footnote{Here $(r-1)+\mathrm{Lip}$ means that $\Phi$ is of class $C^{r-1}$, and its differential of order $(r-1)$ is Lipschitz continuous.}
\begin{equation}\label{def_r_etoile}
r_*=r,\text{ if }r\notin \N,\text{ and }r_*=(r-1)+\mathrm{Lip}, \text{ if }r\in \N.
\end{equation} 
Accordingly, Corollary~\ref{strong dllave} works for $C^r$ Anosov diffeomorphisms, $r \geq 5$, which are $(r-1)$-pinched.
\end{remark}

We also show that we can drop the condition on Jacobians for generic transitive Anosov flows on $3$-manifolds and obtain local rigidity. %which are $C^0$-conjugate are  still smoothly conjugate, unless the conjugacy swaps positive and negative SRB measures of the two flows.
\begin{theoalph}\label{coro_D}
	Let $M$ be a $3$-manifold such that the space $\mathcal{A}$ of $C^\infty$ vector fields on $M$ which generate transitive Anosov flows is non-empty.  Then, there exists a $C^1$-open and $C^\infty$-dense subset $\mathcal{U}\subset \mathcal{A}$ such that for any  $X\in \mathcal{U}$, the Anosov flow $X^t$ generated by $X$ is locally rigid, i.e., if $Y^t$ is an Anosov flow whose generator $Y$ is sufficiently $C^1$-close to $X$, then we have:
	$$
	X^t\text{ and }Y^t\text{ are }C^0\text{-conjugate}\quad\Leftrightarrow\quad X^t\text{ and }Y^t\text{ are }C^\infty\text{-conjugate}.
	$$
\end{theoalph}

In fact, the $C^1$-open and $C^\infty$-dense property underlying Theorem~\ref{coro_D} stems from the lack of $C^1$-smoothness of both the stable and unstable foliations; see Proposition~\ref{prop_generic}.

Note that we always assume that the $3$-manifolds $M$ and $N$ are homeomorphic, hence, diffeomorphic. Therefore we could have assumed from the beginning that both flows live on the same smooth manifold, however, we find that it is better to use distinct notation for conceptual reasons. In fact, in higher dimensions this distinction becomes more important because there exist homeomorphic but non-diffeomorphic manifolds which support conjugate Anosov flows. Accordingly, in higher dimension, one could only conclude that the manifolds are diffeomorphic if rigidity is established.

Finally we have the following application on stratification by finite regularity conjugacy classes of Anosov diffeomorphisms in dimension 2.

\begin{coralph}
	There exists a $C^{1+%\alpha}$, %
	\textup{H\"older}}$
%$\alpha>0$, 
Anosov diffeomorphism $g\colon \mathbb{T}^2\to \mathbb{T}^2$ which is not $C^1$ conjugate to any $C^3$ diffeomorphism. 
\end{coralph}

\subsection*{Acknowledgements} The first author was supported by the Simons Fellowship during the 2024-25 academic year when the bulk of this paper was done. The first author is grateful for excellent working conditions provided by IHES and by the Mathematics Department at Universit\'e Paris-Saclay and especially to Sylvain Crovisier for his hospitality.

The authors are very grateful to Mihajlo Ceki\'{c} and Gabriel Paternain for explaining to us the very interesting quasi-Fuchsian example where SRB measures are swapped by the conjugacy. We also thank James Marshall Reber for his feedback on early drafts of Section~\ref{sec pos livs}.

\section{Detailed statements of results}\label{ssection_deux}

\subsection{Rigidity problem for Anosov flows and prior results}

Consider two transitive Anosov flows $X^t\colon M \to M$ and $Y^t\colon N \to N$ which are $C^0$-conjugate, that is, assume that there is a homeomorphism $\Phi\colon M \to N$ such that 
\begin{equation}
\label{eq_conj}
	\Phi \circ X^t=Y^t \circ \Phi.
\end{equation}
When does the conjugacy $\Phi$ have $C^1$ or better regularity? Although~\eqref{eq_conj} ensures that the orbits of $X^t$ and $Y^t$ have the same behavior from a topological point of view, a continuous conjugacy does not retain all the information about the flows. For instance, if $\Phi$ is not $C^1$, then it may send the physical (SRB) measure to some singular measure which is not physical, or it may change the Hausdorff dimension of some relevant invariant set. Put succinctly: the flows may have different statistical behavior.  
%In the following, we will say that $X^t$ and $Y^t$ are $C^r$-conjugate, $r \geq 1$, if $\Phi$ in~\eqref{eq_conj} can be chosen of class $C^r$. 

Let us recall the basic fact that if $X^t$ and $Y^t$ are transitive, the conjugacy in~\eqref{eq_conj}  is essentially unique within a given orbit equivalence class. This means that any other conjugacy map $\tilde \Phi$ in the same class (which means that $\Phi^{-1} \circ \tilde  \Phi \colon M \to M$ fixes each $X^t$-orbit) has the form $\tilde \Phi=\Phi\circ X^\tau$, for some constant time $\tau \in \mathbb{R}$. Indeed, pick a point $x \in M$ with a dense orbit, then  $\tilde \Phi(x)=\Phi(X^\tau(x))$ for some $\tau \in \mathbb{R}$. Therefore $\tilde \Phi|_{\{X^t(x)\}_{t \in \mathbb{R}}}=\Phi\circ X^\tau|_{\{X^t(x)\}_{t \in \mathbb{R}}}$, by~\eqref{eq_conj}; in fact, the same relation holds for all points, since  the orbit $\{X^t(x)\}_{t \in \mathbb{R}}$ is dense, and both $\tilde \Phi$, $\Phi$ are continuous. Because of this observation, it makes sense to speak about regularity of the conjugacy without specifying the conjugacy map we choose in a given orbit equivalence class. 

Basic obstructions to $\Phi$ being $C^1$ are again carried by periodic orbits, namely, by their multipliers in period. Indeed, if the conjugacy $\Phi$ is $C^1$, then at any periodic point $p=X^T(p)$, %say of period $T>0$, 
we can differentiate the conjugacy equation~\eqref{eq_conj}  and obtain
$$
D\Phi(p)\circ DX^T(p)\circ (D\Phi(p))^{-1}=DY^T(\Phi(p)).
$$
In particular, a necessary condition for $\Phi$ to be $C^1$ is that
\begin{equation}
	\label{eq_eigen}
\forall p=X^T(p) \implies DX^T(p)\text{ and }DY^T(\Phi(p))\text{ have the same eigenvalues.} 
\end{equation} 
It is a well-known classical result that~\eqref{eq_eigen} is, actually, a complete set of moduli if $\dim M=\dim N=3$:
\begin{theorem}[De la Llave-Moriy\'on~\cite{dllMor, dllSRB}, Pollicott~\cite{pollicott1988} ]\label{delallave}
	Let $X^t$, $Y^t$ be two $C^r$, $r \in(1,\infty]\cup\{\omega\}$, transitive Anosov flows on $3$-manifolds which are  continuously conjugate as in~\eqref{eq_conj} and satisfy the assumption~\eqref{eq_eigen}. Then the conjugacy $\Phi$ is $C^{r_*}$ regular, with $r_*>1$ as in~\eqref{def_r_etoile}. 
	\end{theorem}
%Regularity $C^{r_*}$ is defined in the following way
%	$$
%	r_*=r,\text{ if }r\notin \N,\text{ and }r_*=(r-1)+\mathrm{Lip}, \text{ if }r\in \N.
%	$$
%Here $(r-1)+\mathrm{Lip}$ means that $\Phi$ is of class $C^{r-1}$, and its differential of order $(r-1)$ is Lipschitz continuous.

\subsection{Main technical result}

Let $X^t\colon M\to M$ be an Anosov flow on some Riemannian $3$-manifold $M$. 
For $*=s,u$, we define
$$
\lambda^*_x (t):=\|DX^t|_{E^*(x)}\|,\quad \forall\, (x,t) \in M\times\mathbb{R}.
$$
%and let $\mathrm{Jac}_x(t):=\lambda_x^s(t)\lambda_x^u(t)$.
 By compactness, there exists an integer $n \geq 1$ such that $\lambda_x^s(n)<1<\lambda_x^u(n)$, for all $x \in M$.  
%\begin{definition}[Mildly dissipative flows]
%	We say that the Anosov flow $X^t$ is \emph{mildly dissipative} if for any $x\in M$, it holds
%	$$
%	\lambda_x^s(1)^2\lambda_x^u(1)<1,\quad\text{and}\quad \lambda_x^u(1)^2\lambda_x^s(1)>1.
%	$$
%\end{definition}
%In particular, conservative Anosov flows are mildly dissipative.  

\begin{definition}[$k$-pinching]\label{defi d pinched}
	Given $k>1$, we say that the Anosov flow $X^t$ is $k$-\emph{pinched} if there exists an $n\ge 1$ such that for all $x\in M$,
	$$
	\lambda_x^s(n)^k\lambda_x^u(n)<1,\quad\text{and}\quad \lambda_x^u(n)^k\lambda_x^s(n)>1.
	$$
	%We will say that the Anosov flow $X^t$ is \emph{mildly dissipative} if it is $2$-pinched, e.g., if $X^t$ is volume-preserving. 
\end{definition}
To ease the notation,  from now on we will always assume that the integer $n$ appearing in the pinching conditions above can be chosen to be equal to one --- $n=1$. 

Note that given an arbitrary $3$-dimensional Anosov flow $X^t$ then there always exists some $k>1$ such that $X^t$ is $k$-pinched. 
%\begin{definition}[Mildly dissipative flows]\label{def_MD} 
When the Anosov flow $X^t$ is dissipative and $\varrho$-pinched for some $\varrho\in (1,2]$, we say that $X^t$ is  $\varrho$-{\emph{mildly dissipative}}. %with $\varrho\in(1,2]$ if it is dissipative and has the $\varrho$-pinching property: there exists an $n\ge 1$ such that for all $x$
%	$$
%	\lambda_x^s(n)^\varrho\lambda_x^u(n)<1,\quad\text{and}\quad \lambda_x^u(n)^\varrho\lambda_x^s(n)>1.
%	$$
%\end{definition}

Let us recall the definition of SRB measures.
\begin{definition}
	\label{def_SRB}
	Let $X^t\colon M\to M$ be a transitive Anosov flow on a compact manifold~$M.$ The positive SRB measure is the unique Borel invariant probability measure $m_X^+$ whose conditionals along the leaves of the unstable foliation are absolutely continuous with respect to the volume measure on these leaves. The negative SRB measure $m_X^-$ of $X^t$ is defined by the same property along the stable leaves.
\end{definition}

\begin{remark}
	Both measures $m_X^+$ and $m_X^-$ are ergodic. Moreover, by the work of Gurevich-Oseledets~\cite{GurevichOseledets} (in the diffeomorphism case), and Livshits-Sinai~\cite{LivSin}, a transitive Anosov flow $X^t\colon M\to M$ is conservative if and only if $m_X^+=m_X^-$. 
\end{remark}

Our main result is the following:
\begin{theoalph}\label{theo alphc}
	Let $X^t$, $Y^t$ be two $3$-dimensional transitive $C^r$, $r \geq 3$, Anosov flows that are $k$-pinched with $1<k\leq r-1$, and which are $C^0$-conjugate by a homeomorphism $\Phi$ as in~\eqref{eq_conj}. %Assume that they are dissipative, i.e., they do not preserve a volume form.  
	Then, at least one of the following statements holds:
	\begin{enumerate}
		\item\label{cas prem ter} the flows $X^t$ and $Y^t$ are $C^{r_*}$-conjugate;
		\item\label{cas prem} $\Phi$ swaps SRB measures of the two flows, i.e., $\Phi_* m_X^+=m_Y^-$ and $\Phi_* m_X^-=m_Y^+$;%the conjugacy $\Phi$ sends the positive (respectively, negative) SRB measure of $X^t$ to the negative (respectively positive) SRB measure of $Y^t$;
		\item\label{cas prem sec} at least one of the foliations $\mathcal{W}_X^s$, $\mathcal{W}_X^u$, $\mathcal{W}_Y^s$ or $\mathcal{W}_Y^u$ is of class $C^{1+\alpha}$ for some $\alpha\in(0,1)$.

	\end{enumerate}
\end{theoalph}
In fact, case~\ref{cas prem sec} could alternatively be stated in terms of the $C^{1+\alpha}$ regularity of the stable and unstable distributions. 

%Let us now give a variant of Theorem~\ref{strong dllave flow}, assuming matching of stable multipliers. 
%\begin{theoalph}\label{theo alphh}
%	Let $X^t\colon M\to M$ and  $Y^t\colon N\to N$ be two transitive dissipative $C^\infty$ Anosov flows on $3$-manifolds $M$ and $N$. Assume that they are $C^0$-conjugate by a homeomorphism $\Phi\colon M \to N$, $\Phi \circ X^t=Y^t \circ \Phi$. 
%	Assume that for any periodic  point $p=X^T(p)$ stable Jacobians match, i.e.,
%	\begin{equation}\label{stab match}
%		\lambda_p^s(X^t,T)= \lambda_{\Phi(p)}^s(Y^t,T). 
%	\end{equation} %Assume that they are dissipative, i.e., they do not preserve a volume form.  
%	Then, at least one of the following statements holds:
%	\begin{itemize}
%	\item $X^t$ and $Y^t$ are constant roof suspension flows;
%	\item 
%	$X^t$ and $Y^t$ are $C^{\infty}$-conjugate. 
%	\end{itemize}
%\end{theoalph}
 
\begin{addalph}\label{add alph}
In the case~\ref{cas prem sec} above we have more information. Assume, for the sake of concreteness, that $\mathcal W_X^u$ is $C^1$. Then one of the following holds:
\begin{enumerate}
\item both flows are constant roof suspensions and all foliations $\mathcal{W}_X^s$, $\mathcal{W}_X^u$, $\mathcal{W}_Y^s$ and $\mathcal{W}_Y^u$ are of class $C^{1}$;
\item the conjugacy $\Phi$ is $C^r$ along stable leaves and $\mathcal W_Y^u$ is also $C^1$;
\item the stable foliation $\mathcal{W}_Y^s$ is $C^1$;
\item $\Phi_*m_X^-=m_Y^+$. 
\end{enumerate}
In particular, if $X^t$ has one $C^1$ foliation, then so does the conjugate flow $Y^t$, unless, possibly, when $\Phi_*m_X^-=m_Y^+$. 

Moreover, if we additionally assume that both flows are $C^4$ regular and $\frac{5}{4}$-mildly dissipative then, in fact, one of the following holds:
\begin{enumerate}
\item both flows are constant roof suspensions and all foliations $\mathcal{W}_X^s$, $\mathcal{W}_X^u$, $\mathcal{W}_Y^s$ and $\mathcal{W}_Y^u$ are of class $C^{1}$;
\item the conjugacy $\Phi$ is $C^r$ along  stable leaves and $\mathcal W_Y^u$ is also $C^1$;
\item $\Phi_*m_X^-=m_Y^+$ and $\mathcal W_Y^s$ is $C^1$.
\end{enumerate}
In particular, in the $\frac54$-mildly dissipative case, we have that at least one of the foliations $\mathcal{W}_X^s$ or  $\mathcal{W}_X^u$ is $C^1$ and at least one of the foliations $\mathcal{W}_Y^s$ or $\mathcal{W}_Y^u$ is also $C^1$, improving case~\ref{cas prem sec} in Theorem~\ref{theo alphc}.
\end{addalph} 
%\begin{proof}
%	Idea:~\eqref{stab match} implies that $\Phi$ is smooth along stable leaves. Since $X^t$ and $Y^t$ are not constant roof suspension flows, none of them can have two $C^1$ foliations; we have two cases:
%	\begin{itemize}
%		\item if $\mathcal{W}_X^s$ and $\mathcal{W}_Y^u$ are $C^1$, then  $\mathcal{W}_X^u$ and $\mathcal{W}_Y^s$ are not $C^1$ and we can match unstable multipliers of $X^t$ with stable multipliers of $Y^t$, which implies that $Y^t$ is volume-preserving, a contradiction; similarly, 
%		\item 
%	\end{itemize}
%\end{proof}

In fact, our proof of the Addendum also works under $\frac{1+\sqrt 3}{2}$-mildly dissipative assumption and with some more effort should possible to establish for 2-mildly dissipative Anosov flows. We also believe that for $C^\infty$ flows, by working with higher order asymptotic formulae it should be possible to obtain a stronger version of the addendum. Specifically, it should be possible to remove the mildly dissipative assumption altogether and still obtain the latter trichotomy in the addendum. 

\begin{remark}
	Let us discuss the regularity requirements, specifically $C^r$ for $r \geq 3$, imposed on the flow $X^t$ in our results. A key ingredient in our proof is the asymptotic formulae for certain periods, which are derived in Proposition~\ref{coro zxp per}. These formulae involve second-order partial derivatives of the hitting times of the flow, which are computed in specific normal coordinates (see, e.g., Proposition~\ref{propo o good}). Since these coordinates are merely $C^{r-1}$, the flow $X^t$ must be $C^r$ with $r \geq 3$ in Theorem~\ref{theo alphc}.
	Using alternative techniques, we believe that it might be possible to relax the regularity requirement, allowing the flow $X^t$ to be merely $C^r$ for $r > 2$.
	
	The $C^4$ requirement in the finite regularity version of Theorem~\ref{strong dllave flow}, as stated in Remark~\ref{remarque_reg_basse}, arises from the use of a more precise second-order asymptotic formula in the proof (see Proposition~\ref{prop mild dis exp}). This formula requires the hitting times to be $C^{2+\mathrm{Lip}}$ in the normal coordinates.
	Similarly, the $C^5$ requirement in the finite regularity version of Corollary~\ref{strong dllave}, as noted in Remark~\ref{remarque_reg_basse}, stems from suspending the diffeomorphism by its Jacobian (of class $C^4$), yielding a $C^4$ flow to which we then apply the previous result. 
\end{remark}

\subsection{Generic rigidity, local rigidity and rigidity from matching of abstract potentials}

We also show that generic transitive Anosov flows on $3$-manifolds which are $C^0$-conjugate are smoothly conjugate, unless the conjugacy swaps positive and negative SRB measures of the two flows:
\begin{theoalph}\label{theoremC}
	Let $M$ and $N$ be $3$-manifolds which support transitive Anosov flows. Let $\mathcal{A}_M$ and $\mathcal{A}_N$ denote the spaces of $C^\infty$ vector fields on $M$ and $N$, respectively, which generate transitive Anosov flows. Then there exist $C^1$-open and $C^\infty$-dense subsets $\mathcal{U}_M\subset \mathcal{A}_M$ and $\mathcal{U}_N\subset \mathcal{A}_N$\footnote{Actually, the complements $\mathcal{A}_M\setminus\mathcal{U}_M$ and $\mathcal{A}_N\setminus\mathcal{U}_N$ have infinite codimension.} such that, for any $X \in\mathcal{U}_M$ and $Y \in\mathcal{U}_N$, the associated Anosov flows $X^t$ and $Y^t$ satisfy:
	\begin{equation*}%\label{no_cun_feuillet}
	\text{none of the four foliations }\mathcal{W}_X^s,\mathcal{W}_X^u,\mathcal{W}_Y^s,\mathcal{W}_Y^u\text{ is }C^1.
	\end{equation*}
	In particular, if $X^t$ and $Y^t$
	are $C^0$-conjugate via $\Phi\colon M \to N$, $\Phi \circ X^t=Y^t \circ \Phi$,  at least one of  the following holds:
	\begin{enumerate}
			\item the conjugacy $\Phi$ is a  $C^{\infty}$ regular;
		\item $\Phi_* m_X^+=m_Y^-$ and $\Phi_* m_X^-=m_Y^+$. %$\Phi$ sends the positive SRB measure of $X^t$ to the negative SRB measure of $Y^t$, and $\Phi$ sends the negative SRB measure of $X^t$ to the positive SRB measure of $Y^t$;
		
	\end{enumerate}
\end{theoalph}

\begin{remark} If both flows are assumed to be $\frac54$-mildly dissipative, as in the Addendum~\ref{add alph}, then, in fact one can drop the genericity assumption on $Y$ and only assume that $X$ is generic.
\end{remark}

Finally, we give variants of Theorem~\ref{strong dllave flow}  and Corollary~\ref{strong dllave}, where  we replace Jacobian data with abstract smooth potentials data.

\begin{theoalph}\label{theorem_H}
		Let $X^t\colon M\to M$ and  $Y^t\colon N\to N$ be two transitive dissipative $C^\infty$ Anosov flows on $3$-manifolds $M$ and $N$. Assume that they are orbit equivalent via a homeomorphism $\Phi\colon M \to N$, and that there exists periodic  point $p=X^T(p)$ such that
	\begin{equation*}
		\log | \det DX^T(p)|\log |\det DY^{T'}(\Phi(p))|>0
	\end{equation*} 
	where $T'$ is the period of $\Phi(p)$.
	
		There exist open and dense sets $\mathcal{U}_X,\mathcal{U}_Y\subset C^\infty(M,\R^+)$ such that, if we can find $\varphi\in\mathcal{U}_f$, $\xi\in \mathcal{U}_g$ with the property that for any periodic  point $p=X^T(p)$ and the corresponding periodic point $\Phi(p)=Y^{T'}(\Phi(p))$, 
	\begin{equation*}
	\int_0^T\varphi(X^t(p))\, dt=		\int_0^{T'}\xi(Y^t(\Phi(p)))\, dt,
	\end{equation*}
	then 
	$X^t$ and $Y^t$ are $C^\infty$ orbit equivalent. 
\end{theoalph}

\begin{coralph} \label{cor_I}
	Let $f,g\colon \T^2 \to \T^2$ be two $C^\infty$ dissipative Anosov diffeomorphisms that are $C^0$-conjugate by a homeomorphism $h\colon \T^2 \to \T^2$,
	$h\circ f=g \circ h$. Assume that there is a periodic point $p=f^n(p)$ such that
	\begin{equation*}
		\log |\det Df^n(p)|\log |\det Dg^n(h(p))|> 0.
	\end{equation*}
	There exist open and dense sets $\mathcal{U}_f,\mathcal{U}_g\subset C^\infty(\T^2,\R)$ such that, if we can find $\varphi\in\mathcal{U}_f$, $\xi\in \mathcal{U}_g$ with the property that for any periodic  point $p=f^n(p)$,  
	\begin{equation*}
		%\forall\, p \in \mathbb{T}^2,\, f^n(p)=p \implies 
		\sum_{\ell=0}^{n-1} \varphi(f^\ell(p))=\sum_{\ell=0}^{n-1} \xi(g^\ell(h(p))),
	\end{equation*}
	then %at least one of the following statements holds:
	%\begin{enumerate}
		%\item $\varphi,\psi$ are cohomologous to a constant, i.e.,  $\varphi=\gamma\circ f-\gamma+c$ for some continuous function $\gamma\colon \T^2 \to \R$, and some $c \in \R$, and similarly for $\psi$;
		the conjugacy $h$ is $C^\infty$. 
	%\end{enumerate}
\end{coralph}

\subsection{Stratification of the Teichm\"uller space of smooth conjugacy classes according to regularity}

We proceed to explain an application to the structure of smooth conjugacy classes of Anosov diffeomorphisms.
Let $\cX^r(M)$ be the space of Anosov diffeomorphisms on $M$ of some finite regularity $r\ge 1$. Denote by $\Diff^r(M)$ the space of $C^r$ diffeomorphisms of $M$ which are homotopic to the identity. Then $\Diff^r(M)$  acts on $\cX^r(M)$ by conjugation. Usually this action is either free or virtually free. Then one can form the Teichm\"uller space by taking the quotient 
$$
\cT^r(M)=\cX^r(M)/\Diff^r(M).
$$
There is little general understanding of topology and structure of these spaces, but when $\dim M=2$ at least we know that these spaces are Hausdorff.
For any $r> s\ge 1$ we have the stratification map of Teichm\"uller spaces
$
\pi_{r\to s}\colon \cT^{r}(M)\to\cT^s(M)
$
which is induced by inclusion and quotient maps 
$$
\cT^{r}(M)=\cX^{r}(M)/\Diff^{r}(M)\subset \cX^{s}(M)/\Diff^{r}(M)\to  \cX^{s}(M)/\Diff^{s}(M)=\cT^s(M).
$$
It is interesting to understand basic properties of maps $\pi_{r\to s}$. One expects that maps $\pi_{r\to s}$ are injective, at least away from exceptional loci. Injectivity of $\pi_{r\to s}$ is equivalent to a well-studied {\it bootstrap problem} in hyperbolic dynamics. It is was established for $2$-dimensional Anosov diffeomorphisms~\cite{LMM, dllSRB} as well as for one-dimensional~\cite{SS} and higher dimensional expanding maps~\cite{GRHexp} (with some caveats).

Surprisingly, to the best of our knowledge, surjectivity of maps $\pi_{r\to s}$ was not considered in the literature. We provide evidence that, in general, one should not expect $\pi_{r\to s}$ to be surjective. Specifically, using Theorem~\ref{theo alphc} and Cawley's realization theorem~\cite{cawley}, we obtain the following result.

\begin{coralph}\label{cor_cawley}
	The map $\pi_{3\to 1}\colon \cT^{3}(\T^2)\to\cT^1(\T^2)$ is not surjective; that is, there exist $C^1$ Anosov diffeomorphism of the $2$-torus which is not $C^1$ conjugate to any $C^3$ Anosov diffeomorphism.
\end{coralph}

The proof and some further discussion are given in Section~\ref{sec_cawley}. 

A similar stratification can be considered for the Teichm\"uller space of Riemannian metrics on a given manifold, where one identifies two metrics if they are isometric via an isometry which is homotopic to identity. Local description of such Teichm\"uller spaces were given by Ebin~\cite{ebin}

We note that the analogous question about surjectivity in this setting is well-understood~\cite{DK} and the Teichm\"uller space of Riemannian metrics is known to be well-stratified according to regularity. For example, one can realize a particular curvature function of some finite regularity $C^r$ by some Riemannian metric, in which case this metric cannot be more than $C^{r+2}$ regular. In fact, locally on a Riemannian manifold there is a natural harmonic system of coordinates with respect to which the metric is ``as smooth as it can be''. For example, $(1+3x|x|)(dx^2+dy^2)$ is a $C^{1+\textup{Lip}}$ metric and cannot be made any better by coordinate changes. 

One can also hold an analogous discussion for the Teichm\"uller space of expanding maps $f$ on a given manifold. We expect this space to be well-stratified as well. At least in dimension one, we can similarly introduce harmonic coordinates by requiring the  eigenfunction for the Perron-Frobenius operator with respect to  $-\log \mathrm{Jac}(f)$ to be the constant function. This harmonic coordinate system makes the expanding map “as
smooth as it can be'', implying that the Teichm\"uller space of one dimensional expanding maps is stratified according to regularity. 
Specifically, given some target regularity $C^r$, $r>1$ consider any diffeomorphism
$T\colon \mathbb{S}^1\to \mathbb{S}^1$ which is $C^r$ but not $C^{r+\varepsilon}$  and which is not $C^0$ conjugate to $x\mapsto x+\frac{1}{d}$. Then define a $C^r$ expanding map $f\colon \mathbb{S}^1\to \mathbb{S}^1$ of degree $d$ 
by 
$$
f(x)=\sum_{\ell=0}^{d-1}T^\ell(x).
$$
Then $f$ is a covering map whose Deck group $\mathbb Z/d\mathbb Z$ is generated by $T$. Using this observation one can calculate that $1$ is the eigenfunction of the transfer operator $\mathcal L_{-\log f',f}\colon C^{r-1}(\mathbb{S}^1 )\to C^{r-1}(\mathbb{S}^1)$. According to~\cite[Corollary 3.3]{GRHexp} the  expanding map $f$ has optimal smoothness in its $C^1$ conjugacy class. Indeed, if $f$ were $C^1$ conjugate to $\bar f$ which is $C^{r+\varepsilon}$ then the transfer operator $\mathcal L_{-\log \bar f',\bar f}\colon C^{r-1+\varepsilon}(\mathbb{S}^1 )\to C^{r-1+\varepsilon}(\mathbb{S}^1)$ has a $C^{r-1+\varepsilon}$ eigenfunction $u$ corresponding to the maximal eigenvalue (see, e.g.,~\cite[Theorem 3.1]{GRHexp}). By integrating $u$ we can $C^{r+\varepsilon}$-conjugate $\bar f$ to $\hat f$ which is still $C^{r+\varepsilon}$  and preserves the Lebesgue measure. Now we have that $f$ and $\hat f$ are $C^1$ conjugate and both preserve the Lebesgue measure. It follows that this $C^1$ conjugacy is a rotation of $\mathbb{S}^1$. Hence, $f$ is also $C^{r+\varepsilon}$, contradicting our construction of $f$.

\subsection{Further questions}

Let us list a few questions raised by our results. 

\begin{question}\label{question_deux_sept}
	Given $r > 2$, let $X^t$ and $Y^t$ be $3$-dimensional transitive $C^r$ Anosov flows that are $k$-pinched as in Theorem~\ref{theo alphc} and $C^0$ conjugate by a homeomorphism $\Phi$. Similarly to Theorem~\ref{theoremC}, can one say that case~\ref{cas prem} in Theorem~\ref{theo alphc}, where $\Phi$ swaps the SRB measures of $X^t$ and $Y^t$, is exceptional? 
\end{question}
	Indeed, the existence of such a pair $(Y^t,\Phi)$ can be regarded as a symmetry of the flow $X^t$, and we expect a typical Anosov flow to have no such symmetry. In other words, is it true that for a generic $3$-dimensional dissipative $C^r$ Anosov flow $X^t$, there is no $C^r$ Anosov flow $Y^t$ that is $C^0$-conjugate to $X^t$ by some map $\Phi$ which swaps the SRB measures of the flows?  
\begin{question} 
	Furthermore, can one classify exceptional transitive Anosov flows $X^t$ for which a pair $(Y^t,\Phi)$ as in Question~\ref{question_deux_sept} exists? Must $X^t$ be time-reversible, namely, $C^0$-conjugate to the inverse flow $X^{-t}$? Does one always have that at least one of the foliations $\mathcal{W}_X^s$ or $\mathcal{W}_X^u$ is $C^1$, or even more regular?
\end{question} 

\begin{question}
	Recall that according to Theorem~\ref{theo alphc} the case where  either foliation $\mathcal{W}^s$ or $\mathcal{W}^u$ is $C^1$ (but not both) is problematic for the rigidity of $X^t$.  On the other hand, by Theorem~\ref{theoremC}, for an open and dense set of $3$-dimensional transitive Anosov flows, both foliations are not $C^1$. Can one go further and classify such exceptional Anosov flows with $C^1$ strong stable foliation? 
\end{question}

%C1 foliations classifiable? Does SRB swap imply C1 foliations
\begin{question}
In Section~\ref{sect_hitt_tim}, we  introduce a twisted cocycle which is a generalization of the longitudinal Anosov cocycle introduced by Foulon and Hasselblatt for volume preserving Anosov flows in dimension $3$~\cite{FH}. Foulon and Hasselblatt showed that, in the volume preserving case, this cocycle is a coboundary if and only if the distribution $E^s\oplus E^u$ is smooth; this fact was utilized in the preceding  work~\cite{GRH} on the rigidity of $3$-dimensional volume preserving Anosov flows by the first and last author. While in the dissipative case, we  observe in Remark~\ref{remark_fh_un}  that if $E^s$ or $E^u$ is $C^1$, then the twisted counterpart to the cocycle introduced by Foulon and Hasselblatt has to be a coboundary. However, currently we lack understanding of the implications of this cocycle being a coboundary.  
\end{question}

\subsection{Structure of the paper}

In the next section we setup notation and recall various known results which will be used later. We also introduce stable (unstable) templates which can be thought of as the ``time coordinate'' of stable (unstable) distribution restricted to a local unstable (stable) manifold. Most importantly, we prove that the stable (unstable) distribution is $C^1$ regular if and only if the stable (unstable) template has polynomial form for volume expanding (contracting) periodic points.

In Section~\ref{ssection_quatre} we setup a shadowing scheme associated to a homoclinic orbit of a dissipative periodic orbit and prove an asymptotic formula for the periods of shadowing periodic orbits. The leading exponentially small term in this formula will be the main driver for the arguments in the paper. We also derive a version of this formula for a conservative periodic orbit and a more precise two-term formula for a mildly dissipative periodic orbit.

In Section~\ref{sec pos livs}  we recall the positive proportion Livshits Theorem and discuss various related auxiliary results, such as density of a positive proportion set of periodic orbits and full proportion property of volume contacting periodic points with respect to the SRB measure. We also identify equilibrium states with respect to which mildly dissipative periodic points have full proportion.

In Section~\ref{ssection_six} we still consider a single Anosov flow and present our main dichotomy result, namely, that either eigenvalue data at periodic points can be recovered from the periods or the stable (or unstable) foliation is $C^1$ regular. Then in Section~\ref{ssection_sept} we put together all prior ingredients and use this dichotomy to establish the rigidity results about conjugate flows which were stated in Sections~\ref{ssection_une} and~\ref{ssection_deux}.

In Section~\ref{sec_cawley} we give a proof of Corollary~\ref{cor_cawley} which has a different flavor and relies on Cawley's realization result~\cite{cawley}.

Finally, in Section~\ref{section_examples}, we present examples of conjugate Anosov flows which illustrate various aspects of our results.

\section{Background knowledge on adapted charts and templates}\label{ssection_trois}

\subsection{Notation}

We will begin by recalling some standard terminology and notation.

Given an Anosov flow $X^t$, it is well-known that the strong stable bundle $E^s$ and the strong unstable bundle $E^u$ integrate uniquely to invariant foliations $\mathcal{W}^s$ and $\mathcal{W}^u$, respectively. 

{\it Throughout the paper we will always assume that the manifold $M$ and the foliations $\mathcal{W}^s$ and $\mathcal{W}^s$ are oriented.} Indeed, we can do so without loss of generality because one can always pass to an appropriate finite cover so that the manifolds and foliations become orientable.
We will also use the notation $\mathcal{W}_X^*$, $*=s,u$, to emphasize the dependence of these foliations on the flow, when several flows are involved.  Further, we denote by $\mathcal{W}_X^{cs}$ and $\mathcal{W}_X^{cu}$ the weak foliations with $2$-dimensional leaves obtained by flowing the leaves of $\mathcal{W}^{s}$ and $\mathcal{W}^{u}$, respectively. For each $x \in M$, and $*=s,u, cs, cu$, we denote by $\mathcal{W}^*(x)$ the leaf of $\mathcal{W}^*$ containing $x$. We denote by $d_{\mathcal{W}^*}$ the distance along the leaves of $\mathcal{W}^*$ induced by the restriction of the Riemannian metric to $T\mathcal{W}^*$. Finally, for $x \in M$ and $\delta>0$ we let $\mathcal{W}_\delta^*(x):=\{y \in \mathcal{W}^*(x): d_{\mathcal{W}^*}(x,y)<\delta\}$ be the $\delta$-neighbourhood of $x$ within $\mathcal{W}^*(x)$. In the following, we will use the notation $\mathcal{W}_{\mathrm{loc}}^*(x)$ to denote some local leaf $\mathcal{W}_\delta^*(x)$, for $\delta>0$ of order $1$.%i.e., for $*=s,u$,
%$$
%\mathcal{W}^{c*}(x):=\{X^t(y):y \in \mathcal{W}^*(x)\},\quad \forall\, x \in M.
%$$

We denote by $\mathcal{P}$ the set of periodic orbits of the flow $X^t$. Given a periodic orbit $\gamma \in \mathcal{P}$, and any point $p$  in $\gamma$, we denote by $T=T(\gamma)=T(p)>0$ its period, and let 
$$
\mu_\gamma=\mu_p:=\lambda_p^s(T)\in (0,1),\quad \lambda_\gamma=\lambda_p:=\lambda_p^u(T)>1.
$$
be the eigenvalues of the linearized Poincar\'e return map at $p$.
We will also use notation $\mathrm{Jac}_p(T)$ for the Jacobian of the linearized return map,  $\mathrm{Jac}_p(T)=\mathrm{Jac}_\gamma(T):=\det DX^{T}(p)=\mu_\gamma \lambda_\gamma$.

\subsection{Adapted charts}\label{subs norm coorrd}

In the following, we fix $r \geq 3$, and consider a transitive $C^r$ Anosov flow on some $3$-manifold $M$ which is $k$-pinched for some $1 < k\leq  r-1$. Given a $k\in \mathbb{R}$, we will write $[k]$ to denote the integer ceiling $\lceil k-1\rceil$.   

The following is a standard consequence of non-stationary linearization along 1-dimensional  stable and unstable foliations.

\begin{proposition}[Katok-Lewis~\cite{KL}]%[Linearization along $\mathcal{W}^{s/u}$]
	\label{norm foms}
	For $*=s,u$, there exists a continuous family of $C^r$ charts $\{\Phi_x^*\colon T_x \mathcal{W}^*(x)\to \mathcal{W}^*(x)\}_{x\in M}$ such that for any $x\in M$, and for any time $\sigma \in \mathbb{R}$, 
	\begin{enumerate}
		\item $\Phi_x^*(0)=x$, and $D\Phi_x^*(0)=\mathrm{Id}$;
		\item $X^\sigma(\Phi_x^*(\xi))=\Phi_{X^\sigma(x)}^*(\lambda_x^*(\sigma) \xi)$, for any $\xi \in \mathbb{R}$. 
		%\item if $y \in \mathcal{W}^*(x)$, then  $(\Phi_y^*)^{-1}\circ \Phi_x^*$ is an affine map.
	\end{enumerate}
\end{proposition}

The non-stationary linearization along stable and unstable foliations is a standard tool in hyperbolic dynamics and $\Phi^*_x$, $*=s,u$, are given by integrating properly normalized densities of the SRB measure on the stable and unstable leaves.

Now we recall the construction of adapted charts for $3$-dimensional Anosov flows, due to Tsujii~\cite{Tsujii} in the volume preserving case, and Tsujii-Zhang~\cite{TsujiiZhang} in the general case. This construction was later extended to the partially hyperbolic setting by Eskin-Potrie-Zhang~\cite{EPZ}.

%In the following, we recall the construction of adapted charts for three-dimensional Anosov flows, due to Tsujii~\cite{Tsujii} in the volume preserving case, and Tsujii-Zhang~\cite{TsujiiZhang} in the general case; the construction was later extended to the partially hyperbolic setting, by Eskin-Potrie-Zhang~\cite{EPZ}. 
 
\begin{proposition}[Adapted charts~\cite{TsujiiZhang}]\label{propo o good}
	Given a $k$-pinched Anosov flow, there exists a continuous family of  uniformly $C^{r-1}$ charts $\{\imath_x\colon (-1,1)^3\to M\}_{x \in M}$ such that for any $x \in M$, and any time $\sigma \in \mathbb{R}$, we have:
	\begin{enumerate}
		\item\label{normal un} $\imath_x(\xi,0,0)=\Phi_x^s(\xi)$, for any $\xi \in (-1,1)$; 
		\item\label{normal deux} $\imath_x(0,0,\eta)=\Phi_x^u(\eta)$, for any $\eta \in (-1,1)$; 
		\item\label{flw dir} $\imath_x(\xi,t,\eta)=X^t(\imath_x(\xi,0,\eta))$, for any $(\xi,t,\eta)\in (-1,1)^3$;
		%\item $\imath
		%\marginpar{change degree to k-1}
		\item\label{pt cocyc nf} let $F_x^\sigma:=(\imath_{X^\sigma(x)})^{-1}\circ X^\sigma\circ \imath_x=(F_{x,1}^\sigma,F_{x,2}^\sigma,F_{x,3}^\sigma)$; then, $F_{x,2}^\sigma(\xi,t,\eta)=t+\psi^\sigma(\xi,\eta)$, and there exist polynomials $P_x^*(\sigma)(z)=\sum_{\ell=1}^{[k]} \alpha_x^{*,\ell}(\sigma) z^\ell$, $*=s,u$, of degree at most $[k]$, such that for $\xi\in (-\lambda_x^s(\sigma)^{-1},\lambda_x^s(\sigma)^{-1})$ and $\eta \in (-\lambda_x^u(\sigma)^{-1},\lambda_x^u(\sigma)^{-1})$,   
		\begin{align*}
			%C_x(\eta):=
			\begin{bmatrix}
				\partial_1 F_{x,1}^\sigma & \partial_2 F_{x,1}^\sigma\\
				\partial_1 F_{x,2}^\sigma & \partial_2 F_{x,2}^\sigma
			\end{bmatrix}(0,0,\eta)&=\begin{bmatrix}
				\lambda_x^s(\sigma) & 0\\
				P_x^s(\sigma)(\eta) & 1
			\end{bmatrix},\\
			\begin{bmatrix}
				\partial_2 F_{x,2}^\sigma & \partial_3 F_{x,2}^\sigma\\
				\partial_2 F_{x,3}^\sigma & \partial_3 F_{x,3}^\sigma
			\end{bmatrix}(\xi,0,0)&=\begin{bmatrix}
				1 & P_x^u(\sigma)(\xi)\\
				0 & \lambda_x^u(\sigma)
			\end{bmatrix},
		\end{align*}
		where $\lambda_x^s(\sigma):=\|DX^\sigma(x)|_{E^s}\|$, $\lambda_x^u(\sigma):=\|DX^\sigma(x)|_{E^u}\|$.
	\end{enumerate}
\end{proposition}
%We refer to Appendix~\ref{appb} for more details on the construction. 
We present the construction of adapted charts in Appendix~\ref{appb}. It is essentially the same construction as the one given in~\cite{Tsujii, TsujiiZhang}, but we took special care and constructed $C^{r-1}$ regular charts (other than $C^{r-2}$ regular constructed in~\cite{TsujiiZhang}) and also extended the normalization of jets for all times $\sigma$ (other than for time-1 only). 

%Although the proof is very similar to the original proof~\cite{TsujiiZhang}, we include the details of the construction in Appendix~\ref{appb}; indeed, by modifying their construction, we are able to show that the charts $\{\imath_x\}_x$ can be taken $C^{r-1}$ and not only $C^{r-2}$ as claimed in their work. 
% Remark: this is the place where we need higher regularity.

\subsection{Hitting times and their jets}\label{sect_hitt_tim}

The adapted charts provide a family of local transversals to the Anosov flow. Here we setup notation and summarize basic properties of hitting (return) times relative to these transversals.
\begin{definition}
	Given a family $\{\imath_x\colon (-1,1)^3\to M\}_{x \in M}$ as above, we let 
	$$
	\Sigma_x:=\imath_x\left((-1,1)\times \{0\}\times (-1,1)\right).
	$$
	Then $\{\Sigma_x\}_{x \in M}$ is a continuous family of uniformly $C^{r-1}$ transverse sections for the flow $X^t$. By construction, for any point $x \in M$, the transversal $\Sigma_x$ contains local stable and unstable manifolds of the base point $x$, i.e., $\mathcal{W}_{\mathrm{loc}}^*(x)\subset \Sigma_x$, $*=s,u$. 
	
	For any point $x \in M$, and $\sigma\in \mathbb{R}$, there exists a neighborhood $\mathcal{U}(x,\sigma)\subset \Sigma_x$ of $x$ such that  %for any $y \in \mathcal{U}(x,t)$, 
	the Poincar\'e map $\Pi_x^\sigma\colon \mathcal{U}(x,\sigma) \to \Sigma_{X^\sigma(x)}$ is well-defined, with $\Pi_x^\sigma(x)=X^\sigma(x)$. For $y \in \mathcal{U}(x,\sigma)$, we denote by $\tau_x^\sigma(y)\in \R$ the corresponding hitting time, $\Pi_x^\sigma(y)=X^{\tau_x^\sigma(y)}(y)$. 
\end{definition}	 

	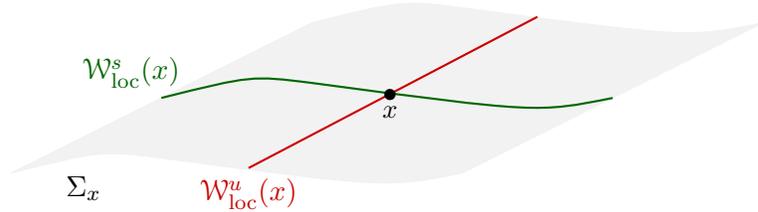
\begin{figure}[h!]
	\centering
	\begin{tikzpicture}[xscale=2, yscale=1]
		\filldraw[black!10!white, opacity=.5] (0,0)--(2,2) .. controls (2.6,2.3) and (2.6,2.3) .. (3,2.2) .. controls (4.5,1.8) and (4.5,1.8) .. (5,2)--(3,0) .. controls (2.5,-.2) and (2.5,-.2) .. (1,.2) .. controls (0.6,0.3) and (0.6,0.3) .. (0,0); 
		%\draw[red!50!orange] (5,4) node[right]{$\cW^{cu}(y)$}; 
		%\fill[black!60!white, opacity=.6] (2.98,1.57)--(3.5,1.45) to[bend right] (3.5,1.7) -- (2.98,1.55);
		\draw[thick, red!80!black] (1.58,0.068) node[below]{$\mathcal{W}^u_{\mathrm{loc}}(x)$}--(3.5,2.08);
		\draw[thick, green!40!black] (1,1) .. controls (1.6,1.3) and (1.6,1.3) .. (2,1.2) .. controls (3.5,0.8) and (3.5,0.8) .. (4,1);
		%\draw[dotted, thick, green!40!black] (1.48,1.51) .. controls (2.08,1.81) and (2.08,1.81) .. (2.48,1.71) .. controls (3.98,1.31) and (3.98,1.31) .. (4.48,1.51);
		%\draw[red!80!black, thick] (3,-1) node[right]{$\cW^u_{\loc}(y)$}--(5,3);
		%\draw[blue!70!red, thick] (3,0)--(5,2) node[right]{$H^s_{x,y}(\cW^u_{\loc}(x))$};
		%\draw[dotted, thick] (3.8,1.4)--(3.8,1.8);
		%\draw[dotted, thick, green!40!black] (1.48,1.51) .. controls (2.08,1.81) and (2.08,1.81) .. (2.48,1.71) .. controls (3.98,1.31) and (3.98,1.31) .. (4.48,1.51);
		%\draw[thick, green!40!black] (2.98,1.55) .. controls (3.8,1.8) and (3.8,1.8) .. (5,2.5);
		%\draw[green!40!black] (4.5,2.3) node[above]{$\mathcal{W}^s_{\mathrm{loc}}(q)$};
		\draw (2.52,1.02) node[below]{\small $x$}; 
		\draw (2.52,1.04) node{$\bullet$};
		%\draw (3.8,1.38) node{$\bullet$};
		%\draw (3.8,1.82) node{$\bullet$};
		%\draw (3.05,1.78) node[left]{\small $q$}; 
		%\draw (3,1.55) node{$\bullet$};
		\draw[green!40!black] (0.8,1) node[above]{$\mathcal{W}^s_{\mathrm{loc}}(x)$};
		%\draw[green!40!black] (1.8,1.7) node[above]{$W^s_{\varepsilon}(q)$};
		%\draw (4,.8) node[right]{\small $y$};
		%\draw (4,1) node{$\bullet$};
		%\draw (2.98,1.55) to[bend right] (3,1.6); 
		%\draw (3.2,1.15) node[right]{$\alpha_p^s(q)$};
		\draw (0.3,-0.2) node[right]{$\Sigma_x$};
	\end{tikzpicture}
	\caption{Transveral $\Sigma_x$.}
\end{figure}
	
	With a slight abuse of notation, we will also denote by $\tau_x^\sigma(\xi,\eta)$ the hitting time in $(\xi,\eta)$-coordinates, i.e., if $y:=\imath_x(\xi,0,\eta)\in \mathcal{U}(x,\sigma)$, we let $\tau_x^\sigma(\xi,\eta):=\tau_x^\sigma(y)$. 
	By construction of our charts, $\tau_x^\sigma(\cdot)$ is a $C^{r-1}$ function which is constant when restricted to the local strong stable and unstable manifolds of $x$, i.e., for $\sigma \geq 0$, 
	\begin{align*}
		\tau_x^\sigma(\cdot,0)|_{(-1,1)}&\equiv \sigma,& \tau_x^\sigma(0,\cdot)|_{(-\lambda_x^u(\sigma)^{-1},\lambda_x^u(\sigma)^{-1})}&\equiv \sigma.
		%\tau_x(0,\cdot,-t)|_{(-1,1)}&\equiv -t,& \tau_x(\cdot,0,-t)|_{(-\lambda_x^s(-t)^{-1},\lambda_x^s(-t)^{-1})}&\equiv -t.
	\end{align*}

\begin{lemma}\label{lemme lien}
	For any time $\sigma\in \mathbb{R}$, any $\xi\in (-\lambda_x^s(\sigma)^{-1},\lambda_x^s(\sigma)^{-1})$, and any $\eta \in (-\lambda_x^u(\sigma)^{-1},\lambda_x^u(\sigma)^{-1})$, we have 
	\begin{equation}\label{der temps} \partial_{2}\tau_x^\sigma(\xi,0)=-P_x^u(\sigma)(\xi),\quad	\partial_1 \tau_x^\sigma(0,\eta)=-P_x^s(\sigma)(\eta),
	\end{equation}
	where $P_x^u(\sigma)(\xi),P_x^s(\sigma)(\eta)$ are the polynomials from Proposition~\ref{propo o good}\eqref{pt cocyc nf}. %; in particular,  
%	$$
%	\partial_{12}\tau_x^1(0,0)=-\alpha_x,
%	$$
%where $\alpha_x=\alpha_x^{s,1}=\alpha_x^{u,1}$ is the first coefficient of $P_x^u(\xi)$,  $P_x^s(\eta)$, %from Proposition~\ref{propo o good}\eqref{pt cocyc nf},
% and  for $j\in \{2,\cdots,k-1\}$, 
%\begin{equation*}
%	\frac{1}{j!}\partial_1^j\partial_{2}\tau_x^1(0,0)=-\alpha_x^{u,j},\quad \frac{1}{j!}\partial_{1}\partial_2^j\tau_x^1(0,0)=-\alpha_x^{s,j}. 
%\end{equation*}
\end{lemma}

\begin{proof}
	Let us prove it for $\sigma \geq 0$ and $\eta \in (-\lambda_x^u(\sigma)^{-1},\lambda_x^u(\sigma)^{-1})$, the other cases being analogous. By using the normal forms given by Proposition~\ref{propo o good}\eqref{pt cocyc nf}, we have for sufficiently small $|\xi|\ll 1$,
	\begin{align*}
	F_x^\sigma(\xi,0,\eta)&=F_x^\sigma(0,0,\eta)+\partial_1 F_x^\sigma(0,0,\eta) \xi +O(\xi^2)\\
	&=(0,0,\lambda_x^u(\sigma)\eta)+(\lambda_x^s(\sigma), P_x^s(\sigma)(\eta),\gamma^\sigma(\xi,\eta))  \xi+O(\xi^2),
	\end{align*}
	for some $\gamma^\sigma(\xi,\eta) \in \mathbb{R}$. 
	Therefore, by Proposition~\ref{propo o good}\eqref{flw dir},  
	$$
	X^{-P_x^s(\sigma)(\eta)\xi}\circ \imath_{X^\sigma(x)}(F_x^\sigma(\xi,0,\eta))=\imath_{X^\sigma(x)}(\lambda_x^s(\sigma)\xi,0 ,\lambda_x^u(\sigma)\eta+\gamma^\sigma(\xi,\eta) \xi)+O(\xi^2),
	$$
	with $\imath_{X^\sigma(x)}(\lambda_x^s(\sigma)\xi,0 ,\lambda_x^u(\sigma)\eta+\gamma^\sigma(\xi,\eta)\xi)\in \Sigma_{X^\sigma(x)}$. Hence
	$$
	\tau_x^\sigma(\xi,\eta)=1-P_x^s(\sigma)(\eta) \xi+O(\xi^2),
	$$
and
	we deduce that $\partial_1 \tau_x^\sigma(0,\eta)=-P_x^s(\sigma)(\eta)$, as claimed. 
\end{proof}

\begin{remark} Although for a fixed $\sigma$ the hitting time $\tau_x^\sigma$ is formally only $C^{r-1}$,
	as a direct consequence of Lemma~\ref{lemme lien}, we have  that $\partial_{2}\tau_x^\sigma(\cdot,0)$ and $\partial_1 \tau_x^\sigma(0,\cdot)$ are actually polynomial functions; in particular,  
	%	$$
	%	\partial_{12}\tau_x^\sigma(0,0)=-\alpha_x,
	%	$$
	 %from Proposition~\ref{propo o good}\eqref{pt cocyc nf},
	 for $j\in \{1,\cdots,[k]\}$, 
	\begin{equation}\label{derivee_taus}
		\frac{1}{j!}\partial_1^j\partial_{2}\tau_x^\sigma(0,0)=-\alpha_x^{u,j}(\sigma),\quad \frac{1}{j!}\partial_{1}\partial_2^j\tau_x^\sigma(0,0)=-\alpha_x^{s,j}(\sigma),
	\end{equation}
	where $\alpha_x^{u,j}(\sigma)$, $\alpha_x^{s,j}(\sigma)$ are the $j$-th coefficients of $P_x^u(\sigma)(\xi)$ and $P_x^s(\sigma)(\eta)$, respectively. Moreover, if the flow $X^t$ is $C^3$, by Schwarz lemma, we have $\alpha_x^{u,1}=\alpha_x^{s,1}$. 
\end{remark}

\begin{remark}\label{remark twist coc}
	%More generally, for any $(x,t)\in M \times \mathbb{R}$, we denote 
	%$$
	%\alpha_x(t):= -\partial_{12}\tau_x^t(0,0),
	%$$ 
	%so that $\alpha_x=\alpha_x(1)$. 
	By the chain rule, the map $(x,\sigma)\mapsto \alpha_x^{s,j}(\sigma)$ is a twisted  cocycle with twisting given by the multiplicative cocycle $(x ,\sigma)\mapsto  \lambda_x^s(\sigma)^j \lambda_x^u(\sigma)$, i.e., for any point $x \in M$, and times $\sigma,\sigma'\in \mathbb{R}$, we have 
	\begin{equation}\label{cocycl_alphaj}
	\alpha_x^{s,j}(\sigma+\sigma')=\alpha_x^{s,j}(\sigma)+\lambda_x^s(\sigma) \lambda_x^u(\sigma)^j\alpha_{X^\sigma(x)}^{s,j}(\sigma'). 
	\end{equation}
	Indeed, for $|\xi|\ll 1$, with the notation of Proposition~\ref{propo o good}, we have the following additivity by the definition of hitting times
	$$
	\tau_x^{\sigma+\sigma'}(\xi,\eta)=\tau_{x}^{\sigma}(\xi,\eta)+\tau_{X^\sigma(x)}^{\sigma'}(F_{x,1}^\sigma(\xi,0,\eta),F_{x,3}^\sigma(\xi,0,\eta)).
	$$
	Recall that
	$
	F_{x,1}^\sigma(\xi,0,\eta)=\lambda_x^s(\sigma)\xi+o(\xi)$,  $F_{x,3}^\sigma(\xi,0,\eta)=\lambda_x^u(\sigma)\eta+o(\xi)$.
	Hence, by differentiating the above equation we have
	$$
	\partial_1\tau_x^{\sigma+\sigma'}(0,\eta)=\partial_1\tau_{x}^{\sigma}(0,\eta)+\lambda_x^s(\sigma)\partial_1\tau_{X^\sigma(x)}^{\sigma'}(0,\lambda_x^u(\sigma)\eta),
	$$
	which yields  
	~\eqref{cocycl_alphaj} by differentiating $j$ times with respect to the second component at $\eta=0$,  by~\eqref{derivee_taus}. 
\end{remark}
	
	\begin{remark}
	The twisted cohomology class of $\alpha_x^{s,j}$ is independent of our choice of $\{\Sigma_x\}_{x \in M}$ given by adapted charts. Indeed, if $\{\tilde\Sigma_x\}_{x \in M}$ is another continuous family of $C^{r-1}$ transversals such that for $x \in M$,  $\Sigma_x\supset\mathcal{W}_{\mathrm{loc}}^*(x)$, $*=s,u$, then for any $x \in M$, there exists a neighborhood $U_x\subset (-1,1)^2$ of $(0,0)$, $\varepsilon_x>0$, and a $C^{r-1}$ function $u_x\colon U_x \to \mathbb{R}$ such that 
	$$
	\tilde \Sigma_x\cap B(x,\varepsilon_x)=\{\imath_x\left(\xi,u_x(\xi,\eta),\eta\right):(\xi,\eta)\in U_x\}.
	$$  
	Therefore, denoting by $\{\tilde\tau_x^\sigma\}_{x,\sigma}$ the hitting times for the family $\{\tilde\Sigma_x\}_{x \in M}$, the corresponding cocycle $M \times \mathbb{R}\ni (x,\sigma)\mapsto \tilde \alpha_x^{s,j}(\sigma):=-\frac{1}{j!}\partial_{1}^j \partial_2 \tilde\tau_x^\sigma(0,0)$ differs from $\alpha_x^{s,j}$ by a twisted coboundary:
	$$
	\tilde\alpha_x^{s,j}(\sigma)-\alpha_x^{s,j}(\sigma)=\partial_{12}u_x(0,0)-\lambda_x^s(\sigma) \lambda_x^u(\sigma)^j\partial_{12}u_{X^\sigma(x)}(0,0). 
	$$
	
	Similarly,  for $j\in \{1,\cdots,[k]\}$, the map $(x,\sigma)\mapsto \alpha_x^{u,j}(\sigma)$ is a twisted  cocycle with twisting given by the multiplicative cocycle $(x ,\sigma)\mapsto \lambda_x^s(\sigma)\lambda_x^u(\sigma)^j$. 
	
	In particular, the families of twisted cocycles $\{\alpha_x^{s,j}\}_{j=1,\cdots,[k]}$ and $\{\alpha_x^{u,j}\}_{j=1,\cdots,[k]}$ can be regarded as  a generalization of the longitudinal Anosov cocycle introduced by Foulon and Hasselblatt for volume preserving Anosov flows in dimension $3$~\cite{FH}. 
	%However, in this paper we do not need to use any properties of this cocycle.  
\end{remark}
 %By the above lemma, it holds $\alpha_x=-\partial_{12}\tau_x^\Sigma(1)$, hence we will extend this notation by setting 
%\begin{equation}\label{der temp bis}
%	\alpha_x(t)=-\int_0^t\partial_{12} \mathfrak{t}_{X^s(x)}^\Sigma\,  \mathrm{Jac}_x(s)\, ds,\quad \forall\, t \in \mathbb{R}. 
%\end{equation}

%\begin{remark}
%	Note that, for $\xi \in (-1,1)$, $\eta \in (-1,1)$, 
%	\begin{align*}
%		\tau_x(\xi+\delta \xi,\delta \eta,t)&=-\alpha_x(t)\xi \delta \eta+  O(\delta \xi\delta \eta),\\
%		\tau_x(\delta \xi,\eta +\delta \eta,t)&=-\alpha_x(t)\eta \delta \xi+  O(\delta \xi\delta \eta).
%	\end{align*}  
%\end{remark}

\subsection{Templates and regularity of $E^s$ and $E^u$}\label{section templates}

Following the terminology of Tsujii-Zhang~\cite{TsujiiZhang}, we consider a family of functions called \emph{templates} which are given by the first jets of local strong stable and unstable manifolds through points on coordinate axes in the adapted charts. Specifically templates are defined in the following way.

\begin{definition}[Templates]
	For any $\xi,\eta\in (-1,1)$, we write: 
	\begin{align}
		\imath_x^{-1}\big( \mathcal{W}_{\mathrm{loc}}^s(\Phi_x^u(\eta))\big)&=\left\{(\tilde\xi,\mathcal{T}^s_x(\eta)\tilde\xi+b_x^s(\tilde\xi,\eta)\tilde\xi^2,\eta+c_x^s(\tilde\xi,\eta)\tilde\xi)\right\}_{\tilde\xi \in (-1,1)},\label{equation stable mnfd normal coord} \\ 
		\imath_x^{-1}\big( \mathcal{W}_{\mathrm{loc}}^u(\Phi_x^s(\xi))\big)&=\left\{(\xi+c_x^u(\xi,\tilde\eta)\tilde\eta,\mathcal{T}^u_x(\xi)\tilde\eta+b_x^u(\xi,\tilde\eta)\tilde\eta^2,\tilde\eta)\right\}_{\tilde\eta \in (-1,1)},\label{equation unstable mnfd normal coord}
	\end{align}
	where $\mathcal{T}_x^s$ and $\mathcal{T}_x^u$ are called the {\emph{stable and unstable templates}}, respectively, and  $b_x^s,b_x^u,c_x^s,c_x^u$ are some functions.  
\end{definition}

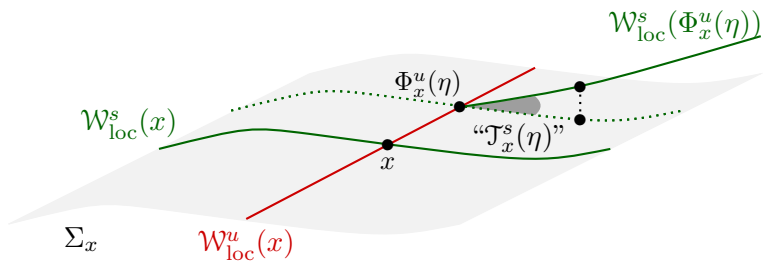
\begin{figure}[h!]
	\centering
	\begin{tikzpicture}[xscale=2, yscale=1]
		\filldraw[black!10!white, opacity=.5] (0,0)--(2,2) .. controls (2.6,2.3) and (2.6,2.3) .. (3,2.2) .. controls (4.5,1.8) and (4.5,1.8) .. (5,2)--(3,0) .. controls (2.5,-.2) and (2.5,-.2) .. (1,.2) .. controls (0.6,0.3) and (0.6,0.3) .. (0,0); 
		%\draw[red!50!orange] (5,4) node[right]{$\cW^{cu}(y)$}; 
		\fill[black!60!white, opacity=.6] (2.98,1.57)--(3.5,1.45) to[bend right] (3.5,1.7) -- (2.98,1.55);
		\draw[thick, red!80!black] (1.58,0.068) node[below]{$\mathcal{W}^u_{\mathrm{loc}}(x)$}--(3.5,2.08);
		%\draw[thick, green!40!black] (1,1) .. controls (1.6,1.3) and (1.6,1.3) .. (2,1.2) .. controls (3.5,0.8) and (3.5,0.8) .. (4,1);
		\draw[dotted, thick, green!40!black] (1.48,1.51) .. controls (2.08,1.81) and (2.08,1.81) .. (2.48,1.71) .. controls (3.98,1.31) and (3.98,1.31) .. (4.48,1.51);
		\draw[thick, green!40!black] (1,1) .. controls (1.6,1.3) and (1.6,1.3) .. (2,1.2) .. controls (3.5,0.8) and (3.5,0.8) .. (4,1);
		%\draw[red!80!black, thick] (3,-1) node[right]{$\cW^u_{\loc}(y)$}--(5,3);
		%\draw[blue!70!red, thick] (3,0)--(5,2) node[right]{$H^s_{x,y}(\cW^u_{\loc}(x))$};
		\draw[dotted, thick] (3.8,1.4)--(3.8,1.8);
		\draw[dotted, thick, green!40!black] (1.48,1.51) .. controls (2.08,1.81) and (2.08,1.81) .. (2.48,1.71) .. controls (3.98,1.31) and (3.98,1.31) .. (4.48,1.51);
		\draw[thick, green!40!black] (2.98,1.55) .. controls (3.8,1.8) and (3.8,1.8) .. (5,2.5);
		\draw[green!40!black] (4.5,2.3) node[above]{$\mathcal{W}^s_{\mathrm{loc}}(\Phi_x^u (\eta))$};
		\draw (2.52,1.02) node[below]{\small $x$}; 
		\draw (2.52,1.04) node{$\bullet$};
		\draw (3.8,1.38) node{$\bullet$};
		\draw (3.8,1.82) node{$\bullet$};
		\draw (3.08,1.87) node[left]{\small $\Phi_x^u (\eta)$}; 
		\draw (3,1.55) node{$\bullet$};
		%\draw[green!40!black] (0.8,1) node[above]{$\mathcal{W}^s_{\mathrm{loc}}(p)$};
		\draw[green!40!black] (1.8,1.7) node[above]{};%{$W^s_{\varepsilon}(q)$};
		%\draw (4,.8) node[right]{\small $y$};
		%\draw (4,1) node{$\bullet$};
		%\draw (2.98,1.55) to[bend right] (3,1.6); 
		\draw (3,1.15) node[right]{$``\mathcal{T}_x^s(\eta)"$};
		\draw[green!40!black] (0.8,1) node[above]{$\mathcal{W}^s_{\mathrm{loc}}(x)$};
		\draw (0.3,-0.2) node[right]{$\Sigma_x$};
	\end{tikzpicture}
	\caption{The template can be informally thought of as a function measuring the angle between the local stable manifold and the transversal $\Sigma_x$.}
\end{figure}

One important point is that templates control the $C^1$ smoothness of the strong  stable distribution $E^s$ along the unstable leaves.
\begin{lemma}\label{link templ es}
	There exists an $\alpha>0$ such that the family of stable templates $\{\mathcal{T}_x^s\}_{x \in M}$ is uniformly $C^{1+\alpha}$ if and only if the strong stable distribution $E^s$ is $C^{1+\alpha}$. 
\end{lemma}

\begin{proof}
	Assume that $E^s$ is $C^{1+\alpha}$. Then for any $x \in M$, with respect to an adapted chart $\imath_x$ at $x$, the restriction of $E^s$ to $\mathcal{W}^u_{\mathrm{loc}}(x)$ is given by
	\begin{equation}\label{expre stabl dist}
	E^s(0,0,\eta)=\mathbb{R} (1,\mathcal{T}^s_x(\eta),c_x^s(0,\eta)),\quad \eta \in (-1,1).
	\end{equation}
	Since the adapted charts are uniformly $C^{r-1}$, we then conclude that the family of stable templates $\{\mathcal{T}_x^s\}_{x \in M}$ is uniformly $C^{1+\alpha}$.
	
	Recall that the weak stable distribution $E^s \oplus \mathbb{R} X$ of a $3$-dimensional Anosov flow is always uniformly $C^{1+\alpha}$ for some $\alpha>0$ (see, e.g.~\cite{PS}).  Equivalently, the family of functions $\{c_x^s(0,\cdot)\}_{x \in M}$ is uniformly $C^{1+\alpha}$. 
	
	Now assume that $\{\mathcal{T}_x^s\}_{x \in M}$ are uniformly $C^{1+\alpha}$. From~\eqref{expre stabl dist} we can see that the restriction of the stable distribution to each local unstable leaf is uniformly $C^{1+\alpha}$. Since $E^s$ is also uniformly $C^{1+\alpha}$ along weak stable leaves we conclude that $E^s$ is globally $C^{1+\alpha}$. 
\end{proof}

\begin{lemma} 
	For any $x \in M$, $\sigma \in \mathbb{R}$ and $\eta \in (-\lambda_x^u(\sigma)^{-1},\lambda_x^u(\sigma)^{-1})$, %$\eta \in (-1,1)$, 
	we have 
	\begin{equation}\label{inv templ}
		P_x^s(\sigma)(\eta)=\lambda^s_x(\sigma)\mathcal{T}^s_{X^\sigma(x)}(\lambda_x^u(\sigma)\eta)-\mathcal{T}^s_{x}(\eta),
		%\lambda_x^s(-1)P_{X^{-1}(x)}^s(\lambda_x^u(-1)\eta)=\mathcal{T}^s_{x}(\eta)-\lambda_x^s(-1)\mathcal{T}^s_{X^{-1}(x)}(\lambda_x^u(-1)\eta).
		%P_x(\eta)=\lambda^s_x(1)\mathcal{T}^s_{X^1(x)}(\lambda_x^u(1)\eta)-\mathcal{T}^s_{x}(\eta).
	\end{equation}
	where $P_x^s(\sigma)(\eta)$ is the polynomial from Proposition~\ref{propo o good}\eqref{pt cocyc nf}. 
\end{lemma}

\begin{proof}
	Equation~\eqref{inv templ} follows from the invariance of the stable foliation $\mathcal{W}^s$, expressed in normal coordinates. Indeed, for $\eta \in (-\lambda_x^u(\sigma)^{-1},\lambda_x^u(\sigma)^{-1})$, the image of $(1,\mathcal{T}^s_{x}(\eta))$ by the differential $DX^\sigma$ in normal coordinates, namely
	$$
	\begin{bmatrix}
		\lambda_x^s(\sigma) & 0\\
		P_x^s(\sigma)( \eta)  & 1
	\end{bmatrix}\begin{bmatrix}
		1 \\
		\mathcal{T}^s_{x}(\eta)
	\end{bmatrix}= \begin{bmatrix}
		\lambda_x^s(\sigma) \\
		P_x^s(\sigma)(\eta)+\mathcal{T}^s_{x}(\eta)
	\end{bmatrix},
	$$
	should be proportional to $(1,\mathcal{T}^s_{X^\sigma(x)}(\lambda_x^u(\sigma)\eta))^\top$, which yields
	\begin{equation*}
		P_x^s(\sigma)(\eta)=\lambda^s_x(\sigma)\mathcal{T}^s_{X^\sigma(x)}(\lambda_x^u(\sigma)\eta)-\mathcal{T}^s_{x}(\eta).\qedhere
	\end{equation*}
\end{proof}

\begin{remark}\label{remark_fh_un}
	If the stable distribution $E^s$ is $C^{1+\alpha}$, for some $\alpha>0$,  then the twisted cocycle $(x,\sigma)\mapsto\alpha_x^{s,1}(\sigma)$ is a twisted coboundary. Indeed, by~\eqref{link templ es}, the family of stable templates $\{\mathcal{T}_x^s\}_{x \in M}$ is uniformly $C^{1+\alpha}$, and then, for any $x \in M$ and $\sigma \in \mathbb{R}$, by differentiating~\eqref{inv templ} and evaluating at $\eta=0$, we have
	$$
	\alpha_x^{s,1}(\sigma)=\lambda_x^s(\sigma)\lambda_x^u(\sigma)(\mathcal{T}_{X^\sigma(x)}^s)'(0)-(\mathcal{T}_x^s)'(0).
	$$
	In other words, similarly to the  observation of Foulon-Hasselblatt~\cite{FH}, the non-vanishing of the cohomology class of the twisted cocycle $(x,\sigma)\mapsto\alpha_x^{s,1}(\sigma)$ is an obstruction to the $C^1$ smoothness of $E^s$. 
\end{remark}

\begin{corollary}
	Let $X^t$ be volume preserving 3-dimensional Anosov flow. Then $E^s$ is $C^{1+\alpha}$, $\alpha>0$, if and only if $E^u$ is $C^{1+\alpha'}$, $\alpha'>0$, in which case $X^t$ is a contact flow or a constant roof suspension flow. 
\end{corollary}
\begin{proof}
	If $X^t$ is volume preserving, then the map $(x,\sigma)\mapsto \alpha_x^{s,1}(\sigma)$ is merely the longitudinal Anosov cocycle introduced by Foulon-Hasselblatt~\cite{FH} (see Remark~\ref{remark twist coc}). Moreover, if $E^s$ is $C^{1+\alpha}$, $\alpha>0$, then by Remark~\ref{remark_fh_un}, the cocycle $(x,\sigma)\mapsto \alpha_x^{s,1}(\sigma)$ is a  coboundary. By the work of Foulon-Hasselblatt~\cite{FH}, we deduce that $X^t$ is either a contact flow or a constant roof suspension flow. In both cases, $E^u$ is also $C^{1+\alpha'}$, for some $\alpha'>0$. 
\end{proof}

\begin{lemma}\label{lemma tem st}
	If $p\in M$ is a volume expanding periodic point, of period $T>0$, with eigenvalues $\mu=\mu_p<1<\lambda=\lambda_p$, $\mathrm{Jac}_p(T)=\mu\lambda>1$, then for any $\eta \in (-1,1)$, we have
	\begin{equation*}
	\mathcal{T}_p^s(\eta)-\lim_{\sigma\to +\infty} \lambda_p^s(-\sigma)\mathcal{T}^s_{X^{-\sigma}(p)}(\lambda_p^u(-\sigma)\eta)%\\
	%-\lim_{n\to +\infty} \mu^{-n}\mathcal{T}^s_{p}(\lambda^{-n}\eta)%&=\sum_{\ell=1}^{+\infty} \lambda_p^s(-\ell)P_{X^{-\ell}(p)}^s(\lambda_p^u(-\ell)\eta)\\
	%=&\,  \mathcal{T}_p^s(\eta)-\lim_{n\to +\infty}\mu^{-n}\mathcal{T}^s_{p}(\lambda^{-n}\eta)=
	=\sum_{j=1}^{[k]} \frac{\alpha_p^{s,j}(T)}{\mu\lambda^j-1}\eta^j. 
	\end{equation*}
	%where for $j \in \{1,\cdots,k-1\}$, we let 
	%$$
	%\alpha_p(T):=-\partial_{12}\tau_p^T(0,0),\quad 
	%\alpha_p^{s,j}(T):=-\frac{1}{j!}\partial_{1}\partial_2^j\tau_p^T(0,0). 
	%$$
%	$$
%	In particular, if $X^t$ is mildly dissipative, then for $\eta \in (-1,1)$, $P_x(\eta)=\alpha_x \eta$, and
%	\begin{equation}\label{forme stable template}
%	\mathcal{T}_p^s(\eta)-\lim_{n\to +\infty} \lambda_p^s(-n)\mathcal{T}^s_{X^{-n}(p)}(\lambda_p^u(-n)\eta)=\frac{\alpha_p(T)\eta }{\mathrm{Jac}_{p}(T)-1}.%,\qquad  
%	%\alpha_p(T):=-\int_{-T}^{0} \partial_{12} \mathfrak{t}_{X^{t}(p)}^\Sigma\,  \mathrm{Jac}_{p}(t)\, dt. 
%	\end{equation}
\end{lemma}

\begin{proof}
	%Fix $x \in M$. After a change of variables in~\eqref{inv templ}, for $\eta \in (-1,1)$ and $\sigma \geq 0$, we obtain
	%\begin{equation*}%\label{inv templ sec}
	%\lambda_x^s(-\sigma)P_{X^{-\sigma}(x)}^s(\sigma)(\lambda_x^u(-\sigma)\eta)=\mathcal{T}^s_{x}(\eta)-\lambda_x^s(-\sigma)\mathcal{T}^s_{X^{-\sigma}(x)}(\lambda_x^u(-\sigma)\eta).
	%\end{equation*} 
	%and then, for any integer $n \geq 1$,
	%$$
	%\sum_{\ell=1}^n\lambda_x^s(-\ell)P_{X^{-\ell}(x)}^s(\lambda_x^u(-\ell)\eta)=\mathcal{T}^s_{x}(\eta)-\lambda_x^s(-n)\mathcal{T}^s_{X^{-n}(x)}(\lambda_x^u(-n)\eta). 
	%$$
	Fix a periodic point $p\in M$, of period $T>0$, with $\mathrm{Jac}_p(T)>1$. After a change of variables in~\eqref{inv templ}, for $\eta \in (-1,1)$ and $\sigma \geq 0$, we obtain 
	\begin{equation}\label{sum poly}
	\mathcal{T}_p^s(\eta)-  \lambda_p^s(-\sigma)\mathcal{T}^s_{X^{-\sigma}(p)}(\lambda_p^u(-\sigma)\eta)=\lambda_p^s(-\sigma)P_{X^{-\sigma}(p)}^s(\sigma)(\lambda_p^u(-\sigma)\eta).% \lambda_p^s(-\ell)P_{X^{-\ell}(p)}^s(\lambda_p^u(-\ell)\eta). 
	\end{equation}
	%Indeed, w
	Let us show that the right hand side has a limit which is a polynomial of degree at most  $[k]$. As before, we use the following notation for coefficients: $P_x^s(\sigma)(\eta)=\sum_{j=1}^{[k]} \alpha_x^{s,j}(\sigma) \eta^\ell$. Then
	\begin{equation}\label{eq_polys_sum}
		\lambda_p^s(-\sigma)P_{X^{-\sigma}(p)}^s(\sigma)(\lambda_p^u(-\sigma)\eta)=\sum_{j=1}^{[k]} \lambda_p^s(-\sigma) \left( \lambda_p^u(-\sigma)\right)^{j} \alpha_{X^{-\sigma}(p)}^{s,j}(\sigma)\eta^j. 
		%&=\sum_{j=0}^{n-1}\sum_{\ell=1}^{[k]}. 
	\end{equation}
	
	Take $\sigma\gg 1$, and write it as $\sigma=nT+\rho$, $\rho \in [0,T)$. By~\eqref{cocycl_alphaj}, the map 
	$
	(x,t)\mapsto \alpha_x^{s,j}(t)
	$
	is a twisted cocycle, with twisting given by $(x,t)\mapsto \lambda_x^s(t)\left(\lambda_x^u(t)\right)^j$,  $j\in \{1,\cdots,[k]\}$, hence 
	\begin{equation*}
	  \alpha_{X^{-\sigma}(p)}^{s,j}(\sigma)
	=  \alpha_{X^{-\sigma}(p)}^{s,j}(\rho)+\lambda_{X^{-\rho}(p)}^s(\rho)\left(\lambda_{X^{-\rho}(p)}^u(\rho) \right)^j \alpha_p^{s,j}(T)\sum_{\ell=0}^{n-1} (\mu \lambda^j)^\ell,
	\end{equation*}
	and 
	\begin{equation*}
		\lambda_p^s(-\sigma) \left( \lambda_p^u(-\sigma)\right)^{j}\alpha_{X^{-\sigma}(p)}^{s,j}(\sigma)
		= \lambda_p^s(-\sigma) \left( \lambda_p^u(-\sigma)\right)^{j}\alpha_{X^{-\rho}(p)}^{s,j}(\rho)+ \alpha_p^{s,j}(T)\frac{1-(\mu \lambda^j)^{-n}}{\mu \lambda^j-1}. 
	\end{equation*}
	%since all $P_{X^{-\sigma}(p)}^s$ are polynomials of degree at most  $[k]$ vanishing at $0$ and with uniformly bounded coefficients,
The map $[0,T]\ni \rho\mapsto \alpha_{X^{-\rho}(p)}^{s,j}(\rho)$ is continuous, hence the family $\{\alpha_{X^{-\rho}(p)}^{s,j}(\rho)\}_{\rho \in [0,T]}$ is bounded, while $\lambda_p^s(-\sigma) \left( \lambda_p^u(-\sigma)\right)^{j}=O(\mathrm{Jac}_p(T)^{-n})$ goes to $0$ as $\sigma\to +\infty$, and so does the first term in the right hand side of the previous equation.  Since $\mu \lambda^j>1$, for $j\geq 1$, we deduce from~\eqref{eq_polys_sum} that 
	\begin{equation*}
	\lim_{\sigma\to +\infty}	\lambda_p^s(-\sigma)P_{X^{-\sigma}(p)}^s(\sigma)(\lambda_p^u(-\sigma)\eta)=\sum_{j=1}^{[k]} \frac{\alpha_p^{s,j}(T)}{\mu\lambda^j-1}\eta^j. 
	\end{equation*}
	We conclude that the two sides in~\eqref{sum poly} have a limit as $\sigma \to +\infty$, which is equal to the polynomial $\sum_{j=1}^{[k]} \frac{\alpha_p^{s,j}(T)}{\mu\lambda^j-1}\eta^j$. 
\end{proof}

\begin{lemma}\label{lemm pol tem}
	Let $p \in M$ be a volume expanding periodic point, of period $T>0$, with $\mu=\mu_p<1<\lambda=\lambda_p$, $\mathrm{Jac}_p(T)=\mu\lambda>1$. Let $\tilde P_p^s$ be the polynomial given by Lemma~\ref{lemma tem st}:
	$$
	\tilde P_p^s(\eta):=\sum_{j=1}^{[k]} \frac{\alpha_p^{s,j}(T)}{\mu\lambda^j-1}\eta^j. %\sum_{j=1}^k\tilde\alpha_p^j \eta^j,%\quad \tilde\alpha_p^j:=-\frac{\partial_1 \partial_2^j \tau_p^T(0,0) }{\mu \lambda^j-1},\, j\in \{1,\cdots,k\}. 
	$$
	Then, the stable distribution $E^s$ is $C^{1+\alpha}$  along $\mathcal{W}_{\mathrm{loc}}^u(p)$ for some $\alpha>0$ if and only if 
	\begin{equation}\label{template matches pol}
	\mathcal{T}_p^s=\tilde P_p^s. 
	\end{equation}
	Moreover, if $E^s$ is $C^\beta$ along the unstable manifold $\mathcal{W}_{\mathrm{loc}}^u(p)$ for some $\beta >  -\frac{\log \mu}{\log\lambda}$,\footnote{By Pugh-Shub-Wilkinson~\cite{PSW}, we also know that for any $\theta<  -\frac{\log \mu}{\log\lambda}$, $E^s$ is always $C^\theta$ along $\mathcal{W}_{\mathrm{loc}}^u(p)$.} then it is automatically $C^{1+\alpha}$ along $\mathcal{W}_{\mathrm{loc}}^u(p)$, $\alpha>0$ (note that $ -\frac{\log \mu}{\log\lambda}\in (0,1)$). 
\end{lemma}

\begin{proof}
	If~\eqref{template matches pol} is satisfied, then  $\mathcal{T}_p^s$ is obviously $C^{1+\alpha}$, $\alpha>0$. Conversely, let us assume that there exists $\alpha>0$ such that $\mathcal{T}_p^s$ is $C^{1+\alpha}$. After possibly taking $\alpha>0$ smaller, by reasoning as in the proof of Lemma~\ref{link templ es}, we deduce that $E^s$ is $C^{1+\alpha}$ along $\mathcal{W}_{\mathrm{loc}}^u(p)$. Since the sections $\{\Sigma_{X^t(p)}\}_{t \in [0,T]}$ are uniformly $C^{r-1}$, and since $E^s(X^t(p))=DX^t(p)E^s(p)$, for $t \in [0,T]$,
	we deduce that the family of templates $\{\mathcal{T}_{X^t(p)}^s\}_{t \in [0,T]}$ is uniformly $C^{1+\alpha}$. Then, according to Lemma~\ref{lemma tem st}, for any $\eta \in (-1,1)$, 
	\begin{align*}
	\tilde P_p^s(\eta)&=\mathcal{T}_p^s(\eta)-\lim_{n\to +\infty} \lambda_p^s(-n)\mathcal{T}^s_{X^{-n}(p)}(\lambda_p^u(-n)\eta)\\ %-\lim_{n\to +\infty} \mu^{-n}\mathcal{T}^s_{p}(\lambda^{-n}\eta)\\
	&=\mathcal{T}_p^s(\eta)-\lim_{n \to +\infty} O
	%\lim_{n\to +\infty}
	\left( (\mu\lambda)^{-\frac n T}\right)%(\mathcal{T}^s_{p})'(0)\eta+o((\mu\lambda)^{-n}) \right)
	=\mathcal{T}_p^s(\eta). 
	\end{align*}
	Finally, if we assume that $E^s$ is $C^\beta$ for some $\beta > -\frac{\log \mu}{\log\lambda}$. Then, 
	$$
	\lambda_p^s(-n)\mathcal{T}^s_{X^{-n}(p)}(\lambda_p^u(-n)\eta)=O\left(\mu^{-\frac n T}\lambda^{-\beta\frac n T}\right)=o(1).
	$$
	As before, we conclude that $\mathcal{T}_p^s=\tilde P_p^s$, and that $E^s$ is $C^{1+\alpha}$ along $\mathcal{W}_{\mathrm{loc}}^u(p)$, for some $\alpha>0$. 
\end{proof}

\begin{lemma}\label{descartes}
	 The strong stable distribution $E^s$ is $C^{1+\alpha}$ for some $\alpha>0$ if and only if there exists a dense set of volume expanding periodic points $p\in M$, $\mathrm{Jac}_{p}(T(p))>1$, such that $\mathcal{T}_{p}^s=\tilde P_{p}^s$. 
\end{lemma}

\begin{proof}
	The direct implication follows immediately from Lemma~\ref{lemm pol tem}. Conversely, let us assume that there exists a dense set $\mathcal{S}\subset M$ of periodic points $p$ such that $\mathcal{T}_{p}^s=\tilde P_{p}^s$. 
	
	%$C^0$ uniformly bounded, say, by some constant $K>0$. 
	For each $(x,\eta) \in M\times (-1,1)$, the stable template $\mathcal{T}_x^s(\eta)$  essentially measures the angle between the stable space $E^s$ at $\Phi_x^u(\eta)$ and the transversal $\Sigma_x$. Because the map $(-1,1)\ni (x,\eta)\mapsto \Phi_x^u(\eta)$ is continuous, the stable space $E^s(y)$ depends continuously on the point $y\in M$, and  %along $\mathcal{W}_{\mathrm{loc}}^u(x)$, while 
	$\{\Sigma_x\}_{x \in M}$ is a continuous family of uniformly $C^{r-1}$ transverse sections, we conclude that the map 
	$$
	\mathcal{T}^s\colon M \times (-1,1)\ni (x,\eta) \mapsto \mathcal{T}_x^s(\eta)
	$$ 
	is continuous.   
	
	Let $I:= \left[-\frac{1}{2},\frac{1}{2}\right]$. We let $(C^0(I,\mathbb{R}),\|\cdot\|_{C^0})$ be the Banach space of continuous functions on $I$, %endowed with the norm $\|\cdot\|_{C^0}$, 
	where $\|\varphi\|_{C^0}:=\sup_{\eta \in I} |\varphi(\eta)|$, for $\varphi \in C^0(I,\mathbb{R})$. %Similarly, we let $(C^r(I,\mathbb{R}),\|\cdot\|_{C^r})$ be the Banach space of $C^r$ functions on $I$, %endowed with the norm $$, 
	%where $\|\varphi\|_{C^r}:=\sum_{\ell=0}^r\|\varphi^{(\ell)}\|_{C^0}$, for $\varphi \in C^r(I,\mathbb{R})$. 
	
	By Lemma~\ref{lemma tem st}, for each periodic point $p \in M$ with $\mathrm{Jac}_p(T(p))>1$, we have that $\tilde P_p^s$ is a polynomial of degree at most $k-1$. Let $\mathbb{R}_{[k]}^I\subset C^{r-1}(I,\mathbb{R})$ be the space of polynomial functions of degree at most $[k]$ on the interval $I$.  It is a finite dimensional vector space, hence all norms on this space are equivalent. For any $Q\colon \eta \mapsto \sum_{\ell=0}^{[k]} q_\ell \eta^\ell \in \mathbb{R}_{[k]}^I$, let $\|Q\|_{\ell^\infty}:=\max_{\ell\in \{0,\cdots,[k]\}}|q_\ell|$. %The norms $\|\cdot\|_{C^0}$ and $\|\cdot\|_{C^1}$ induce norms on $\mathbb{R}_k[X]$, by considering the polynomial functions on $I$ induced by elements in $\mathbb{R}_k[X]$; by a slight abuse of notation, we keep the same notation for these norms.  
	In particular,  % to For any $Q=\sum_{\ell=0}^k q_\ell X^\ell \in \mathbb{R}_k[X]$, let us denote $\|Q\|_{\ell^\infty}:=\max_{\ell\in \{0,\cdots,k\}}|q_\ell|$,   $\|Q\|_{C^0}:=\sup_{\eta \in [-\frac{1}{2},\frac 12]}|Q(\eta)|$, and $\|Q\|_{C^1}:=\|Q\|_{C^0}+\|Q'\|_{C^0}$. Then, 
	there exists a constant $C>0$ such that for any $Q \in \mathbb{R}_{[k]}^I$, $\|Q\|_{\ell^\infty}\leq C \|Q\|_{C^0}$. % $\|\cdot\|_{\ell^\infty}\leq C\|\cdot\|_{C^0}$. In particular, for any $n \in \mathbb{N}$, 
	Therefore, for any periodic points $p,q \in \mathcal{S}$ in the dense set $\mathcal{S}$, we have $\mathcal{T}_*^s=\tilde P_*^s\in \mathbb{R}_{[k]}^I$, for $*=p,q$, %as it coincides with the polynomial function $\tilde P_*^s$,
	 hence
	\begin{equation}\label{equiv normes}
	\|\mathcal{T}_{p}^s-\mathcal{T}_q^s\|_{\ell^\infty}\leq C \|\mathcal{T}_{p}^s-\mathcal{T}_q^s\|_{C^0}. %\|\tilde P_{p_n}^s\|_{C^0}= C \|\mathcal{T}_{p_n}^s\|_{C^0}\leq CK,
	\end{equation}
	Moreover, the map $\mathcal{T}^s$ introduced above is continuous, hence it is uniformly continuous when restricted to the compact set $M \times I$. In particular, $\sup_{x \in M}\|\mathcal{T}_x^s\|_{C^0}\leq K$ for some $K>0$, and for any $\varepsilon>0$, there exists $\eta>0$ such that  
	$$
	|x-y|\leq \eta \implies \|\mathcal{T}_{x}^s-\mathcal{T}_y^s\|_{C^0} \leq \varepsilon.
%	 \|\mathcal{T}_{p_n}^s\|_{C^1}\leq (1+k(k+1)C)K.
	$$ 
	By~\eqref{equiv normes}, we deduce that 
	$$
	\forall\, p,q\in \mathcal{S},\, |p-q|\leq \eta \implies \|\mathcal{T}_{p}^s-\mathcal{T}_q^s\|_{\ell^\infty}\leq C \varepsilon. 
	$$
	By the uniform Cauchy criterion for the finite dimensional vector space $(\mathbb{R}_{[k]}^I,\|\cdot\|_{\ell^\infty})$, and since the set $\mathcal{S}$ is dense in $M$, we conclude that 
	the family $\{\mathcal{T}_{x}^s\}_{x \in M}$ is a family of polynomial functions of degree at most $[k]$, whose coefficients depend continuously on the point $x$. In particular, the family $\{\mathcal{T}_{x}^s\}_{x \in M}$ is uniformly $C^{r-1}$, and 
	%Indeed, for each point $x \in M$, we can take a sequence $(p_n)_{n \in \mathbb{N}}\in \mathcal{S}^\mathbb{N}$ of periodic points in $\mathcal{S}$ such that $\lim_{n \to+\infty} p_n=x$. Let $\varepsilon,\eta>0$ be as above, and let $n_0\in \mathbb{N}$ such that $m,n \geq n_0\implies |p_n-p_m|\leq \eta$. By~\eqref{equiv normes},  for $m,n \geq n_0$, we then have $\|\mathcal{T}_{p_n}^s-\mathcal{T}_{p_m}^s\|_{C^1}\leq C \varepsilon$. Thus, by the Cauchy criterion, the $C^0$-limit $\mathcal{T}_x^s= \lim_{n \to +\infty} \mathcal{T}_{p_n}^s$ actually belongs to $C^1(I,\mathbb{R})$, with $\|\mathcal{T}_x^s\|_{C^1}\leq CK$. Moreover, 
	%$$
%\forall\, x,y\in M,\, |x-y|<\eta\implies \|\mathcal{T}_{x}^s-\mathcal{T}_y^s\|_{C^1}\leq C \varepsilon. 
%	$$
	by Lemma~\ref{link templ es}, we conclude that $E^s$ is $C^{1+\alpha}$ for some $\alpha>0$. 
\end{proof}

\begin{remark}
	In the previous results, we have focused on the stable distribution $E^s$, but clearly, the same results hold for the unstable distribution $E^u$, by reversing time. In particular, the second statement in Lemma~\ref{lemma tem st} can be made for the unstable template $\mathcal{T}^u_p$ along the local stable manifold of a volume contracting periodic point $p$. 
\end{remark}

\section{Asymptotic formula for periods}\label{ssection_quatre}

In the present section, we fix a $3$-dimensional transitive Anosov flow $X^t$ which is of class $C^r$, with $r\geq 3$. 

\subsection{Shadowing setup}

Let $\{\imath_x\}_{x \in M}$ be the uniform charts given by Proposition~\ref{propo o good}, and for $x \in M$, recall that we denote by $\Sigma_x\subset M$ the surface
$$
\Sigma_x:=\imath_{x}((-1,1)\times \{0\}\times (-1,1)).
$$
Also recall that for $(\xi,\eta)\in (-1,1)^2$, $y=\imath_x(\xi,0,\eta)\in \Sigma_x$, we denote by $\{\tau_x^t(\xi,\eta)=\tau_x^t(y)\}_t$ the corresponding family of hitting times, and by $\Pi_x^t \colon y \mapsto X^{\tau_x^t(y)}(y)$ the Poincaré map from $\Sigma_x$ to $\Sigma_{X^t(x)}$. %In that case, we denote  $\tau_x^\Sigma(\xi,\eta,t):=t'$. Let $z:=\imath_x(\xi,0,\eta)\in \Sigma_x$; we also set $\hat{\tau}_x^\Sigma(z,t):=\tau_x^\Sigma(\xi,\eta,t)$, and we denote by $f^t\colon \Sigma_p \to \Sigma_{X^t(p)}$, $z\mapsto X^{\hat{\tau}_x^\Sigma(z,t)}(z)$ the associated Poincar\'e map. 

In the following, we consider a periodic point $p\in M$, of period $T>0$, with multipliers $0<\mu=\lambda_p^s(T)<1<\lambda=\lambda_p^u(T)$. We assume that $p$ is volume expanding, i.e., $\mathrm{Jac}_p(T)=\mu\lambda>1$. All the statements that follow will be given in that context, but they all can be easily  adapted to the case when the periodic point $p$ is volume contracting, i.e., $\mathrm{Jac}_p(T)<1$. 

 We fix some homoclinic point $q \in \mathcal{W}^u_{\mathrm{loc}}(p)$. We fix a time $T'>0$ with the property that $q'=X^{T'}(q)\in  \mathcal{W}^s_{\mathrm{loc}}(p)$. Without loss of generality, we can assume that $q,q'\in \Sigma_p$. 
 Then, for some $\xi_\infty,\eta_\infty\neq 0$, we have
$$
\imath_{p}^{-1}(q)=(0,0,\eta_\infty),\quad \imath_{p}^{-1}(q')=(\xi_\infty, 0, 0). 
$$ 
Let $f=\Pi_p^T\colon \Sigma_p \to\Sigma_p$ be the Poincaré map from $\Sigma_p$ to itself. 
\begin{lemma}\label{lemme shadowing}
	There exist a constant $C_0>0$ and an integer $n_0 \in \mathbb{N}$ such that for $n \geq n_0$, there exists a unique periodic point $p_n\in \Sigma_p$ of period %whose orbit intersects $\Sigma_p$ at  $(n+1)$ points $p_n':=f^{-n}(p_n)$, $f^{-n+1}(p_n), \ldots,  p_n$.
	$$
	T_n\simeq nT+T'
	$$
	and such that % Moreover, we have
	$$
	d(X^t(p_n),X^t(q))\leq C_0\mu^{\frac n 2},\quad \forall\, t \in \left[\frac{-nT}{2},\frac{nT}{2}+T'\right].
	$$ 
\end{lemma}

\begin{proof}
	For $n \geq 0$, we consider the periodic pseudo-orbit $\{X^t(q): t\in  \left[\frac{-nT}{2},\frac{nT}{2}+T'\right]\}$ with a small jump at $t=\frac{nT}{2}+T'$. Because $q$ is homoclinic to $p$, and $\mu\lambda>1$, it is easy to see that for some constant $C>0$ independent of $n$, the jump is bounded above by $C\mu^{\frac n2}$. By Anosov closing lemma, for all sufficiently large $n \geq 0$, we deduce that this pseudo-orbit is $C_0 \mu^{\frac n2}$-shadowed by a unique periodic orbit $\{X^t(p_n):t \in [0,T_n]\}$ of period $T_n\simeq nT+T'$, for some constant $C_0>0$, with $p_n\in \Sigma_p$ close to $q$. 
\end{proof}

%There exist constants $...$ such that for any integer $n \geq n_0$, and for any integer $\ell \in \{0,\cdots,n\}$, we have 
%$$
%\imath_{p}^{-1}(f^{-\ell}(p_n))=(\mu^{n-\ell} \xi_\infty,0,\eta_\infty),\quad \imath_{p}^{-1}(q')=(\xi_\infty, 0, 0). 
%$$  

\subsection{Asymptotic formula for a volume expanding periodic point}\label{section asymp ex}

As above, we let $p\in M$ be a volume expanding periodic point of period $T>0$, with multipliers $0<\mu<1<\lambda$, $\mu\lambda=\mathrm{Jac}_p(T)>1$. We fix some homoclinic point $q \in \Sigma_p\cap \mathcal{W}^u_{\mathrm{loc}}(p)$ and we let $q':=X^{T'}(q)\in \Sigma_p\cap \mathcal{W}^u_{\mathrm{loc}}(p)$ as above. Let $(p_n)_{n \geq n_0}$ be the sequence of periodic points given by Lemma~\ref{lemme shadowing} whose orbits shadow the orbit of $q$. 
The main goal of this section is to derive a certain asymptotic expansion of the period $T_n$ of $p_n$ with respect to $n\gg 1$. We fix $\theta \in (0,1)$ such that 
\begin{equation}\label{def_theta_cst}
	\max\left(\mu^{\frac 32},\mu^{\frac{2\log \lambda}{\log \lambda- \log \mu}}\right)<\theta<\mu.
\end{equation} 
\begin{proposition}\label{coro zxp per}
	As $n\to +\infty$, the period $T_n$ of the periodic point $p_n$ has the following asymptotic expansion:
	\begin{equation*}
		T_n=nT+T'+\xi_\infty\left(\mathcal{T}_p^s(\eta_\infty)-\tilde P_p^s(\eta_\infty)\right)\mu^{n}+O(\theta^n),
	\end{equation*}
	%for some constant $c_\infty\neq 0$, 
	with $\xi_\infty \neq 0$, 
	%\theta:=\max\left(\mu^{\frac 32},\mu^{\frac{2\log \lambda}{\log \lambda- \log \mu}}\right)\in (0,\mu)
	%$,
	and where 
	$$
	\tilde P_p^s(\eta):=-\sum_{j=1}^{[k]} \frac{1}{j!} \frac{\partial_{1}\partial_2^j\tau_p^T(p)}{\mu\lambda^j-1} \eta^j%\quad \alpha_p^{j}(T):=-\partial_{1}\partial_2^j\tau_p^T(p),
	$$
	is the polynomial already introduced in Lemma~\ref{lemma tem st}. 
\end{proposition}

We will split the proof of this result into several lemmata. 
Let us first derive some asymptotic expansion for the coordinates of the periodic points $p_n \in \Sigma_p$ and $p_n':=f^{-n}(p_n)\in \Sigma_p$ in normal charts.  For each integer $n \geq n_0$, let 
 $$
(\xi_n,0,\eta_n):=\imath_p^{-1}(p_n),\quad (\xi_n',0,\eta_n'):=\imath_p^{-1}(p_n').
$$
Let $\kappa^s\in \mathbb{R}$ be such that  $v_q^s=(1,\kappa^s)\in D\imath_p^{-1}(q) E_\Sigma^s(q)$, where $E_\Sigma^s(q):=(E^s \oplus \mathbb{R}X)(q)\cap T_q \Sigma_p$. For each time $t \in [0,nT]$, we let $\Pi_p^{-t}\colon \Sigma_p \to \Sigma_{X^{-t}(p)}$ be the Poincaré map, and we let $\hat \Pi_p^{-t}:=\imath_{X^{-t}(p)}^{-1} \circ \Pi_p^{-t} \circ \imath_{p}$ be its image in normal coordinates. We also denote 
\begin{equation*}
	(\xi_n(-t),\eta_n(-t)):=\hat \Pi_p^{-t}(\xi_n,\eta_n).
	%P_n(-t):=\imath_{X^{-t}(p)}^{-1}(X^{-t} (p_n))=(\xi_n(-t),t_n(-t),\eta_n(-t)). 
\end{equation*}
 %and for any integer $\ell \in \{0,\cdots,n\}$, we let
%$$
%(\xi_n(-\ell T),\eta_n(-\ell T)):=\hat f^{-\ell}(\xi_n,\eta_n).
%$$
We also define the integer 
\begin{equation}\label{def helene}
\ell_n:=\left[\frac{\log \mu}{\log \mu-\log \lambda}n\right]\in \mathbb{N}.
\end{equation} 
Since $\lambda^{-1}<\mu$, we note that $\ell_n< \frac{n}{2}$. Note that the choice of time $\ell_n$ is made so that $\xi_n(-{\ell_n}T)\approx\eta_n(-{\ell_n}T)$; in other words, $\mu^{n-\ell_n}\approx \lambda^{-\ell_n}$. %Let $\varsigma:=|\log_{\mu} \lambda|$ and $\gamma:=\frac{r+\varsigma}{1+\varsigma}>1$. 
\begin{remark}
	In what follows, we always denote by $0<\mu<1<\lambda$ the multipliers of the periodic point $p$ under consideration, with $\lambda^{-1}<\mu$. The constant $\nu \in (0,\mu)$ is auxiliary, its value will be chosen differently for various lemmata.
\end{remark}

\begin{lemma}\label{est unifm}
	Fix a number $\nu\in (0,1)$ such that $\max(\lambda^{-1},\mu^{\frac{3}{2}})\leq \nu<\mu$. 
	Then, for some constant $c_\infty\neq 0$, %c_\infty'\neq 0$, 
	we have 
	\begin{equation}\label{est xi n eta n}
		(\xi_n,\eta_n)-(0,\eta_\infty)=c_\infty \mu^n(1,\kappa^s)+O(\nu^n).
	%	(\xi_n',\eta_n')&=(\xi_\infty,0)+c_\infty' \mu^n(1,0) +O(\nu^n),
	\end{equation}
	For any time $t\in \left[0,nT\right]$, %uch that $\mu^{n-\frac t T}\leq \lambda^{-\frac{t}{T}}$, 
	we have  
	\begin{equation}\label{estimee xi n moins t}
		\begin{array}{rll}
		&(\xi_n(-t),\eta_n(-t))-(0,\eta_\infty  \lambda_p^u(-t))=O(\mu^{n-\frac{t}{T}}),&\text{if }t\leq \ell_n T,\\
		&(\xi_n(-t),\eta_n(-t))-(\xi_\infty \lambda_p^s(nT-t),0)=O(\lambda^{-\frac{t}{T}}),&\text{if } t\geq \ell_n T.
		\end{array}
	\end{equation}
	%In particular, for any $t \in [0,nT]$, we have 
	%\begin{equation}\label{produit xi n eta n}
	%\xi_n(-t)\eta_n(-t)=O((\mu\lambda)^{-\frac t T})\mu^n.
	%\end{equation}
\end{lemma}

\begin{proof}
	We abbreviate as $\hat f:=\hat \Pi_p^{T}$ the Poincar\'e map $f=\Pi_p^T$ in normal coordinates. 
	Since the flow $X^t$ is assumed to be of class $C^{r}$, $r \geq 3$, by~\cite[Theorem 2, case 3]{Stowe}, there exists a change of coordinates $\Phi\colon U\to \mathbb{R}^2$ of regularity class $C^{\frac{3}{2}}$ (in fact, of class $C^{\frac{r}{2}}$), %, with $\gamma:=\frac{r+\varsigma}{1+\varsigma}>1$, $\varsigma:=|\log_{\mu} \lambda|$, 
	defined on some neighborhood $U \subset (-1,1)^2$ of $(0,0)$, with $\Phi(0,0)=(0,0)$, which linearizes $\hat f$:
	\begin{equation}\label{conj to lin}
	\Phi \circ \hat f\circ \Phi^{-1}=L\colon (\tilde \xi,\tilde \eta)\mapsto (\mu \tilde \xi,\lambda \tilde \eta).  
	\end{equation}
	
	Note that $\Phi$ preserves the horizontal and vertical axes $\{(\xi,0):|\xi|<1\}$ and $\{(0,\eta):|\eta|<1\}$. Moreover, $\Phi$ can be extended to $\hat f^{\pm 1}(U)$ by setting $\Phi(\hat f^{\pm 1}(x)):=L^{\pm 1}(\Phi(x))$, for any $x \in U$. By repeating this construction finitely many times, we can extend linearizing chart $\Phi$ to a neighborhood of the horizontal and vertical axes up to $(\xi_\infty, 0)$ and $(0,\eta_\infty)$, respectively.
	
	After replacing $\Phi$ with $\Phi \circ \Lambda$ for some linear map $\Lambda \colon (\xi,\eta)\mapsto (a \xi,b \eta)$, and since $L,\Lambda$ commute, we can also assume that $\Phi$ fixes the points $(\xi_\infty,0)$ and $(0,\eta_\infty)$. Moreover, $\hat f|_{\{\eta=0\}}\equiv L|_{\{\eta=0\}}\colon (\xi,0) \mapsto (\mu \xi,0)$ and  $\hat f|_{\{\xi=0\}}\equiv L|_{\{\xi=0\}}\colon (0,\eta) \mapsto (0,\lambda \eta)$; in particular, for any integer $\ell \geq 0$, we have 
	\begin{equation}\label{eq xi infty eta inf}
		(0,\lambda^{-\ell}  \eta_\infty)=\Phi(0,\lambda^{-\ell}\eta_\infty),\quad (\mu^{\ell}  \xi_\infty,0)=\Phi(\mu^{\ell}\xi_\infty,0).
	\end{equation}
	For each integer $n \geq n_0$, let us denote 
	\begin{equation*}
		%\begin{array}{lll}
%	(0,\tilde \eta_\infty):=\Phi(0,\eta_\infty),& &(\tilde \xi_\infty,0):=\Phi(\xi_\infty,0),\\
	(\tilde \xi_n,\tilde \eta_n):=\Phi(\xi_n,\eta_n),\quad(\tilde \xi_n',\tilde \eta_n'):=\Phi(\xi_n',\eta_n').
	%\end{array}
	\end{equation*}
	By construction of the points $p_n,p_n'$,  we have  $\hat f^n(\xi_n',\eta_n')=(\xi_n,\eta_n)$. Then,  by Lemma~\ref{lemme shadowing}, we have the crude expansion\footnote{Using exponential slacking in the shadowing construction we could use $O(\mu^n)$ instead of $O(\mu^{n/2})$, however, due to loss which occurs later in the proof such precision is not needed here.}
		$$
		%\lim_{n \to +\infty}
		 (\xi_n,\eta_n)=(0,\eta_\infty)+O(\mu^{\frac n2}),\quad %\lim_{n \to +\infty}
		  (\xi_n',\eta_n')=(\xi_\infty,0)+O(\mu^{\frac n2}).  
		$$
		By~\eqref{conj to lin}, we thus obtain
		\begin{equation}\label{eq trois deux}
		\begin{array}{l}
		(\tilde \xi_n,\tilde \eta_n)=(\mu^n \tilde \xi_n',\lambda^n \tilde \eta_n')=\left(\xi_\infty \mu^n+O(\mu^{\frac{3}{2} n}),\eta_\infty+O(\mu^{\frac n2})\right),\\
		(\tilde \xi_n',\tilde \eta_n')=(\mu^{-n} \tilde \xi_n,\lambda^{-n} \tilde \eta_n)=\left(\xi_\infty+O(\mu^{\frac n2}),\eta_\infty\lambda^{-n}+O(\mu^{\frac n2}\lambda^{-n})\right).
		\end{array}
		\end{equation} 
		Recall that $v_q^s=(1,\kappa^s)\in D\imath_p^{-1}(q) E_\Sigma^s(q)$, and let \begin{itemize}
			\item $\tilde v_q^s=(\beta_q^s,\gamma_q^s):=D\Phi(0,\eta_\infty) v_q^s$, with $\beta_q^s\neq 0$; %$v_q^s=(1,\kappa^s)\in E_\Sigma^s(q)$ as above;
			\item $v_{q'}^s:=D(\imath_p^{-1}\circ X^{T'}\circ \imath_p)(0,\eta_\infty) v_q^s\in D\imath_p^{-1}(q') E_\Sigma^s(q')$;
			\item $\tilde v_{q'}^s=(\delta_{q'}^s,0):=D\Phi(\xi_\infty,0)v_{q'}^s$, with $\delta_{q'}^s\neq 0$.
		\end{itemize} 
		Indeed, we observe that $v_{q'}^s,\tilde v_{q'}^s \in (\mathbb{R}\setminus \{0\}) \times \{0\}$. Similarly, let
		\begin{itemize}
			\item $v_{q'}^u=(\kappa^u,1)\in  D\imath_p^{-1}(q') E_\Sigma^u(q')$, where $E^u_\Sigma(q'):=(\mathbb{R}X\oplus E^u)(q')\cap T_{q'} \Sigma_p$;
			\item $\tilde v_{q'}^u=(\beta_{q'}^u,\gamma_{q'}^u):=D\Phi(\xi_\infty,0) v_{q'}^u$, with $\gamma_{q'}^u\neq 0$;
			\item $v_q^u:=D(\imath_p^{-1}\circ X^{-T'}\circ \imath_p)(\xi_\infty,0) v_{q'}^u\in D\imath_p^{-1}(q) E_\Sigma^u(q)$;
			\item $\tilde v_q^u=(0,\delta_{q}^u):=D\Phi(\xi_\infty,0)v_q^u$, with $\delta_q^u\neq 0$. 
		\end{itemize}   
		Indeed, we also have $v_{q}^u,\tilde v_{q}^u \in \{0\}\times(\mathbb{R}\setminus \{0\}) $.
		
		Let $a_n,b_n\in \mathbb{R}$ be such that 
		$$
		(\tilde \xi_n,\tilde \eta_n)-(0, \eta_\infty)=a_n \tilde v_q^s +b_n \tilde v_q^u.
		$$
		By~\eqref{eq trois deux}, and since $\tilde v_q^s=(\beta_q^s,\gamma_q^s)$,  $\tilde v_{q}^u =(0,\delta_{q}^u)$, we see that $a_n=c_\infty \mu^n+O(\mu^{\frac{3}{2} n})$, with $c_\infty:=(\beta_q^s)^{-1}\xi_\infty\neq 0$, and $b_n=O(\mu^{\frac n2})$. Similarly, let $a_n',b_n'\in \mathbb{R}$ be such that  
		$$
		(\tilde \xi_n',\tilde \eta_n')-( \xi_\infty,0)=a_n' \tilde v_{q'}^s +b_n' \tilde v_{q'}^u.
		$$
		By~\eqref{eq trois deux}, and since $\tilde v_{q'}^s=(\delta_{q'}^s,0)$, $v_{q'}^u=(\beta_{q'}^u,\gamma_{q'}^u)$, we see that $b_n'=\tilde c_\infty\lambda^{-n}+O(\mu^{\frac n2}\lambda^{-n})$, for some constant $\tilde c_\infty:=(\gamma_{q'}^u)^{-1}\eta_\infty\neq 0$, and $a_n=O(\mu^{\frac n2})$. 
		
		Let $\Pi\colon U_q \to U_{q'}$ be the Poincar\'e map of $X^t$ between a small neighborhood $U_q\subset \Sigma_p$ of $q$ and a small neighborhood $U_{q'}\subset \Sigma_p$ of $q'$ so that $\Pi(q)=q'$. 
		Let $\hat\Pi:=\imath_p^{-1}\circ \Pi \circ \imath_p$ be its expression  in normal coordinates, and $\tilde\Pi:=\Phi\circ \hat \Pi \circ \Phi^{-1}$ be its expression in the linearizing chart. % Since $(\tilde \xi_\infty,0)=\pi(0,\tilde \eta_\infty)$ and $(\tilde \xi_n',\tilde \eta_n')=\pi(\tilde \xi_n,\tilde \eta_n)$, 
		We then have 
		\begin{align*}
		a_n' \tilde v_{q'}^s +b_n' \tilde v_{q'}^u&=(\tilde \xi_n',\tilde \eta_n')-( \xi_\infty,0)\\
		&=\tilde\Pi	(\tilde \xi_n,\tilde \eta_n)-\tilde\Pi (0,\eta_\infty)\\
		&=D\tilde\Pi(0,\eta_\infty)\left(a_n \tilde v_q^s +b_n \tilde v_q^u\right)+O(\max(|a_n|^{\frac{3}{2}},|b_n|^{\frac{3}{2}}))\\
		&=a_n \tilde v_{q'}^s +b_n \tilde v_{q'}^u+O(\max(\mu^{\frac{3}{2} n},|b_n|^{\frac{3}{2}})).%O(\mu^{\frac 32 n}),
		\end{align*}
		By considering the projection onto $\mathbb{R}\tilde v_{q'}^u$ parallel to $\mathbb{R}\tilde v_{q'}^s$, we thus see that 
		$$
		b_n=O(\nu^n),
		$$
		where $\nu:=\max(\lambda^{-1},\mu^{\frac{3}{2}})\in (0,\mu)$. From the above equation, we also deduce that 
		%under $\Phi \circ \imath_p^{-1} \circ X^{T_n-nT}\circ \imath_p \circ \Phi^{-1}$, we also have 
		$$
		a_n'=a_n+O(\nu^{n})=c_\infty \mu^n+O(\nu^n).%\quad  b_n=b_n'+O(\mu^{\frac 32 n})=O(\nu^n),
		$$ 
		%with $\nu:=\max(\lambda^{-1},\mu^{\frac{3}{2}})\in (0,1)$,  %\quad b_n=b_n'+O(\mu^{\frac 32  n})=\tilde c_\infty \lambda^{-n}+o(\lambda^{-n})).
		Therefore,  
%		$$
%		(\tilde \xi_n',\tilde \eta_n')-( \xi_\infty,0)=(c_\infty \mu^n +O(\nu^n))\tilde v_{q'}^s +(\tilde c_\infty \lambda^{-n} +O(\mu^{\frac n2}\lambda^{-n})) \tilde v_{q'}^u.
%		$$
%		Since $\tilde \xi_n=\mu^n\tilde \xi_n'$, and $0<\lambda^{-1}<\mu$, we deduce that 
%		$$
%		\tilde \xi_n=\xi_\infty \mu^n+O(\mu^{2n}).
%		$$ 
%		Repeating the same argument as above, we can then upgrade the previous estimates to $a_n=c_\infty \mu^n+O(\mu^{2n})$, $a_n'=a_n+O(\mu^{\frac 32  n})=c_\infty \mu^n+O(\mu^{\frac 32  n})$, and 
		\begin{align}
		(\tilde \xi_n,\tilde \eta_n)-(0, \eta_\infty)&=c_\infty \mu^n\tilde v_q^s +O(\nu^n)\nonumber \\ 
		&=(\xi_\infty \mu^n+O(\nu^n),O(\mu^n)), \label{equa trois quatre}\\
		(\tilde \xi_n',\tilde \eta_n')-( \xi_\infty,0)&=(c_\infty \mu^n +O(\nu^n))\tilde v_{q'}^s +(\tilde c_\infty \lambda^{-n} +O(\mu^{\frac n2}\lambda^{-n})) \tilde v_{q'}^u\nonumber\\ 
		&=(O(\mu^n),\eta_\infty \lambda^{-n}+O(\mu^{\frac n2}\lambda^{-n})).\label{equa trois cinq}
		%(c_\infty \mu^n +O(\mu^{\frac 32  n}))\tilde v_{q'}^s +(\tilde c_\infty \lambda^{-n} +o(\lambda^{-n})) \tilde v_{q'}^u.
		\end{align}
		Now we go back to $(\xi_n,\eta_n)$ y using Taylor's formula for the $C^{\frac{3}{2}}$ diffeomorphism $\Phi^{-1}$, and using $0<\lambda^{-1}<\mu$.
		\begin{align*}
			(\xi_n,\eta_n)-(0,\eta_\infty)&= D\Phi(0,\eta_\infty)^{-1} \left(c_\infty\mu^n\tilde v_q^s +O(\nu^{n})\right)+O(\mu^{\frac{3}{2} n})\\ 
			&=c_\infty\mu^n (1,\kappa^s)  +O(\nu^n),
			%(\xi_n',\eta_n')-(\xi_\infty,0)&=
			%D\Phi(\xi_\infty,0)^{-1} \left((c_\infty\mu^n+O(\nu^{  n}))\tilde v_{q'}^s +(\tilde c_\infty \lambda^{-n}+O(\mu^{\frac n2}\lambda^{-n})) \tilde v_{q'}^u\right)+O(\mu^{\frac 32 n})\\ 
			%&=c_\infty' \mu^n(1,0)+ O(\nu^n), 
		\end{align*}
		which concludes the proof of~\eqref{est xi n eta n}. 
		%for some constant $c_\infty'\neq 0$. 
		Moreover, by~\eqref{equa trois quatre}-\eqref{equa trois cinq}, for any integer $\ell \in \{0,\cdots,n\}$, we have
		\begin{align*}
			L^{-\ell}(\tilde \xi_n,\tilde \eta_n)=(\xi_\infty \mu^{n-\ell}+O(\mu^{2n-\ell}),\eta_\infty \lambda^{-\ell}+O(\mu^{n}\lambda^{-\ell})).
%			(\xi_n(-\ell T),\eta_n(-\ell T))=
		\end{align*}
		Applying the diffeomorphism $\Phi^{-1}$, by~\eqref{eq xi infty eta inf}, we then obtain
		\begin{equation}\label{eq pour ell}
		 (\xi_n(-\ell T),\eta_n(-\ell T)-\eta_\infty \lambda^{-\ell})=\hat f^{-\ell}(\xi_n,\eta_n)-(0,\eta_\infty \lambda^{-\ell})=O(\mu^{n-\ell}),
		 \end{equation}
		 if $\ell$ is such that $\mu^{n-\ell}\leq \lambda^{-\ell}$, and
		 \begin{equation*}
		(\xi_n(-\ell T)-\xi_\infty \mu^{n-\ell},\eta_n(-\ell T))=\hat f^{-\ell}(\xi_n,\eta_n)-(\xi_\infty \mu^{n-\ell},0)=O(\lambda^{-\ell}),
		\end{equation*}
		if $\ell$ is such that $\mu^{n-\ell}\geq \lambda^{-\ell}$.  In particular, we note that in either case,  
		$$
		\xi_n(-\ell T)\eta_n(-\ell T)=O((\mu\lambda)^{-\ell})\mu^n.
		$$
		Let us now finish the proof of~\eqref{estimee xi n moins t}. %-\eqref{produit xi n eta n}. 
		Fix a time $t\in [0,nT]$ such that $\mu^{n-\ell}\leq \lambda^{-\ell}$, where $\ell:=[\frac t T]\geq 0$. Let $t=\ell T + t'$, with $t':=t-\ell T \in [0,T)$. We have 
		\begin{align*}
		\hat \Pi_p^{-t'}(\xi_n(-\ell T),\eta_n(-\ell T))&=(\xi_n(-t),\eta_n(-t)),\\
		\hat \Pi_p^{-t'}(0,\eta_\infty \lambda^{-\ell})&=(0,\eta_\infty \lambda^{-\ell} \lambda_p^u(-t'))=(0,\eta_\infty  \lambda_p^u(-t)). 
		\end{align*}
		Recall that $\hat \Pi_p^{-t'}:=\imath_{X^{-t'}(p)}^{-1}\circ \Pi_p^{-t'}\circ \imath_{p}$ is the notation for the Poincar\'e map $\Pi_p^{-t'}\colon \Sigma_p \to \Sigma_{X^{-t'}(p)}$ represented in normal coordinates. The time $t'\in [0,T)$ is uniformly bounded with respect to $n$, hence so are
		$\lambda_p^u(-t')$, $\lambda_p^s(-t')$ and
		 the differential of the map $\hat \Pi_p^{-t'}$; in particular, $\mu^{n-\ell}\simeq \mu^{n-\frac t T}$ and $\lambda^{-\ell}\simeq \lambda_p^u(-t)\simeq \lambda^{-\frac t T}$. By~\eqref{eq pour ell}, we thus obtain
		$$
		(\xi_n(-t),\eta_n(-t)-\eta_\infty  \lambda_p^u(-t))=O(\mu^{n-\frac{t}{T}}),
		$$
		%and then, 
		%$$
		%\xi_n(-t)\eta_n(-t)=O(\lambda_p^u(-t)\mu^{\frac t T})\mu^n=O((\mu\lambda)^{\frac t T})\mu^n,
		%$$
		as we claimed. Moreover, we see that $\mu^{n-\ell}\leq \lambda^{-\ell}$ if and only if $\ell \leq-\frac{\log \mu}{\log \lambda-\log \mu}n$. The case where $\mu^{n-\ell}\geq \lambda^{-\ell}$ is analogous. 
\end{proof} 
 
 \begin{lemma}\label{differentielle pi x un}
 	For any point $x\in M$, %and for any integer $n\in \mathbb{Z}$ %such that $|\lambda_x^u() \eta|<1$, 
 	and $\sigma\geq 0$, 
 	let 
	$$
	\hat \Pi_x^{-\sigma}=(\hat \Pi_{x,1}^{-\sigma},\hat \Pi_{x,3}^{-\sigma})=\imath_{X^{-\sigma}(x)}^{-1}\circ \Pi_x^{-\sigma}\circ \imath_{x}
	$$ 
	be the expression of the Poincaré map $ \Pi_x^{-\sigma}\colon \Sigma_x \to \Sigma_{X^{-\sigma}(x)}$ in normal charts. Then the differential of the first component along the unstable manifold is given by
 	$$
 	d\hat \Pi_{x,1}^{-\sigma}(0,\eta)=\lambda_x^s(-\sigma)d \xi,\quad \forall\,\eta \in (-1,1).
 	$$
 \end{lemma}
 
 \begin{proof}
 	Let us denote by $F_x^{-\sigma}:=(\imath_{X^{-\sigma}(x)})^{-1}\circ X^{-\sigma}\circ \imath_x=(F_{x,1}^{-\sigma},F_{x,2}^{-\sigma},F_{x,3}^{-\sigma})$ the inverse of the time-$\sigma$ map of the flow in normal coordinates based at $x$. Fix any $\eta \in (-1,1)$. % such that $|\lambda_y^u(1)\eta|<1$. 
 	On  one hand, by Proposition~\ref{propo o good}\eqref{normal deux}, $F_x^{-\sigma}$ preserves the vertical line $\{(0,0)\}\times (-1,1)$, hence $\partial_3 F_{x,1}^{-\sigma}(0,0,\eta)=\partial_3 F_{x,2}^{-\sigma}(0,0,\eta)=0$. Moreover, by Proposition~\ref{propo o good}\eqref{pt cocyc nf}, we have 
 		\begin{equation*} 
 		 	\partial_1 F_{x,1}^{-\sigma}(0,0,\eta)=\lambda_x^s(-\sigma),\quad \partial_2 F_{x,1}^{-\sigma}(0,0,\eta)=0.
% 		 	\begin{bmatrix}
% 		 		\partial_1 F_{x,1}^{-1} & \partial_2 F_{x,1}^{-1} \\
% 		 		\partial_1 F_{x,2}^{-1}  & \partial_2 F_{x,2}^{-1} 
% 		 	\end{bmatrix}(0,0,\eta)=\begin{bmatrix}
% 		 		\lambda_x^s(-1) & 0\\
% 		 		\gamma_x^s(\eta) & 1
% 		 	\end{bmatrix},
 		\end{equation*}
 		%for some $\gamma_x^s(\eta)\in \mathbb{R}$. 
 		Therefore, if $|\xi|\ll 1$ is small, we have 
 		\begin{equation*}
 		F_x^{-\sigma}(\xi,0,\eta)=(0,0,\lambda_x^u(-\sigma)\eta) +(\lambda_x^s(-\sigma),\partial_1 F_{x,2}^{-\sigma}(0,0,\eta),\partial_1 F_{x,3}^{-\sigma}(0,0,\eta))\xi+O(\xi^2). 
 		\end{equation*}
 		By Proposition~\ref{propo o good}\eqref{flw dir}, we then have
 		$$
 		(\imath_{X^{-\sigma}(x)})^{-1}\circ X^{-\sigma-\partial_1 F_{x,2}^{-\sigma}(0,0,\eta)}\circ \imath_x(\xi,0,\eta)=
 		(\lambda_x^s(-\sigma)\xi,0,\lambda_x^u(-\sigma)\eta+\partial_1 F_{x,3}^{-\sigma}(0,0,\eta)\xi)+O(\xi^2).
 		$$ 
 		We thus conclude that 
 		$$
 		\hat \Pi_{x}^{-\sigma}(\xi,\eta)=(\lambda_x^s(-\sigma)\xi,\lambda_x^u(-\sigma)\eta+\partial_1 F_{x,3}^{-\sigma}(0,0,\eta)\xi)+O(\xi^2),
 		$$ 
 		and the result follows by differentiating the first coordinate with respect to $\xi$. 
 \end{proof}
 
 In the following, we assume that the constant $\nu \in (0,1)$ from Lemma~\ref{est unifm} is chosen such that $\max(\lambda^{-1},\mu^{\frac{3}{2}})< \nu<\mu$. 
 \begin{corollary}\label{coro good est hitt}
 	%Without loss of generality, we assume that the constant $\nu\in (0,1)$ in 
 	%Lemma~\ref{est unifm} was chosen such that   $\max(\lambda^{-1},\mu^{\gamma},\mu^{1+\vartheta})< \nu<\mu$.  
 	%Let $\ell_n:=\left[-\frac{\log \mu}{\log \lambda-\log \mu}nT\right]\in \mathbb{N}$. Since $\lambda^{-1}<\mu$, we see that $\ell_n< \frac{nT}{2}$. 
 	For any integer $\ell \in \{0,\cdots,\ell_n\}$,  with $\ell_n < \frac{n}{2}$ given in~\eqref{def helene}, we have 
 	\begin{equation}\label{change of xi n ell}
 	\xi_n(-\ell T)=\xi_\infty \mu^{n-\ell}+O(\nu^{n} \mu^{-\ell}).
 	\end{equation}
 	Therefore, for any integer $\ell \in \{0,\cdots,\ell_n\}$,  we have
 	\begin{equation}\label{change of tau n ell}
 	\tau_{p}^{T}(\xi_n(-\ell T),\eta_n(-\ell T))=T-\xi_\infty P_{p}^s(T)(\eta_\infty \lambda^{-\ell}) \mu^{-\ell}(\mu^n+O(\nu^n ))+
 		O(\mu^{2(n-\ell)}).
 	%\tau_{p}^{T}(\xi_n(-\ell T),\eta_n(-\ell T))=T-\xi_\infty P_{p}^s(T)(\eta_\infty\lambda^{-\ell})\mu^{n-\ell}+O(\nu^{n} \mu^{-\ell}),%O(\mathrm{Jac}_p(T)^{-\ell}\nu^n). 
 	\end{equation}
 	%and also 
 %	\begin{equation}\label{change of tau n ell_secondeq}
 	%	\tau_{p}^{T}(\xi_n(-\ell T),\eta_n(-\ell T))=T+O(\lambda^{-\vartheta \ell}\mu^{n-\ell}).%O(\mathrm{Jac}_p(T)^{-\ell}\nu^n). 
 %	\end{equation}
 \end{corollary}
 
 \begin{proof}
 	%By Lemma~\ref{differentielle pi x un}, it suffices to show~\eqref{change of xi n ell} when $\sigma=\ell T$,
 	Let  $\bar\nu:=\max(\lambda^{-1},\mu^{\frac{3}{2}})\in (0,\mu)$.  
 	By induction on $\ell \in \{0,\cdots,\ell_n\}$, let us first show that  
 	\begin{equation}\label{change of xi n ell prem}
 		\xi_n(-\ell T)=\xi_\infty \mu^{n-\ell}+O((\ell+1) \bar\nu^{n}\mu^{-\ell}).
 	\end{equation}
 	%We show~\eqref{change of xi n ell} 
% 	By induction on $0\leq \ell< \frac{\ell_n}{T}$, let us first show that  
% 	\begin{equation}\label{change of xi n ell prem}
% 		\xi_n(-\ell T)=c_\infty \mu^{n-\ell}+O(\nu^n \mu^{-\ell}).
% 	\end{equation}
 	
 	For $\ell=0$, it is true by~\eqref{est xi n eta n}. 
 	Let us assume that it is true for some integer $\ell \in \{0,\cdots,\ell_n-1\}$, and let us show it for $\ell+1$. Recall that by~\eqref{estimee xi n moins t}, we have
 	\begin{equation}\label{ trois deux revi}
 	(\xi_n(-\ell T),\eta_n(-\ell T))-(0,\eta_\infty  \lambda^{-\ell})=O(\mu^{n-\ell} ).%O(\mu^n\lambda_p^s(-\ell)).
 	\end{equation}
 	By applying Lemma~\ref{differentielle pi x un} and our induction hypothesis, since the normal charts are $C^{r-1}$, with $r-1\geq 2$, we have
 	\begin{align*}
 		\xi_n(-(\ell+1)T)&=\hat \Pi_{p,1}^{-T}(\xi_n(-\ell T),\eta_n(-\ell T))\\
 		&=\hat \Pi_{p,1}^{-T}(0,\eta_\infty  \lambda^{-\ell})+  \mu^{-1} \xi_n(-\ell T)+O(\mu^{2(n-\ell)})\\
 		&=c_\infty \mu^{n-(\ell+1)}  +O((\ell +1)\bar\nu^{n}\mu^{-(\ell+1)})+O(\mu^{2(n-\ell)})\\
 		&=c_\infty  \mu^{n-(\ell+1)}  +O((\ell+2)\bar \nu^{n-(\ell+1)}).
 	\end{align*} 
 	Indeed, since $\ell \leq \ell_n < \frac{n}{2}$, and $\mu^{\frac 32}\leq \bar\nu$, we have 
 	$$
 	\frac{\mu^{2(n- \ell )}}{\bar\nu^n\mu^{-\ell}}\leq \frac{\mu^{\frac{3}{2}n}}{\bar\nu^n}\mu^{\frac n 2 -\ell_n}\leq 1,
 	$$ 
 	which justifies the last equality and  concludes the proof of~\eqref{change of xi n ell prem}. Since we assumed that  $\max(\lambda^{-1},\mu^{\frac{3}{2}})< \nu<\mu$,  we deduce that for any integer $\ell \in \{0,\cdots,\ell_n\}$,  with $\ell_n < \frac{n}{2}$, we have  $\ell \bar \nu^n= O(n\bar \nu^n)=O(\nu^n)$, and 
 	\begin{equation*} 
 		\xi_n(-\ell T)=c_\infty \mu^{n-\ell}+O(\nu^{n}\mu^{-\ell}).
 	\end{equation*}
 	
 	Thus, to obtain~\eqref{change of xi n ell} it remains to show that the constant $c_\infty$ is actually equal to $\xi_\infty$. Set %$\ell_n:=-\frac{\log \mu}{\log \lambda-\log \mu}n$, and 
 	$\ell_n':=\lfloor\frac{n}{2}\rfloor$. In particular, %$\frac{\ell_n T}{\ell_n'}\leq \varrho\in (0,1)$ for $n\gg 1$, and  $\mu^{n-\ell_n}\simeq \lambda^{-\ell_n}$, while  
 	$ \lambda^{-\ell_n'}=O(\lambda^{-\frac n2})=o(\mu^{\frac n2})$, while $\mu^n\lambda_p^s(-\ell_n' T)=\mu^{n-\ell_n'}$ is of the same order as $\mu^{\frac n2}$. By~\eqref{estimee xi n moins t}, we thus have
 	$$
 	(\xi_n(-\ell_n'T),\eta_n(-\ell_n'T))=(\xi_\infty \mu^{n-\ell_n'},0)+O(\lambda^{-\frac n2}). 
 	$$
 	By Lemma~\ref{differentielle pi x un}, we  have 
 	$$
 	d\hat \Pi_{p,1}^{-T}(0,\eta_n((-\ell_n'+1)T))=\mu^{-1} d\xi,
 	$$
 	and then, 
 	$$
 	\xi_n((-\ell_n'+1)T)=\xi_\infty \mu^{n-(\ell_n'+1)}+o(\mu^{n-(\ell_n'+1)}).
 	$$
 	By a straightforward induction, we  obtain
 	$$
 	\xi_n=\xi_n(0)=\xi_\infty \mu^n+o(\mu^n).
 	$$
 	Comparing with~\eqref{est xi n eta n}, we conclude that $c_\infty=\xi_\infty$, as claimed.
 	
 	Now, for $\ell \in \{0,\cdots,\ell_n\}$, by~\eqref{change of xi n ell},~\eqref{ trois deux revi},~\eqref{der temps}, and since $\tau_{p}^T$ is $C^2$, with  $\tau_{p}^T(0,\cdot)\equiv T$, 
 	%$$
 	%\sup_{\xi\in [0, \xi_n(-\ell T)]}\partial_{11} \tau_{p}^T(\xi,\eta_n(-\ell T))=O(\eta_n(-\ell T)=O(\lambda^{-\ell}),
 	%$$
 	we have
 	\begin{align*}
 		\tau_{p}^{T}(\xi_n(-\ell T),\eta_n(-\ell T))
 		& =\tau_{p}^{T}(0,\eta_\infty\lambda^{-\ell})+\partial_1 \tau_{p}^T(0,\eta_\infty\lambda^{-\ell})\xi_n(-\ell T)+O(\mu^{2(n-\ell)})\\
 		%&\quad +O(\partial_{11} \tau_{X^{-\ell}(p)}^1(0,\eta_\infty \lambda_p^u(-\ell))\xi_n(-\ell)^2)\\
 		&=T-\xi_\infty P_{p}^s(T)(\eta_\infty \lambda^{-\ell}) \mu^{-\ell}(\mu^n+O(\nu^n ))+
 		O(\mu^{2(n-\ell)}).
 		%&=T-\xi_\infty P_{p}^s(T)(\eta_\infty \lambda^{-\ell}) \mu^{n-\ell} +O(\nu^n \mu^{-\ell}),
 	\end{align*}
 	%where the last equality follows from the previous observation that for $\ell\leq \ell_n$, we have  $\mu^{2(n-\ell)}\leq \nu^n \mu^{-\ell}$. 
 %	To derive the last inequality, we have used that $\partial_1 \tau_{X^{-\ell}(p)}^1(0,0)=0$, so that
% 	$$
% 	\partial_1 \tau_{X^{-\ell}(p)}^1(0,\eta_\infty \lambda_p^u(-\ell))=O( \lambda_p^u(-\ell)),
% 	$$
% 	since $\partial_1 \tau_{X^{-\ell}(p)}^1(0,0)=0$, so that 
% 	\begin{equation*}
 %	\partial_1 \tau_{X^{-\ell}(p)}^1(0,\eta_\infty \lambda_p^u(-\ell))\nu^n \lambda_p^s(-\ell)=O( \lambda_p^u(-\ell))\nu^n \lambda_p^s(-\ell)=O(\mathrm{Jac}_p(T)^{-\frac{\ell}{T}}\nu^n),
% 	\end{equation*}
% 	\marginpar{C3 is used again, also needs fixing,adjusting theta}
% 	while by the same calculation as above, and since $\partial_{11} \tau_{X^{-\ell}(p)}^1(0,0)=0$, we have 
% 	$$
% 	\partial_{11} \tau_{X^{-\ell}(p)}^1(0,\eta_\infty \lambda_p^u(-\ell))\xi_n(-\ell)^2 =O(\lambda_p^u(-\ell))\mu^{2(n-\frac{\ell}{T})}=O(\mathrm{Jac}_p(T)^{-\frac{\ell}{T}}\nu^n).
% 	$$
 	The proof is complete. 
 \end{proof}

Recall the expansions of local stable and unstable manifolds which we used to define the stable and unstable templates in Subsection~\ref{section templates}:
\begin{equation}\label{rappel templates}
\begin{array}{l}
\mathcal{W}_{\mathrm{loc}}^s(\Phi_x^u(\eta))=\left\{ Q_\eta^s(\tilde  \xi):=\imath_x(\tilde\xi,\mathcal{T}^s_x(\eta)\tilde\xi+b_x^s(\tilde\xi,\eta)\tilde\xi^2,\eta+c_x^s(\tilde\xi,\eta)\tilde\xi)\right\}_{\tilde\xi \in (-1,1)},\\
\mathcal{W}_{\mathrm{loc}}^u(\Phi_x^s(\xi))=\left\{Q_\xi^u(\tilde \eta):=\imath_x(\xi+c_x^u(\xi,\tilde\eta)\tilde\eta,\mathcal{T}^u_x(\xi)\tilde\eta+b_x^u(\xi,\tilde\eta)\tilde\eta^2,\tilde\eta)\right\}_{\tilde\eta \in (-1,1)}.
\end{array}
\end{equation}
By construction, at the point $(0,0,\eta_\infty)=\imath_p^{-1}(q)$, we have 
\begin{align*}
	(1,\mathcal{T}_p^s(\eta_\infty),c_p^s(0,\eta_\infty))&\in D\imath_p^{-1}(q)E^s(q),\\
	(1,0,c_p^s(0,\eta_\infty))&\in D\imath_p^{-1}(q)E_\Sigma^{s}(q).
\end{align*} 
%Similarly, at $(\xi_\infty,0,0)=\imath_p^{-1}(q)$, it holds 
%$D\imath_p(c_p^u(\xi_\infty,0),\mathcal{T}_p^s(\xi_\infty),1)\in E^u(q)$, and $D\imath_p(c_p^u(\xi_\infty,0),0,1)\in E^{cu}(q)\cap T_{q}\Sigma_p$. 
In particular, using the notation introduced at the beginning of Subsection~\ref{section asymp ex}, we have 
$$
v_q^s=(1,\kappa^s)=(1,c_p^s(0,\eta_\infty))\in D\imath_p^{-1}(q) E_\Sigma^s(q).
$$  %and $\kappa_3=c_p^s(0,\eta_\infty) \xi_\infty$. 

Recall that we denote by $\Pi\colon U_q \to U_{q'}$ the Poincar\'e map induced by $X^t$ between a neighborhood $U_q\subset \Sigma_p$ of $q$ and a neighborhood $U_{q'}\subset \Sigma_p$ of $q'$, and let $\hat\Pi:=\imath_p^{-1}\circ \Pi \circ \imath_p$ be its expression  in normal coordinates. Now we define the return time $\bar\tau$ on $U_q$. For $z \in U_{q}$, we have $\Pi(z)=X^{\bar\tau(z)}(z)$, for some time $\bar\tau(z)$ close to $\bar \tau(q)=T'$. Writing $z=\imath_{p}(\xi,0,\eta)$, we also let $\bar \tau(\xi,\eta):=\bar\tau(z)$. 

As before, let us also consider
\begin{align*}  v_{q'}^u&:=(\kappa^u,1)=(c_p^u(\xi_\infty,0),1)\in  D\imath_p^{-1}(q') E_\Sigma^u(q'); \\
%\item $v_{q'}^s:=D\hat\Pi(0,\eta_\infty) v_q^s\in D\imath_p^{-1}(q') E_\Sigma^s(q')$;
%\item $\tilde v_{q'}^u=(\beta_{q'}^u,\gamma_{q'}^u):=D\Phi(\xi_\infty,0) v_{q'}^u$, with $\gamma_{q'}^u\neq 0$;
v_q^u&:=D\hat \Pi^{-1}(\xi_\infty,0) v_{q'}^u\in D\imath_p^{-1}(q) E_\Sigma^u(q)=\{0\}\times \mathbb{R}. 
\end{align*} 

%We also denote $v_p^s:=(1,0,c_p^s(0,\eta_\infty))$ and $v_p^u:=D(\bar\pi_{q',q}^{-1})_{q}(c_p^u(\xi_\infty,0),0,1)$. 
\begin{lemma}\label{lemme temps templates}
%For $(\xi,\eta)\in (-1,1)$ with $|(\xi,\eta)|:=|\xi|+|\eta|\ll 1$, w
We have 
\begin{equation}\label{diff tua q q prime}
D\bar\tau(0,\eta_\infty)(\tilde\xi\, v_q^s+ \tilde\eta\,  v_q^u)=\mathcal{T}_p^s(\eta_\infty)\tilde\xi-\mathcal{T}_p^u(\xi_\infty)\tilde\eta. %+O(|(\xi,\eta)|^2).
\end{equation}
\end{lemma}
\begin{proof}
Given $|\tilde\xi|\ll 1$, let us consider the path $L^s$ tangent to $U_q$ and the weak stable manifold at $q$, $L^s\colon \tilde \xi\mapsto \imath_p(\tilde\xi, 0,\eta_\infty+c_p^s(\tilde\xi,\eta_\infty)\tilde\xi)$. By Proposition~\ref{propo o good}\eqref{flw dir}, and with $Q_{\eta_\infty}^s(\tilde\xi)$ as in~\eqref{rappel templates}, we have 
$$
d(X^{\mathcal{T}_p^s(\eta_\infty)\tilde\xi}(L^s(\tilde\xi)),Q_{\eta_\infty}^s( \tilde\xi))=O(\tilde\xi^2). 
$$
Moreover, by definition $Q_{\eta_\infty}^s(\tilde\xi) \in \mathcal{W}_{\mathrm{loc}}^s(q)$, hence $X^{T'}(Q_{\eta_\infty}^s(\tilde\xi) )\in \mathcal{W}_{\mathrm{loc}}^s(q')\subset U_{q'}$, and thus,
$$
\bar\tau(L^s(\tilde \xi))=T'+\mathcal{T}_p^s(\eta_\infty)\tilde\xi+O(\tilde\xi^2).
$$
By definition, $\bar\tau(L^s (0))=\bar \tau(q)=T'$, and $(L^s)'(0)=D\imath_p(0,0,\eta_\infty)v_q^s$, hence $\bar\tau(L^s (\tilde\xi))=T'+D\bar\tau(0,\eta_\infty) (\tilde\xi\, v_q^s) +O(\tilde\xi^2)$. Therefore, 
$$
D\bar\tau (0,\eta_\infty)(v_q^s)=\mathcal{T}_p^s(\eta_\infty). 
$$

Similarly, given $|\tilde\eta|\ll 1$, let us consider the path $L^u$ tangent to $U_{q'}$ and the weak unstable manifold at $q'$, $L^u\colon \tilde \eta\mapsto \imath_p(\xi_\infty+c_p^u(\xi_\infty,\tilde\eta)\tilde\eta, 0,\tilde\eta)$.
For $Q_{\xi_\infty}^u(\tilde\eta)$ as in~\eqref{rappel templates}, we have 
$$
d(X^{\mathcal{T}_p^u(\xi_\infty)\tilde\eta}(L^u(\tilde\eta)),Q_{\xi_\infty}^u( \tilde\eta))=O(\tilde\eta^2). 
$$

Moreover, by definition $Q_{\xi_\infty}^u(\tilde\eta) \in \mathcal{W}_{\mathrm{loc}}^u(q')$, hence $X^{-T'}(Q_{\xi_\infty}^u(\tilde\eta) )\in \mathcal{W}_{\mathrm{loc}}^u(q)\subset U_q$, and thus,
% By definition, we have $P^u( \eta)=\Pi^{-1}(Q^u(\eta))+O( \eta^2)$.
$$
\bar\tau(\hat\Pi^{-1}\circ L^u (\tilde\eta))=T'-\mathcal{T}_p^u(\xi_\infty)\tilde\eta+O(\tilde\eta^2).
$$
(Note that $\{X^{-T'+\mathcal{T}_p^u(\xi_\infty)\tilde\eta}(L^u(\tilde\eta)), |\tilde\eta|\ll 1\}$ is not in $\Sigma_p$, but is tangent to the path $\hat\Pi^{-1}\circ L^u\subset\Sigma_p$ and hence we use $\hat\Pi^{-1}$.)  
By definition, $\bar\tau(\hat\Pi^{-1}\circ L^u (0))=\bar \tau(q)=T'$, and  $(\hat\Pi^{-1}\circ L^u)'(0)=D\imath_p(0,0,\eta_\infty)v_{q}^u$, hence  $\bar\tau(\hat\Pi^{-1}\circ L^u (\tilde\eta))=T'+D\bar\tau(0,\eta_\infty) (\tilde\eta\, v_q^u) +O(\tilde\eta^2)$. As before, we thus conclude that
$$
D\bar \tau (0,\eta_\infty)(v_q^u)=-\mathcal{T}_p^u(\xi_\infty). \qedhere
$$
\end{proof}

Building on the previous lemmata, we will  now give the proof of the main result of this section, namely Proposition~\ref{coro zxp per}, which gives an asymptotic expansion of the period $T_n$ of $p_n$ as $n\to +\infty$. Let us recall the expression of the polynomial introduced in Lemma~\ref{lemma tem st}: 
$$
\tilde P_p^s(\eta):=\sum_{j=1}^{[k]} \frac{\alpha_p^{s,j}(T)}{\mu\lambda^j-1}\eta^j,\quad \alpha_p^{s,j}(T):=-\frac{1}{j!} \partial_{1}\partial_2^j\tau_p^T(0,0).
$$ 
\begin{proof}[Proof of Proposition~\ref{coro zxp per}]
We split the proof into three claims. 
\begin{claim}\label{prem claim prop}
	Recall that $\max(\lambda^{-1},\mu^{\frac{3}{2}})< \nu<\mu$. 
	With the notation of Lemma~\ref{lemme temps templates}, we have:
	$$
	\bar \tau(p_n)=T'+\xi_\infty \mathcal{T}^s_p(\eta_\infty)\mu^n+O(\nu^n). 
	$$ 
\end{claim}
\begin{proof}
	By~\eqref{est xi n eta n},~\eqref{change of xi n ell} and~\eqref{diff tua q q prime} (recall that $c_\infty=\xi_\infty$), we have 
	\begin{align*}
		\bar \tau(p_n)&=\bar \tau(\xi_n,\eta_n)\\
		&=\bar \tau(0,\eta_\infty)+D\bar \tau(0,\eta_\infty) \xi_\infty v_q^s \mu^n + O(\nu^n)\\
		&=T'+\xi_\infty\mathcal{T}_p^s(\eta_\infty)\mu^n+O (\nu^n).\qedhere
	\end{align*}
\end{proof}

Let us recall that for $x \in M$, $t \in \mathbb{R}$, $\tau_x^t(\cdot)$ is the hitting time function associated to the Poincar\'e map from $\Sigma_x$ to $\Sigma_{X^t(x)}$. 
Also let $\ell_n < \frac{nT}{2}$ be the time defined in~\eqref{def helene}. 
\begin{claim}\label{deux claim prop}
	Let 
	\begin{equation*}
	\bar\theta:=\max\left(\lambda^{-1},\mu^{\frac{3}{2}},\mu^{\frac{2\log \lambda}{\log \lambda- \log \mu}}\right)=\max\left(\mu^{\frac 32},\mu^{\frac{2\log \lambda}{\log \lambda- \log \mu}}\right)\in (0,\mu),
	\end{equation*}
	and let $\theta \in (0,1)$ be any number with $\bar \theta<\theta<\mu$. Then, we have 
	$$
	\tau_{p}^{\ell_n T}(X^{-\ell_n T}(p_n)) 
	=\ell_n T-\xi_\infty \tilde P_p^s(\eta_\infty)\mu^n + O(\theta^{n}).
	$$
	%where 
	%$$
	%\theta:=\max\left(\mu^{\frac 32},\mu^{\frac{2\log \lambda}{\log \lambda- \log \mu}}\right)\in (0,\mu). 
	%$$
\end{claim}
\begin{proof}
	It essentially follows from Corollary~\ref{coro good est hitt}. Indeed, by~\eqref{change of tau n ell}, %and since $\mathrm{Jac}_p(T)=\mu \lambda>1$, 
	we have 
	\begin{align}
		&\tau_{p}^{\ell_n T}(X^{-\ell_n T}(p_n)) =\sum_{\ell=1}^{\ell_n}
	\tau_{p}^{T}(\xi_n(-\ell T),\eta_n(-\ell T)) \nonumber\\ 
	&=\sum_{\ell=1}^{\ell_n}\left(T-\xi_\infty P_{p}^s(T)(\eta_\infty \lambda^{-\ell}) \mu^{-\ell}(\mu^n+O(\nu^n ))+
	O(\mu^{2(n-\ell)}).\right) \nonumber\\  
	&=\ell_n T- \xi_\infty \sum_{\ell=1}^{\ell_n}\left(  P_{p}^s(T)(\eta_\infty \lambda^{-\ell}) \mu^{-\ell}\right)\left(\mu^n+O(\nu^n)\right)+O\left(\mu^{2(n-\ell_n)}\right).\label{prem_expr_cla}
	\end{align}
	Since  $P_{p}^s(T)(\eta_\infty \lambda^{-\ell}) \mu^{-\ell}=O\left((\mu\lambda)^{-\ell}\right)$, with $\mu\lambda>1$, we can write 
	\begin{align*}
	\sum_{\ell=1}^{\ell_n}P_{p}^s(T)(\eta_\infty \lambda^{-\ell}) \mu^{-\ell}=\sum_{\ell=1}^{+\infty}\left( P_{p}^s(T)(\eta_\infty \lambda^{-\ell}) \mu^{-\ell}\right)+O\left((\mu\lambda)^{-\ell_n}\right).
	\end{align*}
	By~\eqref{def helene}, we have   
	\begin{equation}\label{calcul theta n}
	(\mu\lambda)^{-\ell_n}\mu^n=O\left((\mu\lambda)^{\frac{\log \mu}{\log \lambda- \log \mu}n}\mu^n\right)=O\left(\mu^{\frac{\log \lambda+ \log \mu}{\log \lambda- \log \mu}n}\mu^n\right)=O\left(\mu^{\frac{2\log \lambda}{\log \lambda- \log \mu}n}\right).
	\end{equation}
	Since $\lambda^{-1}<\mu$, we have  $\mu^{\frac{2\log \lambda}{\log \lambda- \log \mu}}=\lambda^{\frac{2\log \mu}{\log \lambda- \log \mu}}\in (\lambda^{-1},\mu)$, hence 
	$$
	\bar\theta:=\max\left(\lambda^{-1},\mu^{\frac{3}{2}},\mu^{\frac{2\log \lambda}{\log \lambda- \log \mu}}\right)=\max\left(\mu^{\frac 32},\mu^{\frac{2\log \lambda}{\log \lambda- \log \mu}}\right)\in (0,\mu). 
	$$
	By~\eqref{def helene}, we also have
	$$
	\mu^{2(n-\ell_n)}=O\left(\mu^{\frac{2\log \lambda}{\log \lambda-\log \mu}n}\right)=O(\bar \theta^n).
	$$
	Finally, the exact same computation as in the proof of Lemma~\ref{lemma tem st} gives 
	$$
	\sum_{\ell=1}^{+\infty}\left( P_{p}^s(T)(\eta_\infty \lambda^{-\ell}) \mu^{-\ell}\right)+O\left((\mu\lambda)^{-\ell_n}\right)=\sum_{j=1}^{[k]} \frac{\alpha_p^{s,j}(T)}{\mu\lambda^j-1}\eta_\infty^j=\tilde P_p^s(\eta_\infty),
	$$
	with 
	$
	%\alpha_p(T):=-\partial_{12}\tau_p^T(0,0),\quad 
	\alpha_p^{s,j}(T):=-\frac{1}{j!}\partial_{1}\partial_2^j\tau_p^T(0,0)
	$, for $j \in \{1,\cdots,[k]\}$. By~\eqref{prem_expr_cla}, and because $\nu\in (0,1)$ was any number such that  $\max(\lambda^{-1},\mu^{\frac{3}{2}})<\nu<\mu$, this concludes the proof of the claim. 
\end{proof}

\begin{claim}\label{trois claim prop}
	We have 
	$$
	\tau_{p}^{(n-\ell_n)T}(p_n') 
	=nT-\ell_n+O(\theta^n). 
	$$ 
\end{claim}
\begin{proof}
	%Let $\ell_n':=[nT-\ell_n]\in \mathbb{N}$.
	We write
	\begin{equation}\label{eq somme temps}
		\tau_{p}^{(n-\ell_n)T}(p_n') =\sum_{\ell=\ell_n+1}^{n}\tau_{p}^T(\xi_n(-\ell T),\eta_n(-\ell T)).
		%&+\tau_{X^{\ell_n'}(p)}^{nT-\ell_n-\ell_n'}(\xi_n(-nT+\ell_n'),\eta_n(-nT+\ell_n')).\nonumber
	\end{equation}
	By~\eqref{estimee xi n moins t}, for any integer $\ell\in \{\ell_n+1,\cdots,n\}$, we have 
	$$
	(\xi_n(-\ell T),\eta_n(-\ell T))-(\xi_\infty\mu^{n-\ell},0)=O(\lambda^{-\ell}).
	$$
	Since $\tau_{p}^T(\cdot,0)\equiv T$, and $\tau_p^T$ is $C^2$, with $D\tau_p^T(0,0)=0$, we deduce that 
	\begin{align*}
	&\tau_{p}^T(\xi_n(-nT+\ell),\eta_n(-nT+\ell))\\
	&=\tau_{p}^T(\xi_\infty\mu^{n-\ell},0)+O\left(D\tau_{p}^T(\xi_\infty \mu^{n-\ell},0)\lambda^{-\ell}\right)\\
	&=T+O\left((\mu\lambda)^{-\ell}\mu^n\right). 
	\end{align*}
	%Similarly, since $\tau_{X^{\ell_n'}(p)}^t(\cdot,0)\equiv t$, and  $nT-\ell_n-\ell_n'\in [0,1]$, we have 
	%\begin{align*}
	%	&\tau_{X^{\ell_n'}(p)}^{nT-\ell_n-\ell_n'}(\xi_n(nT-\ell_n'),\eta_n(nT-\ell_n'))\\
	%	&=\tau_{X^{\ell_n'}(p)}^{nT-\ell_n-\ell_n'}(\xi_\infty \lambda_p^s(\ell_n'),0)+O\left(D\tau_{X^{\ell_n'}(p)}^{nT-\ell_n-\ell_n'}(\xi_\infty \lambda_p^s(\ell_n'),0)\lambda^{-(n+\frac{\ell_n'}{T})}\right)\\
	%	&=nT-\ell_n-\ell_n'+O(\mu^{\frac {\ell_n'} {T}}\lambda^{-(n+\frac{\ell_n'}{T})})= nT-\ell_n-\ell_n'+O(\mathrm{Jac}_p(-nT+\ell_n'))\mu^n.
	%\end{align*}
	By~\eqref{eq somme temps}, adding up everything, and since  $\mu\lambda>1$, we thus obtain
	\begin{equation*}
		\tau_{p}^{(n-\ell_n)T}(p_n') =nT-\ell_n +O\left((\mu\lambda)^{-\ell_n}\mu^n\right)=nT-\ell_n +O\left(\theta^n\right),%+O\left(\sum_{\ell=0}^{\ell_n'}\mathrm{Jac}_p(-nT+\ell)\right)\mu^n\\
		%&=nT-\ell_n+O(\mathrm{Jac}_p(-\ell_n)\mu^n)= nT-\ell_n+O(\theta^n),
	\end{equation*}
	where the last equality follows from~\eqref{calcul theta n}. 
\end{proof}
Gathering Claims~\ref{prem claim prop}-\ref{deux claim prop}-\ref{trois claim prop}, we can now finish the proof of Proposition~\ref{coro zxp per}.
\begin{align*}
	T_n&=\tau_{p}^{nT}(p_n')+\bar \tau(p_n)\\
	&=\tau_{p}^{(n-\ell_n)T}(p_n')+ \tau_{p}^{\ell_n T}(X^{-\ell_n T}(p_n)) +\bar \tau(p_n)\\
	&=nT-\ell_n+O(\theta^n)+ \ell_n-\xi_\infty \tilde P_p^s(\eta_\infty)\mu^n + O(\theta^{n})+ T'+\xi_\infty \mathcal{T}^s_p(\eta_\infty)\mu^n+O(\nu^n)\\
	&=nT+T'+\xi_\infty\left(\mathcal{T}^s_p(\eta_\infty)-\tilde P_p^s(\eta_\infty)\right)\mu^n +O(\theta^n).\qedhere
\end{align*}
\end{proof}

\subsection{Second order asymptotic formula for mildly dissipative periodic points}
\label{sec_AF}
Here we derive a more precise second order asymptotic formula, but only for mildly dissipative periodic orbits, when the flow $X^t$ is of class $C^r$, $r \geq 4$. This formula will be used in Sections~\ref{sec_74} and~\ref{sec_73}, where Theorem~\ref{strong dllave flow} and Addendum~\ref{add alph} are proved.

Recall that for any periodic orbit $\gamma \in \mathcal{P}$ for $X^t$, we denote by $\mu_\gamma<1<\lambda_\gamma$ its stable and unstable multipliers. Recall Definition~\ref{defi d pinched}, where we have defined $\varrho$-mildly dissipative Anosov flows; now we need to define the same property for periodic orbits. 
\begin{definition}[$\varrho$-mildly dissipative periodic orbits] For any $\varrho\in(1,2]$, the set  $\mathcal{P}^\varrho\subset \mathcal{P}$ of  $\varrho$-mildly dissipative periodic orbits is defined as follows:
	$$
	\mathcal{P}^\varrho:=\left\{\gamma\in \mathcal{P}: \mu_\gamma^\varrho \lambda_\gamma<1\text{ and }\mu_\gamma\lambda_\gamma^\varrho>1\right\}.
	$$ 
	For any $\gamma \in \mathcal{P}^\varrho$, and any point $p \in \gamma$ we also say that $p$ is $\varrho$-mildly dissipative. 
\end{definition}

Recall that the flow $X^t$ is $k$-pinched for some $k$, $1<k\le r-1$.

\begin{proposition}\label{prop mild dis exp}
	If the periodic point $p$ is volume expanding and $\frac{5}{4}$-mildly dissipative, with eigenvalues $\mu<1< \lambda$, then as  $n \to +\infty$, the period $T_n$ of the periodic point $p_n$ has the following asymptotic expansion:
	\begin{equation*}
		T_n=nT+T'+\left(\mathcal{T}_p^s(\eta_\infty)-%\frac{\partial_{12}\tau_p^T(p) }{\mu\lambda-1}\eta_\infty
		\tilde P_p^s(\eta_\infty)\right)\xi_\infty\mu^{n}-\left(\mathcal{T}_p^u(\xi_\infty)-\tilde P_p^u(\xi_\infty)
		%+\frac{\partial_{12}\tau_p^T(p) }{\mu\lambda-1}\xi_\infty
		\right)\eta_\infty\lambda^{-n}+O(\theta^{n}),
	\end{equation*}
	where $\theta:=\mu^{\frac{5}{4}}\in (0,\lambda^{-1})$, and 
	$$
	\tilde P_p^s(\eta):=-\sum_{j=1}^{[k]} \frac{1}{j!}\frac{ \partial_{1}\partial_2^j\tau_p^T(p)}{\mu\lambda^j-1}\eta^j,\quad
	\tilde P_p^u(\xi):=-\sum_{j=1}^{[k]}\frac{1}{j!} \frac{\partial_1^j\partial_{2}\tau_p^T(p)}{\mu^j \lambda-1}\xi^j. 
	$$ 
\end{proposition}

\begin{remark} By optimizing the choice of exponents in the proof which follows one can relax the assumption on mild dissipation and obtain the same formula for $\frac{1+\sqrt 3}{2}$-mildly dissipative volume expanding periodic points. We also believe that with some more efforts one can deduce the same formula for 2-mildly dissipative volume expanding periodic points. 
\end{remark}

Below we explain how the previous arguments can be adapted to prove Proposition~\ref{prop mild dis exp}. 
Let $p=X^T(p)$ be a volume expanding periodic point which is $\frac{5}{4}$-mildly dissipative, i.e., the multipliers $\mu<1<\lambda$ of $p$ satisfy 
\begin{equation}\label{pinching_ helene_milddis}
\mu\lambda>1,\quad \mu^{\frac{5}{4}}\lambda<1. 
\end{equation}
As in~\eqref{def helene}, let 
\begin{equation}\label{def helene_milddis}
	\frac{4n}{9}\lesssim\ell_n:=\left[\frac{\log \mu}{\log \mu-\log \lambda}n\right]\lesssim\frac{n}{2}.
\end{equation}
 
In that case, similarly to estimate~\eqref{est xi n eta n} in Lemma~\ref{est unifm}, and because we assume here that $X^t$ is $C^4$, by~\cite{Stowe}, we can consider $C^2$ linearizing coordinates and obtain the following expansions.
\begin{lemma}\label{est unifm_mild_dis} 
	For some constants $c_\infty,c_\infty',c_\infty'',c_\infty''' \neq 0$, we have the asymptotic formulae
	\begin{equation}\label{est xi n eta n_mild_dis}
		\begin{array}{rcl}
		(\xi_n,\eta_n)-(0,\eta_\infty)&=&c_\infty \mu^n(1,\kappa^s)+c_\infty'\lambda^{-n}(0,1)+O(\mu^{2n}),\\ 
		(\xi_n',\eta_n')-(\xi_\infty,0)&=&c_\infty''\mu^n(1,0)+c_\infty'''\lambda^{-n}(\kappa^u,1) +O(\mu^{2n}).
		\end{array}
	\end{equation} 
\end{lemma}

In fact, arguing as in the proof of Corollary~\ref{coro good est hitt}, we can show that $c_\infty=\xi_\infty$ and $c_\infty'''=\eta_\infty$. 
We will need the following more precise version of 
Lemma~\ref{differentielle pi x un}. 
\begin{lemma}\label{differentielle pi x un_fort}
	For any point $x\in M$,  
	and $\sigma\in \mathbb{R}$, 
	let us denote by $ \Pi_x^{\sigma}\colon \Sigma_x \to \Sigma_{X^{\sigma}(x)}$ the Poincaré map of the flow $X^t$ from $\Sigma_x$ to $\Sigma_{X^\sigma(x)}$, and let $\hat \Pi_x^{\sigma}:=\imath_{X^{\sigma}(x)}^{-1}\circ \Pi_x^{\sigma}\circ \imath_{x}$  be its expression in normal coordinates. For $(\xi,\eta)\in (-1,1)$, with $|\xi|\ll 1$, we have  
	$$
	D\hat \Pi_{p}^{-T}(\xi,\eta)=\begin{bmatrix}
		\mu^{-1} & 0\\
	     O(\eta) & \lambda^{-1}
	\end{bmatrix}+O(\xi).
	$$
	Similarly, for $(\xi,\eta)\in (-1,1)$, with $|\eta|\ll 1$, we have  
	$$
	D\hat \Pi_{p}^{T}(\xi,\eta)=\begin{bmatrix}
		\mu & O(\xi)\\
		0 & \lambda
	\end{bmatrix}+O(\eta).
	$$
\end{lemma}
\begin{proof}
	Let us focus on the first case; the second one is shown analogously.
	As in the proof of Lemma~\ref{differentielle pi x un}, this essentially follows from Proposition~\ref{propo o good}.  In particular, for $\xi=0$, the matrix of the differential $D\hat \Pi_{p}^{-T}(0,\eta)$ is lower triangular, where the only non-normalized coefficient, namely the bottom-left one, is of order $O(\eta)$. Indeed, since the dynamics is normalized along the axes $\{\xi=0\}$ and $\{\eta=0\}$, the matrix at  $(\xi,\eta)=(0,0)$ is diagonal. 
\end{proof}

Fix a constant $\nu\in (\mu^{\frac 32},\lambda^{-1})$. 
Arguing as we did in the proof of Corollary~\ref{coro good est hitt}, starting from~\eqref{est xi n eta n_mild_dis}, and thanks to Lemma~\ref{differentielle pi x un_fort}, we can obtain the following expansions.
\begin{corollary}\label{coro good est hitt_mild_dis} 
	For any integer $\ell \in \{0,\cdots,n\}$,  we have 
	\begin{equation}\label{change of xi n ell_milddssip}
		\begin{array}{ll}
		\left(\xi_n(-\ell T),\eta_n(-\ell T)\right)=\left(\xi_\infty \mu^{n-\ell}+O(\nu^{n}\mu^{-\ell}),\eta_\infty \lambda^{-\ell}+O(\mu^{n})\right),&\quad \forall\, \ell \leq \ell_n,\\
		\left(\xi_n(-\ell T),\eta_n(-\ell T)\right)=\left(\xi_\infty \mu^{n-\ell}+O(\mu^n),%\max(\mu^{2n-\ell},\lambda^{-\ell})),
		\eta_\infty \lambda^{-\ell}+O(\nu^{n} \lambda^{n-\ell})\right), &\quad\forall\, \ell \geq \ell_n. 
	\end{array}
	\end{equation}
	%Fix %$\bar \nu \in (\nu,\lambda^{-1})$
	%$$
	%\ell_n^-:=\left\lfloor\frac{n}{3}\right\rfloor\leq \ell_n. 
	%$$
	Therefore, for any integer $\ell \in \{0,\cdots,\ell_n\}$,  we have
	\begin{equation}\label{change of tau n ell_mild_dis}
		\tau_{p}^{T}(\xi_n(-\ell T),\eta_n(-\ell T))=T-\xi_\infty P_{p}^s(T)(\eta_\infty \lambda^{-\ell}) \mu^{n-\ell}+O(\mu^{\frac{4}{3}n}),
	\end{equation} 
	%for any integer $\ell \in \{\ell_n^-,\cdots,\ell_n\}$,  we have
	%\begin{equation}\label{change of tau n ell_mild_dis_ter}
	%	\tau_{p}^{T}(\xi_n(-\ell T),\eta_n(-\ell T))=T-\alpha_p \xi_\infty \eta_\infty  (\mu\lambda)^{-\ell}\mu^n+O(\nu^n),
		%+O(\mu^{n-\ell}\lambda^{-2\ell}),
	%\end{equation} 
	%where $\alpha_p=\alpha_p ^{s,1}(T)=\alpha_p ^{u,1}(T)=-\partial_{12} \tau_p^T(0,0)$; 
	%for any $\ell \in \{\ell_n,\cdots,\ell_n^+\}$,  we have
	%\begin{equation}\label{change of tau n ell_mild_dis_quat}
	%	\tau_{p}^{T}(\xi_n(-\ell T),\eta_n(-\ell T))=T-\alpha_p \xi_\infty \eta_\infty \mu^n (\mu\lambda)^{-\ell}+ O(\mu^{2(n-\ell)}\lambda^{-\ell}),
	%\end{equation} 
	and for any integer $\ell \in \{\ell_n,\cdots,n\}$,  we have
	\begin{equation}\label{change of tau n ell_mild_dis_bis}
		\tau_{p}^{T}(\xi_n(-\ell T),\eta_n(-\ell T))=T-P_p^u(T)(\xi_\infty \mu^{n-\ell})\eta_\infty\lambda^{-\ell} +O(\mu^{\frac{4}{3}n}). 
		%\tau_{p}^{T}(\xi_n(-\ell T),\eta_n(-\ell T))=T-\eta_\infty P_{p}^u(T)(\xi_\infty \mu^{n-\ell}) \lambda^{-\ell}(1+O(\nu^n \lambda^n))+
		%O(\lambda^{-2\ell}). 
	\end{equation} 
\end{corollary}

\begin{proof}
	Let $\ell \in \{0,\cdots,\ell_n\}$. 
	Since $\tau_p^T$ is $C^3$, by Taylor formula, we have the following expansion for the hitting times
	\begin{align*}
		\tau_{p}^{T}(\xi_n(-\ell T),\eta_n(-\ell T))&=\tau_{p}^{T}(0,\eta_n(-\ell T))+\partial_1\tau_{p}^{T}(0,\eta_n(-\ell T))\xi_n(-\ell T)\\
		&+O\left(\sup_{|\xi|\leq\xi_n(-\ell T)} \partial_{11}\tau_{p}^{T}(\xi,\eta_n(-\ell T))\xi_n(-\ell T)^2\right)\\
		&=T-P_p^s(T)(\eta_n(-\ell T))\xi_n(-\ell T)+O\left(\eta_n(-\ell T)\xi_n(-\ell T)^2\right),
	\end{align*}
	where we have used that 
	$
	\partial_{11}\tau_{p}^{T}(\xi,0)=0
	$. 
	By~\eqref{change of xi n ell_milddssip}, we then obtain
	\begin{align*}
		&\tau_{p}^{T}(\xi_n(-\ell T),\eta_n(-\ell T))=T-P_p^s(T)(\eta_n(-\ell T))\xi_n(-\ell T)+O\left(\eta_n(-\ell T)\xi_n(-\ell T)^2\right)\\
		&=T-P_p^s(T)(\eta_\infty \lambda^{-\ell})\xi_\infty\mu^{n-\ell} +O\left(\max\left(\mu^{2n-\ell},\nu^n (\mu\lambda)^{-\ell},\mu^{n-\ell}\nu^n,\mu^{2(n-\ell)}\lambda^{-\ell}\right)\right)\\
		&=T-P_p^s(T)(\eta_\infty \lambda^{-\ell})\xi_\infty\mu^{n-\ell} +O(\nu^n). 
	\end{align*}
	Since $\nu$ can be chosen arbitrarily close to $\mu^{\frac 32}$, we can assume that $\nu < \mu^{\frac{4}{3}}$. 
	Let us now consider the case where $\ell \in \{\ell_n,\cdots,n\}$. 
	By Taylor formula, we have 
	\begin{align*}
		\tau_{p}^{T}(\xi_n(-\ell T),\eta_n(-\ell T))&=\tau_{p}^{T}(\xi_n(-\ell T),0)+\partial_2\tau_{p}^{T}(\xi_n(-\ell T),0)\eta_n(-\ell T)\\
		&+O\left(\sup_{|\eta|\leq\eta_n(-\ell T)} \partial_{22}\tau_{p}^{T}(\xi_n(-\ell T),\eta)\eta_n(-\ell T)^2\right)\\
		&=T-P_p^u(T)(\xi_n(-\ell T))\eta_n(-\ell T)+O\left(\xi_n(-\ell T)\eta_n(-\ell T)^2\right),
	\end{align*}
	where we have used that 
	$
	\partial_{22}\tau_{p}^{T}(0,\eta)=0
	$. 
	By~\eqref{change of xi n ell_milddssip}, we then obtain
	\begin{align*}
		&\tau_{p}^{T}(\xi_n(-\ell T),\eta_n(-\ell T))=T-P_p^u(T)(\xi_n(-\ell T))\eta_n(-\ell T)+O\left(\xi_n(-\ell T)\eta_n(-\ell T)^2\right)\\
		&=T-P_p^u(T)(\xi_\infty \mu^{n-\ell})\eta_\infty\lambda^{-\ell} +O\left(\max\left(\mu^{n}\lambda^{-\ell},\nu^n (\mu\lambda)^{n-\ell},\mu^{n}\nu^n \lambda^{n-\ell},\mu^{n-\ell}\lambda^{-2\ell}\right)\right)\\
		&=T-P_p^u(T)(\xi_\infty \mu^{n-\ell})\eta_\infty\lambda^{-\ell} +O(\nu^n (\mu\lambda)^{n-\ell_n}). 
	\end{align*}
	Since $\nu$  can be chosen arbitrarily close to $\mu^{\frac 32}$, by~\eqref{pinching_ helene_milddis}-\eqref{def helene_milddis}, and since $\mu^{n-\ell_n}\simeq \lambda^{-\ell_n}$, we can ensure that 
	$$
	\nu^n(\mu\lambda)^{n-\ell_n}< \mu^{\frac{1}{4}n} \mu^{\frac{5}{2}\ell_n}< \mu^{\frac{4}{3}n}<\lambda^{-n}. \qedhere
	$$
	%Estimate~\eqref{change of tau n ell_mild_dis} is shown as in Corollary~\ref{coro good est hitt}, while estimate~\eqref{change of tau n ell_mild_dis_bis} is shown similarly, based on the second line of~\eqref{change of xi n ell_milddssip}. To obtain~\eqref{change of tau n ell_mild_dis_ter}, we note that for $\ell \in \{\ell_n^-,\cdots,\ell_n\}$, we have $\xi_n(-\ell T)=\mu^{n-\ell}=O(\eta_n(-\ell T)) =O(\lambda^{-\ell})$; since $\tau_p^T$ is $C^3$, and  $\partial_1\tau_{p}^{T}(0,0)=\partial_{11}\tau_{p}^{T}(0,0)=\partial_2\tau_{p}^{T}(0,0) =\partial_{22}\tau_{p}^{T}(0,0) =0$, we then have 
%	\begin{align*}
	%	&\tau_{p}^{T}(\xi_n(-\ell T),\eta_n(-\ell T))=\tau_{p}^{T}(0,0)+\partial_{12}\tau_{p}^{T}(0,0) \xi_n(-\ell T)\eta_n(-\ell T)+O(\mu^{n-\ell}\lambda^{-2\ell})\\
	%	&=T-\alpha_p \xi_\infty \eta_\infty \mu^{n}(\mu\lambda)^{-\ell}+O(\max(\mu^{2n-\ell},\nu^n(\mu\lambda)^{-\ell},\nu^n \mu^{n-\ell}, \mu^{n-\ell}\lambda^{-2\ell}))\\
	%	&=T-\alpha_p \xi_\infty \eta_\infty \mu^{n}(\mu\lambda)^{-\ell}+O(\nu^n),
	%\end{align*}
	%given our choice of $\ell_n^-$. 
	%Estimate~\eqref{change of tau n ell_mild_dis_quat} is shown in the same way. 
\end{proof}
\begin{proof}[End of the proof of Proposition~\ref{prop mild dis exp}]
We argue as in Claims~\ref{prem claim prop}-\ref{deux claim prop}-\ref{trois claim prop}. 
By~\eqref{change of tau n ell_mild_dis}, we have 
\begin{align*}
	&\tau_{p}^{\ell_n T}(X^{-\ell_n T}(p_n)) =\sum_{\ell=1}^{\ell_n}
	\tau_{p}^{T}(\xi_n(-\ell T),\eta_n(-\ell T)) \nonumber\\ 
	&=\ell_n T- \xi_\infty \sum_{\ell=1}^{\ell_n}\left(  P_{p}^s(T)(\eta_\infty \lambda^{-\ell}) \mu^{n-\ell}\right)+O(n\mu^{\frac 43 n}). 
\end{align*}
Let us write $P_p^s(T)(\eta)=\alpha_p \eta + R_p^s(T)(\eta)$, where $\alpha_p=\alpha_p^{s,1}(T)=\alpha_p^{u,1}(T)=-\partial_{12} \tau_p^T(0,0)$, and $R_p^s(T)(\eta)=O(\eta^2)$. We deduce that
\begin{align}\label{final_form_mild_dis}
	\tau_{p}^{\ell_n T}(X^{-\ell_n T}(p_n))&=\ell_n T-  \alpha_p \xi_\infty \eta_\infty \mu^n  \frac{1-(\mu\lambda)^{-\ell_n}}{\mu \lambda-1}-\xi_\infty \sum_{\ell=1}^{+\infty}\left(  R_{p}^s(T)(\eta_\infty \lambda^{-\ell}) \mu^{n-\ell}\right) \nonumber\\
	&+O(\mu^{n-\ell_n}\lambda^{-2\ell_n})+O(n\mu^{\frac 43 n})\nonumber\\
	&=\ell_n T-\xi_\infty \tilde{P}_{p}^s(\eta_\infty) \mu^{n}+\alpha_p \xi_\infty \eta_\infty \frac{\mu^{n-\ell_n} \lambda^{-\ell_n}}{\mu \lambda-1} + O(\theta^n),
\end{align}
where we recall that $\theta:=\mu^{\frac 54}<\lambda^{-1}$. 

Similarly, by~\eqref{change of tau n ell_mild_dis_bis}, we have 
 \begin{align*}
 	\tau_{p}^{(n-\ell_n) T}(p_n')&=\tau_{p}^{(n-\ell_n) T}(X^{-nT}(p_n)) =\sum_{\ell=\ell_n+1}^{n}
 	\tau_{p}^{T}(\xi_n(-\ell T),\eta_n(-\ell T)) \nonumber\\ 
 	&=(n-\ell_n) T- \eta_\infty \sum_{\ell=\ell_n+1}^{n}\left(  P_{p}^u(T)(\xi_\infty \mu^{n-\ell}) \lambda^{-\ell}\right)+O(n\mu^{\frac 43 n}). 
 \end{align*}
We deduce that
\begin{align}\label{final_form_mild_dis_bis}
	\tau_{p}^{(n-\ell_n) T}(p_n')&=(n-\ell_n) T-  \eta_\infty\lambda^{-n} \sum_{j=1}^{[k]}\alpha_p^{u,j}(T) \xi_\infty^j  \frac{1-(\mu^j \lambda)^{n-\ell_n}}{1-\mu^j \lambda} 
	+O(n\mu^{\frac 43 n}) \nonumber\\
	&=(n-\ell_n) T-\eta_\infty\lambda^{-n}\sum_{j=1}^{[k]}\alpha_p^{u,j}(T) \xi_\infty^j  \frac{1}{1-\mu^j \lambda} \nonumber\\
	&-  \alpha_p \xi_\infty \eta_\infty   \frac{\mu^{n-\ell_n}\lambda^{-\ell_n}}{\mu \lambda-1}
	+O(\mu^{2(n-\ell_n)}\lambda^{-\ell_n})+O(n\mu^{\frac 43 n}) \nonumber\\
	&=(n-\ell_n) T+\eta_\infty \lambda^{-n} \tilde{P}_p^u(\xi_\infty)-\alpha_p \xi_\infty \eta_\infty   \frac{\mu^{n-\ell_n}\lambda^{-\ell_n}}{\mu \lambda-1}+O(\theta^n),
\end{align}
where (recall~\eqref{derivee_taus}) 
$$
\tilde{P}_p^u(\xi)=\sum_{j=1}^{[k]}   \frac{\alpha_p^{u,j}(T)}{\mu^j \lambda-1}\xi^j=-\sum_{j=1}^{[k]}\frac{1}{j!} \frac{\partial_1^j\partial_{2}\tau_p^T(0,0)}{\mu^j \lambda-1}\xi^j. 
$$ 
Finally, by~\eqref{diff tua q q prime} and~\eqref{est xi n eta n_mild_dis}, the estimate in Claim~\ref{prem claim prop} can be improved under the mildly dissipative assumption in the following way.
$$
	\bar \tau(p_n)=T'+\xi_\infty \mathcal{T}^s_p(\eta_\infty)\mu^n-\eta_\infty \mathcal{T}_p^u(\xi_\infty) \lambda^{-n}+O(\mu^{2n}). 
$$  
To conclude the proof of Proposition~\ref{prop mild dis exp}, it remains to add to $\bar \tau(p_n)$ the expressions obtained in~\eqref{final_form_mild_dis}-\eqref{final_form_mild_dis_bis}. 
\end{proof}

\subsection{Asymptotic formula in the volume preserving case}

Let us now assume that the Anosov flow $X^t$ is $C^{r}$, $r \geq 3$, and volume preserving. In particular, $X^t$ is $2$-pinched in the sense of Definition~\ref{defi d pinched}, and by  Lemma~\ref{lemme lien}, for any time $\sigma\in \mathbb{R}$, the polynomials $P_x^s(\sigma),P_x^u(\sigma)$ are merely linear maps:
\begin{equation}\label{pol vol pre}
	P_x^s(\sigma)(\eta)\equiv \partial_{12}\tau_x^\sigma(x)\eta,\quad P_x^u(\sigma)(\xi)\equiv \partial_{12}\tau_x^\sigma(x)\xi. 
\end{equation}
Let $p \in M$ be a periodic point, with period $T>0$ and multipliers $0<\mu<1<\mu^{-1}$. As above we consider a sequence of periodic points $(p_n)_{n \geq n_0}$ whose orbits shadow some orbit homoclinic to $p$. 

\begin{proposition}\label{prop volume pre point}
	As $n\to +\infty$, the period $T_n$ of the periodic point $p_n$ has the following asymptotic expansion:
	\begin{equation*}%\label{asymp vol pre} 
		T_n=nT+T'+\xi_\infty\eta_\infty \partial_{12}\tau_p^T(p) n\mu^{n}+\left(\xi_\infty \mathcal{T}^s_p(\eta_\infty)-\eta_\infty \mathcal{T}^u_p(\xi_\infty)\right)\mu^n+o(\mu^n),
	\end{equation*}
	where $\xi_\infty\eta_\infty\neq 0$. 
	
	Moreover, if $X^t$ is $C^r$, $r>3$, then the remainder is of order $O(\nu^n)$, with $\nu\in (0,\mu)$. 
\end{proposition}

\begin{proof}[Outline of the proof]
We will not give a detailed proof of Proposition~\ref{prop volume pre point} as it is similar to the proofs of Proposition~\ref{coro zxp per}, resp. Proposition~\ref{prop mild dis exp}, at a volume expanding, resp. $\frac{5}{4}$-midly dissipative volume expanding periodic point $p$. We will comment below how the previous statements can be adapted to the volume preserving case. We use the same notation as in Subsection~\ref{section asymp ex}. 

By~\cite{Stowe}, the dynamics near $p$ can be $C^\frac{3}{2}$-linearized, and the formulae in Lemma~\ref{est unifm} are to be replaced with:  
\begin{align*}%\label{est xi n eta n bis}
	(\xi_n,\eta_n)-(0,\eta_\infty)&=\left(\xi_\infty (1,\kappa^s)+\tilde c_\infty(0,1)\right)\mu^n +O(\mu^{\frac 32 n}),\\
		(\xi_n',\eta_n')-(\xi_\infty,0)&=\left(\tilde c_\infty' (1,0)+\eta_\infty(\kappa^u,1)\right)\mu^n +O(\mu^{\frac 32 n}),
\end{align*}
for some constants $\tilde c_\infty,\tilde c_\infty'\neq 0$. %Note that here, the constants in front of the vectors $v_q^s=(1,\kappa^s)$ and $v_{q'}^u=(\kappa^u,1)$ can be explicited  in terms of $\xi_\infty,\eta_\infty$. 
%and  $\nu:=\max(\lambda^{-1},\mu^{\frac 32})\in (0,\mu)$. 
Moreover, for any time $t\in \left[0,nT\right]$, %uch that $\mu^{n-\frac t T}\leq \lambda^{-\frac{t}{T}}$, 
we have  
\begin{equation*}%\label{estimee xi n moins t bis}
	\begin{array}{rll}
		&(\xi_n(-t),\eta_n(-t))-(0,\eta_\infty  \lambda_p^u(-t))=O(\mu^{n-\frac{t}{T}}),&\text{if }t\leq \frac{nT}{2},\\
		&(\xi_n(-t),\eta_n(-t))-(\xi_\infty \lambda_p^s(nT-t),0)=O(\mu^{\frac{t}{T}}),&\text{if } t\geq \frac{nT}{2}.
	\end{array}
\end{equation*}
%Now, similarly to Lemma~\ref{differentielle pi x un}, for any point $x\in M$, 
%the Poincaré map $\hat \Pi_x$ and its inverse $\hat \Pi_x^{-1}$ expressed in normal coordinates satisfy
%\begin{align*}
%D\hat F_{x,1}^{-1}(0,\eta)&=\lambda_x^s(-1)d \xi,\qquad \forall\,\eta \in (-1,1),\\
%D\hat F_{x,1}(\xi,0)&=\lambda_x^u(1)d \eta,\qquad \forall\,\xi\in (-1,1). 
%\end{align*}
Then, the estimates in Corollary~\ref{coro good est hitt} now become: given $\nu \in (\mu^{\frac{3}{2}},\mu)$, then 
for any integer $\ell \leq \ell_n$, with $\ell_n:=\left[\frac{n}{2}\right]$,  we have 
\begin{align*}
	\xi_n(-\ell T)&=\xi_\infty \mu^{n-\ell}+O(\nu^n \mu^{-\ell}), %\quad \forall\, \ell \leq \frac{nT}{2},
	\\
	\eta_n(-(n-\ell)T)&=\eta_\infty \mu^{n-\ell}+O(\nu^n \mu^{-\ell}). %\quad \forall\, \ell \leq \frac{nT}{2}.
\end{align*}
Using~\eqref{pol vol pre}, we thus have 
\begin{align*}%\label{change of tau n ell}
	\tau_{p}^{T}(\xi_n(-\ell T),\eta_n(-\ell T))&=1+\xi_\infty\eta_\infty \partial_{12}\tau_{p}^T(p)\mu^n+O(\nu^n)+o(\mu^{2(n-\ell)}),\\ %\quad \forall\, \ell \leq \frac{nT}{2},\\
	\tau_{p}^{T}(\xi_n(-(n-\ell)T),\eta_n(-(n-\ell)T))&=1+\xi_\infty\eta_\infty \partial_{12}\tau_{p}^T(p)\mu^n+O(\nu^n)+o(\mu^{2(n-\ell)}). %\quad \forall\, \ell \geq \frac{nT}{2}. 
\end{align*}

Note that Lemma~\ref{lemme temps templates} applies in the volume preserving case, as it is insensitive to whether or not the point $p$ is dissipative. 

Gathering the previous estimates, and following the steps of Proposition~\ref{coro zxp per}, we obtain the following asymptotic expansions:
	\begin{align*}
	\bar \tau(p_n)&=T'+\left(\xi_\infty \mathcal{T}^s_p(\eta_\infty)-\eta_\infty \mathcal{T}^u_p(\xi_\infty)\right)\mu^n+O(\mu^{\frac 32 n}),\\
	\tau_{p}^{\ell_n T}(X^{-\ell_n T}(p_n)) 
	&=\ell_n T+\ell_n \xi_\infty \eta_\infty  \partial_{12}\tau_{p}^T(p)\mu^n + o(\mu^n),\\
	%&=\ell_n T+\xi_\infty \eta_\infty \partial_{12}\tau_{X^{-\ell_n}(p)}^{\ell_n}(X^{-\ell_n}(p))\mu^n + O(n\mu^{\frac 32 n}),\\
	\tau_{p}^{(n-\ell_n)T}(p_n') 
	&=(n-\ell_n)T+(n-\ell_n) \xi_\infty \eta_\infty  \partial_{12}\tau_{p}^T(p)\mu^n + o(\mu^n).%\xi_\infty \eta_\infty\left(\sum_{\ell=0}^{\ell_n-1}\ \partial_{12}\tau_{X^{\ell}(p)}^1(X^{\ell}(p))\right)\mu^n\\
	%&\quad +\xi_\infty\eta_\infty\partial_{12}\tau_{X^{\ell_n}(p)}^{nT-2\ell_n}(X^{\ell_n}(p))\mu^n+ O(n\mu^{\frac 32 n})\\
	%&=(nT-\ell_n)+\xi_\infty \eta_\infty \partial_{12}\tau_{p}^{nT-\ell_n}(p)\mu^n + O(n\mu^{\frac 32 n}).
	\end{align*}  
	%Note that we have used that in the volume preserving setting, the map $M \times \mathbb{R} \ni (x,t)\mapsto \partial_{12} \tau_x^t(x)$ is an additive cocycle (see Remark~\ref{remark twist coc}). 
	We conclude the proof of Proposition~\ref{prop volume pre point} by adding up the above expansions. %; indeed, still by the cocycle property, we have 
	%$$
	%\partial_{12}\tau_{X^{-\ell_n}(p)}^{\ell_n}(X^{-\ell_n}(p))+\partial_{12}\tau_{p}^{nT-\ell_n}(p)=\partial_{12}\tau_{p}^{nT}(p)=n\partial_{12} \tau_p^T(p). \qedhere
	%$$
\end{proof}

\begin{remark}\label{remarque_contact}
Let us consider the case where $X^t$ is a contact flow. Fix a periodic point $p=X^T(p)$ and homoclinic points $q$, $q'=X^{T'}(q)$ as above. As observed by Foulon-Hasselblatt~\cite{FH}, the coefficient $\partial_{12}\tau_p^T(p)$ in Proposition~\ref{prop volume pre point} vanishes; but we claim that the contact property ensures that the coefficient $(\xi_\infty \mathcal{T}^s_p(\eta_\infty)-\eta_\infty \mathcal{T}^u_p(\xi_\infty))$ of the next term in the asymptotic formula above is non-zero. Indeed, after possibly replacing $q,q'$ with $X^{-\ell T}(q),X^{\ell' T}(q')$ for some large integers $\ell,\ell' \gg 1$, without loss of generality, we can assume that $0<|\xi_\infty|\approx |\eta_\infty|\ll 1$ are small. We claim that the coefficient 
$$
\tilde \zeta_p(q):=\big(\xi_\infty \mathcal{T}^s_p(\eta_\infty)-\eta_\infty \mathcal{T}^u_p(\xi_\infty)\big)
$$ is non-zero. Indeed, since $E^s, E^u$ are $C^{1+\alpha}$, $\alpha>0$, so are the stable and unstable templates, and we can expand $\tilde \zeta_p(q)=\left((\mathcal{T}^s_p)'(0)-(\mathcal{T}^u_p)'(0)\right)\xi_\infty \eta_\infty+o(\xi_\infty \eta_\infty)$. On the one hand, $\left((\mathcal{T}^s_p)'(0)-(\mathcal{T}^u_p)'(0)\right)\xi_\infty \eta_\infty$ is the first order approximation of the temporal displacement between $\mathcal{W}_{\mathrm{loc}}^s(q)$ and $\mathcal{W}_{\mathrm{loc}}^u(q')$. On the other hand, by the contact property, it is also the first order approximation of the transverse area, which has to be of order $\xi_\infty \eta_\infty$. We deduce that $(\mathcal{T}^s_p)'(0)-(\mathcal{T}^u_p)'(0)\neq 0$, hence $\tilde \zeta_p(q)\neq 0$. 
\end{remark}

\label{sec_vp}

\section{On asymptotic proportions of certain periodic points}\label{sec pos livs}

In the following, we always assume that $X^t\colon M \to M$ is a transitive Anosov flow. 

\subsection{Pressure, equilibrium states and SRB measures}\label{sec_51}

Let us recall that, given a H\"older continuous function $\psi$, the pressure $P(\psi)\in \mathbb{R}$ is defined by 
$$
P_X(\psi):=\sup_{\mu \in \mathcal{M}(X^t)} \left(h_\mu(X^t)+\int_M \psi\, d \mu\right),
$$
where $\mathcal{M}(X^t)$ is the set of Borel invariant probability measures of $X^t$, and $h_\mu(X^t)$ denotes the metric entropy of the time-one map of the flow $X^t$ with respect to $\mu$. The equilibrium state $\mu_\psi\in \mathcal{M}(X^t)$ associated to $\psi$ is  the unique measure in $\mathcal{M}(X^t)$ on which the above supremum is achieved. It is known to be ergodic and fully supported. 

We have defined the SRB measures in Definition~\ref{def_SRB} and now we need to recall few more well-known facts about them. The positive SRB measure $m_X^+$ of a transitive Anosov flow $X^t$ can also be characterized as the unique equilibrium state of the geometric potential defined by
$$
\psi^u\colon x \mapsto - \frac{d}{dt}\Big |_{t=0} \log \|DX^t(x)|_{E^u}\|,
$$
 while $m_X^-$  is the equilibrium state for the potential 
 $$\psi^s\colon x \mapsto  \frac{d}{dt}\Big|_{t=0} \log \|DX^t(x)|_{E^s}\|.
 $$ 
 Also recall that $P_X(\psi^u)=P_X(\psi^s)=0$~\cite{bowen}.

\subsection{Positive proportion Livshits Theorem of Marshall Reber and Dilsavor}

We denote by $\mathcal{P}$ the set of periodic orbits of $X^t$, and given a periodic orbit $\gamma \subset \mathcal{P}$, we denote by $T(\gamma)>0$ its period. %and we let $\bar \gamma\subset M$ be the subset of $M$ consisting of points in $\gamma$.
For any Hölder function $\psi\colon M \to \mathbb{R}$, and for any $\gamma \in \mathcal{P}$, we let 
$$
T_\psi(\gamma):=\int_0^{T(\gamma)} \psi(\gamma(s))\, ds.
$$ 
We say that $\psi$ is a coboundary if there exists a H\"older continuous $\kappa \colon M \to \mathbb{R}$ (which is smooth along the flow) such that 
$$
\psi=\frac{d}{dt}\Big|_{t=0}\left(\kappa \circ X^t\right). 
$$

Fix a Hölder function $\psi \colon M \to \mathbb{R}$, and a positive number $\Delta>0$. For any $T>0$, and for any subset $\mathcal{S}\subset \mathcal{P}$, we denote by $\mathcal{S}_{T,\Delta}\subset \mathcal{S}$ the set of periodic orbits $\gamma\in \mathcal{S}$ with $T(\gamma)\in (T,T+\Delta]$, and we let 
$$
\Sigma_{\psi,T,\Delta}(\mathcal{S}):=\sum_{\gamma \in \mathcal{S}_{T,\Delta}} e^{T_\psi(\gamma)}.
$$

It is well-known that
$$
P(\psi)=\lim_{T\to +\infty} \frac{1}{T}\log \Sigma_{\psi,T,\Delta}(\mathcal{P}),
$$
and if $P(\psi)\geq 0$, we also have the Bowen formula, which expresses the equilibrium state as a weak-$^*$ limit of discrete measure supported on periodic orbits~\cite{ParPol}
$$
\mu_\psi=\lim_{T\to +\infty} \frac{1}{\Sigma_{\psi,T,\Delta}(\mathcal{P})}\sum_{\gamma \in \mathcal{P}_{T,\Delta}} e^{T_\psi(\gamma)} \delta_{\gamma}.
$$
Given a subset $\mathcal{S}\subset \mathcal{P}$, we say that $\mathcal{S}$ has \emph{positive proportion} with respect to $(\psi, \mu_\psi)$ if 
$$
\limsup\limits_{T \to +\infty} \frac{\Sigma_{\psi,T,\Delta}(\mathcal{S})}{\Sigma_{\psi,T,\Delta}(\mathcal{P})}>0.
$$
\begin{remark}
\label{remark_pp}
 If we replace $\psi$ with $\psi+c$, where $c$ is a constant, then $P(\psi+c)=P(\psi)+c$ and $\mu_{\psi+c}=\mu_\psi$. One advantage of working with geodesics whose length is in the interval $(T, T+\Delta]$ is that we have obvious inequalities
 $$
 e^{cT} \Sigma_{\psi,T,\Delta}(\mathcal{P}) \le \Sigma_{\psi+c,T,\Delta}(\mathcal{P}) \le e^{c(T+\Delta)} \Sigma_{\psi,T,\Delta}(\mathcal{P})
 $$
 and similarly for the subset $\mathcal S$. It immediately follows that $\mathcal S$ has positive proportion with respect to $(\psi, \mu_\psi)$ if and only if it has positive proportion with respect to $(\psi+c, \mu_\psi)$. It is also well-known that two potentials have the same equilibrium state if and only if their difference is cohomologous to a constant. It follows that the property of having a positive proportion depends only the equilibrium state $\mu_\psi$ and is independent of particular choice of the potential $\psi$.
\end{remark}

Let us now recall the positive proportion Livshits Theorem due to Marshall Reber and Dilsavor~\cite{DilMReber}.
\begin{theorem}[Positive proportion Livshits Theorem~\cite{DilMReber}]\label{pos prop thm}
	Let $X^t$ be a transitive $3$-dimensional Anosov flow which is not a constant roof suspension, and let $\varphi \colon M \to \mathbb{R}$ be a Hölder continuous function. If there exists a Hölder continuous function $\psi \colon M \to \mathbb{R}$ such that the set $\mathcal{S}^\varphi:=\{\gamma \in \mathcal{P}:T_\varphi(\gamma)=0\}$ has positive proportion for $\mu_\psi$, then $\varphi$ is a coboundary. 
\end{theorem}

In fact, the result of Marshall Reber and Dilsavor is valid in any dimension under the condition that stable and unstable distributions do not jointly integrate. In dimension 3 this is equivalent to the flow not being a constant roof suspension~\cite[Theorem 3.4]{Plante}.\footnote{In fact, the case of constant roof suspension reduces to the diffeomorphism case, hence, there are no exception for positive proportion Livshits Theorem in dimension 3.}

\subsection{Positive proportion and density}

Here we show that  sets of periodic orbits of positive proportion are dense.

\begin{proposition}\label{pos prop open}
	Let $U\subset M$ be a non-empty open set. Let  $\mathcal{P}^U\subset \mathcal{P}$ be the set of periodic orbits $\gamma \in \mathcal{P}$ such that $\gamma \cap U\neq \varnothing$. Then for any Hölder potential $\psi \colon M \to \mathbb{R}$ with pressure $P(\psi)\geq 0$, the set $\mathcal{P}^U$ has full proportion for the equilibrium state $\mu_\psi$. 
\end{proposition}

\begin{proof}
	Assume by contradiction that it is not the case. Let $U\subset M$ be a non-empty open set, and let $\psi \colon M \to \mathbb{R}$ be a Hölder potential with $P(\psi)\geq 0$ such that the complement $\mathcal{P}\setminus\mathcal{P}^U$ has positive proportion for $\mu_\psi$. The set $F:=M\setminus U$ is closed; moreover,  $\mathcal{P}^F:=\mathcal{P}\setminus\mathcal{P}^U$ consists of periodic orbits $\gamma \in \mathcal{P}$ such that $\gamma\subset F$. Fix $\Delta>0$. For each $T>0$, and for $*\in \{U,F\}$, let us consider the invariant probability measure
	$$
	\mu_{\psi,T,\Delta}^*:=\frac{1}{\Sigma_{\psi,T,\Delta}(\mathcal{P}^*)}\sum_{\gamma \in \mathcal{P}_{T,\Delta}^*} e^{T_\psi(\gamma)} \delta_{\gamma}.
	$$ 
	In particular, the support of $\mu_{\psi,T,\Delta}^F$ is contained in the closed set $F\subset M$. Moreover,
	$$
	\mu_{\psi,T,\Delta}:=\frac{1}{\Sigma_{\psi,T,\Delta}(\mathcal{P})}\sum_{\gamma \in \mathcal{P}_{T,\Delta}} e^{T_\psi(\gamma)} \delta_{\gamma}=\frac{\Sigma_{\psi,T,\Delta}(\mathcal{P}^U)}{\Sigma_{\psi,T,\Delta}(\mathcal{P})}\mu_{\psi,T,\Delta}^U+\frac{\Sigma_{\psi,T,\Delta}(\mathcal{P}^F)}{\Sigma_{\psi,T,\Delta}(\mathcal{P})}\mu_{\psi,T,\Delta}^F.
	$$
	By weak-$*$ compactness, and the assumption that $\mathcal{P}^F$ has positive proportion for $\mu_\psi$, we can take a sequence $(T_n)_n\to +\infty$ such that 
	$$
	%\mu_{\psi,T_n,\Delta}\rightharpoonup_n \mu_\psi,\quad
	 \mu_{\psi,T_n,\Delta}^*\rightharpoonup_n\mu_\psi^*,\, *\in \{U,F\},\quad \lim\limits_{n \to +\infty}\frac{\Sigma_{\psi,T_n,\Delta}(\mathcal{P}^F)}{\Sigma_{\psi,T_n,\Delta}(\mathcal{P})}=\rho>0.
	$$
	Because $P(\psi)\geq 0$ we can apply Bowen formula $\mu_\psi=\lim_{T\to\infty} \mu_{\psi,T,\Delta}$, and by taking limit along $T_n$, $n\to\infty$, in the above decomposition we have
	$$
	\mu_{\psi}=(1-\rho)\mu_{\psi}^U+\rho \mu_{\psi}^F.
	$$
	Since $\rho>0$, and $\mu_{\psi}$ is ergodic, we have $\rho=1$, hence $\mu_{\psi}=\mu_\psi^F$. But  $\mu_\psi^F$ is supported in $F$ while the equilbrium state $\mu_\psi$ has full support, a contradiction. 
\end{proof}

We have the following direct consequence of the above result.
\begin{corollary}\label{coro den posi}
\label{cor_dense}
For any Hölder potential $\psi \colon M \to \mathbb{R}$ with pressure $P(\psi)\geq 0$, and any set $\mathcal{S}\subset \mathcal{P}$ of periodic orbits with positive proportion for $\mu_\psi$, the set $\{p\in \gamma:\gamma \in \mathcal{S}\}$ is dense in $M$. 
\end{corollary}

\subsection{Volume contracting periodic orbits and the SRB measure}
\label{sec_43}
 
Here we show that if the Anosov flow  $X^t\colon M \to M$ is dissipative, then volume contracting periodic orbits have full proportion with respect to the positive SRB measure $m_X^+$ and, similarly, volume expanding points have full proportion with respect to the negative SRB measure $m_X^-$.

%We have defined the SRB measures in Definition~\ref{def_SRB} and now we need to recall few more well-known facts about them. The positive SRB measure $m_X^+$ of a transitive Anosov flow $X^t$ can also be characterized as the unique equilibrium state of the geometric potential defined by
%$$
%\psi^u(x)=-\lim_{t\to 0}\frac1t \log\lambda_x^u(t).
%$$
Recall that we have $P_X(\psi^u)=0$ and by the Bowen formula
$$
\mu_{\psi^u, T,\Delta}=\frac{1}{\Sigma_{\psi^u,T,\Delta}(\mathcal{P})}\sum_{\gamma \in \mathcal{P}_{T,\Delta}} \lambda_\gamma^{-1}\delta_{\gamma}
\rightharpoonup_T m_X^+.
$$

If we denote by $\chi^s(m_X^+)$ and $\chi^u(m_X^+)$ the stable and unstable Lyapunov exponents of $m_X^+$ then $h(m_X^+)=\chi^u(m_X^+)$ by the Pesin formula and, since we are assuming that $X^t$ is dissipative, $h(m_X^+)<-\chi^s(m_X^+)$ by the Margulis-Ruelle inequality (strict inequality holds because equality holds only for the negative SRB measure~\cite{ledrappier84} and $m_X^+\neq m_X^-$). Hence for $\psi\colon x\mapsto\frac{d}{dt}|_{t=0} \log\textup{Jac}_x(t)$, we have
\begin{equation}\label{eq_negative}
\int_M \psi\, dm_X^+=\chi^s(m_X^+)+\chi^u(m_X^+)<0.
\end{equation}

Now we decompose the set of periodic orbits as the disjoint union $\mathcal P=\cC\cup\cV\cup\cE$, where $\gamma\in\cC$ are volume contracting, $\textup{Jac}_\gamma(T(\gamma))<1$, $\gamma\in\cV$ are volume preserving,  $\textup{Jac}_\gamma(T(\gamma))=1$, and $\gamma\in\cE$ are volume expanding,  $\textup{Jac}_\gamma(T(\gamma))>1$. We split the approximating measure $\mu_{\psi^u, T,\Delta}$ according to this decomposition
$$
\mu_{\psi^u, T,\Delta}=\frac{\Sigma_{\psi^u,T,\Delta}(\cC)}{\Sigma_{\psi^u,T,\Delta}(\mathcal{P})}\mu_{\psi^u,T,\Delta}^\cC+\frac{\Sigma_{\psi^u,T,\Delta}(\cV)}{\Sigma_{\psi^u,T,\Delta}(\mathcal{P})}\mu_{\psi^u,T,\Delta}^\cV+\frac{\Sigma_{\psi^u,T,\Delta}(\cE)}{\Sigma_{\psi^u,T,\Delta}(\mathcal{P})}\mu_{\psi^u,T,\Delta}^\cE.
$$
After passing to suitable subsequence $T_n\to\infty$, $n\to\infty$ we can pass to a limit in the above formula and obtain
$$
m_X^+=\rho\mu_{\psi^u}^\cC+\varsigma \mu_{\psi^u}^\cV+(1-\rho-\varsigma)\mu_{\psi^u}^\cE.
$$
Now note that, by construction, we have inequalities $\mu_{\psi^u,T,\Delta}^\cC(\psi)<0$, $\mu_{\psi^u,T,\Delta}^\cV(\psi)=0$ and $\mu_{\psi^u,T,\Delta}^\cE(\psi)>0$ which persist under passing to the limit: $\mu_{\psi^u}^\cC(\psi)\le 0$, $\mu_{\psi^u}^\cV(\psi)=0$ and $\mu_{\psi^u}^\cE(\psi)\ge 0$. Comparing to~\eqref{eq_negative} we can conclude that $\rho>0$. Finally, since $m_X^+$ is ergodic it cannot be a non-trivial linear combination of invariant measures and we conclude that $\rho=1$ and $\varsigma=0$ which means, according to our definition, that volume contracting periodic orbits have full proportion with respect to $m_X^+$. Similarly, volume expanding orbits from $\cE$ have full proportion with respect to $m_X^-$.
\begin{lemma}
\label{lemma_full_proportion}
Let $X^t$ be a dissipative transitive Anosov flow. Then volume contracting periodic orbits $\cC\subset\mathcal{P}$ have full proportion with respect to the positive SRB measure $m_X^+$, that is,
$$
\lim_{T \to +\infty} \frac{\Sigma_{\psi^u,T,\Delta}(\mathcal{C})}{\Sigma_{\psi^u,T,\Delta}(\mathcal{P})}=1.
$$
Similarly, volume expanding periodic orbits $\cE\subset\mathcal{P}$ have full proportion with respect to the negative  SRB measure $m_X^-$.
\end{lemma}
We leave it to the reader to adjust the argument to show that the limit in fact exists, without passage to a subsequence.

As a direct corollary of Corollary~\ref{coro den posi} and Lemma~\ref{lemma_full_proportion} we have:
\begin{corollary}\label{coro dens vol expa pts}
	Let $X^t$ be a dissipative transitive Anosov flow. Then both volume contracting and volume expanding periodic points are dense in $M$. 
\end{corollary}

\subsection{Volume preserving periodic  orbits}

Let $X^t$ be a dissipative transitive Anosov flow. We say that a periodic orbit $\gamma \in \mathcal{P}$ is \emph{volume preserving} if for $p \in \gamma$, we have $\mathrm{Jac}_p(T(\gamma))=1$. 

%\begin{definition}
%	Let us recall that for any $x\in M$, $t \in \mathbb{R}$, and $*\in \{s,u\}$, we defined
%	$$
%	\lambda_x^*(t):=\|DX^t(x)|_{E^*}\|,
%	$$
%	and that, for any periodic orbit $\gamma \in \mathcal{P}$, for any point $p$ in $\gamma$, we let 
%	$$
%	\mu(\gamma)=\mu(p):=\lambda_p^s(T(\gamma))\in (0,1),\quad \lambda(\gamma)=\lambda(p):=\lambda_p^u(T(\gamma))>1.
%	$$
%	We say that $\gamma$ is \emph{resonant} if there exist integers $k,\ell \in \mathbb{Z}$ such that 
%	\begin{equation}\label{k et l res}
%	\mu(\gamma)^k \lambda(\gamma)^\ell=1.
%	\end{equation}
%	Note that since $\mu(\gamma)<1<\lambda(\gamma)$ we can always assume that $k,\ell \geq 0$. Let $\Pi:=\{(k,\ell)\in \mathbb{N}^2:k,\ell\text{ are relatively prime}\}$. For any $(k,\ell) \in \Pi$, let $\mathcal{R}^{k,\ell}\subset \mathcal{P}$ be the set of periodic orbit $\gamma \in \mathcal{P}$ satisfying~\eqref{k et l res}; note that if $(k,\ell)\neq (k',\ell')\in \Pi$, then $\mathcal{R}^{k,\ell}\cap \mathcal{R}^{k',\ell'}=\emptyset$. For any integer $N\in \mathbb{N}$, we also let $\mathcal{R}^N:=\sqcup_{(k,\ell)\in \Pi:k+\ell\leq N}\mathcal{R}^{k,\ell}$ be the set of periodic orbits with a resonance of order  $k+\ell \leq N$. 
%\end{definition}

\begin{proposition}\label{prop zero prop vp}
	Assume that $X^t$ is not a constant roof suspension. Then, for any H\"older potential $\psi \colon M \to \mathbb{R}$, the set $\mathcal{V}\subset \mathcal{P}$ of volume preserving periodic orbits has zero proportion with respect to $\mu_\psi$. 
\end{proposition}
\begin{proof}
	Fix $\Delta>0$. Assume that for some Hölder potential $\psi \colon M \to \mathbb{R}$, the set $\mathcal{V}$ has positive proportion. 
%	Since there are finitely many resonances of order less than or equal to $N$,   
%	$$
%	0<\limsup_{T\to +\infty}\frac{\Sigma_{\psi,T,\Delta}(\mathcal{R}^N)}{\Sigma_{\psi,T,\Delta}(\mathcal{P})}\leq \sum_{(k,\ell)\in \Pi:k+\ell\leq N}\limsup_{T\to +\infty}\frac{\Sigma_{\psi,T,\Delta}(\mathcal{R}^{k,\ell})}{\Sigma_{\psi,T,\Delta}(\mathcal{P})}.
%	$$
%	In particular, 
%	there exists $(k,\ell)\in \Pi$ with $k+\ell\leq N$ such that $\mathcal{R}^{k,\ell}$ has positive proportion for $\mu_\psi$. 
In particular, for a set of $\gamma\in \mathcal{P}$ of positive proportion, for $p \in\gamma$, 
	$$
	\int_{0}^{T(\gamma)} \frac{d}{dt}\Big|_{t=0}\log\mathrm{Jac}_{X^u(p)}(t)\, du=0. 
	$$
	By the positive proportion Livshits Theorem (Theorem~\ref{pos prop thm}), we conclude that $\frac{d}{dt}|_{t=0}\log\mathrm{Jac}_{x}(t)$ is a coboundary. Thus, there exists a continuous function $\kappa \colon M \to \mathbb{R}$ such that for any $(x,t)\in M \times \mathbb{R}$, 
	$$
	\log\mathrm{Jac}_{x}(t)=\kappa(X^t(x))-\kappa(x).
	$$ 
	In particular, for any periodic orbit $\gamma \subset \mathcal{P}$, for any $p\in \gamma$, we have 
	$
	\mathrm{Jac}_p(T(\gamma))=1
	$, 
	and then $X^t$ is volume preserving (see~\cite{LivSin}), which contradicts our assumption. 
%	Otherwise, either $k>\ell$ or $k<\ell$. Let us assume that $k>\ell$, the other case being similar. Let us fix $t_0>0$ such that for any $x\in M$ and $t\geq t_0$, we have $\lambda_x^u(t)-2\sup_{y \in M} |\kappa(y)|\geq 1$. Then, for any $x \in M$ and $t\geq t_0$,  
%	$$
%	\mathrm{Jac}_x(t)=\exp\left(\frac{1}{k}\left((k-\ell)\log \lambda_x^u(t)+\kappa(X^{t}(x))-\kappa(x)\right)\right)>1,
%	$$
%	which is incompatible with the fact that $X^t$ is a diffeomorphism of $M$. 
\end{proof}

\subsection{Mildly dissipative equilibrium states}\label{subsec full prop mild diss}

Let $X^t\colon M \to M$ be a dissipative transitive Anosov flow on a $3$-dimensional manifold $M$.  Consider the logarithmic infinitesimal Jacobian 
$$\psi\colon x \mapsto  \frac{d}{dt}\left|_{t=0}\right. \log \det DX^t(x).$$
Then $\psi=\psi^s-\psi^u$ and we consider the following one parameter family $\{\varphi_t\}_{t \in \mathbb{R}}$ of H\"older potentials %\marginpar{don'trecall}
$$
\varphi_t:=t\psi^s + (1-t) \psi^u= \psi^u + t \psi.
$$ 
For each $t \in \mathbb{R}$ we denote by $m_t$ the equilibrium measure for $X^t$ associated to the potential $\varphi_t$. The path of equilibrium states $\{m_t, \, t\in[0,1]\}$ connects the positive SRB measure $m_0=m_X^+$ to the negative SRB measure $m_1=m_X^-$. %Recall that the pressure of a H*older potential $\varphi\colon M \to \mathbb{R}$ is defined as $P(\varphi):=\sup_{\mu \in \mathcal{M}_X} h_\mu(X) 
Recall that the function $\bar P\colon t \mapsto P(\varphi_t)$ which assigns to $t \in \mathbb{R}$ the pressure of the potential $\varphi_t$ is smooth, with  $\bar P'(t)=\int_M \psi\, dm_t$. In fact, $\bar P$ is strictly convex, since $\psi$ is not cohomologous to a constant (see Parry-Pollicott~\cite[Proposition 4.10, Proposition 4.12]{ParPol}). Recall that by the entropy formula for SRB measures $\bar P(0)=\bar P(1)=0$. Let $t_0=t_0(X^t)\in (0,1)$ be the point at which $\bar P$ achieves its minimum, so that 
\begin{equation}\label{der press}
	\bar P'(t_0)=\int_M \psi\, dm_{t_0}=0.
\end{equation}

Recall that for any periodic orbit $\gamma \in \mathcal{P}$ for $X^t$, we denote by $\mu_\gamma<1<\lambda_\gamma$ its stable and unstable multipliers, and that for any $\varrho\in[1,2)$, the set  $\mathcal{P}^\varrho\subset \mathcal{P}$ of  $\varrho$-mildly dissipative periodic orbits is defined as follows:
$$
\mathcal{P}^\varrho:=\left\{\gamma\in \mathcal{P}: \mu_\gamma^\varrho \lambda_\gamma<1\text{ and }\mu_\gamma\lambda_\gamma^\varrho>1\right\}.
$$ 
\begin{lemma}\label{full prop mild diss}
	For any $\varrho\in(1,2]$, the equilibrium state $m_{t_0}$ gives full proportion to the collection~$\mathcal{P}^\varrho$.
\end{lemma}
\begin{proof}%[Proof of the claim:]
	Assume by contradiction that there exists $\varrho\in(1,2]$ such that the set $\mathcal{P}^\varrho$ does not have full proportion for $m_{t_0}$. We would like to use Bowen formula for $m_{t_0}$, however it has negative pressure. Thus we add constant a consider the family $\phi_t+c$. Then $P(\phi_t+c)=P(\phi_t)+c$ and for a large enough $c$ we have $P(\psi_{t_0}+c)>0$. The equilibrium states remain the same, in particular, $\mu_{\phi_{t_0}+c}=m_{t_0}$. And the property of having positive proportion remains the same under such adjustment (recall Remark~\ref{remark_pp}).
	
	Let us decompose $\mathcal{P}\setminus \mathcal{P}^\varrho=\bar{\mathcal{P}}^{\varrho,1}\cup \bar{\mathcal{P}}^{\varrho,2}$, where
	$$
	\bar{\mathcal{P}}^{\varrho,1}:=\left\{\gamma\in \mathcal{P}: \mu_\gamma^\varrho \lambda_\gamma\geq 1\right\},\quad \bar{\mathcal{P}}^{\varrho,2}:=\left\{\gamma\in \mathcal{P}: \mu_\gamma\lambda_\gamma^\varrho\leq 1\right\}.
	$$ 
	Fix $\Delta>0$. As previously, for any $T>0$, and for any subset $\mathcal{S}\subset \mathcal{P}$, we denote by $\mathcal{S}_{T,\Delta}\subset \mathcal{S}$ the set of periodic orbits $\gamma\in \mathcal{S}$ with $T(\gamma)\in (T,T+\Delta]$. Let 
	$$\Sigma_{\varphi_{t_0}+c,T,\Delta}(\mathcal{S}):=\sum_{\gamma\in \mathcal{S}_{T,\Delta}} e^{T_{\varphi_{t_0}}(\gamma)+c T(\gamma)},$$
	 and  let
	$$
	m_{\varphi_{t_0}+c,T,\Delta}^{\mathcal{S}}:=\frac{1}{\Sigma_{\varphi_{t_0}+c,T,\Delta}(\mathcal{S})}\sum_{\gamma \in \mathcal{S}_{T,\Delta}} e^{T_{\varphi_{t_0}}(\gamma)+c T(\gamma)} \delta_{\gamma},
	$$
	so that
	\begin{equation}\label{decm po}
		m_{\varphi_{t_0}+c,T,\Delta}^{\mathcal{P}}=
		\frac{\Sigma_{\varphi_{t_0}+c,T,\Delta}(\mathcal{P}^\varrho)}{\Sigma_{\varphi_{t_0}+c,T,\Delta}(\mathcal{P})}m_{\varphi_{t_0}+c,T,\Delta}^{\mathcal{P}^\varrho}+\sum_{i=1}^2\frac{\Sigma_{\varphi_{t_0}+c,T,\Delta}(\bar{\mathcal{P}}^{\varrho,i})}{\Sigma_{\varphi_{t_0}+c,T,\Delta}(\mathcal{P})}m_{\varphi_{t_0}+c,T,\Delta}^{\bar{\mathcal{P}}^{\varrho,i}}.
	\end{equation}
	Take a suitable sequence $(T_n)_{n\geq 0}$ such that $\lim_n T_n=+\infty$, and 
	$$
	m_{\varphi_{t_0}+c,T_n,\Delta}^{*}\rightharpoonup_n m^{*},\quad *=\mathcal{P}^\varrho,\bar{\mathcal{P}}^{\varrho,1},\bar{\mathcal{P}}^{\varrho,2}.
	$$
	Since we have assumed that $\bar{\mathcal{P}}^{\varrho,1}\cup\bar{\mathcal{P}}^{\varrho,2}$ has positive proportion with respect to $m_{t_0}$, by passing to a further subsequence we also arrange that for $i=1,2$, 
	$$
	\lim_{n \to +\infty} \frac{\Sigma_{\varphi_{t_0}+c,T_n,\Delta}(\bar{\mathcal{P}}^{\varrho,i})}{\Sigma_{\varphi_{t_0}+c,T_n,\Delta}(\mathcal{P})}=\rho_i,
	$$
	 with $\rho_1+\rho_2>0$.  By~\eqref{decm po} and the Bowen formula, we then have
	$$
	m_{t_0}=(1-\rho_1-\rho_2)m^{\mathcal{P}^\varrho}+\rho_1m^{\bar{\mathcal{P}}^{\varrho,1}}+\rho_2 m^{\bar{\mathcal{P}}^{\varrho,2}}. 
	$$
	Without loss of generality, we assume that $\rho_1>0$, since the case $\rho_2>0$ is symmetric.
	By ergodicity of the equilibrium measure $m_{t_0}$, we conclude that $m_{t_0}= m^{\bar{\mathcal{P}}^{\varrho,1}}$. By definition, for any $\gamma \in \bar{\mathcal{P}}^{\varrho,1}$, we have 
	$$
	\log \mathrm{Jac}_\gamma\geq -(\varrho-1)\log \mu_\gamma,
	$$
	which can be written as
	$$
	\int_\gamma\psi(\gamma(s))\, ds\geq -(\varrho-1)\int_\gamma\psi^s(\gamma(s))\, ds.
	$$
	Then, for any $T_n \geq 0$, 
	\begin{align*}
	\int_M \psi\, dm_{\varphi_{t_0}+c,T_n,\Delta}^{\bar{\mathcal{P}}^{\varrho,1}}&\ge e^{cT_n}\int_M \psi\, dm_{\varphi_{t_0},T_n,\Delta}^{\bar{\mathcal{P}}^{\varrho,1}}\geq -e^{cT_n}(\varrho-1)\int_M \psi^s\, dm_{\varphi_{t_0},T_n,\Delta}^{\bar{\mathcal{P}}^{\varrho,1}}\\
	&\ge -\frac{e^{cT_n}}{e^{c(T_n+\Delta)}}(\varrho-1)\int_M \psi^s\, dm_{\varphi_{t_0}+c,T_n,\Delta}^{\bar{\mathcal{P}}^{\varrho,1}}>0.
	\end{align*}
	%	$$
	%\int_M \psi\, dm_{\varphi_{t_0},T_n,\Delta}^{\bar{\mathcal{P}}^{\varrho,1}}\geq -(\varrho-1)\int_M \psi^s\, dm_{\varphi_{t_0},T_n,\Delta}^{\bar{\mathcal{P}}^{\varrho,1}}.
	%$$
	Taking the limit as $T_n\to +\infty$, and using that $m_{\varphi_{t_0}+c,T_n,\Delta}^{\bar{\mathcal{P}}^{\varrho,1}}\rightharpoonup_n m_{t_0}$, we deduce that 
	$$ 
	\int_M \psi\, dm_{t_0}\geq -(\varrho-1)\int_M \psi^s\, dm_{t_0}>0.
	$$
	The last inequality is strict because integral of $\psi^s$ is negative with respect to any invariant measure. Indeed, $\psi^s$ is cohomologous to the average Jacobian function $\frac1{T_0}\int_0^{T_0}\psi^s(X^t(\cdot))dt$ which is negative for large enough $T_0$.
	
	The last inequality contradicts
	%	On the other hand, since $\mu_{t_0}= \mu_{\bar{\mathcal{P}}_X^{\varrho,1}}$, by 
	\eqref{der press}. We conclude that the set $\mathcal{P}^\varrho=\left\{\gamma\in \mathcal{P}: \mu_\gamma^\varrho \lambda_\gamma<1\text{ and }\mu_\gamma\lambda_\gamma^\varrho>1\right\}$ has full proportion, as claimed. 
	%, we also have that 
	%	$$
	%	\int_M \psi_X\, d\mu_{\bar{\mathcal{P}}_X^{\varrho,1}}=0,
	%	$$
	%	a contradiction. 
	\end{proof}
	
Given $\varrho>1$, we introduce the two subsets $\mathcal{E}^\varrho,\mathcal{C}^\varrho\subset \mathcal{P}^\varrho$, 
$$
\mathcal{E}^\varrho:=\left\{\gamma\in \mathcal{P}: \mu_\gamma\lambda_\gamma>1\text{ and }\mu_\gamma^\varrho \lambda_\gamma<1\right\},\quad \mathcal{C}^\varrho:=\left\{\gamma\in \mathcal{P}: \mu_\gamma\lambda_\gamma<1\text{ and }\mu_\gamma\lambda_\gamma^\varrho>1\right\}.
$$
\begin{corollary}\label{cor full prop mild diss}
	 For any $\varrho>1$, there exist $t_\mathcal{E}^\varrho,t_\mathcal{C}^\varrho\in [0,1]$ arbitrarily close to $t_0$ such that the expanding subset $\mathcal{E}^\varrho$ has positive proportion relative to $m_{t_\mathcal{E}^\varrho}$ and the contracting subset $\mathcal{C}^\varrho$ has positive proportion relative to $m_{t_\mathcal{C}^\varrho}$. 
\end{corollary}
\begin{proof}
	Let us prove the result for $\mathcal{E}$; the other case being completely analogous. 
	Fix any $\varrho>1$. The pressure function $\bar P$ is strictly convex, and has a critical point at $t_0$, by~\eqref{der press}. Then, for any small $\varepsilon>0$, there exists a unique $t_\varepsilon\in [t_0,1]$, with $\lim_{\varepsilon\to 0} t_\varepsilon=t_0$, such that 
	\begin{equation}\label{der press eps}
		\bar P'(t_\varepsilon)=\int_M \psi\, dm_{t_\varepsilon}=\varepsilon>0. 
	\end{equation} 
	%In particular, the collection $\mathcal{E}=\{\gamma \in \mathcal{P}:\mu_\gamma \lambda_\gamma>1\}$ of volume expanding periodic orbits has positive proportion relative to $m_{t_\varepsilon}$. 
	
	%Let us then decompose $\mathcal{P}^\varrho=\mathcal{E}^\varrho\cup \mathcal{V}^\varrho \cup \mathcal{C}^\varrho$ into volume expanding, volume preserving, and volume contracting periodic orbits, respectively.
 	By Proposition~\ref{prop zero prop vp}, the set of volume preserving periodic orbits has zero proportion with respect to $m_{t_\varepsilon}$. Arguing in the exactly same way as in the proof of Lemma~\ref{full prop mild diss}, we deduce that if $\varepsilon$ is sufficiently small, then, the set %$\mathcal{E}^\varrho:=\mathcal{E}\cap \mathcal{P}^\varrho$ has positive proportion hence 
 	$\mathcal{E}^\varrho \cup \mathcal{C}^\varrho\subset \mathcal{P}^\varrho$ has full proportion relative to $m_{t_\varepsilon}$. 
	
	As in the previous proof we pick a constant $c$ such that $P(\varphi_{t_\varepsilon}+c)>0$ so that we can apply Bowen formula. Let us show that the set $\mathcal{C}^\varrho$ cannot have full proportion. Indeed, fix $\Delta>0$; using similar notation as in the proof of Lemma~\ref{full prop mild diss}, for any $T>0$, we let 
 	$$
 	m_{\varphi_{t_\varepsilon}+c,T,\Delta}^{\mathcal{C}^\varrho}:=\frac{1}{\Sigma_{\varphi_{t_\varepsilon}+c,T,\Delta}({\mathcal{C}^\varrho})}\sum_{\gamma \in \mathcal{C}^\varrho_{T,\Delta}} e^{T_{\varphi_{t_\varepsilon}}(\gamma)+cT(\gamma)} \delta_{\gamma}.
 	$$
	Note that by definition of $\mathcal{C}^\varrho$ we have $\int \psi dm_{\varphi_{t_\varepsilon}+c,T,\Delta}^{\mathcal{C}^\varrho}<0$.
 	If $\mathcal{C}^\varrho$ is full-proportion relative to $m_{t_\varepsilon}$, then, %for an appropriate subsequence $(T_n)_{n\geq 0}$ with $\lim_n T_n=+\infty$, we have
	$$
	0\ge \lim_{T \to +\infty}\int_M \psi\, dm_{\varphi_{t_\varepsilon}+c,T,\Delta}^{\mathcal{C}^\varrho}=\int_M \psi\, dm_{{t_\varepsilon}},
	$$
	which contradicts~\eqref{der press eps}. We conclude that the set $\mathcal{E}^\varrho$ has positive proportion relative to the equilibrium state $m_{t_\varepsilon}$. In fact, it is easy to see that  $\mathcal{E}^\varrho$ has full proportion relative to $m_{t_\varepsilon}$.
	%, we see
	%On the other hand 	
	%$$
	%\int_M \psi\, dm_{\varphi_{t_0},T_n,\Delta}^{\mathcal{E}^\varrho}\ge 0.\qedhere
	%$$
\end{proof}

\section{A dichotomy: $C^1$-smoothness of a strong subbundle or recovery of eigendata}\label{ssection_six}

As previously, we consider a transitive $C^r$, $r\ge 3$, Anosov flow $X^t$ on some $3$-manifold, which is $k$-pinched for some $1<k\leq r-1$. We let $p\in M$ be a volume expanding periodic point of period $T>0$, with multipliers $0<\mu=\mu_p<1<\lambda=\lambda_p$, $\mu\lambda=\mathrm{Jac}_p(T)>1$. We fix some homoclinic point $q \in  \mathcal{W}^u_{\mathrm{loc}}(p)$, we let  $T'=T_q'\in \mathbb{R}$ and $q':=X^{T'}(q)\in \Sigma_p\cap \mathcal{W}^s_{\mathrm{loc}}(p)$ as in Section~\ref{ssection_quatre}, and let $(p_n)_{n \geq n_0}$ be the sequence of periodic points given by Lemma~\ref{lemme shadowing}, whose orbits shadow the orbit of $q$. We use the same notation as in Subsection~\ref{section asymp ex}. Let $\Sigma_p:=\imath_p((-1,1)\times \{0\}\times (-1,1))$, let $\{\mathcal{T}_p^s(\cdot)\}$ be the stable template along $\mathcal{W}_{\mathrm{loc}}^u(p)\cap \Sigma_p$, and let $\tilde P_p^s$ be the polynomial of degree at most $[k]$ introduced in Lemma~\ref{lemma tem st}. By Proposition~\ref{coro zxp per}, the periods $(T_n=T_n^q)_{n \geq n_0}$ obey the following asymptotics: 
\begin{equation}\label{asymptotiques des periodes}
	T_n=nT+T'+\zeta_p(q)\mu^{n}+O(\theta^n),
\end{equation}
%\marginpar{needs to be adjusted here if we change theta}
with $\max\left(\mu^{\frac 32},\mu^{\frac{2\log \lambda}{\log \lambda- \log \mu}}\right)<\theta<\mu$, and 
\begin{equation*}%\label{fct zeta}
	\zeta_p(q):=\xi_\infty\left(\mathcal{T}_p^s(\eta_\infty)-\tilde P_p^s(\eta_\infty)\right),
\end{equation*} 
where $(0,\eta_\infty):=\imath_p^{-1}(q)$ and $(\xi_\infty,0):=\imath_p^{-1}(q')$ denote the normal coordinates of $q$ and $q'$,  respectively. Note that the excursion time $T'=T_q'$ can always be recovered from the periods $(T_n=T_n^q)_{n \geq n_0}$, since $T'=\lim_{n \to +\infty} (T_n-nT)$. 
\begin{definition}[Recovery of eigendata from periods]\label{defi recov}
	Let  
	\begin{align*}
	\mathcal{H}^u(p)=\mathcal{H}_X^u(p)&:=\{q \in \mathcal{W}^u_{\mathrm{loc}}(p): q\text{ is homoclinic to }p\},\\
	\mathcal{H}_{\mathrm{good}}^u(p)=\mathcal{H}_{X,\mathrm{good}}^u(p)&:=\{q\in \mathcal{H}^u(p):\zeta_p(q)\neq 0\},
	\end{align*}
	and let
	\begin{equation*}
		\Gamma_p=\Gamma_p^X\colon\left\{
		\begin{array}{rcl}
		\mathcal{H}^u(p) &\to& \mathbb{R}\cup \{-\infty\},\\
		q &\mapsto& \limsup\limits_{n \to +\infty} \frac{1}{n}\log |T_n^q-nT-T_q'|. 
		\end{array}
		\right. 
	\end{equation*}
	We say that the stable eigenvalue $\mu=\mu_p\in (0,1)$ of the periodic point $p$ can be~\emph{recovered from the periods} if $\mathcal{H}_{\mathrm{good}}^u(p)\neq \varnothing$. 
	\end{definition}

	Indeed, by~\eqref{asymptotiques des periodes}, if $\mathcal{H}_{\mathrm{good}}^u(p)\neq \varnothing$, then $\mu$ can be computed from the periods $(T_n=T_n^q)_{n \geq n_0}$ as follows:
	\begin{equation}\label{first eq recov}
	\log \mu=\sup_{\tilde q\in \mathcal{H}^u(p)} \Gamma_p(\tilde q)=\Gamma_p(q),\quad \forall\, q \in \mathcal{H}_{\mathrm{good}}^u(p),
\end{equation}
	while if $\mathcal{H}_{\mathrm{good}}^u(p)= \varnothing$ then, by~(\ref{asymptotiques des periodes})
	\begin{equation}\label{second eq recov}
	\sup_{\tilde q\in \mathcal{H}^u(p)} \Gamma_p(\tilde q)\leq\log \theta<\log \mu. 
\end{equation}
	%We denote by $\mathcal{H}_{\text{good}}^u\subset \Sigma_p\cap \mathcal{W}^u_{\mathrm{loc}}(p)$ the set of homoclinic points $q$ as above, i.e., such that $\mathcal{T}_p^s(\eta_\infty)-\tilde P_p^s(\eta_\infty)\neq 0$. Since the function $\mathcal{T}_p^s-\tilde P_p^s$ is continuous, for any $q\in\mathcal{H}_{\text{good}}^u$, there exists $\epsilon>0$ such that any homoclinic point $\tilde q \in B(q,\epsilon)\cap \mathcal{W}^u_{\mathrm{loc}}(p)$ is still in $\mathcal{H}_{\text{good}}^u$. 

%\begin{remark}\label{remark open set}
%If the stable eigenvalue $\mu=\mu_p\in (0,1)$ of the periodic point $p$ can be recovered from the periods, then the set of   homoclinic points $q \in \Sigma_p\cap \mathcal{W}^u_{\mathrm{loc}}(p)$ satisfies $\mathcal{T}_p^s(\eta_\infty)-\tilde P_p^s(\eta_\infty)\neq 0$. Since the function $\mathcal{T}_p^s-\tilde P_p^s$ is continuous, there exists a neighborhood $U_q \subset \Sigma_p\cap \mathcal{W}^u_{\mathrm{loc}}(p)$ such that for any homoclinic point $\tilde q\in U_q$, denoting $(0,\tilde \eta_\infty):=\imath_p^{-1}(\tilde q)$, we also have $\mathcal{T}_p^s(\tilde \eta_\infty)-\tilde P_p^s(\tilde \eta_\infty)\neq 0$. 
%\end{remark}
\begin{remark}\label{remark change time}
	By reversing time, we can analogously define the symmetric notion of recovery of the unstable eigenvalue at a volume contracting periodic point. Indeed, fix a periodic point $p\in M$, of period $T>0$, with multipliers $0<\mu_p<1<\lambda_p$, with   $\mathrm{Jac}_p(T)=\mu_p\lambda_p<1$. Let 
	$$
	\mathcal{H}^s(p)=\mathcal{H}_X^s(p):=\{q' \in  \mathcal{W}^s_{\mathrm{loc}}(p): q'\text{ is homoclinic to }p\}.
	$$ 
	For any point $q' \in \mathcal{H}^s(p)$, similarly to Lemma~\ref{lemme shadowing}, we can define a sequence of periodic points $(p_n)_n$ whose orbits shadow the orbit of $q'$. Denoting by $(T_n=T_n^{q'})_n$ their periods, formula~\eqref{asymptotiques des periodes}  becomes 
	\begin{equation}\label{replace vol contr}
	T_n=nT+T'+\zeta_p'(q')\lambda^{-n}+O(\theta^n),
	\end{equation}
	for some $T'=T_{q'}'\in \mathbb{R}$ and $\theta\in (0,\lambda^{-1})$, and where 
	\begin{equation*}%\label{fct zeta}
	\zeta_p'(q'):=-\eta_\infty\left(\mathcal{T}_p^u(\xi_\infty)-\tilde P_p^u(\xi_\infty)\right),
	\end{equation*} 
	denoting  $(\xi_\infty,0):=\imath_p^{-1}(q')$, $q:=X^{-T'}(q')$,  $(0,\eta_\infty):=\imath_p^{-1}(q)$, where $\{\mathcal{T}_p^u(\cdot)\}$  is the unstable template along $\mathcal{W}_{\mathrm{loc}}^s(p)\cap \Sigma_p$, and $\tilde P_p^u$ is a  polynomial of degree at most $[k]$ (uniform in $p$) similar to those introduced in Lemma~\ref{lemma tem st}. We also denote by $\mathcal{H}_{\mathrm{good}}^s(p)=\mathcal{H}_{X,\mathrm{good}}^s(p)$ the set of points $q' \in \mathcal{H}^s(p)$ such that $\zeta_p'(q')\neq 0$, and we let 
	$$
	\Gamma_p=\Gamma_p^X\colon 
		\mathcal{H}^s(p)\to \mathbb{R}\cup \{-\infty\},\quad
		q' \mapsto \limsup_{n \to +\infty} \frac{1}{n}\log |T_n^{q'}-nT-T_{q'}'|.
	$$ 
	 Similarly, if $\mathcal{H}_{\mathrm{good}}^s(p)\neq \varnothing$, then by~\eqref{replace vol contr}, 
	\begin{equation}\label{first eq recov unst}
	-\log \lambda=\sup_{\tilde q'\in \mathcal{H}^s(p)} \Gamma_p(\tilde q')=\Gamma_p(q'), \quad \forall\, q'\in \mathcal{H}_{\mathrm{good}}^s(p).
	\end{equation}
	In that case, we say that the unstable eigenvalue $\lambda_p$ can be~\emph{recovered from the periods}.
\end{remark}

\begin{lemma}\label{lemme recov single pt}
	For any volume expanding periodic point $p$, we have the following dichotomy:
	\begin{itemize}
		\item either the stable eigenvalue $\mu_p\in (0,1)$ can be recovered  from the periods $(T_n)_{n \geq n_0}$;
		\item or $\mathcal{T}_p^s=\tilde P_p^s$; moreover, the strong stable distribution $E^s$ is $C^{1+\alpha}$ along $\mathcal{W}_{\mathrm{loc}}^u(p)\cap \Sigma_p$ for some $\alpha>0$.
	\end{itemize}
\end{lemma}

\begin{proof}
	Assume that $\mu$ cannot be recovered from the periods $(T_n)_{n \geq n_0}$. By definition, and since homoclinic points are dense along $\mathcal{W}^u(p)$, we thus have $\mathcal{T}_p^s(\eta)-\tilde P_p^s(\eta)=0$ for a dense set of $\eta \in (-1,1)$. By continuity of the function $\mathcal{T}_p^s-\tilde P_p^s$, we deduce that $\mathcal{T}_p^s=\tilde P_p^s$. 
	Lemma~\ref{lemm pol tem} then implies that the strong stable distribution $E^s$ is $C^{1+\alpha}$ along $\mathcal{W}_{\mathrm{loc}}^u(p)\cap \Sigma_p$, for some $\alpha>0$.
\end{proof}

This result has the following global counterpart. Before stating the result, let us recall that by Corollary~\ref{coro dens vol expa pts}, volume expanding periodic points are dense in $M$. 
\begin{proposition}\label{propo glob dich}
	The following global dichotomy holds:
	\begin{itemize}
		\item either there exists a non-empty open set $V\subset M$ such that for any volume expanding periodic point 
			$p\in V$, the stable eigenvalue $\mu_p\in (0,1)$ can be recovered from the periods;
%			\item for any volume-contracting periodic point 
%			$p\in V$, the unstable multiplier $\lambda_p>1$ can be recovered from the periods;
		%\end{itemize}
		\item or there exists a dense set of  volume expanding periodic points whose stable eigenvalue cannot be recovered from the periods; in that case, the stable %and the unstable 
		distribution $E^s$ is $C^{1+\alpha}$, for some $\alpha>0$. 
	\end{itemize}
\end{proposition}
\begin{proof}
	
	Assume that we are not in the first case, namely, there exists a dense set $P_{\mathrm{exp}}\subset M$ of  volume expanding periodic points $p$ whose stable eigenvalue $\mu_p\in (0,1)$ cannot be recovered from the periods. Thus, by Lemma~\ref{lemme recov single pt}, for any $p \in P_{\mathrm{exp}}$, we have $\mathcal{T}_p^s=\tilde P_p^s$. The conclusion then follows from Lemma~\ref{descartes}. 
\end{proof}

\begin{remark}\label{rem psw}
	Let us recall that if the stable distribution $E^s$ is $C^{1+\alpha}$, $\alpha>0$, then so is the stable foliation $\mathcal{W}^s$ (see e.g. Pugh-Shub-Wilkinson~\cite[Section 6]{PSW}). 
\end{remark}

Let us conclude this part by recalling the classical fact that $3$-dimensional transitive Anosov flows whose stable and unstable distributions $E^s$ and $E^u$ are $C^1$ can be classified:

\begin{lemma}\label{lem plante fh}
	Let $X^t$ be a $3$-dimensional transitive Anosov flow. 
	Assume that both $E^s$ and $E^u$ are $C^1$. Then,
	\begin{itemize}
		\item either $X^t$ is a contact flow;
		\item or $E^s\oplus E^u$ is integrable, and $X^t$ is topologically conjugate to the suspension of an Anosov diffeomorphism. 
	\end{itemize}
\end{lemma}

\begin{proof}
	Let $\omega\colon M \to \mathbb{R}$ be the canonical $1$-form such that 
	$$
	\ker \omega=E^s \oplus E^u,\quad \omega(X)\equiv 1.
	$$
	Clearly, $\omega$ is $X^t$-invariant. Assume now that $E^s$, $E^u$ are $C^1$. Then so is $\omega$, and the form $\omega \wedge d\omega$ is $X^t$-invariant. In particular, since $X^t$ is transitive, we have the following dichotomy:
	\begin{itemize}
		\item either $\omega\wedge d\omega$ is a volume form, in which case, $X^t$ is a contact flow; 
		\item or $\omega\wedge d\omega\equiv 0$.
	\end{itemize}
	In the latter case, by Frobenius Theorem, the distribution $E^s \oplus E^u$ is integrable. Moreover, by Plante~\cite[Theorem 3.1]{Plante}, we can then conclude that $X^t$ is topologically conjugate to a constant roof suspension of some Anosov diffeomorphism on $\mathbb{T}^2$. 
%	$\bullet$ If $X^t$ is conservative (but not contact), then Foulon-Hasselblatt conclude that $E^s \oplus E^u$ is integrable, and $X^t$ is a suspension flow. 
%	
%	$\bullet$ If $X^t$ is dissipative, let us show that $E^s\oplus E^u$ is also integrable. Indeed, given $x \in M$ and two unit vectors $v^s\in E^s(x)$, $v^u\in E^u(x)$, let us show that $[v^s,v^u]\in E^s \oplus E^u(x)$. By Corollary~\ref{coro dens vol expa pts}, we can approximate $x$ by a sequence of dissipative periodic points $(p_n)_n$, and similarly, we approximate $v^*$ with a unit vector $v_n^* \in E^*(p_n)$, $*=s,u$. Since $(X^t)^* \beta=\beta$, denoting by $T_n>0$ the period of $p_n$, we then have
%	$$
%	\beta_{p_n}([v_n^s,v_n^u])=(X^{T_n})^*\beta_{p_n}([v_n^s,v_n^u])=\mathrm{Jac}_{p_n}(T_n)\beta_{p_n}([v_n^s,v_n^u]),
%	$$
%	hence $\beta_{p_n}([v_n^s,v_n^u])=0$, as $\mathrm{Jac}_{p_n}(T_n)\neq 1$. It follows that 
%	$$
%	\beta_x([v^s,v^u])=\lim_{n \to +\infty} \beta_{p_n}([v_n^s,v_n^u])=0,
%	$$
%	i.e., $[v^s,v^u]\in \ker \beta(x)= E^s\oplus E^u(x)$. By Frobenius Theorem, we deduce that $E^s \oplus E^u$ is integrable. Then, by Plante~\cite[Theorem 3.1]{Plante}, we conclude that $X^t$ is topologically conjugate to the suspension of an Anosov diffeomorphism. 
\end{proof}

\section{Proofs of main results: matching of eigendata}\label{ssection_sept}

Let $X^t$, $Y^t$ be two transitive $C^r$, $r \geq 3$, Anosov flows on $3$-manifolds that are $C^0$-conjugate by a homeomorphism $\Phi$ as in~\eqref{eq_conj}. Let $p\in M$ be periodic for $X^t$, of period $T>0$, let $q \in \mathcal{W}_{\mathrm{loc}}^u(p)$ be homoclinic to $p$, let $q'=X^{T'}(q) \in \mathcal{W}_{\mathrm{loc}}^s(p)$, and let $(p_n)$, $n\ge n_0$, be the sequence of periodic points given by Lemma~\ref{lemme shadowing} whose orbits shadow the orbit of $q$. The points $\Phi(q) \in \mathcal{W}_{\mathrm{loc}}^u(\Phi(p))$, $\Phi(q')=Y^{T'}(\Phi(q)) \in \mathcal{W}_{\mathrm{loc}}^s(\Phi(p))$ are homoclinic to $\Phi(p)$. Let $\{\tilde\imath_x\}$, $x\in M$, be a family of adapted charts for the flow $Y^t$, and let $\{\tilde\Sigma_x\}$ $x\in M$, be the associated  family of transverse sections.  Let also $(\tilde p_n)$, $n\ge n_1$, be the sequence of periodic points for $Y^t$ given by Lemma~\ref{lemme shadowing} associated to the points $\Phi(p)$, $\Phi(q)$ and $\Phi(q')$. Let $n_2=\max\{n_0,n_1\}$.

\begin{lemma}\label{lemme mathcing perio points}
	For all $n\ge n_2$, the points $\Phi(p_n)$ and $\tilde p_n$ are in the same (periodic) orbit. 
\end{lemma}

\begin{proof}
	Let $\delta_Y>0$ be some expansivity constant for the flow $Y^t$. 
	By construction, for $n\gg 1$, large, the point $\tilde p_n$ is the unique periodic point in $\tilde \Sigma_{\Phi(p)}$ near $\Phi(q)$ whose orbit shadows the pseudo-orbit $\{Y^{t}(\Phi(p))\}_{t \in [-\frac{nT}{2},\frac{nT}{2}+T']}$ with a jump at time $\frac{nT}{2}+T'$. Moreover, by Lemma~\ref{lemme shadowing}, for $n\gg 1$ large, the orbit of $\tilde p_n$ stays $\frac{\delta_Y}{10}$-close to this pseudo-orbit. The pseudo-orbit $\{X^{t}(p)\}_{t \in [-\frac{nT}{2},\frac{nT}{2}+T']}$ is sent to the pseudo-orbit $\{Y^{t}(\Phi(p))\}_{t \in [-\frac{nT}{2},\frac{nT}{2}+T']}$ by $\Phi$. Moreover, by Lemma~\ref{lemme shadowing}, and by the continuity of $\Phi$, for $n \gg 1$ large, the orbit of $\Phi(p_n)$ stays  $\frac{\delta_Y}{10}$-close to the pseudo-orbit $\{Y^{t}(\Phi(p))\}_{t \in [-\frac{nT}{2},\frac{nT}{2}+T']}$. Therefore, the orbits of $\tilde p_n$ and $\Phi(p_n)$ $\frac{\delta_Y}{2}$-shadow each other, hence they are actually equal. 
\end{proof}

As above, we denote by $T_n^q>0$ the period of the orbit of $p_n$. Similarly, we denote by $T_n^{\Phi(q)}>0$ the period of the orbit of $\tilde p_n$. By Lemma~\ref{lemme mathcing perio points}, for large $n$,
 we have 
 \begin{equation}\label{mathc period}
 	T_n^q=T_n^{\Phi(q)}. 
 \end{equation}
 
 \subsection{The case when one of the flows is volume preserving}
 
 Here we consider the case when one flow is conservative and the conjugate flow is dissipative. In Section~\ref{section_examples}, we give certain non-trivial examples when this indeed happens. First we establish consequences of the existence of a conjugacy.
  
 \begin{proposition}\label{conjugaison vol pre dissip} Assume that $X^t$ is a $3$-dimensional dissipative Anosov flow which is conjugate to a conservative Anosov flow $Y^t$ by some homeomorphism $\Phi$. Then, %at least one of 
 	the following statements hold:
 	\begin{enumerate}
 		\item both distributions $E_Y^s$ and $E_Y^u$ are of class $C^{1+\alpha}$, for some $\alpha>0$;
 		\item at least one of the distributions $E_X^s$ and $E_X^u$ is of class $C^{1+\alpha}$, for some $\alpha>0$. 
 	\end{enumerate}
 \end{proposition}
 
 \begin{remark}\label{conjugaison vol pre dissip rk}  In fact, the above dichotomy can be refined into the following trichotomy:
 	\begin{enumerate}
 		\item either both $X^t$ and $Y^t$ are constant roof suspension flows, and all four distributions $E_X^s,E_X^u$, $E_Y^s$ and $E_Y^u$ are $C^1$;
 		\item or $Y^t$ is a contact flow, $E_X^u$ is $C^{1+\alpha}$, for some $\alpha>0$, $E_X^s$ is not $C^1$, and $\Phi$ is smooth along stable leaves;
 		\item or $Y^t$ is a contact flow, $E_X^s$ is $C^{1+\alpha}$, for some $\alpha>0$, $E_X^u$ is not $C^1$, and $\Phi$ is smooth along unstable leaves.
 	\end{enumerate}%In the former case, as recalled in Lemma~\ref{lem plante fh}, by~\cite{FH, Plante}, one can further conclude that $Y^t$ is either a contact flow or a constant roof suspension flow. 
 \end{remark}
 
 \begin{proof}
 	If both $E_X^s$ and $E_X^u$ are $C^{1+\alpha}$, $\alpha>0$, then by Lemma~\ref{lem plante fh}, we have that $X^t$ is a constant roof suspension flow, and then so is $Y^t$; in particular, both $E_Y^s$ and $E_Y^u$ are $C^{1+\tilde\alpha}$, $\tilde \alpha>0$. Hence, in what follows we will assume that neither $X^t$ nor $Y^t$ is a constant roof suspension flow. 
 	
 	If there exists a dense set of volume expanding periodic points of $X^t$ at which the stable eigenvalues cannot be recovered, then  by Proposition~\ref{propo glob dich},  the distribution $E_X^s$ is $C^{1+\alpha}$, for some $\alpha>0$. Similarly, if there exists a dense set of volume contracting periodic points of $X^t$ at which the unstable eigenvalue cannot be recovered then $E_X^u$ is $C^{1+\alpha}$, for some $\alpha>0$. 
 	
 	Hence, using Corollary~\ref{cor_dense} we are free to assume that for a full proportion set of volume expanding periodic points of $X^t$, with respect to the measure $m_X^-$, the stable eigenvalue can be recovered. That is,  for such points $p$, there exists a homoclinic point $q \in \mathcal{H}_{\text{good}}^u(p)$ such that the periods $(T_n^q)_{n}$ of the periodic points $(p_n)_{n}$ given by Lemma~\ref{lemme shadowing} whose orbits shadow the orbit of $q$ satisfy (see~\eqref{asymptotiques des periodes}): 
 	\begin{equation}\label{asymp_vp_case_un}
 		T_n^q=nT+T_q'+\zeta_p(q)\mu_p^{n}+O(\theta_p^n),
 	\end{equation}
 	where we recall that $\mu_p\in (0,1)$ is the stable eigenvalue of $p$, $\theta_p \in (0,\mu_p)$, and the coefficient $\zeta_p(q)\neq 0$. 
 	
 	On the other hand, by the asymptotic formula obtained in Proposition~\ref{prop volume pre point} for the volume preserving flow $Y^t$ at the periodic point  $\Phi(p)=X^{T}(\Phi(p))$, and by~\eqref{mathc period}, for $n\gg 1$, we have
 	\begin{equation}\label{asymp_vp_case}
 	T_n^q =nT +T_q'+C_\infty \alpha_{\Phi(p)}(T) n\mu_{\Phi(p)}^n +O(\mu_{\Phi(p)}^n),
 	\end{equation}
 	where $\mu_{\Phi(p)}\in (0,1)$ is the stable eigenvalue of $\Phi(p)$, $C_\infty\neq 0$, and $\alpha_{\Phi(p)}(T)$ is the value at $(\Phi(p),T)$ of the longitudinal Anosov cocycle $(x,t)\mapsto \alpha_x(t)$ (see Remark~\ref{remark twist coc}).  
 	Clearly, since the convergence rate of $n\mu_{\Phi(p)}^n$ is not precisely exponential we see that the above formulae could only be compatible when $\alpha_{\Phi(p)}(T)=0$. Hence the longitudinal Anosov cocycle vanishes on a full proportion (with respect to the equilibrium measure $\Phi_*m_X^-$) set of periodic points $\Phi(p)$, and by Theorem~\ref{pos prop thm} the longitudinal Anosov cocycle is a coboundary. Then, by~\cite{FH}, we conclude that both  $E_Y^s$ and $E_Y^u$ are of class $C^{1+\alpha}$, $\alpha>0$.
 	
 	As we have already recalled in Lemma~\ref{lem plante fh}, using the results from~\cite{FH, Plante}, one can then conclude that $Y^t$ is either a contact flow or a constant roof suspension flow, the latter case being ruled out by our assumption. By the formula obtained in Proposition~\ref{prop volume pre point}, formula~\eqref{asymp_vp_case} can then be refined as 	\begin{equation}\label{asymp_vp_case_deux}
 		T_n^q =nT +T_q'+\tilde{\zeta}_{\Phi(p)}(q) \mu_{\Phi(p)}^n+o(\mu_{\Phi(p)}^n),
 	\end{equation} 
 	with $\tilde{\zeta}_{\Phi(p)}(q)\neq 0$, due to the quantitative non-joint integrability of $E_Y^s$ and $E_Y^u$ in the contact case, see Remark~\ref{remarque_contact}.  Comparing~\eqref{asymp_vp_case_un} with~\eqref{asymp_vp_case_deux}, we deduce that for a full proportion set of volume expanding periodic orbits $\gamma$, stable multipliers of $\gamma$ and $\Phi(\gamma)$ match. Applying Theorem~\ref{pos prop thm}, %, comparing~\eqref{asymp_vp_case_un} with~\eqref{asymp_vp_case_deux}, 
 	we conclude that for any periodic point $p=X^T(p)$, stable multipliers of $p$ and $\Phi(p)$ match, hence conjugacy $\Phi$ is smooth along stable leaves, by the classical argument~\cite{dllSRB}. 
 	
 	To finish the proof, let us note that in the above situation, $E_X^u$ has to be $C^{1+\alpha}$, for some $\alpha>0$. Otherwise, by the same reasoning, for any periodic point $p=X^T(p)$, the unstable multiplier of $p$ would have to match the unstable multiplier of $\Phi(p)$. Together with the matching of stable multipliers this immediately implies that all periodic orbits of $X^t$ are volume preserving (Jacobian is 1). This, by~\cite{LivSin},  implies that $X^t$ is volume preserving, contrary to our assumption. 
 \end{proof}

\subsection{The trichotomy: proof of Theorem~\ref{theo alphc}}\label{sub proof thm e}
Let $X^t$, $Y^t$ be two transitive $C^r$ Anosov flows on $3$-manifolds as before,  that are $C^0$-conjugate by a homeomorphism $\Phi$. Assume that they are $k$-pinched, $k\leq r-1$, and that none of the four foliations $\mathcal{W}_X^s,\mathcal{W}_X^u$, $\mathcal{W}_Y^s,\mathcal{W}_Y^u$ is $C^{1}$. 
	
	By Proposition~\ref{conjugaison vol pre dissip}, either both $X^t$ and $Y^t$ are volume preserving, or both $X^t$ and $Y^t$ are dissipative. 
	
	If both $X^t$ and $Y^t$ are volume preserving, then, by the work of Gogolev-Rodriguez Hertz~\cite{GRH},  $X^t$ and $Y^t$ are $C^{r_*}$-conjugate. 

	Let us now consider the case where both $X^t$ and $Y^t$ are dissipative. For $Z=X,Y$, we denote by
	$\mathcal P_Z$ the set of periodic orbits for $Z^t$; we decompose the set of periodic orbits ---  $\mathcal P_Z=\cE_Z\cup\cV_Z\cup \cC_Z$, where $\gamma\in\cE_Z$ are volume expanding, $\textup{Jac}_\gamma^Z(T(\gamma))>1$, $\gamma\in\cC_Z$ are volume contracting, $\textup{Jac}_\gamma^Z(T(\gamma))<1$, and  $\gamma\in\cV_Z$ are volume preserving,  $\textup{Jac}_\gamma^Z(T(\gamma))=1$.  Then, by Proposition~\ref{propo glob dich}, for $Z=X,Y$, there exist non-empty open sets $V_Z^{\cE},V_Z^{\cC}\subset M$ such that
	\begin{itemize}
		\item for any volume expanding periodic point 
		$p\in V_Z^{\cE}\cap \cE_Z$, the stable eigenvalue $\mu_p^Z\in (0,1)$ of $Z^t$ at $p$ can be recovered from the periods;
		\item for any volume contracting periodic point 
		$p\in V_Z^{\cC}\cap \cC_Z$, the unstable eigenvalue $\lambda_p^Z>1$ of $Z^t$ at $p$ can be recovered from the periods. 
	\end{itemize}  
	Since $Z^t$ is transitive, there exists $t_Z\in \mathbb{R}$ such that $V_Z:=V_Z^{\cE}\cap Z^{-t_Z}( V_Z^{\cC})\neq \varnothing$ is a non-empty open set. Moreover, for any volume contracting point $p \in V_Z^{\cC}$, the point $Z^{-t_Z}(p)$ is still a volume contracting periodic point, and its unstable eigenvalue can still be recovered from the periods. By a similar argument, we can also assume that $V:=V_X\cap \Phi^{-1}(V_Y )\neq \varnothing$. 
	
	Let $m_X^-$ be the negative SRB measure for the flow $X^t$; recall that it is the equilbrium measure for the potential $\psi_X^s\colon x \mapsto  \frac{d}{dt}\left|_{t=0}\right. \log \|DX^t(x)|_{E^s}\|$. By functoriality, the pushforward measure $\Phi_* m_X^-$ is also an equilibrium measure for $Y^t$, associated to the potential $\psi_X^s \circ \Phi^{-1}$. 
	By Proposition~\ref{pos prop open}, and Subsection~\ref{sec_43}, the set $\mathcal{P}_X^{V,\cE}:=\mathcal{P}_X^{V}\cap \cE_X\subset \mathcal{P}_X$ of volume expanding periodic orbits for $X^t$ crossing the open set $V$ has full proportion for the  measure $m_X^-$. By construction of $V$, we deduce that for the full proportion set $\mathcal{P}_X^{V,\cE}$ of periodic orbits $\gamma$ for $X^t$, the stable eigenvalue $\mu_\gamma\in (0,1)$ can be recovered from the periods. More precisely, denoting by $\Gamma_p^X$ the function introduced in Definition~\ref{defi recov} for the flow $X^t$, for any $\gamma \in \mathcal{P}_X^{V,\cE}$, there exist $p \in V\cap \gamma$ and $q \in \mathcal{H}_{X,\mathrm{good}}^u(p)$ such that $\log \mu_\gamma=\Gamma_p^X(q)$. 
	
	Let us also consider the set $\Phi(\mathcal{P}_X^{V,\cE})=\mathcal{P}_Y^{V,\cE}\cup \mathcal{P}_Y^{V,\cV} \cup  \mathcal{P}_Y^{V,\cC}\subset \mathcal{P}_Y$, where $ \mathcal{P}_Y^{V,\cE}:=\Phi(\mathcal{P}_X^{V,\cE})\cap \mathcal{E}_Y$, $\mathcal{P}_Y^{V,\cV}:=\Phi(\mathcal{P}_X^{V,\cE})\cap \mathcal{V}_Y$ and
	$\mathcal{P}_Y^{V,\cC}:=\Phi(\mathcal{P}_X^{V,\cE})\cap \mathcal{C}_Y$ are volume expanding, volume preserving, and volume contracting periodic orbits for $Y^t$, respectively. By Proposition~\ref{prop zero prop vp}, the set $\mathcal{P}_Y^{V,\cE}\cup  \mathcal{P}_Y^{V,\cC}$ has full proportion with respect to the equilibrium measure $\Phi_* m_X^-$. We have two cases:
	\begin{enumerate}
		\item\label{premier cas} $\mathcal{P}_Y^{V,\cE}\subset \mathcal{P}_Y$ is a set of positive proportion for the equilibrium measure $\Phi_* m_X^-$;
		\item\label{deuxi cas}
		$\mathcal{P}_Y^{V,\cC}\subset \mathcal{P}_Y$ is a set of positive proportion for the equilibrium measure $\Phi_* m_X^-$.
		%\item\label{troisi cas}
		%both $\mathcal{P}_Y^{\text{exp}}$ and $\mathcal{P}_Y^{\text{cont}}$ have positive proportion for the equilibrium measure $\Phi_* m_X^-$. 
	\end{enumerate}
	Although it is not obvious, we will see that these two cases are actually mutually exclusive. Below we treat these two cases separately. 
	  
\textbf{Case}~\ref{premier cas}. % Let us consider case~\eqref{premier cas}. 
For any periodic orbit $\gamma \in \mathcal{P}_X^{V,\cE}\cap \Phi^{-1}(\mathcal{P}_Y^{V,\cE})$, there exist points $p \in V\cap \gamma$, $q \in \mathcal{H}_{X,\mathrm{good}}^u(p)$, and $\tilde q \in \mathcal{H}_{Y,\mathrm{good}}^u(\Phi(p))$ such that $\log \mu_\gamma=\Gamma_p^X(q)$ and $\log \mu_{\Phi(\gamma)}=\Gamma_{\Phi(p)}^Y(\tilde q)$, where $\mu_{\Phi(\gamma)}\in (0,1)$ is the stable eigenvalue of the volume expanding periodic point $\Phi(p)$ of $Y^t$, and $\Gamma_p^X$ ($\Gamma_{\Phi(p)}^Y$) are the functions introduced in Definition~\ref{defi recov} for the flows  $X^t$ ($Y^t$) at $p$  ($\Phi(p)$).  
In fact, since both functions  $\Gamma_{p}^X$ and $\Gamma_{\Phi(p)}^Y$ are entirely determined by the lengths of shadowing periodic orbits, and by~\eqref{mathc period} these lengths are the same for any homoclinic point $q_X\in \mathcal{H}_X^u(p)$ and the  corresponding homoclinic point $\Phi(q_X)\in \mathcal{H}_Y^u(\Phi(p))$ we have
$$
\Gamma_p^X( q_X)=\Gamma_{\Phi(p)}^Y( \Phi(q_X)).
$$  
Using~\eqref{first eq recov} and~\eqref{second eq recov} we deduce that 
$$
\log\mu_\gamma=\Gamma^X_p(q)=\Gamma^Y_{\Phi(p)}(\Phi(q))\le
\Gamma^Y_{\Phi(p)}(\tilde q)=\log\mu_{\Phi(\gamma)}.
$$
Similarly,
$$
\log\mu_{\Phi(\gamma)}=\Gamma^Y_{\Phi(p)}(\tilde q)=\Gamma^X_p(\Phi^{-1}(\tilde q))\le \Gamma^X_p(q)=\log\mu_\gamma.
$$
Hence
\begin{equation}\label{mathcin st eig}
\mu_\gamma=\mu_{\Phi(\gamma)}.
\end{equation}
%\begin{equation}\label{mathcin st eig}
%	\log \mu_\gamma=%\Gamma_X(p):=
%	\sup_{q_X\in \mathcal{H}_X^u(p)} \Gamma_p^X( q_X)=\sup_{q_Y\in \mathcal{H}_Y^u(\Phi(p))} \Gamma_{\Phi(p)}^Y( q_Y) %=:\Gamma_Y(\Phi(p))
%	=\log \mu_{\Phi(\gamma)}.
%\end{equation} 
In particular, denoting by $\varphi\colon M \to \mathbb{R}$ the Hölder potential 
$$
\varphi:=\psi_X^s-\psi_Y^s\circ \Phi\colon x \mapsto  \frac{d}{dt}\left|_{t=0}\right. \left(\log \|DX^t(x)|_{E^s}\|-\log \|DY^t(\Phi(x))|_{E^s}\|\right),
$$
equation~\eqref{mathcin st eig} implies that for any periodic orbit $\gamma$ in the positive proportion set  $\mathcal{P}_X^{V,\cE}\cap \Phi^{-1}(\mathcal{P}_Y^{V,\cE})$, we have $T_\varphi(\gamma)=0$. By the positive proportion Livshits Theorem (Theorem~\ref{pos prop thm}), we thus conclude that $\varphi$ is a coboundary. In particular, for any periodic point $p \in M$ of $X^t$, denoting by $\mu_p,\mu_{\Phi(p)}\in (0,1)$ the stable eigenvalues at $p$ and $\Phi(p)$ for the flows $X^t$ and $Y^t$, respectively, we have 
\begin{equation}\label{mathc sta}
\mu_p=\mu_{\Phi(p)}. 
\end{equation}

\textbf{Case}~\ref{deuxi cas}.
Let us now consider the second case.
By the definition of $V$, and by~\eqref{first eq recov unst}, for any periodic orbit $\gamma \in \mathcal{P}_X^{V,\cE}\cap \Phi^{-1}(\mathcal{P}_Y^{V,\cC})$, there exist points $p \in V\cap \gamma$, $q \in \mathcal{H}_{X,\mathrm{good}}^u(p)$, and $\tilde q' \in \mathcal{H}_{Y,\mathrm{good}}^s(\Phi(p))$ such that $\log \mu_\gamma=\Gamma_p^X(q)$ and $-\log \lambda_{\Phi(\gamma)}=\Gamma_{\Phi(p)}^Y(\tilde q')$, where $\lambda_{\Phi(\gamma)}>1$ denotes the unstable eigenvalue of the volume contracting periodic point $\Phi(p)$ of $Y^t$, and $\Gamma_{\Phi(p)}^Y$ is the function introduced in Remark~\ref{remark change time} for the flow $Y^t$ at $\Phi(p)$. Let us denote by $(T_n^q)_n$ and $(T_n^{\tilde q'})_n$ the periods of the periodic orbits which shadow the orbits of $q$ and $\tilde q'$, respectively. Arguing as previously, and by
\eqref{mathc period},  %with
%\begin{equation*} 
%	T_n^q=T_n^{\Phi(\tilde q')},\quad \forall\, n \gg 1,
%\end{equation*}
we deduce that 
\begin{equation*}
	\log \mu_\gamma=%\Gamma_X(p):=
	\sup_{q_X\in \mathcal{H}_X^u(p)} \Gamma_p^X( q_X)=\sup_{q_Y\in \mathcal{H}_Y^s(\Phi(p))} \Gamma_{\Phi(p)}^Y( q_Y)%=:\Gamma_Y(\Phi(p))
	=-\log \lambda_{\Phi(\gamma)}.
\end{equation*} 
Since the set   $\mathcal{P}_X^{V,\cE}\cap \Phi^{-1}(\mathcal{P}_Y^{V,\cC})$ has positive proportion, we conclude that for any periodic point $p \in M$ for $X^t$, we have
\begin{equation}\label{mathc unsta}
	\mu_p=\lambda_{\Phi(p)}^{-1},
\end{equation}
where $\mu_p\in (0,1)$ is the stable multiplier at $p$ and $\lambda_{\Phi(p)}>1$ is the unstable multiplier at $\Phi(p)$ for the flows $X^t$ and $Y^t$, respectively.

In particular, if cases~\ref{premier cas} and~\ref{deuxi cas} were occuring simultaneously, from~\eqref{mathc sta}-\eqref{mathc unsta}, we would conclude that the flow $Y^t$ is volume preserving, contrary to our assumption.

To finish the proof we need to repeat the whole argument again, but instead of starting with $m_X^-$ and considering volume expanding periodic points for $X^t$, we can start with $m_X^+$ and consider volume contracting periodic points. Specifically, by construction of $V$, and Proposition~\ref{pos prop open}, for the positive SRB measure $m_X^+$, there exists a full-proportion set $\mathcal{P}_X^{V,\cC}$ of volume contracting orbits $\gamma$ for $X^t$ (crossing the open set $V$) whose unstable eigenvalue $\lambda_\gamma>1$ can be recovered from the periods. Similarly, we have  $\Phi(\mathcal{P}_X^{V,\cC})=\widetilde{\mathcal{P}}_Y^{V,\cE}\cup \widetilde{\mathcal{P}}_Y^{V,\cV} \cup  \widetilde{\mathcal{P}}_Y^{V,\cC}$, where $\widetilde{\mathcal{P}}_Y^{V,\cE}$, $\widetilde{\mathcal{P}}_Y^{V,\cV}$ and
$\widetilde{\mathcal{P}}_Y^{V,\cC}$ are volume expanding, volume preserving, and volume contracting periodic orbits for $Y^t$, respectively, and the set $\widetilde{\mathcal{P}}_Y^{V,\cE}\cup \widetilde{\mathcal{P}}_Y^{V,\cC}\subset \mathcal{P}_Y$ has full proportion for the equilibrium measure $\Phi_* m_X^+$. For any periodic point
$p \in M$ for $X^t$, we denote by $\mu_p<1<\lambda_p$ its eigenvalues, and we denote by $\mu_{\Phi(p)}<1<\lambda_{\Phi(p)}$ the eigenvalues of the periodic point $\Phi(p)$ of $Y^t$. Splitting into two cases as previously, we deduce that 
\begin{enumerate}
	\item either for any periodic point $p$ of $X^t$, we have $\lambda_p=\lambda_{\Phi(p)}$;
	\item or for any periodic point $p$  of $X^t$, we have $\lambda_p=\mu_{\Phi(p)}^{-1}$. 
\end{enumerate}
Notice that now we have arrived at four cases:
\begin{enumerate}
\item for any periodic point $p$ we have $\mu_p=\mu_{\Phi(p)}$ and $\lambda_p=\lambda_{\Phi(p)}$;
\item for any periodic point $p$ we have $\mu_p=\lambda_{\Phi(p)}^{-1}$ and $\lambda_p=\mu_{\Phi(p)}^{-1}$;
\item for any periodic point $p$ we have $\mu_p=\mu_{\Phi(p)}$ and $\lambda_p=\mu_{\Phi(p)}^{-1}$;
\item for any periodic point $p$ we have $\mu_p=\lambda_{\Phi(p)}^{-1}$ and $\lambda_p=\lambda_{\Phi(p)}$. 
\end{enumerate}
In the latter two cases we conclude that $X^t$ is volume preserving, which gives a contradiction. 
In the first case we apply Theorem~\ref{delallave} and conclude that the flows $X^t$ and $Y^t$ are smoothly conjugate. While in the second case, swapping of the eigenvalues at corresponding periodic points implies swapping of the positive and negative SRB measures of the two flows. Indeed, by the Livschits theorem, both $\psi^s_X-\psi^u_Y\circ\Phi$ and $\psi^u_X-\psi^s_Y\circ\Phi$ are cohomologous to $0$ (recall the Definition~\ref{def_SRB}) and $\Phi$ must swap the SRB measures by functoriality property of equilibrium states.
\qed

\subsection{Jacobian rigidity for flows: proof of Theorem~\ref{strong dllave flow}}
\label{sec_74}
We begin with a remark on a finite regularity version of Theorem~\ref{strong dllave flow}.
\begin{remark}\label{remark_theoremA_regularity}
	Theorem~\ref{strong dllave flow} was stated for $C^\infty$ flows for convenience, but it also works in finite regularity, namely, $C^r$ flows, $r \geq 4$, which are $k$-pinched for some $1<k\leq r-1$. 
\end{remark}
Let $X^t\colon M\to M$ and  $Y^t\colon N\to N$ be two transitive dissipative $C^\infty$ Anosov flows on $3$-manifolds $M$ and $N$. Assume that they are $C^0$-conjugate by a homeomorphism $\Phi\colon M \to N$, $\Phi \circ X^t=Y^t \circ \Phi$. 
Assume that for any periodic  point $p=X^T(p)$ Jacobians match, i.e.,
\begin{equation}\label{meme_jac}
	\det DX^T(p)= \det DY^T(\Phi(p)).
\end{equation} 
We apply Theorem~\ref{theo alphc} and conclude that $X^t$ and $Y^t$ are $C^\infty$-conjugated, except, possibly, in one of the following two cases:
\begin{enumerate}
	\item the positive and negative SRB measures of the flows $X^t$, $Y^t$ are swapped by $\Phi$; 
	\item at least one of the foliations $\mathcal{W}_Z^s$ and $\mathcal{W}_Z^u$ is $C^{1+\alpha}$ for some $\alpha>0$, $Z=X,Y$.
\end{enumerate}
We claim that the former case never happens; indeed, swapping of SRB measures implies that any volume expanding periodic point for $X^t$ is mapped to a volume contracting periodic point for $Y^t$, which is ruled out by~\eqref{meme_jac}. 

Let us then assume that we are in the latter case. After possibly reversing time, without loss of generality, we can assume that $\mathcal{W}_X^u$ is $C^{1+\alpha}$, $\alpha>0$. 

\begin{claim}\label{claim_eq_susp_flows}
	The foliation $\mathcal{W}_X^s$ is $C^{1}$ if and only if both $X^t$ and $Y^t$ are constant roof suspension flows. 
\end{claim}
\begin{proof}[Proof of the claim:]
	The reverse implication is clear, so let us focus on the direct one, and assume that $\mathcal{W}_X^s$ is $C^{1}$. Now we have that both $\mathcal{W}_X^s$ and $\mathcal{W}_X^u$ and we can apply Lemma~\ref{lem plante fh}. Taking into account that $X^t$ is dissipative we have that $X^t$ must be (topologically conjugate to) a constant roof suspension flow (over an Anosov diffeomorphism of $\mathbb{T}^2$). Let us denote by $E_Z^s$ and $E_Z^u$ the (strong) stable and unstable bundles of the flow $Z^t$, for $Z=X,Y$. In particular, the distribution $E_X^s\oplus E_X^u$ is integrable. Since the latter condition is topological, it is preserved by the $C^0$-conjugacy map $\Phi$, hence the distribution $E_Y^s \oplus E_Y^u$ is integrable too. By Plante~\cite[Theorem 3.1]{Plante}, we then conclude that $Y^t$ is also a constant roof suspension flow.  % Then we have that $\log \det Df$ is continuously cohomologous to a constant $c_f\in \mathbb{R}$. If $c_f=0$, then the diffeomorphism $f$ is conservative, contrary to our assumption. On the other hand, if $\log \det Df=\gamma\circ f -\gamma+c_f$ for a constant $c_f>0$ and a continuous transfer function $\gamma \colon \mathbb{T}^2 \to \mathbb{R}$, then we can take $n_f\in \mathbb{N}$ large enough such that for any $n \geq n_f$, $\log \det Df^n=\gamma\circ f^n -\gamma+n c_f>0$, which is incompatible with $f$ being a diffeomorphism. The case where $c_f<0$ is impossible for similar reasons.  
\end{proof}

First assume that that $X^t$ and $Y^t$ are not constant roof suspension flows. By the above claim, the foliation $\mathcal{W}_X^s$ cannot be $C^{1}$. 
Similarly, the flow $Y^t$ cannot have both foliations $\mathcal{W}_Y^s$ and $\mathcal{W}_Y^u$ of class $C^1$; we then have two cases:
\begin{enumerate}
	\item\label{cas_un_dicht} %$\mathcal{W}_Y^u$ is $C^{1+\alpha}$, for some $\alpha>0$, but  
	$\mathcal{W}_Y^s$ is not $C^1$;
	\item\label{cas_deux_dicht} %$\mathcal{W}_Y^s$ is $C^{1+\alpha}$, for some $\alpha>0$, but  
	$\mathcal{W}_Y^u$ is not $C^1$.
\end{enumerate}

\textbf{Case}~\ref{cas_un_dicht}.  Let us assume that $\mathcal{W}_Y^s$ is not $C^1$. By Proposition~\ref{propo glob dich} and the transitivity of the flows $X^t,Y^t$, we deduce that there exists a non-empty open set $V\subset M$ such that for any volume expanding periodic orbit $\gamma$ for $X^t$ crossing $V$, the periodic orbit $\Phi(\gamma)$ is also volume expanding because of the Jacobian matching~\eqref{meme_jac}. Moreover, the stable eigenvalues $\mu_\gamma^X,\mu_{\Phi(\gamma)}^Y\in (0,1)$ of $X^t$ and $Y^t$ at $\gamma$ and $\Phi(\gamma)$, respectively, can be recovered from the periods. Let $\mathcal{P}_X^{V,\cE}$ be the set of volume expanding periodic orbits for $X^t$ crossing $V$. We have 
\begin{equation}\label{methcing}
	\log \mu_\gamma=%\Gamma_X(p):=
	\sup_{\mathcal{H}_X^u(p)} \Gamma_p^X=\sup_{\mathcal{H}_Y^u(\Phi(p))} \Gamma_{\Phi(p)}^Y %=:\Gamma_Y(\Phi(p))
	=\log \mu_{\Phi(\gamma)},\quad \forall\, \gamma \in \mathcal{P}_X^{V,\cE},\, p \in \gamma.  
\end{equation} 
Consider the H\"older potential $\varphi\colon M \to \mathbb{R}$,
$$
\varphi\colon x \mapsto  \frac{d}{dt}\left|_{t=0}\right. \left(\log \|DX^t(x)|_{E^s}\|-\log \|DY^t(\Phi(x))|_{E^s}\|\right).
$$
By Proposition~\ref{pos prop open}, and Subsection~\ref{sec_43}, the set $\mathcal{P}_X^{V,\cE}$ has full proportion with respect to the negative SRB measure $m_X^-$, and by~\eqref{methcing}, for any $\gamma\in\mathcal{P}_X^{V,\cE}$, we have $T_\varphi(\gamma)=0$. By the positive proportion Livshits Theorem (Theorem~\ref{pos prop thm}), we conclude that $\varphi$ is a coboundary. Therefore, for any periodic point $p$ for $X^t$, the stable multipliers of $X^t$ and $Y^t$ at $p$ and $\Phi(p)$ match.  %yielding the same conclusion for $f$ and $g$. 
By~\eqref{meme_jac}, we deduce that for any periodic point $p$ of $X^t$, the stable and unstable multipliers of $p$ and $\Phi(p)$ match, and hence, the conjugacy map $\Phi$ between $X^t$ and $Y^t$ is $C^\infty$~\cite{LMM, dll2}, as claimed. \\

\textbf{Case}~\ref{cas_deux_dicht}.  Let us now consider the second case, i.e., $\mathcal{W}_X^u$ %and $\mathcal{W}_Y^s$ are 
is $C^{1+\alpha}$, 
for some $\alpha>0$, but neither $\mathcal{W}_X^s$ nor $\mathcal{W}_Y^u$ is $C^1$. Our goal is to show that this case is actually ruled out by the Jacobian assumption~\eqref{meme_jac}. 
 
For $Z=X,Y$, recall that the potentials for SRB measures $m_Z^-$ and $m_Z^+$ are $\psi_Z^s\colon x \mapsto  \frac{d}{dt}\left|_{t=0}\right. \log \|DZ^t(x)|_{E_Z^s}\|$ and $\psi_Z^u\colon x \mapsto  -\frac{d}{dt}\left|_{t=0}\right. \log \|DZ^t(x)|_{E_Z^u}\|$, respectively. %and $\psi_X\colon x \mapsto  \frac{d}{dt}\left|_{t=0}\right. \log \det DX^t(x)=(\log \mathrm{Jac}_x)'(0)$, 
%and consider the one parameter family $\{\varphi_t\}_{t \in [0,1]}$ of potentials, with 
%$$
%\varphi_t:=t\psi_X^s + (1-t) \psi_X^u. %= \psi_X^u + t \psi_X.
%$$ 
For each $t \in \mathbb{R}$, we denote by $m_t$ the equilibrium measure for $X^t$ associated to the Hölder potential $\varphi_t:=t\psi_X^s + (1-t) \psi_X^u$.  %Recall that the pressure of a H*older potential $\varphi\colon M \to \mathbb{R}$ is defined as $P(\varphi):=\sup_{\mu \in \mathcal{M}_X} h_\mu(X) 
%Recall that that the function $\bar P\colon t \mapsto P(\varphi_t)$ which to $t \in [0,1]$ assigns the pressure of the potential $\varphi_t$ is convex. In fact (see e.g. Parry-Pollicott~\cite[Proposition 4.10, Proposition 4.12]{ParPol}), $\bar P$ here is strictly convex, since $\psi_X$ is not cohomologous to a constant, and  $
%\bar P'(t)=\int_M \psi_X\, d\mu_t$. 
%As in Subsection~\ref{subsec full prop mild diss}, let $t_0=t_0(X^t)\in [0,1]$ be the point at which $\bar P\colon t \mapsto P(\varphi_t)$ achieves its minimum.  
%\begin{equation}\label{der press}
%	\bar P'(t_0)=\int_M \psi_X\, d\mu_{t_0}=0.
%\end{equation} 

%Recall that we denote by $\mathcal{P}_X$ the set of periodic orbits for the flow $X^t$, and that for any periodic orbit $\gamma \in \mathcal{P}_X$, we let $\mu_\gamma<1<\lambda_\gamma$ be its stable and unstable multipliers. %For any $\varrho>1$, let $\mathcal{P}_X^\theta\subset \mathcal{P}_X$ be the set of $\theta$-mildly dissipative periodic orbits, defined as follows:
%$$
%\mathcal{P}_X^\theta:=\left\{\gamma\in \mathcal{P}_X: \mu_\gamma^\theta \lambda_\gamma<1\text{ and }\mu_\gamma\lambda_\gamma^\theta>1\right\}.
%$$  

Let $\gamma\in \mathcal{P}_X$ be a periodic orbit of $X^t$, of period $T>0$, with multipliers $0<\mu=\mu_\gamma<1<\lambda=\lambda_\gamma$. We fix $p\in \gamma$, a homoclinic point $q \in  \mathcal{W}^u_{\mathrm{loc}}(p)$, we let $T'=T_q'\in \mathbb{R}$, so that $q':=X^{T'}(q)\in \Sigma_p\cap \mathcal{W}^s_{\mathrm{loc}}(p)$. Let $(p_n)_{n \geq n_0}$ be the sequence of periodic points given by Lemma~\ref{lemme shadowing} whose orbits shadow the orbit of $q$. We use the same notation that we used in Section~\ref{section asymp ex}. Recall that $\Sigma_p:=\imath_p((-1,1)\times \{0\}\times (-1,1))$ is the local transversal, that $\{\mathcal{T}_p^s(\cdot)\}$ is the stable template along $\mathcal{W}_{\mathrm{loc}}^u(p)\cap \Sigma_p$, and that $\{\mathcal{T}_p^u(\cdot)\}$ is the unstable template along $\mathcal{W}_{\mathrm{loc}}^s(p)\cap \Sigma_p$. Also recall from Lemma~\ref{lemm pol tem}
$$
\tilde P_p^s(\eta):=-\sum_{j=1}^{[k]} \frac{1}{j!}\frac{ \partial_{1}\partial_2^j\tau_p^T(p)}{\mu\lambda^j-1}\eta^j,\quad
\tilde P_p^u(\xi):=-\sum_{j=1}^{[k]} \frac{1}{j!}\frac{ \partial_{1}^j\partial_2\tau_p^T(p)}{\mu^j\lambda-1}\xi^j. 
$$ 

%By Lemma~\ref{full prop mild diss},  
Recall that for any $\varrho>1$, we denote by $\mathcal{P}_X^\varrho\subset \mathcal{P}_X$ the set of $\varrho$-mildly dissipative periodic orbits,  
$$
\mathcal{P}_X^\varrho:=\left\{\gamma\in \mathcal{P}_X: \mu_\gamma^\varrho \lambda_\gamma<1\text{ and }\mu_\gamma \lambda_\gamma^\varrho>1\right\}.
$$  
%has full proportion for the equilibrium state $m_{t_0}$. 
As in the proof of
Proposition~\ref{prop mild dis exp}, we will consider $\varrho_0$-mildly dissipative periodic orbits, with $\varrho_0:=\frac{5}{4}$. By Corollary~\ref{cor full prop mild diss}, there exists $t_\mathcal{C}^{\varrho_0}\in [0,1]$ such that the subset $\mathcal{C}_X^{\varrho_0}\subset \mathcal{P}_X^{\varrho_0}$ of volume contracting periodic orbits, 
$$
%\mathcal{E}^\varrho:=\left\{\gamma\in \mathcal{P}: \mu_\gamma\lambda_\gamma>1\text{ and }\mu_\gamma^\varrho \lambda_\gamma<1\right\},\quad 
\mathcal{C}_X^{\varrho_0}:=\left\{\gamma\in \mathcal{P}_X: \mu_\gamma\lambda_\gamma<1\text{ and }\mu_\gamma\lambda_\gamma^{\varrho_0}>1\right\}
$$
has positive proportion with respect to $m:=m_{t_\mathcal{C}^{\varrho_0}}$.  
\begin{claim}
	If the periodic point $p$ is $\varrho_0$-mildly dissipative, and $\mu\lambda\neq 1$,
	then as  $n \to +\infty$, the periods $T_n$ of the periodic point $p_n$ admit the following asymptotic expansion:
	\begin{equation}\label{devpt mild di}
		T_n=nT+T'+%\frac{\partial_{12}\tau_p^T(p) }{\mu\lambda-1}\eta_\infty
		\zeta_p(q)\mu^{n}+O(\theta^n), %-\left(\mathcal{T}_p^u(\xi_\infty)-\tilde P_p^u(\xi_\infty)
		%+\frac{\partial_{12}\tau_p^T(p) }{\mu\lambda-1}\xi_\infty
		%\right)\eta_\infty\lambda^{-n}+o(\lambda^{-n}). 
	\end{equation}
	for some $\theta \in (0,\mu)$, where $\zeta_p(q):=\left(\mathcal{T}_p^s(\eta_\infty)-\tilde P_p^s(\eta_\infty)\right)\xi_\infty$. 
	
	Moreover, there exists a subset $\tilde{\mathcal{P}}^{\varrho_0}_X\subset \mathcal{P}_X^{\varrho_0}$ of full proportion with respect to $m$ such that for any $\gamma\in \tilde{\mathcal{P}}^{\varrho_0}_X$ and any $p \in \gamma$, there exists a homoclinic point $q \in  \mathcal{W}^u_{\mathrm{loc}}(p)$ such that $\zeta_p(q)\neq 0$. 
\end{claim}
\begin{proof}
	If $p$ is volume expanding, namely $\mu\lambda>1$, then the claim follows directly from Proposition~\ref{coro zxp per}. So consider the case when $p$ is volume contracting, namely, $\mu\lambda<1$. By Proposition~\ref{prop mild dis exp} applied to the inverse flow $X^{-t}$, the point $p$ becomes volume expanding, hence %there exists $\varrho_0=\varrho_0(X^t)>1$ such that if the periodic point $p$ is volume expanding and $\varrho_0$-mildly dissipative, with eigenvalues $\mu<1< \lambda, $ then 
	as $n \to +\infty$, the periods $T_n$ of the periodic point $p_n$ obey the following asymptotic expansion:
	\begin{equation*}
		T_n=nT+T'+\left(\mathcal{T}_p^u(\xi_\infty)-\tilde P_p^u(\xi_\infty)
		%+\frac{\partial_{12}\tau_p^T(p) }{\mu\lambda-1}\xi_\infty
		\right)\eta_\infty\lambda^{-n}+
		\left(\mathcal{T}_p^s(\eta_\infty)-%\frac{\partial_{12}\tau_p^T(p) }{\mu\lambda-1}\eta_\infty
		\tilde P_p^s(\eta_\infty)\right)\xi_\infty\mu^{n}+O(\theta^{n}). 
	\end{equation*}
	Since $\mathcal{W}_X^u$ %and $\mathcal{W}_Y^s$ are 
	is $C^{1+\alpha}$, $\alpha>0$, by Lemma~\ref{lemm pol tem}, for any choice of $p$ and $q$, the term $\zeta_p'(q):=(\mathcal{T}_p^u(\xi_\infty)-\tilde P_p^u(\xi_\infty))\eta_\infty$ vanishes, hence the expansion of $T_n$ is of the form~\eqref{devpt mild di} as well. 
	
	Now, by Proposition~\ref{prop zero prop vp}, we know that the set of periodic orbits $\gamma \in \mathcal{P}_X$ for which  $\mu_\gamma\lambda_\gamma=1$ has zero proportion with respect to $m$. 
	Let $\bar{\mathcal{P}}^{\varrho_0}_X\subset \mathcal{P}_X^{\varrho_0}$ be the subset of periodic orbits $\gamma\in \mathcal{P}^{\varrho_0}_X$ such that $\mu_\gamma \lambda_\gamma\neq 1$, and such that for any $p \in \gamma$, and any homoclinic point $q \in  \mathcal{W}^u_{\mathrm{loc}}(p)$, we have $\zeta_p(q)=0$. If $\bar{\mathcal{P}}^{\varrho_0}_X$ has positive proportion, then by Corollary~\ref{cor_dense}, and arguing as in Lemma~\ref{descartes}, we conclude that the stable bundle $E_X^s$ is $C^{1+\alpha}$ for some $\alpha>0$, hence also $\mathcal{W}_X^s$ (see Remark~\ref{rem psw}), contrary to our assumption on $X^t$. 
\end{proof}
%As in Lemma~\ref{full prop mild diss}, l
\begin{claim}
 The potential $\psi_X^s-\psi_Y^u\circ \Phi$ is a coboundary. 
\end{claim}
\begin{proof}
Let 
$%\tilde{\mathcal{E}}_X^\varrho,
\tilde{\mathcal{C}}_X^{\varrho_0}:=\mathcal{C}_X^{\varrho_0}\cap\tilde{\mathcal{P}}_X^{\varrho_0}$ 
be the subset of volume %expanding and volume 
contracting periodic orbits within $\tilde{\mathcal{P}}_X^{\varrho_0}$. % respectively:
%$$
%\tilde{\mathcal{E}}_X^\varrho:=\left\{\gamma\in \mathcal{P}_X: \mu_\gamma\lambda_\gamma>1\right\}\cap \tilde{\mathcal{P}}_X^\varrho,\quad
%\tilde{\mathcal{C}}_X^\varrho:=\left\{\gamma\in \mathcal{P}_X: \mu_\gamma\lambda_\gamma<1\right\}\cap \tilde{\mathcal{P}}_X^\varrho.
%$$
%\marginpar{get rid of X superscript for mu and lambda?}
By %Lemma~\ref{full prop mild diss} 
Corollary~\ref{cor full prop mild diss} and the preceding claim, %each of the two sets 
$\tilde{\mathcal{C}}_X^{\varrho_0}$  %$\tilde{\mathcal{E}}_X^\varrho,\tilde{\mathcal{C}}_X^\varrho$
has positive proportion with respect to $m$.  

By~\eqref{meme_jac}, any $\gamma\in\Phi(\tilde{\mathcal{C}}^{\varrho_0}_X)$ is a volume contracting periodic orbit for the flow $Y^t$.  Moreover,  $\Phi(\tilde{\mathcal{C}}^{\varrho_0}_X)$ has positive proportion with respect to the equilibrium state $\tilde{m}:=\Phi_*m$. We claim that for a subset $\tilde{\mathcal{C}}^{\varrho_0}_Y\subset \Phi(\tilde{\mathcal{C}}^{\varrho_0}_X)$ of full proportion with respect to $\tilde m$ within $\Phi(\tilde{\mathcal{C}}^{\varrho_0}_X)$, the unstable multiplier $\lambda_{\tilde\gamma}>1$ of any periodic orbit $\tilde\gamma \in \tilde{\mathcal{C}}^{\varrho_0}_Y$ can be recovered from the periods. Indeed, otherwise, by Lemma~\ref{lemme recov single pt} (for volume contracting periodic points in place of volume expanding periodic points), and by Corollary~\ref{cor_dense}, we deduce that the unstable distribution $E_Y^u$ of $Y^t$ (hence also the unstable foliation $\mathcal{W}_Y^u$) is $C^{1+\alpha}$, for some $\alpha>0$, contrary to our assumption. 

Arguing as in the proof of Theorem~\ref{theo alphc} in Subsection~\ref{sub proof thm e}, for any periodic orbit $\gamma$ in the set  $%\tilde{\mathcal{C}}^{\varrho_0}_X \cap
 \Phi^{-1}(\tilde{\mathcal{C}}^{\varrho_0}_Y)$, the stable multiplier $\mu_\gamma\in (0,1)$ of $\gamma$ for $X^t$ matches the inverse of the unstable multiplier $\lambda_{\Phi(\gamma)}>1$ of $\Phi(\gamma)$ for $Y^t$, i.e., 
$$
\mu_\gamma=\lambda_{\Phi(\gamma)}^{-1}.
$$
Since the set $%\tilde{\mathcal{C}}^{\varrho_0}_X \cap
 \Phi^{-1}(\tilde{\mathcal{C}}^{\varrho_0}_Y)$ has positive proportion for $m$, by the positive proportion Livshits Theorem (Theorem~\ref{pos prop thm}), we conclude that $\psi_X^s-\psi_Y^u\circ \Phi$ is a coboundary. 
\end{proof}
Let us now explain how to reach a contradiction. By~\eqref{meme_jac}, the potential 
$$
(\psi_X^s-\psi_X^u)-(\psi_Y^s-\psi_Y^u)\circ \Phi
$$
is a coboundary. Combining with the last claim, we deduce that 
\begin{equation}\label{pot phi deux}
\psi_Y^s\circ \Phi\text{ is cohomologous to }2\psi_X^s-\psi_X^u=\varphi_2. 
\end{equation}
As we have already discussed in Subsection~\ref{subsec full prop mild diss}, the function $\bar P \colon t \mapsto P(\varphi_t)$ is strictly convex, due to the fact that $X^t$ is dissipative. But it always vanishes at $0$, $1$, and by~\eqref{pot phi deux}, it also vanishes at $2$, a contradiction. \\

To finish the proof, it remains to  %assume that the foliation $\mathcal{W}_X^s$ is $C^1$. By the above claim, 
consider the case where $X^t$ and $Y^t$ are topologically conjugate to constant roof suspensions over two Anosov diffeomorphisms $f$ and $g$ of $\mathbb{T}^2$, respectively. In particular, $f$ and $g$ are conjugate by some homeomorphism $h \colon \mathbb{T}^2 \to \mathbb{T}^2$, $h\circ f=g\circ h$, and by~\eqref{meme_jac}, for any periodic point $p=f^n(p)$, we have 
\begin{equation}\label{meme_jac_ter}
	\det Df^n(p)=\det Dg^n(h(p)).
\end{equation} 
Fix a sufficiently large constant $\kappa>0$ such that $\log \det Df+\kappa>0$ and $\log \det Dg+\kappa>0$. 
Let us consider the suspension flows $\tilde X^t$, $\tilde Y^t$ over $f$ and $g$, respectively, with roof functions given by $\log \det Df+\kappa$ and $\log \det Dg+\kappa$. Then, $\tilde X^t$ and $\tilde Y^t$ are transitive Anosov flows (recall that $f$ and $g$ are always transitive) which are also dissipative. Since $f$ and $g$ are conjugate, the suspension flows $\tilde X^t$ and $\tilde Y^t$ are orbit equivalent; actually, by~\eqref{meme_jac_ter} and Livshits Theorem, the flows $\tilde X^t$ and $\tilde Y^t$ are $C^0$-conjugate. % by some $C^0$ map $\tilde \Phi$, $\tilde \Phi\colon \tilde X^t=\tilde Y^t\circ \tilde \Phi$. 

We claim that neither $\tilde X^t$ nor $\tilde Y^t$ is a constant roof suspension flow. Indeed, if $\tilde X^t$ is conjugate to a constant roof suspension flow, then we have that $\log \det Df$ is continuously cohomologous to a constant $c_f\in \mathbb{R}$. If $c_f=0$, then the diffeomorphism $f$ is conservative, contrary to our assumption. On the other hand, if $\log \det Df=\gamma\circ f -\gamma+c_f$ for a constant $c_f>0$ and a continuous transfer function $\gamma \colon \mathbb{T}^2 \to \mathbb{R}$, then we can take $n_f\in \mathbb{N}$ large enough such that for any $n \geq n_f$, $\log \det Df^n=\gamma\circ f^n -\gamma+n c_f>0$, which is incompatible with $f$ being a diffeomorphism. 
The case where $c_f<0$ is impossible for similar reasons. 

Now we have that $\tilde X^t$ and $\tilde Y^t$ are  transitive dissipative $C^\infty$ Anosov flows on $3$-manifolds which are $C^0$-conjugate by a homeomorphism $\tilde \Phi$, $\tilde\Phi \circ \tilde X^t=\tilde Y^t \circ \tilde\Phi$, and by~\eqref{meme_jac_ter}, for any periodic point $p=\tilde X^T(p)$, $T>0$, Jacobians match, i.e., 
\begin{equation*}%\label{meme_jac_quattro}
	\det D\tilde X^T(p)=\det D\tilde Y^T(\tilde \Phi(p)).
\end{equation*} 
Since $\tilde X^t$ and $\tilde Y^t$ are not  constant roof suspension flows, by the previous discussion, we deduce that they are $C^\infty$-conjugate. This implies that for any periodic point $p=\tilde X^T(p)$, stable and unstable multipliers of $p$ and $\tilde \Phi(p)$ match. %and similarly for the conjugate Anosov diffeomorphisms $f$ and $g$. By Theorem~\ref{strong dllave flow}, we conclude that $X^t$ and $Y^t$ are $C^\infty$-conjugated. In particular, for each periodic orbit $\gamma$ of $X^t$, the multipliers of $\gamma$ and $\Phi(\gamma)$ match; t
This in turn implies that for any periodic point $p$ of $f$, the multipliers of $p$ and $h(p)$ match, and hence, the conjugacy map $h$ between $f$ and $g$ is $C^\infty$~\cite{LMM, dll2}. Since the initial conjugated flows $X^t$ and $Y^t$ are constant roof suspension flows over $f$ and $g$, we conclude that $X^t$ and $Y^t$ are, in fact, $C^\infty$-conjugate. The proof is complete.  \qed

\subsection{Jacobian rigidity for diffeomorphisms: proof of Corollary~\ref{strong dllave}}
It is a direct consequence of Theorem~\ref{strong dllave flow}. Indeed, let $f,g\colon \T^2 \to\T^2$ be two dissipative $C^\infty$ Anosov diffeomorphisms that are $C^0$-conjugate by a homeomorphism $h\colon \T^2 \to \T^2$, $h\circ f=g \circ f$. Assume that for any periodic  point $p=f^n(p)$ Jacobians match, i.e.,
\begin{equation}\label{meme_jac_bis}
	\det Df^n(p)= \det Dg^n(h(p)).
\end{equation}
Let us show that $f$ and $g$ are $C^\infty$-conjugate. 

Fix a sufficiently large constant $\kappa>0$ such that $\log \det Df+\kappa>0$ and $\log \det Dg+\kappa>0$. 
Let $X^t\colon M\to M$, $Y^t\colon N \to N$ be the suspension flows over $f$ and $g$, respectively, with roof functions given by $\log \det Df+\kappa$ and $\log \det Dg+\kappa$. Then, $X^t$ and $Y^t$ are transitive Anosov flows (recall that $f$ and $g$ are always transitive) which are also dissipative. Since $f$ and $g$ are conjugate, the suspension flows $X^t$ and $Y^t$ are orbit equivalent; actually, by~\eqref{meme_jac_bis} and Livshits Theorem, the flows $X^t$ and $Y^t$ are conjugate by some $C^0$ map $\Phi \colon M \to N$. 
 
By Theorem~\ref{strong dllave flow}, we conclude that $X^t$ and $Y^t$ are $C^\infty$-conjugated. In particular, for each periodic orbit $\gamma$ of $X^t$, the multipliers of $\gamma$ and $\Phi(\gamma)$ match; this in turn implies that for any periodic point $p$ of $f$, the multipliers of $p$ and $h(p)$ match, and hence, the conjugacy map $h$ between $f$ and $g$ is $C^\infty$~\cite{LMM, dll2}, as claimed. 

Although it is not needed for the proof, let us note that in fact, neither $X^t$ nor $Y^t$ is a constant roof suspension flow. Indeed, as we observed in the proof of Theorem~\ref{strong dllave flow} the case where $\log \det Df$ is cohomologous to a constant is ruled by our assumption that $f$ is dissipative. %  by our assumption. If $c_f=0$, then the diffeomorphism $f$ is conservative, contrary to our assumption. On the other hand, if $\log \det Df=\gamma\circ f -\gamma+c_f$ for a constant $c_f>0$ and a continuous transfer function $\gamma \colon \mathbb{T}^2 \to \mathbb{R}$, then we can take $n_f\in \mathbb{N}$ large enough such that for any $n \geq n_f$, $\log \det Df^n=\gamma\circ f^n -\gamma+n c_f>0$, which is incompatible with $f$ being a diffeomorphism. 
%The case where $c_f<0$ is impossible for similar reasons.  

\subsection{Improving to two $C^1$ foliations: proof of Addendum~\ref{add alph}}
\label{sec_73}

Given $r \geq 3$, let $X^t$, $Y^t$ be two $3$-dimensional transitive $C^r$ Anosov flows that are $k$-pinched, $1<k\leq r-1$, and which are $C^0$-conjugate by a homeomorphism $\Phi$ as in~\eqref{eq_conj}.  

If both flows $X^t$ and $Y^t$ are conservative, then by~\cite{GRH}, at least one of the following holds:
\begin{itemize}
	\item $X^t$ and $Y^t$ are constant roof suspension flows, and all four foliations $\mathcal{W}_X^s,\mathcal{W}_X^u$, $\mathcal{W}_Y^s$ and $\mathcal{W}_Y^u$ are $C^1$;
	\item the conjugacy $\Phi$ is smooth. 
\end{itemize}
Now, if $X^t$ is dissipative but $Y^t$ is conservative, then the result follows from Proposition~\ref{conjugaison vol pre dissip}. In the following, we thus assume that both $X^t$ and $Y^t$ are dissipative, and that at least one of the foliations $\mathcal{W}_X^s,\mathcal{W}_X^u$ is of class $C^{1}$. Up to reversing time, without loss of generality, we will assume that $\mathcal{W}_X^u$ is  $C^{1}$. 

On the one hand, if $\mathcal{W}_X^s$ is also $C^{1}$, then as we saw in Claim~\ref{claim_eq_susp_flows}, both $X^t$ and $Y^t$ are constant roof suspension flows, in which case, all four foliations $\mathcal{W}_X^s$, $\mathcal{W}_X^u$, $\mathcal{W}_Y^s$ and $\mathcal{W}_Y^u$ are $C^{1}$.

On the other hand, if $\mathcal{W}_X^s$ is not $C^{1}$, then by Lemma~\ref{lem plante fh}, $X^t$ and $Y^t$ are not constant roof suspension flows, and at least one of the foliations $\mathcal{W}_Y^s$ and $\mathcal{W}_Y^u$ is not $C^1$. To conclude the proof of the first half of Addendum~\ref{add alph}, it remains to show that if the stable foliation $\mathcal{W}_Y^s$ is not $C^1$, then one of the following holds:
\begin{enumerate} [label=\alph*.]
	\item\label{cas_a} the conjugacy $\Phi$ is smooth along stable leaves and 
	the foliation $\mathcal W_Y^u$ is also $C^1$; 
	\item\label{cas_b} $\Phi_*m_X^-=m_Y^+$. 
\end{enumerate} 
In the following, we thus assume that $\mathcal{W}_X^u$ is $C^1$, but neither $\mathcal{W}_X^s$ nor $\mathcal{W}_Y^s$ is. We split the proof into two cases.
 In the following, given a periodic point $p=X^T(p)$, we denote by $\mu_p:=\lambda_{X,p}^s(T)<1$, resp.  $\lambda_p:=\lambda_{X,p}^u(T)>1$  its stable, resp. unstable multipliers for $X^t$, and by $\mu_{\Phi(p)}:=\lambda_{Y,\Phi(p)}^s(T)<1$, resp.  $\lambda_{\Phi(p)}:=\lambda_{Y,\Phi(p)}^u(T)>1$, the stable, resp. unstable multipliers of $\Phi(p)=Y^T(\Phi(p))$ for $Y^t$. 
%$\bullet$ Let us consider the first case, namely $\mathcal{W}_X^u$ is $C^{1+\alpha}$, and $\mathcal{W}_X^s$, $\mathcal{W}_Y^s$ are not $C^1$. 

\textbf{Case} $1$. Assume that for some equilibrium state $m$ for $X^t$, $m$ and $\Phi_*m$ give positive proportion to the set of volume expanding periodic orbits of $X^t$ and $Y^t$ respectively. Then, arguing in exactly same way as in the first case~\ref{premier cas} of the proof of Theorem~\ref{theo alphc}, we deduce that 
\begin{equation}\label{sta_matching}
	\mu_p=\mu_{\Phi(p)},\quad  \forall\, p=X^T(p). 
\end{equation}
The following proposition shows that we are necessarily in case~\ref{cas_a} above:
\begin{proposition}
	Assume that the flow $X^t$ has $C^1$ unstable foliation $\mathcal{W}_X^u$, and that stable multipliers of $X^t$ and $Y^t$ at corresponding periodic points match as in~\eqref{sta_matching}. 
	% i.e., 
	%\begin{equation}\label{sta_matching}
	%	\lambda_p^s(T)=\lambda_{\Phi(p)}^s(T),\quad  \forall\, p=X^T(p). 
	%\end{equation}
	Then the conjugacy $\Phi$ is smooth along stable leaves, and the unstable foliation $\mathcal{W}_Y^u$ of $Y^t$ is $C^1$. 
\end{proposition}
\begin{proof}
	By~\eqref{sta_matching}, the conjugacy $\Phi$ is smooth along stable leaves (see~\cite{dllSRB}). But $\Phi$ is always smooth along flow lines; by Journé's lemma~\cite{Journe}, we deduce that $\Phi$ is also smooth along the leaves of $\mathcal{W}_X^{cs}$. 
	
	By~\cite[Section 6]{PSW}, $\mathcal{W}_Y^u$ is $C^1$ if and only if local unstable holonomy maps of $Y^t$ are uniformly $C^1$. %Since local unstable holonomy maps are uniformly $C^1$ along local flow lines, by Journé's lemma~\cite{Journé}, it thus suffices to show that they are uniformly $C^1$ along local stable leaves. 
	For $Z=X,Y$, given two points $x,y$ with $y \in \mathcal{W}^u_{Z,\mathrm{loc}}(x)$, we denote by
	$$
	\textup{Hol}_{Z,x,y}^u\colon \mathcal{W}_{Z,\mathrm{loc}}^{cs}(x)\to \mathcal{W}^{cs}_{Z,\mathrm{loc}}(y)
	$$ 
	 the local unstable holonomy map. The conjugacy map $\Phi$ sends the foliations $\mathcal{W}_X^u,\mathcal{W}_X^{cs}$ to the corresponding foliations $\mathcal{W}_Y^u,\mathcal{W}_Y^{cs}$ respectively; as a result, for any points $x,y$ with $y \in \mathcal{W}^u_{X,\mathrm{loc}}(x)$, we have
	\begin{equation}\label{correspondance_hol_maps}
		\Phi\big|_{\mathcal{W}_{X,\mathrm{loc}}^{cs}(y)} \circ \textup{Hol}_{X,x,y}^u\circ \left( \Phi\big|_{\mathcal{W}_{X,\mathrm{loc}}^{cs}(x)}\right)^{-1}=\textup{Hol}_{Y,\Phi(x),\Phi(y)}^u. 
	\end{equation}
	Since $\mathcal{W}_X^u$ is $C^1$, by~\cite[Section 6]{PSW}, the local holonomy maps $\textup{Hol}_{X,x,y}^u$, $y \in \mathcal{W}^u_{X,\mathrm{loc}}(x)$, are uniformly $C^1$. Moreover, as recalled above, $\Phi$ is uniformly $C^1$ along local leaves of $\mathcal{W}_X^{cs}$. By~\eqref{correspondance_hol_maps}, we deduce that local holonomy maps $\textup{Hol}_{Y,\tilde x,\tilde y}^u$, $\tilde y \in \mathcal{W}^u_{Y,\mathrm{loc}}(\tilde x)$, are uniformly $C^1$, which concludes the proof. 
\end{proof}

\textbf{Case} $2$. %Let us consider the equilibrium states $(\Phi^{-1})_* m$ 
Assume now that for any equilibrium state $m$ for $X^t$ such that the set of volume expanding periodic orbits of $X^t$ has positive proportion with respect to $m$ (such as the negative SRB measure $m_X^-$, by Lemma~\ref{lemma_full_proportion}), the equilibrium state $\Phi_*m$ gives full proportion to the set of volume contracting periodic orbits of $Y^t$.    

Let us first assume that $\mathcal{W}_Y^u$ is not $C^1$. Then, considering the equilibrium states $m_X^-$, $\Phi_* m_X^-$, and arguing in exactly same way as in case~\ref{deuxi cas} of the proof of Theorem~\ref{theo alphc}, we deduce that 
\begin{equation*} 
	\mu_p=\lambda_{\Phi(p)}^{-1},\quad  \forall\, p=X^T(p).
\end{equation*}
This implies that $\Phi_* m_X^-=m_Y^+$, i.e., we are in case~\ref{cas_b} above. 

On the other hand, if $\mathcal{W}_Y^u$ is also $C^1$, then according to the following proposition, we are in case~\ref{cas_a}, which completes the proof of the first point of Addendum~\ref{add alph}:
\begin{proposition}
	If $X^t$ and $Y^t$ are not constant roof suspension flows, and both $\mathcal{W}_X^u$ and $\mathcal{W}_Y^u$ are $C^1$, then the conjugacy $\Phi$ is $C^1$ along stable leaves. 
\end{proposition}
\begin{proof}
	As above, for $Z=X,Y$, given two points $x,y$ with $y \in \mathcal{W}^u_{Z,\mathrm{loc}}(x)$, we denote by 
	$
	\textup{Hol}_{Z,x,y}^u%\colon \mathcal{W}_{Z,\mathrm{loc}}^{cs}(x)\to \mathcal{W}^{cs}_{Z,\mathrm{loc}}(y)
	$ 
	the local unstable holonomy map. Let us also denote by 
	$$
	\textup{Hol}_{Z,x,y}^{cu}\colon \mathcal{W}_{Z,\mathrm{loc}}^{s}(x)\to \mathcal{W}^{s}_{Z,\mathrm{loc}}(y)
	$$
	the holonomy map along the leaves of $\mathcal{W}_Z^{cu}$. In particular, for any $z \in \mathcal{W}_{Z,\mathrm{loc}}^s(x)$, there exists a unique $\delta_Z(z)=\delta_{Z,x,y}(z)\in \mathbb{R}$ such that 
	$$
	\textup{Hol}_{Z,x,y}^u(z)=Z^{\delta_Z(z)} \circ \textup{Hol}_{Z,x,y}^{cu}(z). 
	$$
	Moreover, the function $\delta_Z\colon \mathcal{W}_{Z,\mathrm{loc}}^s(x)\to \mathbb{R}$ is $C^1$; indeed, $\mathcal{W}_Z^u$ is $C^1$, hence by~\cite[Section 6]{PSW}, so is the unstable holonomy map $\textup{Hol}_{Z,x,y}^u$. 
	Since the conjugacy map $\Phi$ preserves dynamical foliations, we see that 
	\begin{equation}\label{mathc_fct}
	\delta_{Y,\Phi(x),\Phi(y)} \circ \Phi\big|_{\mathcal{W}_{\mathrm{loc}}^s(x)}=\delta_{X,x,y} . 
	\end{equation}
	\begin{claim}
		There exists a point $y \in \mathcal{W}_{X,\mathrm{loc}}^u(x)$ such that 
		\begin{equation}\label{non_van_pt}
		\delta_{Y,\Phi(x),\Phi(y)}'(\Phi(x))\neq 0. 
		\end{equation}
		%For any interval $I\subset \mathcal{W}_{\mathrm{loc}}^s(\Phi(x))$ of positive length, $\delta_{Y,\Phi(x),\Phi(y)}|_I\not \equiv 0$. 
	\end{claim} 
	\begin{proof}
		Otherwise for any $y \in \mathcal{W}_{X}^u(x)$, we have that $D\textup{Hol}^u_{Y,\Phi(x),\Phi(y)} E_Y^s(\Phi(x))=E_Y^s(\Phi(y))$. By minimality of the unstable foliation $\mathcal{W}_Y^u$, we deduce that $E_Y^s \oplus E_Y^u$ is integrable, hence $Y^t$ is a constant roof suspension flow, by Lemma~\ref{lem plante fh}, a contradiction. See e.g.~\cite[Lemma 5.8]{alvarez2024rigidity} for more details. 
	\end{proof}
Let us fix a point $y$ such that~\eqref{non_van_pt} holds; 
then, by the implicit function theorem,  $\delta_{Y,\Phi(x),\Phi(y)}|_{U}$ is a local $C^1$-diffeomorphism, for some neighborhood $U\subset \mathcal{W}_{Y,\mathrm{loc}}^s(\Phi(x))$ of $\Phi(x)$. By~\eqref{mathc_fct}, we deduce that  $\Phi\big|_{\Phi^{-1}(U)}=(\delta_{Y,\Phi(x),\Phi(y)}|_U)^{-1}\circ \delta_{X,x,y}|_{\Phi^{-1}(U)}$ is $C^1$. Since this is true for any point $x$, we deduce that $\Phi$ is $C^1$. 
\end{proof}

%The proof relies on the same idea as the above proof of Theorem~\ref{theo alphc} but it is not sufficient to work with SRB measures only. We will look for a different equilibrium state $m$ such that volume expanding points for $X^t$ and volume contracting points for $Y^t$ both have positive proportion with respect to $m$ and $\Phi_*m$, respectively; and, consequently we are able to match stable eigenvalues for $X^t$ with unstable eigenvalues for $Y^t$.

Let us now show the second half of Addendum~\ref{add alph}, when  both flows are $C^4$ regular and $\frac{5}{4}$-mildly dissipative. We will also need to use the second order asymptotic formula derived in Section~\ref{sec_AF}. By ~\cite{GRH} and Proposition~\ref{conjugaison vol pre dissip}, we are free to assume that both flows $X^t$ and $Y^t$ are dissipative.  According to the preceding proof, we see that the case which remains to be considered is when
\begin{itemize}
	\item $\mathcal{W}_X^u$ is $C^1$ regular but $\mathcal{W}_X^s$ is not;
	\item the conjugacy $\Phi$ satisfies $\Phi_* m_X^-=m_Y^+$. 
\end{itemize}
In particular, by the latter property,  we have \begin{equation}\label{premi-mathc}
	\mu_p=\lambda_{\Phi(p)}^{-1},\quad  \forall\, p=X^T(p).
\end{equation}
Let us show that $\mathcal{W}_Y^s$ is $C^1$. Assume by contradiction that it is not the case. By   Lemma~\ref{lemma_full_proportion}, the negative SRB measure $m_Y^-$  gives full proportion to the set of volume expanding periodic orbits of $Y^t$. Since we assume that $\mathcal{W}_Y^s$ is not $C^1$, by Lemma~\ref{lemme recov single pt} and  Corollary~\ref{cor_dense}, for a full proportion set of volume expanding periodic points of $Y^t$ with respect to $m_Y^-$, the stable eigenvalues can be recovered. That is, for such points $\Phi(p)=Y^T(\Phi(p))$, with multipliers $0<\mu_{\Phi(p)}<1<\lambda_{\Phi(p)}$,  there exists a homoclinic point $\Phi(q) \in \mathcal{H}_{\text{good}}^u(\Phi(p))$ such that the periods $(T_n)_{n}$ of the periodic points $(\Phi(p_n))_{n}$ given by Lemma~\ref{lemme shadowing} whose orbits shadow the orbit of $\Phi(q)$ satisfy the asymptotic formula (see~\eqref{asymptotiques des periodes}): 
\begin{equation*}%\label{dev_per_y}
	T_n=nT+T'+\zeta_{\Phi(p)}(\Phi(q))\mu_{\Phi(p)}^{n}+O(\theta_{\Phi(p)}^n),
\end{equation*}
with $\zeta_{\Phi(p)}(\Phi(q))\neq 0$, and $\theta_{\Phi(p)}\in (0,\mu_{\Phi(p)})$.  

Recall that we denote by  $\mu_p<1< \lambda_p$ the eigenvalues of the periodic point $p=X^T(p)$. The flows $X^t$ and $Y^t$ are globally $\frac{5}{4}$-mildly dissipative and we can apply the two-term asymptotic formula deduced in Proposition~\ref{prop mild dis exp}. 
Since $\mathcal{W}_X^u$ is $C^1$, by
Lemma~\ref{lemm pol tem}, the unstable template $\mathcal{T}_p^u$ is equal to the (linear) polynomial $\tilde{P}_p^u$ defined therein. Hence the leading term in the formula vanishes and we are left with only one exponential term in the asymptotics,
\begin{equation}\label{dev_per_x}
	T_n=nT+T'+\left(\mathcal{T}_p^s(\eta_\infty)-%\frac{\partial_{12}\tau_p^T(p) }{\mu\lambda-1}\eta_\infty
	\tilde P_p^s(\eta_\infty)\right)\xi_\infty\mu_p^{n}+O(\theta_p^{n}),
\end{equation}
where $(0,\eta_\infty)$ and $(\xi_\infty,0)$ are the normal coordinates of the points $q$ and $X^{T'}(q)$, respectively, $\theta_p\in (0,\mu_p)$, and where $\mathcal{T}_p^s$ is the stable template, and $\tilde{P}_p^s$ is the polynomial defined in Lemma~\ref{lemm pol tem}. Since $\mathcal{W}_X^s$ is not $C^1$, by Lemma~\ref{lemm pol tem} and Corollary~\ref{cor_dense}, for a positive proportion set of periodic orbits $\gamma$ relative to $(\Phi^{-1})_* m_Y^-$, $p\in \gamma$, and for any homoclinic point $q$ in an open subset of $\mathcal{W}_{\mathrm{loc}}^u(p)$, we have $\mathcal{T}_p^s(\eta_\infty)-%\frac{\partial_{12}\tau_p^T(p) }{\mu\lambda-1}\eta_\infty
\tilde P_p^s(\eta_\infty)\neq 0$, hence $\mu_p$ can be recovered from~\eqref{dev_per_x}. By the previous discussion, we conclude that $\mu_p=\mu_{\Phi(p)}$ for a positive proportion set of periodic orbits relative to $m_Y^-$. Therefore, by the positive proportion Livshits Theorem~\ref{pos prop thm}, we conclude that   
\begin{equation*} 
	\mu_p=\mu_{\Phi(p)},\quad  \forall\, p=X^T(p).
\end{equation*}
Comparing to~\eqref{premi-mathc}, and applying Livshits Theorem, we conclude that the flow $Y^t$ is volume preserving. We have thus arrived at a contradiction; in other words, the foliation $\mathcal{W}_Y^s$ is $C^1$, as claimed. The proof of Addendum~\ref{add alph} is complete. \qed

\subsection{Exceptionality of $C^1$ foliations: proofs of Theorem~\ref{theoremC} and Theorem~\ref{coro_D}} The main technical result in this section is that having $C^1$ stable (and similarly unstable) foliation is not a generic property in the space of Anosov vector fields. 

\begin{proposition}\label{prop_generic}
	Let $M$ be a $3$-manifold which supports an Anosov flow. Denote by $\mathcal{A}_M$ the spaces of $C^\infty$ vector fields on $M$, which generate transitive Anosov flows. Then there exists a $C^1$-open and $C^\infty$-dense subset $\mathcal{U}_M^s\subset \mathcal{A}_M$  such that for any $X \in\mathcal{U}_M^s$ %both 
	the (strong) stable foliation $\mathcal W_X^s$ %and the (strong) unstable foliation $\mathcal W^u$ are 
	is not $C^1$ regular.
\end{proposition}

Using this proposition we can easily derive Theorem~\ref{theoremC} and Theorem~\ref{coro_D}. 
\begin{proof}[Proof of Theorem~\ref{theoremC}] We apply Proposition~\ref{prop_generic} for Anosov flows on $M$: first for the stable and then for the unstable foliation (we can reverse the time and apply the proposition), we obtain two open and dense sets whose intersection $\mathcal{U}_M$ forms an open and dense set on which both the stable and the unstable foliations are not $C^1$. Similarly we obtain an open and dense set of Anosov flows on $N$ on which both foliations are not $C^1$. Now, if we have conjugate flows from these sets, then applying Theorem~\ref{theo alphc} immediately yields   Theorem~\ref{theoremC}. 
\end{proof}
\begin{proof}[Proof of Theorem~\ref{coro_D}] 
Let $M$ and $\mathcal{U}_M\subset \mathcal{A}_M$ be as in Theorem~\ref{theoremC}. Let us note that if $X^t$, $Y^t$ are $C^0$-conjugate Anosov flows with respective generators $X,Y\in \mathcal{U}_M$, then by Theorem~\ref{theoremC}, the conjugacy $\Phi$ is $C^\infty$, unless  $\Phi$ swaps positive and negative SRB measures of the two flows, i.e., 
\begin{equation}\label{echange_SRB}
	\Phi_* m_X^+=m_Y^-,\quad\Phi_* m_X^-=m_Y^+.
\end{equation} 
If $X^t$ is conservative, then~\eqref{echange_SRB} implies that $Y^t$ is also conservative, hence by~\cite{GRH}, $X^t$ and $Y^t$ are smoothly conjugate (note that the exceptional case where $X^t$ and $Y^t$ are constant roof suspension flows is ruled out by the assumption that $X,Y\in \mathcal{U}_M$, so that the invariant foliations of $X^t$, $Y^t$ are not $C^1$). 

Otherwise, if $X^t$ is dissipative, then so is $Y^t$. Fix any volume expanding periodic point $p=X^T(p)$. If $X,Y$ are sufficiently $C^1$-close, then  the periodic point $\Phi(p)=Y^T(\Phi(p))$ is also volume expanding, while~\eqref{echange_SRB} implies that $\Phi(p)$ has to be volume contracting, a contradiction. Therefore,~\eqref{echange_SRB} does not occur, and $\Phi$ is $C^\infty$. 
\end{proof}

The proof of Proposition~\ref{prop_generic} consists of two steps. First we will show in Lemma~\ref{lemma_perturb} that if the stable foliation $\mathcal W^s$ is $C^1$ then this property can be destroyed by arbitrarily $C^\infty$-small perturbations. The second step, in Lemma~\ref{lemma_stable_C1}, is to prove that the property of the stable foliation of not being $C^1$ is stable under $C^1$ perturbations of the flow.

\begin{lemma}\label{lemma_perturb} 
Let $X^t\colon M\to M$ be a $3$-dimensional transitive Anosov flows with $C^1$ stable foliation $\mathcal W_X^s$. Then there exists a $C^\infty$-small perturbation $Y^t\colon M\to M$ whose stable foliation $\mathcal W_Y^s$ is not $C^1$. In fact, we can choose $Y^t$ to be dissipative and such that there is a periodic point $p$ for $Y^t$  such that $E^s_Y$ is not $C^1$ along $\mathcal{W}_{Y,\textup{loc}}^u(p)$. 
\end{lemma}
\begin{remark} A similar lemma for higher dimensional Anosov flows appears in~\cite[Lemma 4.2(2)]{FMT}.
\end{remark}
\begin{proof}
	Without loss of generality, after a $C^\infty$-small perturbation of $X^t$ if needed, we can assume that $p=X^T(p)$ is a volume expanding periodic point, with multipliers $\mu<1<\lambda$, $\mu\lambda>1$. 
	
The perturbation $Y^t$ will be a time change of $X^t$ localized in a neighborhood of a periodic point.
We define a local transversal to the flow in the following way
$$
\Sigma_p=\bigcup_{x\in \mathcal W^u_{\mathrm{loc}}(p)} \mathcal W^s_{\mathrm{loc}}(x).
$$
Since the stable foliation is $C^1$ this transversal is also $C^1$ with constant return time $T$. We consider the Poincar\'e return map $\Pi\colon \hat\Sigma_p\to\Sigma_p$, where $\hat \Sigma_p\subset\Sigma_p$ is an appropriately small sub-transversal which contains $p$. Further, recall that we can $C^1$ linearize the Poincar\'e map so that it takes the form $\Pi(\xi,\eta)=(\mu\xi,\lambda\eta)$ with the local stable manifolds being the axes: $\mathcal W^u_{\mathrm{loc}}(p)=\{(0,\eta): |\eta|<\varepsilon_0\}$ and $\mathcal W^s_{\mathrm{loc}}(p)=\{(\xi,0): |\xi|<\varepsilon_0\}$.

Now we pick a homoclinic point $(0,\bar\eta)\in\Sigma_p$ such that its forward orbit under $X^t$ intersects $\Sigma_p$ only on the local stable manifold $\mathcal W^s_{\mathrm{loc}}(p)$. To construct perturbation we pick a smooth function $\rho\colon \Sigma_p\to\mathbb R$ with the following properties:
\begin{enumerate}
\item the function $\rho$ is $C^\infty$ small;
\item\label{pppte_deux} the function $\rho$ vanishes on the local unstable manifold: $\rho(0,\eta)=0$ for all $|\eta|<\varepsilon_0$;
\item\label{pppte_trois} $\frac{\partial}{\partial\xi}\rho(0,\lambda^{-1}\bar\eta)<0$;
\item the function $\rho$ has localized support: $\textup{supp}(\rho)\subset B_\delta((0,\lambda^{-1}\bar\eta))=B_\delta(\Pi^{-1}(0,\bar\eta))\subset\hat \Sigma_p$, where $\delta>0$ is sufficiently small so that $\textup{supp}(\rho)$ is disjoint with  $\Pi(\textup{supp}(\rho))$ and $\Pi^{-1}(\textup{supp}(\rho))$.
\end{enumerate}
We define $Y^t$ as a local reparametrization of $X^t$. Namely, locally at $p$, the flow $Y^t$ is defined as the suspension of $\Pi$ with the roof function $T+\rho$. Since $\rho$ has localized support we can let the flows be the same away from $p$. It is easy to see that since $\rho$ is $C^\infty$ small we also have that $Y$ and $X$ are $C^\infty$ close. 

Further, note that since $\rho$ vanishes on both axes we have that local stable and unstable manifolds through $p$ remain the same for $Y$. Also, since the forward orbit of $(0,\bar\eta)$ is disjoint with $\textup{supp}(\rho)$ we have that $\mathcal W^s_{X,\mathrm{loc}}(0,\bar\eta)=\mathcal W^s_{Y,\mathrm{loc}}(0,\bar\eta)$.

We now proceed to calculate the stable distribution of $Y^t$ along the forward orbit of $(0,\bar\eta)$ and show that it is not $C^1$ at $p$. As usual, we will use $(\xi,t,\eta)$ coordinates to express the stable vector field along $\mathcal W^u_{\mathrm{loc}}(p)$, which we will normalize to have unit first coordinate. For the stable vector based at $(0,\eta)\in\mathcal W^s_{\mathrm{loc}}(p)$ we write $v_X^s(\eta)=(1,0, c^s(\eta))$ for the vector which spans the stable distribution of $X^t$. By the above observations we also have that $E^s_Y(0,\bar\eta)=\textup{span}(v^s_Y(\bar\eta))$ with $v^s_Y(\bar \eta)=v^s_X(\bar\eta)=(1,0, c^s(\bar\eta))$. To obtain the stable vector at $\lambda^{-1}\bar\eta$ we go backwards and apply $DY^{-T}(0,\bar\eta)$:
$$
DY^{-T}(0,\bar\eta)v^s_Y(\bar\eta)=(\mu^{-1},\ell,\lambda^{-1}c^s(\bar\eta)).
$$
After normalizing we have $v^s_Y(\lambda^{-1}\bar\eta)=(1,\mu\ell, \mu\lambda^{-1}c^s(\bar\eta))$. Here $\ell$ defined by $\ell:=-\partial\rho/\partial v^s_X(0, \lambda^{-1}\bar\eta)$ is positive since $\partial\rho/\partial e(0, \lambda^{-1}\bar\eta)<0$ for any unit vector $e$ based at $(0,\lambda^{-1}\bar\eta)$ which points into positive quadrant. Indeed, this fact follows directly from the properties~\eqref{pppte_deux} and~\eqref{pppte_trois} imposed on $\rho$ above. (In fact, $v_X^s$ is almost horizontal vector and $\partial\rho/\partial v^s_X(0, \lambda^{-1}\bar\eta)\simeq\partial\rho/\partial\xi(0, \lambda^{-1}\bar\eta)$.) Calculating $v_Y^s(\lambda^{-\ell}\bar\eta)$ further for $\ell\geq 2$ is straightforward since return time is now constant again and the second coordinate of the vector under iteration remains the same:
$$
DY^{-(\ell-1)T}(0,\bar\eta)v^s_Y(\lambda^{-1}\bar\eta)=(\mu^{-\ell+1}, \mu\ell,\mu\lambda^{-\ell}c^s(\bar\eta)).
$$
After normalizing we obtain 
$$
v_Y^s(\lambda^{-\ell}\bar\eta)=(1,\mu^\ell\ell, \mu^\ell\lambda^{-\ell}c^s(\bar\eta)).
$$
Recall that $v_Y^s(0)=(1,0,0)$. Since $\mu>\lambda^{-1}$ we see that the second coordinate of $v_Y^s$ cannot be better than H\"older with exponent $-\frac{\log\mu}{\log\lambda}$. Hence the stable distribution (and hence foliation) of $Y^t$ is not $C^1$ at $p$.
\end{proof}

\begin{lemma}\label{lemma_stable_C1} Let $X^t\colon M\to M$ be a $3$-dimensional transitive Anosov flow whose stable distribution $E_X^s$ is not $C^1$ along $\mathcal{W}_X^u(p)$ for some volume expanding periodic point $p=X^T(p)$, $\mathrm{Jac}_p(T)>1$. Then there exists a $C^1$-small neighborhood $\mathcal U$ of $X$ in the space of smooth Anosov vector fields such that for any $Y\in \mathcal U$ the stable distribution $E_Y^s$ is also not $C^1$ regular.
\end{lemma}

The idea of the proof is to use the characterization of the $C^1$-smoothness of the stable distribution provided by Lemma~\ref{lemm pol tem}. Namely, we know for $X^t$ that the asymptotic formula~\eqref{asymptotiques des periodes} at $p$ holds with a non-zero coefficient $\zeta_p(q)$ by the exponential term. We would like to claim that the coefficient $\zeta_p(q)$ varies continously as we vary the generator of the flow in $C^1$ topology, as this implies stability of non-$C^1$-smoothness of stable distribution. While this strategy is sound the approach faces a technical obstacle. Formula~\eqref{asymptotiques des periodes} was derived with the help of adapted charts given by Proposition~\ref{propo o good}. The adapted charts depend continuously on the flow (in $C^1$-topology on charts), but only if one varies the flow continuously in $C^2$ topology (this can be checked by examining the construction of adapted charts in Appendix~\ref{appb}). Hence, directly from the proof, we can only say that $\zeta_p(q)$ varies continuously as we continuously vary the generator of the flow in $C^2$ topology. To overcome this issue
	 we will reinterpret the main coefficient in asymptotic formulae for periods in simpler crude charts which have the advantage of varying continuously with the flow in the $C^1$ topology. \\

Let $\{\imath_x\colon (-1,1)^3\to M\}_{x \in M}$ be a family of $C^{r-1}$ adapted charts as in Proposition~\ref{propo o good}, let $\{ \Sigma_x\}_{x\in M}$ be the associated family of transverse sections, $\Sigma_x:=\imath_x((-1,1)\times \{0\}\times (-1,1))$, let $\tau_p\colon \Sigma_p  \to \mathbb{R}$ be the first return time to $\Sigma_p$, and let $\Pi\colon \Sigma_p \to \Sigma_p$, $\imath_p(\xi,0,\eta)=x \mapsto X^{\tau_p(x)}(x)$ 
be the first return map of $X^t$ to the section $\Sigma_p$. 
%$\{\tau_x^t\}_{x,t}$ be the associated family of hitting times. 
For each $x \in M$, we denote by $\mathcal{T}_x^s(\cdot)$ the stable template along $\mathcal{W}_{\mathrm{loc}}^u(x)\cap \Sigma_x$ as in~\eqref{equation stable mnfd normal coord}. 

Fix a volume expanding periodic point $p\in M$, of period $T>0$, with multipliers $0<\mu<1<\lambda$, $\mu\lambda>1$. We let $\tilde P_p^s$ be the polynomial introduced in Lemmata~\ref{lemma tem st}-\ref{lemm pol tem}:
$$
\tilde P_p^s(\eta):=-\sum_{j=1}^{[k]} \frac{1}{j!} \frac{\partial_{1}\partial_2^j\tau_p(p)}{\mu\lambda^j-1} \eta^j.$$
 For any point $q=\imath_p(0,0,\eta)\in  \mathcal{W}^u_{\mathrm{loc}}(p)$, 
we let
\begin{equation}\label{fc_zeta_un}
	\zeta_p(q):=\mathcal{T}_p^s(\eta)-\tilde P_p^s(\eta). 
\end{equation} 
Recall that if, moreover, $q$ is a homoclinic point, then $\zeta_p(q)$ is the coefficient that appears in front of the exponentially small term in the asymptotic formula derived in Proposition~\ref{coro zxp per}. 

Now, let $\hat\imath_p\colon (-1,1)^3\to M$ be any other smooth chart such that $\hat\Sigma_p:=\hat\imath_p((-1,1)\times \{0\}\times (-1,1))$ is a transverse section at $p$, with  $\mathcal{W}_{\mathrm{loc}}^s(p),\mathcal{W}_{\mathrm{loc}}^u(p)\subset\hat\Sigma_p$ being the axes, and $X^t \circ\hat \imath_p(\xi,0,\eta)=\hat\imath_p(\xi,t,\eta)$, for $\xi,t,\eta\in (-1,1)$. 
\begin{remark}\label{continuous_dep_sections}
	Since $\mathcal{W}_{\mathrm{loc}}^s(p)$ and $\mathcal{W}_{\mathrm{loc}}^u(p)$ depend $C^1$ continuously on $X$ in $C^1$ topology. As we vary $X$ we can ensure that the parametrized sections $\hat{\Sigma}_{p(X)}(X)$ also depend $C^1$-continuously on $X$ in the $C^1$ topology. 
\end{remark}
Let   %$\mathcal{T}_p^s(\cdot)$, 
 $\hat{\mathcal{T}}_p^s(\cdot)$ be the associated stable template along %$\mathcal{W}_{\mathrm{loc}}^u(p)\cap \Sigma_p$ and 
 $\mathcal{W}_{\mathrm{loc}}^u(p)\cap \hat\Sigma_p$, %respectively, 
 namely, for  $q=%\tilde\imath_p(0,0,\eta)=
 \hat\imath_p(0,0,\hat\eta)\in \mathcal{W}_{\mathrm{loc}}^u(p)$, we can write
\begin{equation*} 
	%\begin{array}{l}
		\mathcal{W}_{\mathrm{loc}}^s(q)=\left\{\hat\imath_p(\hat\xi,\hat{\mathcal{T}}^s_p(\hat\eta)\hat\xi+\hat b_x^s(\hat\xi,\hat\eta)\hat\xi^2,\hat\eta+\hat c_x^s(\hat\xi,\hat\eta)\hat\xi)\right\}_{\hat\xi \in (-1,1)}. 
		%\mathcal{W}_{\mathrm{loc}}^u(\Phi_x^s(\xi))=\left\{Q_\xi^u(\tilde \eta):=\imath_x(\xi+c_x^u(\xi,\tilde\eta)\tilde\eta,\mathcal{T}^u_x(\xi)\tilde\eta+b_x^u(\xi,\tilde\eta)\tilde\eta^2,\tilde\eta)\right\}_{\tilde\eta \in (-1,1)}.
	%\end{array}
\end{equation*}
Denote by $\hat\tau_p\colon \hat\Sigma_p  \to \mathbb{R}$ the first return time to $\hat \Sigma_p$, and let $\hat\Pi\colon \hat \Sigma_p \to \hat \Sigma_p$, $\hat\imath_p(\hat\xi,0,\hat\eta)=x \mapsto X^{\hat\tau_p(x)}(x)$ be the first return map of $X^t$ to the section $\hat \Sigma_p$. As previously, we slightly abuse notation and write $\hat\tau_p(x)=\hat\tau_p(\hat\xi,\hat\eta)$. For any $q=\imath_p(0,0,\eta)=\hat\imath_p(0,0,\hat\eta)\in \mathcal{W}_{\mathrm{loc}}^u(p)$, denote 
\begin{align*}
		v_p^s(q)&:=D\imath_p(0,0,\eta)(1,0, c_x^s(0,\eta))\in E^{cs}(q)\cap T_q \Sigma_p,\\
	\hat v_p^s(q)&:=D\hat \imath_p(0,0,\hat\eta)(1,0,\hat c_x^s(0,\hat\eta))\in E^{cs}(q)\cap T_q \hat\Sigma_p.
\end{align*}%by $\tilde v_p^s(q)$ the vector $\tilde v_p^s(q):=D\tilde \imath_p(0,0,\eta)(1,0,\tilde c_x^s(0,\eta))\in E^{cs}(q)\cap T_q \tilde\Sigma_p$ and 
%let %$\partial_s \tilde \tau_p(q):=D\tilde \tau_p(q) \tilde v_p^s(q)$.
%$\partial_s \tilde \tau_p(q):=D\tilde \tau_p(q) \tilde v_p^s(q)$. For any 
For any $\ell \geq 1$, we have
\begin{align*}
	%D(\imath_p^{-1})(\Pi^{-\ell}(q))
	D\Pi^{-\ell}(q)v_p^s(q)&=\mu^{-\ell} v_p^s(\Pi^{-\ell}(q)),\\
	%D(\hat\imath_p^{-1})(\hat \Pi^{-\ell}(q))
	D\hat \Pi^{-\ell}(q)\hat v_p^s(q)&=\hat \lambda_{p,q}^s(-\ell) \hat v_p^s(\hat \Pi^{-\ell}(q)),
\end{align*} %=(\tilde \lambda_{p,q}^s(-\ell),0,\tilde \lambda_{p,q}^u(-\ell))$, and define 
where $\hat \lambda_{p,q}^s(-\ell)$ is defined by the above formula, and $0<\mu<1$ is the stable multiplier at $p$. Note that we have used the properties given by Proposition~\ref{propo o good}.  Now, define
\begin{equation}\label{def_tilde_zeta}
\hat \zeta_p(q):=\hat{\mathcal{T}}_p^s(\hat\eta)+\sum_{\ell=1}^{+\infty} \hat \lambda_{p,q}^s(-\ell)\partial_1 \hat \tau_p(\hat \Pi^{-\ell}(q)). 
\end{equation}
(This definition is fully analogous to~\eqref{fc_zeta_un}, only using $\hat\Sigma_p$ instead of $\Sigma_p$.)
\begin{lemma}\label{lemme propr gen}
	There exists a positive function $\vartheta_p\colon \mathcal{W}_{\mathrm{loc}}^u(p)\to \mathbb{R}^+\setminus \{0\}$ such that $\zeta_p=\vartheta_p \cdot \hat\zeta_p$. 
	%The zero sets of $\zeta^s$ and $\tilde\zeta^s$  are equal: 
	%$$
	%\{q\in \mathcal{W}_\mathrm{loc}^u(p):\zeta^s(q)=0\}=\{q\in \mathcal{W}_\mathrm{loc}^u(p):\tilde\zeta^s(q)=0\}.
	%$$ 
\end{lemma}

\begin{remark} Validity of this lemma is fairly clear from the role $\zeta_p$ and $\hat\zeta_p$. Namely, these are coefficients in the asymptotics~\eqref{devpt mild di}, relative to two different transversals, $\Sigma_p$ and $\hat\Sigma_p$. Notice that while the chart is important for  the derivation of the asymptotic formula, the formula itself is coordinate free since all that matters are periods and the eigenvalue $\mu$. Hence the coefficient by $\mu^n$ is the same relative to either chart. Hence,  the  lemma is the observation that relative to $\hat\Sigma_p$ this coefficient is given by $\vartheta_p\cdot\hat\zeta_p$. We still give a formal proof starting with the definition of $\vartheta_p$.
\end{remark}

\begin{proof}
	With the notation introduced in~\eqref{equation stable mnfd normal coord}, 
	for any point $q=\imath_p(0,0,\eta)=\hat\imath_p(0,0,\hat \eta)\in \mathcal{W}_\mathrm{loc}^u(p)$, we have
	$$
	D \imath_p(0,0,\eta)(1,\mathcal{T}^s_p(\eta), c_x^s(0,\eta))=\vartheta_p(q) \cdot D\hat \imath_p(0,0,\hat\eta)(1,\hat{\mathcal{T}}^s_p(\hat\eta),\hat c_x^s(0,\hat\eta))\in E^s(q),
	$$
	for some function $\vartheta_p\colon \mathcal{W}_{\mathrm{loc}}^u(p)\to \mathbb{R}^+$. 
	
For a small neighborhood $U_p$ of $\mathcal{W}_{\mathrm{loc}}^u(p)$, we write $\hat{\Sigma}_p\cap U_p=\{(X^{\gamma_p(z)}(z)):z \in \Sigma_p \cap U_p\}$. Again, we abuse notation by writing $\gamma_p(z)=\gamma_p(\xi,\eta)$. 
%Let $\{\tilde \tau_p^t\}_{p,t}$ be the collection of hitting times for the family $\{\tilde{\Sigma}_{\tilde p}\}_{\tilde p \in M}$. %, and $\tilde{\mathfrak{t}}_p:=\frac{d}{dt}|_{t=0} \tilde{\tau}_p^t$. 
%	Let %$\tilde f_{p,q}\colon \tilde\Sigma_p\to \tilde\Sigma_q$, $z \mapsto X^{\tilde\tau_{p,q}(z)}(z)$ be the Poincaré map from $\tilde\Sigma_p$ to $\tilde\Sigma_q$ and let 
%	$\tilde v_p^s(q):=D\tilde \pi_{p}|_{\mathcal{W}_{\mathrm{loc}}^s(q)}(q)v_0^s(q)\in T_q \tilde W_p^s(q)$. 
Then we have the following relationship between the templates %$\partial_s \tilde\sigma(p,q):=D\tilde\sigma_{p}|_{\mathcal{W}_{\mathrm{loc}}^s(q)}(q) \tilde v_p^s(q)$ satisfies 
	\begin{equation}\label{der tilde tau}
		\mathcal{T}_p^s(\eta)=\partial_1\gamma_p(0,\eta)+\vartheta_p(q)\hat{\mathcal{T}}_p^s(\hat \eta).
	\end{equation}
	Moreover, for any point $z \in \Sigma_p \cap U_p$, and any integer $n\geq 0$ such that $\Pi^{-1}(z),\cdots,\Pi^{-n}(z)$ are well defined, we let $\hat z:=X^{\gamma_p(z)}(z)$, so that 
	\begin{equation*}
		\sum_{\ell=1}^n \tau_p(\Pi^{-\ell}(z))=\sum_{\ell=1}^n\hat\tau_p(\hat\Pi^{-\ell}(\hat z))+\gamma_{p} (\Pi^{-n}(z))-\gamma_p(z).
	\end{equation*}
	%where $\mathcal{P}_p^{-t}\colon \Sigma_p\ni z \mapsto X^{\tau^{-t}_p(z)}(z)\in \Sigma_{X^{-t}(p)}$ denotes the Poincaré map from $\Sigma_p$ to $\Sigma_{X^{-t}(p)}$; we thus obtain
	Differentiating the above equation at $q=\imath_p(0,0,\eta)=\hat\imath_p(0,0,\hat \eta)\in \mathcal{W}_\mathrm{loc}^u(p)$ along the weak stable direction (namely, along $v_p^s(q)$) yields
	\begin{equation}\label{somme_differentiee}
		\sum_{\ell=1}^n \mu^{-\ell}\partial_1\tau_p(\Pi^{-\ell}(q))=\sum_{\ell=1}^n \hat \lambda_{p,q}^s(-\ell)\partial_{1}\hat\tau_p(\hat\Pi^{-\ell}(q))\vartheta_p(q)+\mu^{-n}\partial_1\gamma_{p} (0,\lambda^{-n}\eta)-\partial_1\gamma_p(0,\eta),
	\end{equation}
	where $0<\mu<1<\lambda$ are the multipliers at $p$. In the above formula, we have used that the holonomy map $X^{\gamma_p(\cdot)}(\cdot)$ from $\Sigma_p$ to $\hat \Sigma_p$ preserves the weak stable foliation $\mathcal{W}^{cs}$, and that the return times $\tau_p$ and $\hat \tau_p$ are flat along $\mathcal{W}_\mathrm{loc}^u(p)$, which in our charts is the $\eta$-axis.  
	
	Since $\gamma_p$ is smooth, by Taylor expansion, we also have 
	$$
	\mu^{-n}\partial_1\gamma_{p} (0,\lambda^{-n}\eta)=O((\mu\lambda)^{-n}),
	$$
	which goes to $0$ as $n\to+\infty$ since $\mu\lambda>1$. 
	
	Addding up~\eqref{der tilde tau} and~\eqref{somme_differentiee}, and letting $n\to +\infty$ gives
	\begin{equation}\label{last_comput}
	\mathcal{T}_p^s(\eta)+\sum_{\ell=1}^{+\infty} \mu^{-\ell}\partial_1\tau_p(\Pi^{-\ell}(q))=\vartheta_p(q)\left(\hat{\mathcal{T}}_p^s(\hat \eta)+\sum_{\ell=1}^{+\infty} \hat \lambda_{p,q}^s(-\ell)\partial_{1}\hat\tau_p(\hat\Pi^{-\ell}(q))\right).
	\end{equation}
	Moreover, by~\eqref{der temps}, we have $
	\partial_1 \tau_p(0,\eta)=\sum_{j=0}^{[k]} \frac{1}{j!} \partial_1\partial_2^j \tau_p(0,0)\eta^j$, hence
	$$
	\sum_{\ell=1}^{+\infty} \mu^{-\ell}\partial_1\tau_p(\Pi^{-\ell}(q))=\sum_{\ell=1}^{+\infty} \mu^{-\ell}\partial_1\tau_p(0,\lambda^{-\ell}\eta)=\sum_{j=0}^{[k]} \frac{1}{j!} \partial_1\partial_2^j \tau_p(0,0)\eta^j\sum_{\ell=1}^{+\infty}(\mu\lambda^j)^{-\ell}=-\tilde P_p^s(\eta).
	$$
	With the notation introduced in~\eqref{fc_zeta_un}-\eqref{def_tilde_zeta}, we see that the left hand side of~\eqref{last_comput} is indeed equal to $\zeta_p(q)$, while the right hand side is equal to $\vartheta_p(q)\hat\zeta_p(q)$, which gives the posited equality $\zeta_p(q)=\vartheta_p(q)\hat \zeta_p(q)$. 
\end{proof}

We can finish the proof of Lemma~\ref{lemma_stable_C1}:
\begin{proof}[Proof of Lemma~\ref{lemma_stable_C1}]
Let $X^t\colon M\to M$ be a $3$-dimensional transitive Anosov flow whose stable distribution $E_X^s$ is not $C^1$ along $\mathcal{W}_X^u(p)$ for some volume expanding periodic point $p=X^T(p)$, $\mathrm{Jac}_p(T)>1$. 

Then, by Lemma~\ref{lemm pol tem},  there exists a point $q \in \mathcal{W}_{\mathrm{loc}}^u(p)$ such that  
\begin{equation}\label{zeta_non_vanish_q}
	\zeta_p(q)=\mathcal{T}_p^s(q)-\tilde P_p^s(q)\neq 0. 
\end{equation}
By Lemma~\ref{lemme propr gen}, we deduce that $\hat \zeta_p(q)\neq 0$. 

By Remark~\ref{continuous_dep_sections}, if $Y$ is $C^1$-close to $X$, then, for the Anosov flow $Y^t$, we can choose a transverse section  $\hat{\Sigma}_{p(Y)}(Y)$ adapted to $Y^t$ which is close to $\hat \Sigma_p$ in the $C^1$ topology. The corresponding return time to $\hat{\Sigma}_{p(Y)}(Y)$ is also $C^1$ close to $\hat\tau_p$. From the expression of $\hat \zeta_p$ in~\eqref{def_tilde_zeta}, we see that all ingredients --- template, multipliers, derivatives of the return time --- are close to those for $X^t$. Hence the associated function $\hat \zeta_{p(Y)}$ for $Y^t$ is $C^0$-close to the function $\hat \zeta_p$. Let us denote by $q(Y)\in \mathcal{W}_{Y,\mathrm{loc}}^u (p(Y))$ the continuation of $q$. Given the sections $\hat{\Sigma}_{p(Y)}(Y)$ for $Y^t$, let us denote by $\zeta_{p(Y)}$ the function for $Y^t$ analogous to the one defined in~\eqref{fc_zeta_un} for $X^t$. By~\eqref{zeta_non_vanish_q} and Lemma~\ref{lemme propr gen} and observed continuity, we deduce that for all $Y$ is sufficiently $C^1$-close to $X$, we have %$\hat \zeta_Y^s(q(Y))\neq 0$, hence by Lemma~\ref{lemme propr gen}, we deduce that $\hat \zeta^s(q)\neq 0$.
$$
\hat \zeta_{p(Y)}(q(Y))\neq 0 
$$
and, hence, by applying Lemma~\ref{lemme propr gen} again $\zeta_{p(Y)}(q(Y))\neq 0.$
Then, by Lemma~\ref{lemm pol tem}, we have that the stable distribution $E_Y^s$ is not $C^1$ along $\mathcal{W}_{Y,\mathrm{loc}}^u(p(Y))$.
\end{proof}

\subsection{Proofs of Theorem~\ref{theorem_H} and Corollary~\ref{cor_I}} The proofs are based on exactly the same ideas and arguments so we will be brief. Since $\phi$ and $\xi$ are positive functions we can consider reparametrized flows $X_\phi^t$ and $Y^t_\xi$ generated by $\frac1\phi X$ and $\frac1\xi Y$, respectively. Then the matching condition of $\phi$- and $\xi$-weights becomes matching of periods for $X_\phi^t$ and $Y_\xi^t$ under the orbit equivalence $\Phi$. Hence, by Livshits Theorem, we can promote $\Phi$ to a conjugacy $\bar \Phi$ such that $\bar\Phi\circ X_\phi^t=Y_\xi^t\circ\bar\Phi$. This puts us into a position to apply Theorem~\ref{theo alphc} and conclude smoothness of $\bar\Phi$ apart from exceptional cases. Hence, once exceptional cases are ruled out, we have that homeomorphism $\bar\Phi$ is the posited smooth orbit equivalence for $X^t$ and $Y^t$. 

The exceptional cases are taken care of by assumptions in Theorem~\ref{theorem_H}. Indeed, if $\bar\Phi$ swaps the SRB measures then we have corresponding relationship between the stable and unstable logarithmic infinitesimal Jacobians. Namely, $\psi^u_{X_\phi}$ is cohomologous to $\psi^s_{Y_\xi}\circ \bar\Phi$ and $\psi^s_{X_\phi}$ is cohomologous to $\psi^u_{Y_\xi}\circ\bar\Phi$. Subtracting we have that the infinitesimal full Jacobian $\psi^s_{X_\phi}-\psi^u_{X_\phi}$ is cohomologous to $(\psi^u_{Y_\xi}-\psi^s_{Y_\xi})\circ\bar\Phi$. In particular, this implies (in fact, equivalent) that for every periodic point $p=X_\phi^T(p)$ we have
$$
\log|\det DX_\phi^T(p)|=-\log|\det DY_\xi^T(\bar\Phi(p))|,
$$
contradicting the assumption on the existence of a periodic point with logarithmic Jacobians of the same sign.

Finally, in Lemma~\ref{lemma_perturb} we proved that for a generic reparametrization $X^t_\phi$ of $X^t$ neither the stable $\mathcal{W}^s_{X_\phi}$ nor the unstable foliation $\mathcal{W}^u_{X_\phi}$ is $C^1$. Hence for a generic choice of $\phi$ and $\xi$ none of the foliations $\mathcal{W}^s_{X_\phi}$, $\mathcal{W}^u_{X_\phi}$, $\mathcal{W}^s_{Y_\xi}$ and $\mathcal{W}^u_{Y_\xi}$ is $C^1$ regular, thus ruling out last exceptional case (Case 3) in Theorem~\ref{theo alphc}.

To obtain Corollary~\ref{cor_I} note that if $\phi$ and $\xi$ is a pair of functions with matching sums over all matching periodic orbits, then so are $\phi+c$ and $\xi+c$. For a sufficiently large constant $c$ we have $\phi+c>0$ and $\xi+c>0$. Now we can suspend $f$ and $g$ using $\phi+c$ and $\xi+c$, respectively, to obtain conjugate flows $X^t_\phi$ and $Y^t_\xi$ as in the proof of Theorem~\ref{theorem_H}. Hence $X^t_\phi$ and $Y^t_\xi$ are smoothly conjugate which implies that $f$ and $g$ are smoothly conjugate.\qed

 \section{On $C^{1+\textup{H\"older}}$ Anosov diffeomorphisms which are not $C^1$ conjugate to more regular ones}\label{sec_cawley}
 
 Here we provide a proof Corollary~\ref{cor_cawley}. In fact, we expect that stronger results should hold true and $\pi_{2\to 1}$ is not surjective either. We briefly discuss possibility of such stronger results at the end of this section.
 
 Throughout we will fix  a hyperbolic automorphism $L\colon\T^2\to\T^2$ and work only with Anosov diffeomorphisms homotopic to $L$. Accordingly, we consider the Teichm\"uller spaces $\cT^r_L$  of $C^r$ conjugacy classes of $C^r$ Anosov diffeomorphisms which are homotopic to $L$.

 We will say that a diffeomorphism is $C^{1+\textup{H}}$-regular if it is $C^1$ with H\"older continuous differential for some positive H\"older exponent. Throughout this section we will use $Jf$, $J^sf$ and $J^uf$ to denote the Jacobian, the stable Jacobian and the unstable Jacobian of an Anosov diffeomorphism $f$, respectively.  Recall the following realization result of Cawley~\cite{cawley}.

 \begin{theorem}[\cite{cawley}]\label{thm_cawley}
Given a pair of $C^{1+\alpha}$ potentials $\varphi,\psi\colon\T^2\to\R$ such that $P_L(\varphi)=P_L(\psi)=0$ there exists a $C^{1+\textup{H}}$ Anosov diffeomorphism $g$ conjugate to $L$ via $h$, $h\circ g=L\circ h$ such that
\begin{itemize}
	\item $-\log J^ug$ is cohomologous to $\varphi\circ h$ over $g$;
	\item $\log J^sg$ is cohomologous to $\psi\circ h$ over $g$.
\end{itemize}
 \end{theorem}
  We note that in this result we can replace the ``base-point'' automorphism $L$ with any Anosov diffeomorphism $f$ homotopic to $L$, since any such $f$ is H\"older conjugate to $L$.

\begin{remark}\label{remark_FG}
	 A different and very clean proof of this result (when $\varphi=\psi$) was also recently given by Kucherenko and Quas~\cite{QK}. Yet another version of Cawley's construction also appeared in~\cite[Appendix~A]{FG}. The latter construction is more geometric and allows for continuous realization of finite-parameter families of potentials into the space of $C^{1+\textup{H}}$ Anosov diffeomorphisms. In particular, if $\varphi$ and $\psi$ are $C^0$ close to constant then the diffeomorphism $g$ can be constructed to be $C^1$ close to the ``base-point'' automorphism $L$.
	 \end{remark}
 
 We proceed with a proof of Corollary~\ref{cor_cawley}.
 \begin{proof}
 	The spaces $\cT^r_L$ are connected components of $\cT^r(\T^2)$ which are respected by $\pi_{3\to1}$. Hence it is enough to verify that the restriction $\pi_{3\to1}\colon \cT^3_L\to\cT^1_L$ is not surjective. 
 	
 	The starting point of the construction is a dissipative Anosov diffeomorphism $f\colon \T^2\to \T^2$ which is sufficiently $C^1$ close to $L$ (to be specified in the course of construction).
 We will use Cawley's realization over $f$ (that is, having $f$ as a ``base-point'' instead of $L$)  to produce a $C^{1+\textup{H}}$ Anosov diffeomorphism, which is not $C^1$ conjugate to a $C^{3}$ diffeomorphism. Recall that this would mean that $\pi_{3\to 1}$ is not onto.
 
 We impose the following conditions on the dissipative Anosov diffeomorphism $f$ which will make it possible to apply Theorem~\ref{theo alphc} for a suspension flow over $f$:
 \begin{itemize}
 	\item $f$ is $C^3$ regular;
 		\item $f$ is sufficiently close to $L$ in $C^1$ topology so that a suspension flow is $\frac 54$-mildly dissipative and $\log Jf$ is sufficiently close to 0;
 			\item let $X^t$ denote the suspension flow over $f$ with the roof function $1+\log Jf$; we can pick $f$ so that both stable and unstable foliations of $X^t$ are not $C^1$.
 \end{itemize}
 
 The last property can be arranged easily using arguments which almost the same as those used to prove Lemma~\ref{lemma_perturb}. Assume that, say, the stable foliation of $X^t$ is $C^1$. In Lemma~\ref{lemma_perturb} we have perturbed the roof function to destroy the $C^1$ property. In the current setup, the roof, obviously, cannot be perturbed independently of the base dynamics $f$, but we can still use the same approach. Namely we can consider a volume expanding periodic point and the $C^1$ stable distribution along the unstable manifold of this periodic point. Then we can perturb $f$ in the neighborhood of a homoclinic point of such a periodic point to ensure that the derivative of the Jacobian of $f$ at the homoclinic point along the stable direction changes slightly. Once such perturbation is made, we have almost in the same setup as that in Lemma~\ref{lemma_perturb}. The only difference is that the perturbation is not a time change since we have perturbed the base dynamics as well. However it is easy to see that this difference does not affect the arguments of Lemma~\ref{lemma_perturb} which go through to yield that the stable foliation is no longer $C^1$ regular. It could happen now that the unstable foliation is $C^1$. The $C^1$ regularity of the unstable foliation can be destroyed in the same way, using another, even smaller perturbation, while the non-$C^1$-smoothness of the stable foliation persists under this last perturbation, by Lemma~\ref{lemma_stable_C1}.

  		Now let $\phi$ be any H\"older potential with zero pressure, $P_f(\phi)=0$, which is not cohomologous to $-\log {J}^uf$ and which is sufficiently close to constant. Let $c=-P_f(\log {J}f+\phi)\in\R$. We note that $c$ is close to 0 since we have assumed that $\log Jf$ is close to 0. Let $\psi=\log Jf+\phi+c$. Then we clearly have  
 		\begin{itemize}
 		\item $P_f(\phi)=0$;
 		\item $P_f(\psi)=0$;
 		\item $\psi-\phi=c+\log Jf$. 
 	\end{itemize}
 	By applying Theorem~\ref{thm_cawley} with ``base-point'' $f$,  there exist a $C^{1+\textup{H}}$ Anosov diffeomorphism $g\colon\T^2\to\T^2$ and a bi-H\"older homeomorphism $h:\T^2\to\T^2$ such that \begin{itemize}
 		\item $h\circ g=f\circ h$;
 		\item $\log Jg$ is close to 0 and $g$ is $\frac 54$-mildly dissipative;
 		\item $-\log J^ug$ is cohomologous to $\phi\circ h$ over $f$;
 		\item $\log J^sg$ is cohomologous to $\psi\circ h$ over $f$; 
 		\item $\log Jg$ is cohomologous to $c+\log Jf\circ h$;
 	\end{itemize}
 	where the last item is the immediate consequence of the preceding two items. 
 	The second property follows from the discussion in Remark~\ref{remark_FG} combined with the fact that the potential $\varphi$ was chosen sufficiently close to constant: then $g$ can be chosen sufficiently $C^1$ close to $f$, which in turn was assumed to be $C^1$ close to $L$.
 	
 	Now, if $g$ was $C^1$ conjugate to a $C^{3}$ diffeomorphism, then all of the above items would still be true for this new $C^3$ diffeomorphism. Thus, without loss of generality, we can assume that $g$ itself is  $C^3$ regular. 
 	
 	Since $c$ and $\log Jf$ are both close to 0 we have that  $1-c+\log Jg>0$ and $1+\log Jf>0$. Consider the suspension of $g$ and $f$ by $1-c+\log Jg$ and by $1+\log Jf$, respectively. By the last item above these suspension flows have matching periods and hence are $C^0$ conjugate by the  Livshits theorem. This puts us in the position to apply Theorem~\ref{theo alphc} to the suspension flows with $r=3$ and $k=2$.\footnote{One caveat is that formally speaking the suspension flows are $C^2$ because we have suspended using the Jacobians which are $C^2$ functions, however the $C^3 $ regularity is needed only for the construction of $C^2$ adapted charts, and this construction still goes through for such suspension flows over $C^3$ diffeomorphisms. Indeed, the loss of the derivative in the construction of the adapted charts occurs in Lemma~\ref{lemma_B3}. By inspecting the proof one can check that this lemma can be proved using the base $C^3$ dynamics instead of using the flow directly, resulting in $C^2$ charts after the first adjustment. Recall that the further adjustments in the proof in Appendix~\ref{appb} do not result in any further loss of regularity.}
 	Further,  both of the suspension flows were constructed to be $\frac 54$-mildly dissipative. Hence, Addendum~\ref{add alph} also applies in this setting.
 	
 	We recall that we have the following cases provided by Theorem~\ref{theo alphc} and Addendum~\ref{add alph}:
 		\begin{enumerate}
 			\item the flows $X^t$ and $Y^t$ are $C^{3_*}$-conjugate;
 		\item the conjugacy swaps SRB measures of the flows;
 		\item at least one of the foliations $\mathcal{W}_X^s$ and $\mathcal{W}_X^u$  is $C^{1+\alpha}$, $\alpha>0$; similarly, at least one of foliation $\mathcal{W}_Y^s$, $\mathcal{W}_Y^u$  is $C^{1+\alpha}$.
 	\end{enumerate}

 	If the conjugacy is at least $C^1$, then $-\log J^ug$ would be cohomologous to $-\log J^uf\circ h$ contradicting to the fact that $\phi$ was chosen to be not cohomologous to $-\log J^uf$.
 	
 	In the second case, from swapping of SRB measures, we have that $\phi$ is cohomologous to $\log J^sf$ and $\psi$ is cohomologous to $-\log J^uf$. Then we obtain that $\psi-\phi$ is cohomologous to $-\log Jf$. On the other hand recall that $\psi-\phi$ is cohomologous to $c+\log Jf$. This implies that $\log Jf$ is cohomologous to the constant  $c/2$. Hence $c=0$ and $f$ is conservative, again contradicting our assumptions.
 	
 	Finally, we recall that we have arranged the suspension flow of $f$ with the roof $1+\log Jf$ to have both stable and unstable foliations to be non-$C^1$, which rules out the last case. 
 	
 	We have arrived at a contradiction in each case. Hence $g$ is not conjugate to any $C^3$ diffeomorphism.
 \end{proof}
 Observe that since there is an open set of potentials $\phi$ which can be used in the construction, that we in fact have an infinite dimensional family $\EuScript{C}(f)$ of $C^{1+\textup{H}}$ diffeomorphisms which are not $C^1$-conjugate to $C^3$ diffeomorphisms.

\begin{remark} 
	We believe that by going through the above arguments very meticulously, tacking care of all regularities while using a Lipschitz potential $\phi$ we can improve Corollary~\ref{cor_cawley}. Namely, for any $\varepsilon>0$ the map $\pi_{2+\varepsilon\to 2-\varepsilon}$ is not onto, that is, there exist $C^{2-\varepsilon}$ Anosov diffeomorphisms which are not conjugate to any $C^{2+\varepsilon}$ diffeomorphism. Proving non-surjectivity that both $\pi_{3\to2}$ or $\pi_{2\to1}$ or for higher regularities requires a different approach.
\end{remark}

\section{Examples}\label{section_examples}

In this section we present examples of pairs of Anosov flows which are $C^0$ conjugate but not $C^1$ conjugate. For all such examples we will always have that the strong stable (or strong unstable) distribution is $C^k$ for some $k\ge 1$. In particular we will see that swapping of SRB measures case indeed occurs as was explained to us by Ceki\'c and Paternain.

\subsection{Perturbing along the strong stable foliation}\label{sec_examles}
Let $X^t\colon M\to M$ be a smooth transitive Anosov flow on a $3$-dimensional manifold with $C^k$, $k\ge 1$, orientable stable distribution $E^s$. We denote by $S$ a $C^k$ vector field which generates $E^s$. Let $\rho\colon M\to R$ be a smooth $C^1$-small function. Then the vector field $X_\rho=X+\rho S$ generates an Anosov flow $X_\rho^t$. The flow $X_\rho^t$ is a perturbation of $X^t$ and we proceed to point out several properties.
\begin{enumerate}
\item {\it The flows $X^t$ and $X_\rho^t$ have the same strong stable distribution $E^s$.}\\
Indeed, since the stable distribution of $X_\rho^t$ must be close to $E^s$ it suffices to check that $E^s$ invariant under $X_\rho^t$, that is, $DX_\rho^t(E^s)=E^s$ for all $t$. To check this we show that the derivative of $S$ along $X_\rho$ is in $E^u$, and then the invariance follows by integration. We have
$$
L_{X_\rho}S=[X_\rho,S]=[X,S]+[\rho S,S]=L_XS-S(\rho)S\in E^s
$$
since $S$ is $X^t$-invariant.

\item {\it The flows $X^t$ and $X_\rho^t$ are $C^0$ conjugate.}\\
One could explicitly exhibit the conjugacy which slides the points along strong stable leaves, but it is simpler to see that periods on periodic orbits remain the same and, hence, by Livshits theorem, there exists a $C^0$ conjugacy. Recall that these flows share the strong stable foliation. Hence, if $p=X^T(p)$ is a periodic point then corresponding periodic $\bar p$ for $X_\rho^t$ belongs to the same leaf $\mathcal{W}^s(p)$. We have $X^T(\mathcal{W}^s(p))=\mathcal{W}^s(p)$. Because $X_\rho-X\in E^s$, the same is true for $X_\rho^t$: $X_\rho^T(\mathcal{W}^s(p))=\mathcal{W}^s(p)$. In particular, $X^T(\bar p)=\bar p$.

\item {\it For an appropriate choice of $\rho$ the flows $X^t$ and $X_\rho^t$ are not $C^1$ conjugate.}\\
Consider a periodic orbit $\gamma$ of period $T$ and a local weak-stable leaf $\mathcal{W}^{cs}_{X,\mathrm{loc}}(\gamma)$ which is a cylinder. We can put linearizing coordinates on this leaf so that $\mathcal{W}^{cs}_{X,\mathrm{loc}}(\gamma)$ is identified with $ (-\eps,\eps)\times [0,T]$, $(x,T)\sim (\lambda x, 0)$ and $X$ is given by $\frac{\partial}{\partial t}$. Here $\lambda$ is the stable eigenvalue of $\gamma$. We now define $\rho$ on $\mathcal{W}^{cs}_{X,\mathrm{loc}}(\gamma)$ by the formula
$$
\rho(x,t)=c_0\lambda^{-t/T}x,\quad c_0>0.
$$
Provided that the constant $c_0$ is sufficiently small such $\rho$ can be extended to the whole of $M$ with a small $C^1$ norm.

Notice that $\rho$ vanishes on $\gamma$. Hence $\gamma$ remains periodic under $X_\rho$. We can check that the stable eigenvalue has increased in value. Indeed, if $x\in(0,\eps/2)$ then we can estimate
$$
\int_0^T \rho(X_\rho^t(0,x))\, dt\ge \int_0^T\rho(x,t)\, dt\ge c_0 Tx.
$$
Hence we have that $X_\rho^T(x,0)=(\bar x,T)$ with $\bar x\ge x+c_0Tx$. If $\bar\lambda$ is the stable eigenvalue of $X_\rho$ at $\gamma$ we have
$$
\bar\lambda=\lim_{x\to 0}\frac{\lambda \bar x}{x}\ge \lambda+\lambda c_0 T>\lambda.
$$
We conclude that $X^t$ and $X_\rho^t$ have different stable eigenvalue at $\gamma$ and, hence, cannot be $C^1$ conjugate.
\end{enumerate}

\begin{remark} Because the unstable eigenvalue of $\gamma$ remained the same and the stable eigenvalue was perturbed, we have that $X_\rho^t$ is dissipative. From this fact it is easy to conclude that the strong unstable distribution of $X_\rho^t$ is {\it not} $C^1$. Indeed, if the unstable distribution is also $C^1$, then by Lemma~\ref{lem plante fh}, the flow $X_\rho^t$ has to be a constant a constant roof suspension. % let $\omega$ be the invariant 1-form such that $\omega(X_\rho)=1$ and $\ker \omega$ is the sum of strong stable and unstable distributions. If the strong unstable distribution is also $C^1$ then $\omega$ is $C^1$ and $d\omega$ is a $C^0$ 2-form. Hence $\omega\wedge d\omega$ is an $X_\rho^t$ invariant $3$-form. Since $X_\rho^t$ is dissipative we can conclude that $\omega\wedge d\omega\equiv 0$. By Frobenius theorem $\ker \omega$ is integrable and then work of Plante~\cite{Plante} implies that $X_\rho^t$ is a constant roof suspension.
\end{remark}

\begin{remark} The above construction applies to any contact flow $X^t$, however the perturbed flow $X_{\rho}^t$, generally speaking, is only $C^{1+\textup{H\"older}}$ regular due to the $C^{1+\textup{H\"older}}$ regularity of the strong stable distribution.
\end{remark}

\begin{remark} If $X^t$ is the geodesic flow on a surface of constant negative curvature and $\rho\in C^\infty(M)$ then $X_\rho^t$ is a $C^\infty$ smooth flow since the stable horocyclic foliation is $C^\infty$. By making two different perturbations one can easily produce a pair of $C^\infty$ dissipative Anosov flows $X_{\rho_1}^t$ and $X_{\rho_2}^t$ which are merely $C^0$ conjugate.
\end{remark}

\subsection{Perturbing along both strong foliations} We point out that the above example has two very special features:
\begin{enumerate}
\item the flows $X^t$ and $X^t_\rho$ share the strong stable foliation;
\item the flows $X^t$ and $X^t_\rho$ have matching unstable eigenvalues and hence the $C^0$ conjugacy between them is, in fact, smooth along unstable foliation.
\end{enumerate}
Both of these features can be destroyed by further modifying the construction in the following way.

Let $X^t$ and $X^t_\rho$ be as before. Assume that the strong unstable distribution of $X^t$ is also $C^k$, $k\ge 1$, and generated by a vector field $U$. Then using the same idea, for an appropriate $\varsigma\colon M\to R$ the flow $X_\varsigma^t$ given by $X_\varsigma=X+\varsigma U$ is merely $C^0$ conjugate to $X^t$ and has the same stable eigenvalues as $X^t$.

By transitivity of the conjugacy relation, we now have that $X_\rho^t$ and $X_\varsigma^t$ are conjugate to each other. Because $X^t$ and $X_\rho^t$ have different stable eigenvalue data, we have that $X_\varsigma^t$ and $X_\rho^t$ also have different eigenvalue data. Similarly, because $X^t$ and $X_\varsigma^t$ have different unstable eigenvalue data, we have that $X_\rho^t$ and $X_\varsigma^t$ also have different eigenvalue data. Also it is easy to verify that $X_\rho^t$ and $X_\varsigma^t$ do not share any invariant foliations.

\subsection{Swapping SRB measures: Ghys/Ceki\'c-Paternain example}

For the sake of specificity let $X^t$ be the geodesic flow on a surface of constant negative curvature $-1$. If $\gamma$ is  a periodic orbit $\gamma$, then the eigenvalues of the Poincar\'e return map at $\gamma$ are given by $e^{\pm|\gamma|}$, where $|\gamma|$ denotes the length of $\gamma$. We can perturb the hyperbolic metric in the Teichmüller space to obtain another surface which is not isometric to the initial metric. Denote the corresponding geodesic flow by $\bar X^t$. 

Now we explain the so-called quasi-Fuchsian Anosov flow construction due to Ghys~\cite{ghys}. Ceki\'c and Paternain recently revisited the quasi-Fuchsian flows and re-interpreted them as thermostat flows~\cite{CekicPaternain}. While we do not need the thermostat interpretation, we follow closely Ceki\'c and Paternain~\cite{CekicPaternain} and explain how to construct the flows $Y^t$ and $Z^t$ by ``taking the bracket'' of $X^t$ and $\bar X^t$ and why the SRB measures are being swapped under the conjugacy.
Let $H$ be the orbit equivalence given by structural stability, which is $C^0$-close to $id_M$ and takes orbits of $X^t$ to the orbits of $\bar X^t$ in orientation preserving manner.
For any $p\in M$ consider local weak unstable leaves of $X^t$ and $\bar X^t$ at $p$ and $H(p)$, respectively, and the (strong) stable holonomy map between them given by sliding along the leaves of $\mathcal W^s_X$
$$
\textup{Hol}^s\colon \mathcal{W}^{cu}_{X,\mathrm{loc}}(p)\to \mathcal{W}^{cu}_{\bar X,\mathrm{loc}}(H(p)).
$$
Define the generating vector field of $Y^t$ by
$$
Y(\textup{Hol}^s(p))=D\textup{Hol}^s(X(p)).
$$
Similarly, let $\textup{Hol}^u\colon \mathcal{W}^{cs}_{X,\mathrm{loc}}(p)\to \mathcal{W}^{cs}_{\bar X,\mathrm{loc}}(H(p))$ be the strong unstable holonomy given by sliding along the leaves of $\mathcal W^u_X$ and define
$$
Z(\textup{Hol}^u(p))=D\textup{Hol}^u(X(p)).
$$

\begin{figure}[!ht]
	\centering
	\includegraphics[width=0.6\textwidth]{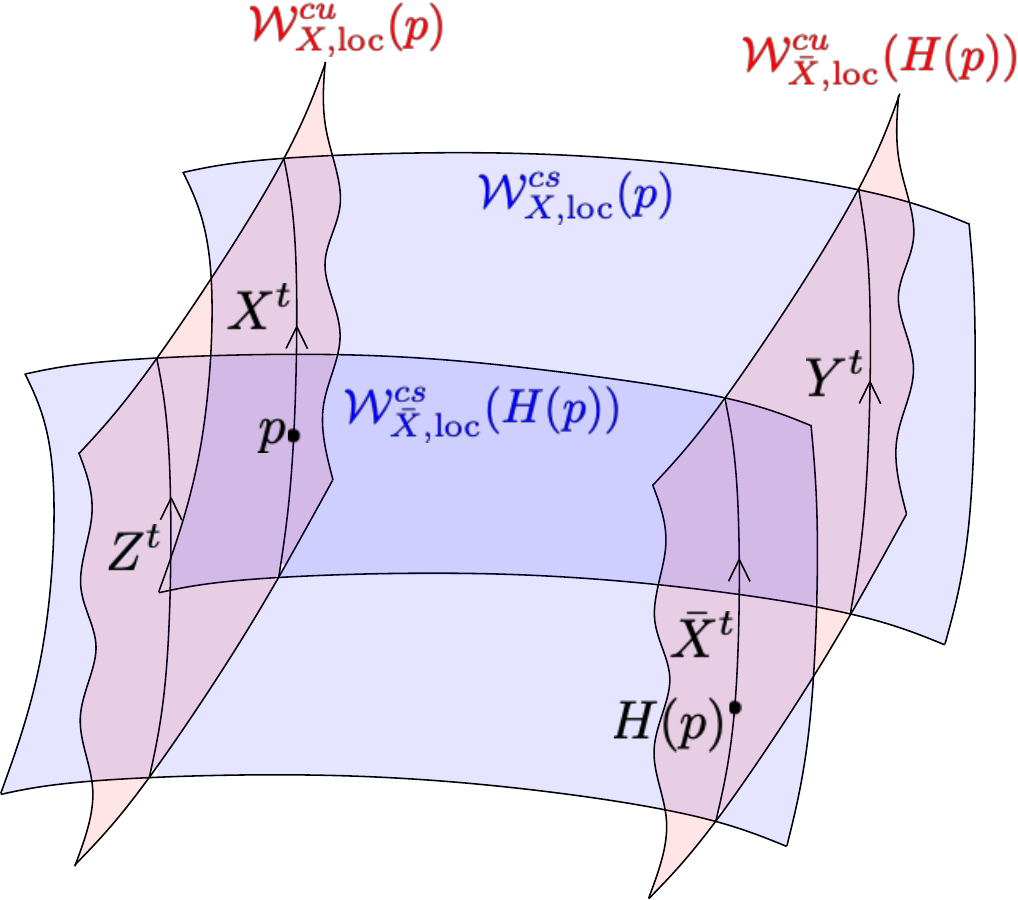}
	\caption{The construction of quasi-Fuchsian flows by ``taking the bracket'' of the flows $X^t$ and $\bar X^t$ to produce conjugate flows $Y^t$ and $Z^t$.}
	\label{fig:ghys}
\end{figure}

Since all foliations and holonomies involved are $C^\infty$ we have that both $Y$ and $Z$ are $C^\infty$ vector fields. The orbit foliation of $Y^t$ is given by the intersection the weak foliations $\mathcal{W}^{cs}_X\cap \mathcal{W}^{cu}_{\bar X}$ and that of $Z^t$ by the intersection $\mathcal{W}^{cu}_X\cap \mathcal{W}^{cs}_{\bar X}$. Both flows ``borrow'' their parametrization from $X^t$ and hence both are conjugate to $X^t$, and, hence, to each other. Specifically, the conjugacy between $X^t$ and, say $Y^t$ is given by
$$
p\mapsto \mathcal{W}^s_{X,\mathrm{loc}}(p)\cap \mathcal{W}^{cu}_{\bar X, \mathrm{loc}}(H(p)).
$$

Since this conjugacy is given by sliding along strong stable leaves we can notice the following. Let $\gamma$ be a periodic orbit for $X^t$ and denote by $\bar\gamma$, $\gamma_Y$ and $\gamma_Z$ the corresponding periodic orbits for $\bar X^t$, $Y^t$ and $Z^t$, respectively. The stable holonomy map
$$
\textup{Hol}^s\colon \mathcal{W}^{cu}_{X,\mathrm{loc}}(\gamma)\to \mathcal{W}^{cu}_{Y, \mathrm{loc}}(\gamma_Y)
$$
conjugates the local dynamics of $X^t$ and $Y^t$. Hence, $\gamma$ and $\gamma_Y$ have the same unstable eigenvalue $\lambda(\gamma_Y)=e^{|\gamma|}$. By the same observation we have that the stable eigenvalue of $\gamma_Y$ is $\mu(\gamma_Y)=e^{-|\bar\gamma|}$ and, similarly, for $\gamma_Z$ we have $\lambda(\gamma_Z)=e^{|\bar\gamma|}$, $\mu(\gamma_Z)=e^{-|\gamma|}$.

Hence we have $\mu(\gamma_Y)=\lambda(\gamma_Z)^{-1}$ and $\lambda(\gamma_Y)=\mu(\gamma_Z)^{-1}$. So, if $\Phi$ is the conjugacy, $\Phi\circ Y^t=Z^t\circ \Phi$, $\Phi(\gamma_Y)=\gamma_Z$, then we see that $\Phi$ swaps the stable and the inverses of unstable eigenvalues at every periodic orbit. Now, by the standard de la Llave argument, using the Livshits theorem and the equilibrium state description of SRB measures (see Section~\ref{sec_43}) we conclude that $\Phi_*(m^+_Y)=m^-_Z$ and $\Phi_*(m^-_Y)=m^+_Z$.

Finally, since we took a non-isometric perturbation of the initial hyperbolic metric, there exist an $X^t$-periodic orbit $\gamma_0$ such that the corresponding $\bar X^t$ periodic orbit $\bar\gamma_0$ has a different length, $|\bar\gamma_0|\neq |\gamma_0|$. It immediately follows that corresponding periodic orbis for $Y^t$ and $Z^t$ are both dissipative --- one volume expanding and one volume contracting. Hence both $Y^t$ and $Z^t$ are dissipative flows. 

\subsection{Swapping SRB measures using Cawley's realization} Another way to construct an example of conjugate Anosov flows with a conjugacy which swaps the SRB measures is to use Cawley's realization Theorem~\ref{thm_cawley}~\cite{cawley}. Start with any dissipative Anosov diffeomorphism $f\colon\T^2\to\T^2$. Then by Theorem~\ref{thm_cawley} we can construct a dissipative Anosov diffeomorphism $g\colon\T^2\to\T^2$ such that
\begin{itemize}
	\item $h\circ g=f\circ h$;
		\item $-\log J^ug$ is cohomologous to $\log J^sf\circ h$ over $g$;
	\item $\log J^sg$ is cohomologous to $-\log J^uf\circ h$ over $g$.
\end{itemize}
The latter two properties provide the swapping property of the SRB measures as they are equilibrium states for the corresponding potentials: $h_*m_g^+=m_f^-$, $h_*m_g^-=m_f^+$. Suspending both $f$ and $g$ with a constant roof 1, we obtain conjugated Anosov flows $X^t$ and $Y^t$ with the same swapping property of SRB measures.

The issue with this example is that diffeomorphism $g$ is merely $C^{1+\textup{H}}$. Accordingly the suspension flows $Y^t$ has the same low regularity which doesn't fall into the setting of this paper as we always assume $C^r$ regularity for some $r\geq 3$. It is not clear to us whether such suspension examples can be made more regular.

\appendix

\section{Smooth rigidity of orbit equivalences}\label{appa}

The following theorems are essentially due to Rafael de la Llave~\cite{dllSRB} and, independently, to Mark Pollicott~\cite{Pol}. Even though neither reference considered orbit equivalences which are not conjugacies, very similar arguments which exploit SRB measure yield the following results.
\begin{theorem}\label{thmA1}
Let $X^t$, $Y^t$ be two $C^r$, $r>1$, transitive Anosov flows on $3$-manifold, that are $C^0$ orbit equivalent via an orbit equivalence $\Phi$ which preserves the time direction. For any periodic point $p=X^T(p)$ let $T'$ be the period of $\Phi(p)$ under $Y^t$ and assume that $DX^T(p)$ and $DY^{T'}(\Phi(p))$ have the same eigenvalues. Then, there exists an orbit equivalence which is $C^{r_*}$, with $r_*$ as in~\eqref{def_r_etoile}. 
\end{theorem}
\begin{theorem}\label{thmA2}
Let $X^t$, $Y^t$ be two $C^r$, $r>1$, transitive Anosov flows on $3$-manifold, that are $C^0$ orbit equivalent via an orbit equivalence $\Phi$ which preserves the time direction. Assume that $\Phi$ and $\Phi^{-1}$ absolutely continuous. Then, there exists an orbit equivalence which is $C^{r_*}$.
\end{theorem}

\begin{proof}[Sketch of the proof of Theorem~\ref{thmA1}]
We first smooth out the orbit equivalence along the flow so that $\Phi$ is $C^r$ smooth along the orbits. 

Then $\dot \Phi$, the derivative of $\Phi$ along the orbits is a well defined positive H\"older continuous function which is $C^{r-1}$ when restricted to an orbit of $X^t$. Let $\rho^{-1}=\dot \Phi\circ \Phi^{-1}$. Then we have
$$
D\Phi(X)=\rho^{-1} Y\circ \Phi,
$$
and if we consider the reparametrization $\bar Y^t$ given by $\bar Y=\rho^{-1} Y$ then $\Phi$ conjugates $X^t$ and $\bar Y^t$: $\Phi\circ X^t=\bar Y^t\circ \Phi$. 

Denote by $m_X=m_X^+$ and $m_Y=m_Y^+$ the (positive) SRB measures for $X^t$ and $Y^t$, respectively.
\begin{lemma} \label{lemmaA2} The pushforward measure $\Phi_*m_X$ is in the same measure class as $m_Y$. Specifically,
$$
\Phi_*m_X=\frac{\rho m_Y}{\int\rho dm_Y}.
$$
\end{lemma}

\begin{proof}
Recall that $m_X$ and $m_Y$ are equilibrium states for the geometric potentials $\psi^u_X$ and $\psi^u_Y$, respectively (see Section~\ref{sec_51}). Accordingly, since $X^t$ is conjugate to $\bar Y^t$ via $\Phi$, we have that $\Phi_*m_X$ is the equilibrium state over $\bar Y^t$ for the potential $\psi^u_X\circ \Phi^{-1}$.

 We also have that $ \bar Y^t$ is a reparametrization of $Y^t$ with $ \bar Y=\rho^{-1} Y$ and, hence, by~\cite[Proposition 4.3]{GRH_abelian} we have that the measure
 $$
 m_{\bar Y}=\frac{\rho m_Y}{\int\rho dm_Y}
 $$
 is the equilibrium state for the potential $\rho^{-1}\psi^u_{Y}$ with respect to $\bar Y^t$. We now claim that the potentials $\rho^{-1}\psi^u_{Y}$ and $\psi^u_X\circ \Phi^{-1}$ are cohomologous. Indeed, let $\bar\gamma$ be a periodic orbit of $\bar Y^t$ of period $T$. Then
 $$
 \int_{\bar \gamma}\psi_X^u\circ \Phi^{-1}=\int_{\Phi^{-1}(\bar\gamma)}\psi_X^u=-\log \textup{Jac}_{\Phi^{-1}(\bar\gamma)}^u(X^T).
 $$
 Now let $\gamma$ be the corresponding orbit of $Y^t$ of some period $T'=\int_0^T\rho^{-1}(\bar\gamma(t))dt$. Then
 $$
 \int_{\bar\gamma} \rho^{-1}\psi^u_{Y}= \int_{\gamma} \psi^u_{Y}=-\log \textup{Jac}_{\gamma}^u(Y^{T'}).
 $$
 By the eigenvalue data assumption the above integrals are equal. Since we have it for any periodic $\bar\gamma$ we can use the Livshits theorem to conclude that $\rho^{-1}\psi^u_{Y}$ and $\psi^u_X\circ \Phi^{-1}$ are cohomologous, and, hence, have the same equilibrium states $\Phi_*m_X= m_{\bar Y}$.
\end{proof}
\begin{remark} We have used~\cite[Proposition~4.3]{GRH_abelian}, which is stated for smooth reparametrizations, however the proof also works for H\"older reparametrizations. Alternatively, a symbolically inclined reader, can arrive at the above lemma by considering the suspension models for both flows and noticing that both models have the same base subshifts but different H\"older roofs. Then the eigenvalue data assumption yields matching of the SRB measures on the subshifts and the lemma follows by passing to equilibrium states of the suspension, see~\cite[Proposition 6.1]{ParPol}.
\end{remark}

The same argument also proves that $\Phi$ preserves the measure class of the negative SRB measures. These properties allow to argue that $\Phi$ is a $C^{r_*}$ diffeomorphism ``\`{a} la de la Llave''~\cite{dllSRB} after adjusting $\Phi$ along the orbits. Namely, the needed property is the following one: let $S$ be a local $C^r$ section for $X^t$ then $\Phi(S)$ is also a $C^r$ section for $Y^t$. However, a priori, $\Phi(S)$ is only topological section for $Y^t$. Locally in a chart, one can easily adjust $\Phi$ along the flow to make sure that $\Phi(S)$ is $C^r$. Then it is easy make this property global by using a partition of unity.

Once such adjustment is made one follows the de la Llave argument to show that preservation of positive and negative SRB measure classes yield $C^r$ smoothness of one-dimensional restrictions of $\Phi$ to strong stable and strong unstable manifolds. Another adjustment to this argument needs to be made to account for the fact that strong stable (unstable) manifolds of $X^t$ do not map to strong stable (unstable) manifolds of $Y^t$. However, they map under orbit equivalence to $C^r$ curves, thanks to the previous adjustment, which are contained corresponding weak submanifolds, which is good enough. The proof concludes with an application of the Journ{\'e}'s regularity lemma~\cite{Journe} as in~\cite{dllSRB}.
\end{proof}

The proof of Theorem~\ref{thmA2} begins in the same way. Then one still needs to prove Lemma~\ref{lemmaA2}, but without using the assumption on periodic data. This can be done with an argument of de la Llave~\cite[Lemma~4.6]{dllSRB}. Namely, for any continuous function $\varphi$ any point $x$, we have
$$
\frac1T\int_0^T\varphi(\bar Y^t(\Phi(x)))\, dt=\frac 1T\int_0^T\varphi\circ\Phi(X^t(x))\, dt.
$$
Since SRB measures are attractors, we have for Lebesgue almost every $x$,
$$
\lim_{T\to\infty} \frac 1T\int_0^T\varphi\circ\Phi(X^t(x))\, dt=\int \varphi\circ \Phi\, dm_X.
$$
Similarly, for Lebesgue almost every $y$
$$
\frac1T\int_0^T\varphi(\bar Y^t(y))\, dt=\int \varphi\, dm_{\bar Y}.
$$
Since $\Phi$ and  $\Phi^{-1}$ are absolutely continuous we have that for Lebesgue almost every $x$ and corresponding $y=\Phi(x)$ the above three formulae hold true and hence
$$
\int \varphi\circ \Phi\, dm_X
=\int \varphi\, dm_{\bar Y},
$$
which implies that $\Phi_*m_X=m_{\bar Y}$. The same argument also proves that $\Phi$ preserves the measure class of the negative SRB measures. The last step of the proof of Theorem~\ref{thmA2} is the same as that of Theorem~\ref{thmA1}.

\section{Adapted charts}\label{appb}

Recall that by Proposition~\ref{norm foms} we have non-stationary linearizing charts $\{\Phi_x^s\}_{x \in M},\{\Phi_x^u\}_{x \in M}$ along stable and unstable manifolds. From the classical construction of such charts~\cite{KL}, it is clear that they are as regular as the flow, namely $C^r$. It is standard to extend these non-stationary linearizations to actual $3$-dimensional charts $\{\jmath_x\colon (-1,1)^3\to M\}_{x \in M}$ such that the first three properties  of Proposition~\ref{propo o good} hold true: for any $x \in M$,  
\begin{enumerate}
	\item\label{normal un} $\jmath_x(\xi,0,0)=\Phi_x^s(\xi)$, $\xi \in (-1,1)$; 
	\item\label{normal deux} $\jmath_x(0,0,\eta)=\Phi_x^u(\eta)$,  $\eta \in (-1,1)$; 
	\item\label{flw dir} $\jmath_x(\xi,t,\eta)=X^t(\jmath_x(\xi,0,\eta))$,  $(\xi,t,\eta)\in (-1,1)^3$. 
\end{enumerate}
Relative to the charts $\{\jmath_x\}_{x \in M}$, the dynamics $\tilde F^\tau_x=\jmath_{X^\tau(x)}^{-1}\circ X^\tau \circ \jmath_x$ takes the form
$$
\tilde F^\tau_x(\xi,t,\eta)=(\tilde F^\tau_{x,1}(\xi,\eta),t+\tilde \psi^\tau_x(\xi,\eta),\tilde F^\tau_{x,3}(\xi,\eta)),
$$
with all coordinate maps $\tilde F^\tau_{x,1},\tilde \psi^\tau_x,\tilde F^\tau_{x,3}$ being $C^r$. 

We first will proof the discrete, $\tau$-time version of Proposition~\ref{propo o good} and then show that the resulting charts are, in fact, independent of $\tau$ and deduce the continuous time normal form posited in Proposition~\ref{propo o good}.

\begin{proposition}\label{prop_C1}
	Fix a time $\tau>0$. There exists a continuous family of  uniformly $C^{r-1}$ charts $\{\imath_x\colon (-1,1)^3\to M\}_{x \in M}$ such that for any $x \in M$, we have:
	\begin{enumerate}
		\item\label{normal un} $\imath_x(\xi,0,0)=\Phi_x^s(\xi)$, for any $\xi \in (-1,1)$; 
		\item\label{normal deux} $\imath_x(0,0,\eta)=\Phi_x^u(\eta)$, for any $\eta \in (-1,1)$; 
		\item\label{flw dir} $\imath_x(\xi,t,\eta)=X^t(\imath_x(\xi,0,\eta))$, for any $(\xi,t,\eta)\in (-1,1)^3$;
		%\item $\imath
		%\marginpar{change degree to k-1}
		\item\label{pt cocyc nf} let $F^\tau_x:=(\imath_{X^\tau(x)})^{-1}\circ X^\tau\circ \imath_x=(F^\tau_{x,1},F^\tau_{x,2},F^\tau_{x,3})$; then, $F^\tau_{x,2}(\xi,t,\eta)=t+\psi^\tau(\xi,\eta)$, and there exist polynomials $P_x^*(\tau)(z)=\sum_{\ell=1}^{[k]} \alpha_x^{*,\ell}(\tau) z^\ell$, $*=s,u$, of degree at most $[k]$, such that for any $\xi,\eta \in (-1,1)$, we have
		\begin{align*}
			%C_x(\eta):=
			\begin{bmatrix}
				\partial_1 F^\tau_{x,1} & \partial_2 F^\tau_{x,1}\\
				\partial_1 F^\tau_{x,2} & \partial_2 F^\tau_{x,2}
			\end{bmatrix}(0,0,\eta)&=\begin{bmatrix}
				\lambda_x^s(\tau) & 0\\
				P_x^s(\tau)(\eta) & 1
			\end{bmatrix},\\
			\begin{bmatrix}
				\partial_2 F^\tau_{x,2} & \partial_3 F^\tau_{x,2}\\
				\partial_2 F^\tau_{x,3} & \partial_3 F^\tau_{x,3}
			\end{bmatrix}(\xi,0,0)&=\begin{bmatrix}
				1 & P_x^u(\tau)(\xi)\\
				0 & \lambda_x^u(\tau)
			\end{bmatrix},
		\end{align*}
		where we recall that $\lambda_x^s(\tau):=\|DX^\tau(x)|_{E^s}\|$, $\lambda_x^u(\tau):=\|DX^\tau(x)|_{E^u}\|$.
	\end{enumerate}
\end{proposition}

	\begin{addendum}\label{add_tau_independent}\label{add_C2}
	The charts are independent of $\tau$ and the  polynomials satisfy the following twisted cocycle equations over the flow
	$$
	P^s_x(\tau_1+\tau_2)(\eta)=P_x^s(\tau_1)(\eta)+\lambda_x^s(\tau_1)P_{X^{\tau_1}(x)}^s(\tau_2)(\lambda_x^u(\tau_1)\eta);
	$$
		$$
	P^u_x(\tau_1+\tau_2)(\xi)=P_x^u(\tau_1)(\xi)+\lambda_x^u(\tau_1)P_{X^{\tau_1}(x)}^u(\tau_2)(\lambda_x^s(\tau_1)\xi);
	$$
for $\tau_1, \tau_2>0$.
\end{addendum}

\begin{proof}[Proof of Proposition~\ref{prop_C1}]

The proof will proceed via three chart adjustments. 

\begin{lemma}\label{lemma_B3}
	There exists a family $\{h_x\}_{x \in M}$ of uniformly $C^{r-1}$ diffeomorphisms,
	$$
	h_x\colon (\xi,t,\eta)\mapsto (\rho^\tau_x(\eta)\xi,t,\sigma^\tau_x(\xi)\eta),
	$$
	such that the adjusted charts $\hat \jmath_x:=\jmath_x \circ h_x$ put dynamics in the form 
	$$\hat F^\tau_x=h_{X^\tau(x)}^{-1} \circ \tilde F^\tau_x\circ h_x\colon (\xi,t,\eta)\mapsto (\hat F^\tau_{x,1}(\xi,\eta),t+\hat \psi^\tau_x(\xi,\eta),\hat F^\tau_{x,3}(\xi,\eta)),$$
	 with
	\begin{equation}\label{property_un}
		\partial_\xi \hat F^\tau_{x,1}(0,\eta)=\lambda_x^s(\tau),\quad \partial_\eta \hat F^\tau_{x,3}(\xi,0)=\lambda_x^u(\tau). 
	\end{equation} 
	Here, $\rho^\tau_x,\sigma^\tau_x\colon (-1,1) \to \mathbb{R}_+$ are   $C^{r-1}$ positive functions that satisfy $\rho^\tau_x(0)=\sigma^\tau_x(0)=1$. 
	%where $\rho_x,\sigma_x\colon (-1,1) \to \mathbb{R}_+$ are two positive functions. The adjusted charts $\jmath_x \circ h_x$ put dynamics in the form $\hat F_x=h_{X^1(x)}^{-1} \circ \tilde F_x\circ h_x\colon (\xi,t,\eta)\mapsto (\hat F_{x,1}(\xi,\eta),t+\hat \psi_x(\xi,\eta),\hat F_{x,3}(\xi,\eta))$, 
\end{lemma}
\begin{proof}
	We seek the adjustment $h_x$ in the form
	$$
	h_x\colon (\xi,t,\eta)\mapsto (\rho^\tau_x(\eta)\xi,t,\sigma^\tau_x(\xi)\eta),
	$$
	where $\rho^\tau_x,\sigma^\tau_x\colon (-1,1) \to \mathbb{R}_+$, $\rho^\tau_x(0)=\sigma^\tau_x(0)=1$, are two positive functions. The adjusted charts $\jmath_x \circ h_x$ put dynamics in the form $\hat F^\tau_x=h_{X^\tau(x)}^{-1} \circ \tilde F^\tau_x\circ h_x\colon (\xi,t,\eta)\mapsto (\hat F^\tau_{x,1}(\xi,\eta),t+\hat \psi^\tau_x(\xi,\eta),\hat F^\tau_{x,3}(\xi,\eta))$,  with new coordinate maps given by
	\vspace{2mm}
	\begin{equation}\label{hat_F_one}
		\left\{
		\begin{array}{rcl}
			\hat F^\tau_{x,1}\colon (\xi,\eta)&\mapsto& (\rho^\tau_{X^\tau(x)}(\tilde F^\tau_{x,3}(\rho^\tau_x(\eta)\xi,\sigma^\tau_x(\xi)\eta)))^{-1}\tilde F^\tau_{x,1}(\rho^\tau_x(\eta)\xi,\sigma^\tau_x(\xi)\eta),\\
			\hat \psi^\tau_x\colon(\xi,\eta)&\mapsto& \tilde \psi^\tau_x(\rho^\tau_x(\eta)\xi,\sigma^\tau_x(\xi)\eta),\\
			\hat F^\tau_{x,3}\colon(\xi,\eta)&\mapsto& (\sigma^\tau_{X^\tau(x)}(\tilde F^\tau_{x,1}(\rho^\tau_x(\eta)\xi,\sigma^\tau_x(\xi)\eta)))^{-1}\tilde F^\tau_{x,3}(\rho^\tau_x(\eta)\xi,\sigma^\tau_x(\xi)\eta).
		\end{array}
		%t+\tilde \psi_x(\rho_x(\eta)\xi,\eta),\tilde F_{x,3}(\rho_x(\eta)\xi,\eta)).
		\right.
	\end{equation}
	%and 
	\vspace{2mm}
	%\begin{equation}\label{hat_F_trois}
	%	\hat \psi_x\colon(\xi,\eta)\mapsto (\rho_x(\eta)\xi,\eta),\quad \hat F_{x,3}\colon(\xi,\eta)\mapsto \tilde F_{x,3}(\rho_x(\eta)\xi,\eta). 
	%\end{equation}
	Our goal now is to use the first adjustment to arrange~\eqref{property_un}. 
	%\begin{equation}\label{property_un}
	%\partial_\xi \hat F_{x,1}(0,\eta)=\lambda_x^s(1),\quad \partial_\eta \hat F_{x,3}(\xi,0)=\lambda_x^u(1). 
	%\end{equation} 
	
	Differentiating the expression of $\hat F^\tau_{x,1}$ obtained   in~\eqref{hat_F_one}, while using $\tilde F^\tau_{x,1}(0,\cdot)\equiv 0$,  $\sigma^\tau_x(0)=1$ and $\tilde F^\tau_{x,3}(0,\eta)=\lambda_x^u(\tau)\eta$,  %$\partial_\eta\tilde F_{x,1}(0,\eta)=0$, 
	we rewrite the first equation of~\eqref{property_un} as
	$$
	(\rho^\tau_{X^\tau(x)}(\lambda_x^u(\tau)\eta))^{-1}\rho^\tau_x(\eta) \partial_\xi \tilde F^\tau_{x,1}(0,\eta)=\lambda_x^s(\tau).
	$$
	Observe that the above equation only involves the function $\rho^\tau_x$. 
	First we change the base point from $x$ to $X^{-\tau}(x)$ and then take $\log$:
	$$
	\log \rho^\tau_{x}(\eta)-\log \rho^\tau_{X^{-\tau}(x)}(\lambda_x^u(-\tau)\eta)=\log\partial_\xi \tilde F^\tau_{X^{-\tau}(x),1}(0,\lambda_x^u(-\tau)\eta)-\log\lambda_{X^{-\tau}(x)}^s(\tau),
	$$
	which we can solve for $\rho^\tau_x$ using the telescopic sum
	\begin{equation}\label{eq_rho_solution}
	\log \rho^\tau_{x}(\eta)=\sum_{\ell=1}^{+\infty} \left(\log\partial_\xi \tilde F^\tau_{X^{-\ell\tau}(x),1}(0,\lambda_x^u(-\ell\tau)\eta)-\log\lambda_{X^{-\ell\tau}(x)}^s(\tau)\right). 
	\end{equation}
	We see that the series converge and that the function $\rho^\tau_x$  obtained in this way indeed satisfies $\rho^\tau_x(0)=1$. Further, differentiating $0\leq j \leq r-1$ times with respect to $\eta$ also yields converging series
	$$
	\sum_{\ell=1}^{+\infty} \left(\lambda_x^u(-\ell\tau)\right)^j \partial_\eta^j\left(\log\partial_\xi \tilde F^\tau_{X^{-\ell\tau}(x),1}\right)(0,\lambda_x^u(-\ell\tau)\eta). 
	$$
	Hence, by Weierstrass M-test, the infinite series gives the posited $C^{r-1}$ solution $\rho^\tau_x$ yielding the first  property in~\eqref{property_un}. Also note that the exponential rate of convergence is uniform in $x$, hence $\rho^\tau_x$ is $C^{r-1}$ uniformly in $x\in M$. Note that the only restriction we have put so far on the second function $\sigma^\tau_x$ is that $\sigma^\tau_x(0)=1$. 
	
	In the same manner, we can solve the second equation in~\eqref{property_un} by choosing suitably the function $\sigma^\tau_x$. 
	Differentiating the expression of $\hat F^\tau_{x,3}$ obtained   in~\eqref{hat_F_one}, and using $\tilde F^\tau_{x,3}(\cdot,0)\equiv 0$,  $\rho^\tau_x(0)=1$, and $\tilde F_{x,1}^\tau(\xi,0)=\lambda_x^s(\tau)\xi$,  %$\partial_\eta\tilde F_{x,1}(0,\eta)=0$, 
	we get 
	$$
	(\sigma^\tau_{X^\tau(x)}(\lambda_x^s(\tau)\xi))^{-1}\sigma^\tau_x(\xi) \partial_\eta \tilde F^\tau_{x,3}(\xi,0)=\lambda_x^u(\tau).
	$$
	This again can be solved for $\sigma^\tau_x$ using the telescopic sum
	$$
	\log \sigma^\tau_{x}(\xi)=\sum_{\ell=0}^{+\infty} \left(\log\lambda_{X^{\ell\tau}(x)}^u(\tau)-\log\partial_\eta \tilde F^\tau_{X^{\ell\tau}(x),3}(\lambda_x^s(\ell\tau)\xi,0)\right). 
	$$
	As above, we see that the function $\sigma^\tau_x$ defined in this way satisfies $\sigma^\tau_x(0)=1$ and gives the posited $C^{r-1}$ solution $\sigma^\tau_x$ yielding the second  property in~\eqref{property_un}. Again we observe that $\sigma^\tau_x$ are $C^{r-1}$ unifomly in $x$.
\end{proof}

Unlike the previous step, we will make consecutive adjustments: first in order to replace 
$
\partial_\xi \hat\psi^\tau_x(0,\eta)
$ by a polynomial in $\eta$ of degree at most $[k]$, and then the final analogous adjustment to replace $\partial_\eta \hat\psi^\tau_x(\xi,0)$ with a polynomial of degree at most $[k]$. 
We recall that $k$ is the pinching exponent as in Definition~\ref{defi d pinched}, i.e., for $n \gg 1$,
%to ease the notation in the following, we assume that the integer $n$ in Definition~\ref{defi d pinched} can be chosen equal to $1$, so that 
\begin{equation}\label{pinc_cond_last}
	\lambda_x^s(n)^k\lambda_x^u(n)<1,\quad\lambda_x^u(n)^k\lambda_x^s(n)>1,\quad \forall\, x \in M.
\end{equation}
\begin{lemma}\label{forme_polyn}
	There exists a family $\{u_x\}_{x \in M}$ of uniformly $C^{r-1}$ diffeomorphisms of the form
	$$
	u_x\colon (\xi,t,\eta)\mapsto (\xi,t+\phi^\tau_x(\eta)\xi+\kappa^\tau_x(\xi)\eta,\eta),
	$$
	such that the adjusted charts $\imath_x:=\hat\jmath_x \circ u_x$ put dynamics in the form 
	$$F^\tau_x=u_{X^\tau(x)}^{-1} \circ \hat F^\tau_x\circ u_x\colon (\xi,t,\eta)\mapsto ( F^\tau_{x,1}(\xi,\eta),t+ \psi^\tau_x(\xi,\eta),F^\tau_{x,3}(\xi,\eta)),$$
	 such that for any $\xi\in (-1,1)$ and $\eta\in (-1,1)$, 
	\begin{equation*} 
		\begin{array}{rl}
			\partial_\xi  F_{x,1}(0,\eta)=\lambda_x^s(1), &\partial_\eta F_{x,3}(\xi,0)=\lambda_x^u(1),\\
			\partial_\xi \psi^\tau_x(0,\eta)=P_x^s(\tau)(\eta), & \partial_\eta \psi^\tau_x(\xi,0)=P_x^u(\tau)(\xi). 
		\end{array}  
	\end{equation*}
\end{lemma}
\begin{proof}
	Let us first adjust the charts so that 
	$
	\partial_\xi \hat\psi^\tau_x(0,\eta)
	$ is replaced by a polynomial of degree at most $[k]$. 
	We seek the adjustment in the form
	$$
	g_x\colon (\xi,t,\eta)\mapsto (\xi,t+\varphi^\tau_x(\eta)\xi,\eta),
	$$
	with $\varphi^\tau_x(0)=0$. The adjusted chart $\bar \jmath_x:=\hat\jmath_x \circ g_x$ puts dynamics in the form 
	$$
	\bar F^\tau_x=g_{X^{\tau}(x)}^{-1}\circ \hat F^\tau_x \circ g_x\colon(\xi,t,\eta)\mapsto (\bar F^\tau_{x,1}(\xi,\eta),t+\bar \psi^\tau_x(\xi,\eta),\bar F^\tau_{x,3}(\xi,\eta)),
	$$
	where $\bar F^\tau_{x,1}=\hat F^\tau_{x,1}$, $\bar F^\tau_{x,3}=\hat F^\tau_{x,3}$, and 
	\begin{equation}\label{expression_bar_psi}
		\bar \psi^\tau_x\colon(\xi,\eta)\mapsto \hat \psi^\tau_x(\xi,\eta)+\varphi^\tau_x(\eta)\xi-\varphi^\tau_{X^\tau(x)}(\hat F^\tau_{x,3}(\xi,\eta))\hat F^\tau_{x,1}(\xi,\eta). 
	\end{equation}
	Differentiating with respect to $\xi$ at $(0,\eta)$ while using $\hat F^\tau_{x,1}(0,\cdot)\equiv 0$, $\hat F^\tau_{x,3}(0,\eta)=\lambda_x^u(\tau) \eta$, and~\eqref{property_un}, we obtain 
	\begin{equation*}
		\partial_\xi\bar \psi^\tau_x (0,\eta)= \partial_\xi\hat \psi^\tau_x(0,\eta)+\varphi^\tau_x(\eta)-\varphi^\tau_{X^\tau(x)}(\lambda_x^u(\tau)\eta)\lambda_x^s(\tau). 
	\end{equation*}
	Formally, $\partial_\xi \hat \psi^\tau_x(0,\cdot)$ is a $C^{r-2}$ function. We claim it is in fact $C^{r-1}$.
	\begin{claim}
		The function $\partial_\xi \hat \psi^\tau_x(0,\cdot)$ is $C^{r-1}$ uniformly in $x\in M$. 
	\end{claim}
	\begin{proof}[Proof of the claim:]
		To check this recall that $\hat \psi^\tau_x(\xi,\eta)=\tilde \psi^\tau_x(\rho^\tau_x(\eta)\xi,\sigma^\tau_x(\xi)\eta)$, where $\tilde \psi^\tau_x$ is $C^r$. Calculating with the chain rule gives
		$$
		\partial_\xi \hat \psi^\tau_x(\xi,\eta)=\rho^\tau_x(\eta)\partial_\xi \tilde \psi^\tau_x(\rho^\tau_x(\eta)\xi,\sigma^\tau_x(\xi)\eta)+{\sigma^\tau_x}'(\xi)\eta\cdot\partial_\eta \tilde \psi^\tau_x(\rho^\tau_x(\eta)\xi,\sigma^\tau_x(\xi)\eta),
		$$
		recalling that $\sigma^\tau_x(0)=1$, and evaluating at $(0,\eta)$ yields
		$$
		\partial_\xi \hat \psi^\tau_x(0,\eta)=\rho^\tau_x(\eta)\partial_\xi \tilde \psi^\tau_x(0,\eta)+{\sigma^\tau_x}'(0) \eta\cdot \partial_\eta \tilde \psi^\tau_x(0,\eta),
		$$
		which is clearly $C^{r-1}$ uniformly in $x$.  
	\end{proof}
	Since $r \geq k+1$, the above claim ensures that $\partial_\xi \hat \psi^\tau_x(0,\cdot)$ is at least $C^k$ and we can use the Taylor expansion 
	\begin{equation}\label{eq_taylor}
	\partial_\xi \hat \psi^\tau_x(0,\eta)=\hat P_x^s(\tau)(\eta)+\omega^\tau_x(\eta),
\end{equation}
	where $\omega^\tau_x$ is $C^{r-1}$, $\omega^\tau_x(\eta)=O(\eta^k)$ uniformly in $x$, and $\hat P_x^s(\tau)(\eta)$ is a polynomial of degree at most $[k]$. Our goal now is to find $\varphi^\tau_x$ such that 
	$$
	\partial_\xi \bar \psi^\tau_x(0,\eta)=\hat P_x^s(\tau)(\eta). 
	$$
	Hence we need to solve 
	$$
	\varphi^\tau_{X^\tau(x)}(\lambda_x^u(\tau)\eta)\lambda_x^s(\tau)=\varphi^\tau_x(\eta)+\omega^\tau_x(\eta). 
	$$
	Changing the base point we have
	$$
	\varphi^\tau_{x}(\eta)=\lambda_x^s(-\tau)\varphi^\tau_{X^{-\tau}(x)}(\lambda_x^u(-\tau)\eta)+\lambda_x^s(-1)\omega^\tau_{X^{-\tau}(x)}(\lambda_x^u(-\tau)\eta). 
	$$
	We solve for $\varphi^\tau_x$ by summing into the past
	$$
	\varphi^\tau_x(\eta)= \sum_{\ell=1}^{+\infty} \lambda_x^s(-\ell\tau)\omega^\tau_{X^{-\ell\tau}(x)}(\lambda_x^u(-\ell\tau)\eta). 
	$$
	Indeed, since $\omega^\tau_x(\eta)=O(\eta^k)$ uniformly in $x$, the $k$-pinching~\eqref{pinc_cond_last} guarantees that the series converges yielding the posited adjustment. Note that $\varphi^\tau_x(0)=0$ and also differentiating formally, using Weierstrass M-test, and since $\omega^\tau_x$ is $C^{r-1}$ uniformly in $x$, we have that $\varphi^\tau_x$ is $C^{r-1}$ uniformly in $x$.

	Finally, we will adjust the charts once more in the similar way to replace 
	$
	\partial_\eta \bar\psi^\tau_x(\xi,0)
	$ by a polynomial of degree at most $[k]$. We seek for this final adjustment in the following form 
	$$
	(\xi,t,\eta)\mapsto(\xi,t+\kappa^\tau_x(\xi)\eta,\eta),
	$$
	which puts the dynamics in the form 
	$$
	F^\tau_x\colon(\xi,t,\eta)\mapsto (F^\tau_{x,1}(\xi,\eta),t+\psi^\tau_x(\xi,\eta),F^\tau_{x,3}(\xi,\eta)). 
	$$
	\begin{claim}
		The function $\partial_\eta\bar \psi^\tau_x(\cdot,0)$ obtained above is still $C^{r-1}$ uniformly $x$. 
	\end{claim}
	\begin{proof}[Proof of the claim:]
		By~\eqref{hat_F_one}-\eqref{expression_bar_psi}, and since $\bar F^\tau_{x,1}=\hat F^\tau_{x,1}$, $\bar F^\tau_{x,3}=\hat F^\tau_{x,3}$, we have 
		\begin{align*}%\label{expression_bar_psi}
			&\partial_\eta\bar \psi^\tau_x(\xi,\eta)={\rho^\tau_x}'(\eta)\xi\,\partial_\xi\tilde \psi^\tau_x(\rho^\tau_x(\eta)\xi,\sigma^\tau_x(\xi)\eta)+\sigma^\tau_x(\xi)\partial_\eta\tilde \psi^\tau_x(\rho^\tau_x(\eta)\xi,\sigma^\tau_x(\xi)\eta)\\
			&+{\varphi^\tau_x}'(\eta)\xi-{\varphi^\tau_{X^\tau(x)}}'(\hat F^\tau_{x,3}(\xi,\eta))\partial_\eta \hat F^\tau_{x,3}(\xi,\eta)\hat F^\tau_{x,1}(\xi,\eta)-\varphi^\tau_{X^\tau(x)}(\hat F^\tau_{x,3}(\xi,\eta))\partial_\eta \hat F^\tau_{x,1}(\xi,\eta).
		\end{align*}
		Since $\rho^\tau_x(0)=1$, $\hat F^\tau_{x,3}(\xi,0)=0$, $\varphi^\tau_{X^\tau(x)}(0)=0$, $\hat F^\tau_{x,1}(\xi,0)=\lambda_x^s (\tau)\xi$, and $\partial_\eta \hat F^\tau_{x,3}(\xi,0)=\lambda_x^u(\tau)$  (by~\eqref{property_un}), evaluating the former expression of $\partial_\eta\bar \psi^\tau_x$ at $(\xi,0)$ yields
		\begin{align*}%\label{expression_bar_psi}
			\partial_\eta\bar \psi^\tau_x(\xi,0)&={\rho^\tau_x}'(0)\xi\,\partial_\xi\tilde \psi^\tau_x(\xi,0)+\sigma^\tau_x(\eta)\partial_\eta\tilde \psi^\tau_x(\xi,0)\\
			&+{\varphi^\tau_x}'(0)\xi-{\varphi^\tau_{X^\tau(x)}}'(0)\lambda_x^u(\tau)\lambda_x^s(\tau)\xi,
		\end{align*}
		which is clearly $C^{r-1}$ uniformly in $x$. 
	\end{proof}
	Since $r \geq k+1$, the above claim ensures that $\partial_\eta \bar \psi^\tau_x(\cdot,0)$ is at least $C^k$ and we can use the Taylor expansion 
	$$
	\partial_\eta\bar \psi^\tau_x(\xi,0)=P_x^u(\tau)(\xi)+\bar\omega^\tau_x(\xi),
	$$
	where $\bar\omega^\tau_x$ is $C^{r-1}$, $\bar\omega^\tau_x(\xi)=O(\xi^k)$ uniformly in $x$, and $P_x^u(\tau)(\xi)$ is a polynomial of degree at most $[k]$. We then look for $\kappa^\tau_x$ such that 
	$
	\partial_\eta \psi^\tau_x(\xi,0)=P_x^u(\tau)(\xi). 
	$
	As above, the solution $\kappa^\tau_x$ is defined by a series, which actually converges by the $k$-pinching~\eqref{pinc_cond_last}. Moreover, $\kappa^\tau_x(0)=0$, and since $\bar\omega^\tau_x$ is $C^{r-1}$ uniformly in $x$, the function $\kappa^\tau_x$ is also $C^{r-1}$ uniformly in $x$. 
	
	Similarly to what was obtained in~\eqref{expression_bar_psi}, we have 
	\begin{equation*}%\label{expression_bar_psi}
		\psi^\tau_x\colon(\xi,\eta)\mapsto \bar \psi^\tau_x(\xi,\eta)+\kappa^\tau_x(\xi)\eta-\kappa^\tau_{X^\tau(x)}(\bar F^\tau_{x,1}(\xi,\eta))\bar F^\tau_{x,3}(\xi,\eta). 
	\end{equation*}
	Let us check that the final adjustment did not destroy the polynomial form of $\partial_\xi \bar \psi^\tau(0,\eta)$ obtained in the previous step. We calculate
	$$
	\partial_x \psi^\tau_x(0,\eta)={\kappa^\tau_x}'(0)\eta+\partial_\xi \bar \psi^\tau_x(0,\eta)-{\kappa^\tau_{X^\tau(x)}}'(0)\partial_\xi \bar F^\tau_{x,1}(0,\eta)\bar F^\tau_{x,3}(0,\eta)-\kappa^\tau_{X^\tau(x)}(0)\partial_\xi \bar F^\tau_{x,3}(0,\eta).   
	$$
	Note that the last term vanishes since $\kappa^\tau_{X^\tau(x)}(0)=0$, and we obtain
	$$
	\partial_x \psi^\tau_x(0,\eta)=\hat P_x^s(\tau)(\eta)+\left({\kappa^\tau_x}'(0)-{\kappa^\tau_{X^\tau(x)}}'(0)\lambda_x^s(\tau)\lambda_x^u(\tau)\right)\eta, 
	$$
	which is the sought polynomial $P_x^s(\tau)(\eta)$ of degree at most $[k]$. This concludes the proof of Lemma~\ref{forme_polyn}. 
\end{proof}
The proof of Proposition~\ref{prop_C1} is complete. 
\end{proof}

\begin{proof}[Proof of Addendum~\ref{add_C2}]
	We first prove the following lemma.
	
	\begin{lemma}
		For any $\tau>0$ and any rational number $p/q$ we have
		$$
		\rho^\tau_x=\rho^{\tau'}_x,\,\,\,\textup{and}\,\,\, 	\sigma^\tau_x=\sigma^{\tau'}_x, x\in M,
		$$
		where $\tau'=\frac pq\tau$.
	\end{lemma}
	\begin{proof}
		Clearly the rational case is implied by the integer case. For the sake simplicity, let us  consider the case when $\tau'=2\tau$, the general case being essentially the same calculation.
		One can check that 	$\rho^{2\tau}_x=\rho^{\tau}_x$ by using the series formula, or, instead we can show that these functions satisfy the same equation. Indeed, from the proof of Proposition~\ref{prop_C1} we have that $\rho^\tau_x$ satisfies
			$$
		\log \rho^\tau_{x}(\eta)-\log \rho^\tau_{X^{-\tau}(x)}(\lambda_x^u(-\tau)\eta)=\log\partial_\xi \tilde F^\tau_{X^{-\tau}(x),1}(0,\lambda_x^u(-\tau)\eta)-\log\lambda_{X^{-\tau}(x)}^s(\tau).
		$$
	Changing the base-point to $X^{-\tau}(x)$ the same equation reads
	\begin{multline*}
	\log  \rho^\tau_{X^{-\tau}(x)}(\tilde\eta)-\log \rho^\tau_{X^{-2\tau}(x)}(\lambda_{X^{-\tau}(x)}^u(-\tau)\tilde \eta)=\\
	\log\partial_\xi \tilde F^\tau_{X^{-2\tau}(x),1}(0,\lambda_{X^{-\tau}(x)}^u(-\tau)\tilde\eta)-\log\lambda_{X^{-2\tau}(x)}^s(\tau).
\end{multline*}
	Using $\tilde \eta=\lambda_x^u(-\tau)\eta$ and adding the above two equation while using the cocycle property of non-stationary linearizations gives
	\begin{multline*}
		\log \rho^\tau_{x}(\eta)-\log \rho^\tau_{X^{-2\tau}(x)}(\lambda_{x}^u(-2\tau) \eta)=
		\log\partial_\xi \tilde F^\tau_{X^{-\tau}(x),1}(0,\lambda_x^u(-\tau)\eta)+\\
			\log\partial_\xi \tilde F^\tau_{X^{-2\tau}(x),1}(0,\lambda_{x}^u(-2\tau)\eta)
		-\log\lambda_{X^{-2\tau}(x)}^s(2\tau).
	\end{multline*}
Now we recall that $\tilde F^{2\tau}_{X^{-2\tau}(x)}=	\tilde F^{\tau}_{X^{-\tau}(x)}	\circ \tilde F^{\tau}_{X^{-2\tau}(x)}$. Specifically for the first coordinate we have
$$ \tilde F^{2\tau}_{X^{-2\tau}(x), 1}(\xi,\eta)= 	\tilde F^{\tau}_{X^{-\tau}(x),1}(\tilde F^{\tau}_{X^{-2\tau}(x),1}(\xi,\eta), \tilde F^{\tau}_{X^{-2\tau}(x),3}(\xi,\eta)).
$$ 
Differentiating with respect to $\xi$ and evaluating at $\lambda_x^u(-2\tau)\eta$ while recalling that $\tilde F^\tau_{x,1}(0,\cdot)\equiv 0$ gives
\begin{multline*}
\log \partial_\xi  \tilde F^{2\tau}_{X^{-2\tau}(x), 1}(0,\lambda_x^u(-2\tau)\eta)=\\
 \log \partial_\xi 	\tilde F^{\tau}_{X^{-\tau}(x),1} (0,\lambda_x^u(-\tau)\eta)+\log\partial_\xi \tilde F^{\tau}_{X^{-2\tau}(x),1}(0,\lambda_x^u(-2\tau)\eta).
\end{multline*}
Hence the equation we derived on $\rho^\tau_x$ simplifies to
\begin{multline*}
		\log \rho^\tau_{x}(\eta)-\log \rho^\tau_{X^{-2\tau}(x)}(\lambda_{x}^u(-2\tau) \eta)=\\
\log \partial_\xi  \tilde F^{2\tau}_{X^{-2\tau}(x), 1}(0,\lambda_x^u(-2\tau)\eta)
-\log\lambda_{X^{-2\tau}(x)}^s(2\tau).
\end{multline*}
It remains to observe that $\rho^{2\tau}_x$ satisfies the same equation. Since this equation has a unique solution given by~\eqref{eq_rho_solution} satisfying $\rho_x(0)=1$ we conclude that indeed $\rho^{2\tau}_x=\rho^\tau_x$.

By identical arguments $\sigma^\tau_x=\sigma^{\tau'}_x$.
	\end{proof}

	Since the initial charts $\jmath_x$ vary contionously with respect to $x\in M$ in $C^r$ topology we have that the functions $\rho^\tau_x$ and $\sigma^\tau_x$ given by~\eqref{eq_rho_solution}  vary continuously in $C^{r-1}$ topology with respect to $\tau$.  Now given any $\tau$ consider a sequence of rationals $p_n/q_n\to\tau$, $n\to\infty$. Then for any $x\in M$
	$$
	\rho_x^\tau=\lim_{n\to\infty}\rho_x^{p_n/q_n}=\rho_x^1
	$$
	by the preceding lemma. Hence, indeed, $\rho_x^\tau$ does not depend on $\tau$.  Similarly, $\sigma_x^\tau$ does not depend on $\tau$, and we can conclude that the charts $\hat\jmath_x=h_x\circ\jmath_x$, $x\in M$, do not depend on the choice of $\tau>0$.

	\begin{lemma}
	For any $\tau>0$ and any rational number $p/q$ we have
	$$
	\varphi^\tau_x=\varphi^{\tau'}_x,\,\,\,\textup{and}\,\,\, 	\kappa^\tau_x=\kappa^{\tau'}_x, x\in M,
	$$
	where $\tau'=\frac pq\tau$.
\end{lemma}

\begin{proof} Similarly to the previous lemma, it is enough to check the case when $q=1$ and we only give a proof in the case when $\tau'=2\tau$, the case $\tau'=p\tau$ being fully analogous but requiring summing $m$ equations instead of just 2.
	
Recall that with respect to adjusted charts $\hat\jmath_x$ dynamics has the form
 $\hat F^\tau_x\colon (\xi,t,\eta)\mapsto (\hat F^\tau_{x,1}(\xi,\eta),t+\hat \psi^\tau_x(\xi,\eta),\hat F^\tau_{x,3}(\xi,\eta))$. Since $\hat F^{2\tau}_x=\hat F^\tau_{X^\tau(x)}\circ \hat F^\tau_x$ we have
 $$
 \hat\psi^{2\tau}_x(\xi,\eta)=\hat\psi^\tau_x(\xi,\eta)+\hat\psi^\tau_{X^\tau(x)}(F^\tau_{x,1}(\xi,\eta), \hat F^\tau_{x,3}(\xi,\eta)).
 $$
Differentiating with respect to $\xi$ and then evaluating at $(0,\eta)$, and using that $\hat\psi_x^\tau(0,\cdot)\equiv 0$ we have
\begin{equation}\label{eq_sum}
 \partial_\xi\hat\psi^{2\tau}_x(0,\eta)=\partial_\xi\hat\psi^\tau_x(0,\eta)+\partial_\xi\hat\psi^\tau_{X^\tau(x)}(0, \lambda^u_x(\tau)\eta)\lambda_x^s(\tau)
\end{equation}
Now we consider Taylor expansions for each of the three terms in this equation according to~\eqref{eq_taylor}. The Taylor remainders on both sides of the equation are the same, hence we deduce an equation on the remainders
$$
\omega^{2\tau}_x(\eta)=\omega_x^\tau(\eta)+\lambda_x^s(\tau)\omega^\tau_{X^\tau(x)}(\lambda^u_x(\tau)\eta).
$$

Similarly to the proof of the previous lemma, we will show that $\varphi^{2\tau}_x$ and $\varphi^\tau_x$ satisfy the same equation and, hence, must be equal by uniqueness of the solution. Recall that the equation for $\varphi^\tau_x$ is
$$
	\varphi^\tau_{X^\tau(x)}(\lambda_x^u(\tau)\eta)\lambda_x^s(\tau)=\varphi^\tau_x(\eta)+\omega^\tau_x(\eta). 
$$
We use this equation twice to derive an equation on $\varphi^{2\tau}_x$:
$$
	\varphi^\tau_{X^{2\tau}(x)}(\lambda_{X^\tau(x)}^u(\tau)\tilde\eta)\lambda_{X^\tau(x)}^s(\tau)=\varphi^\tau_{X^\tau(x)}(\tilde\eta)+\omega^\tau_{X^\tau(x)}(\tilde\eta).
$$
Substituting $\tilde\eta=\lambda_x^u(\tau)\eta$ and multiplying by $\lambda^s(\tau)$ gives:
\begin{multline*}
	\varphi^\tau_{X^{2\tau}(x)}(\lambda_{x}^u(2\tau)\eta)\lambda_{x}^s(2\tau)=\varphi^\tau_{X^\tau(x)}(\lambda_x^u(\tau)\eta)\lambda^s(\tau)+\lambda^s(\tau)\omega^\tau_{X^\tau(x)}(\lambda^u(\tau)\eta)=\\
	\varphi^\tau_x(\eta)+\omega^\tau_x(\eta)+\lambda^s(\tau)\omega^\tau_{X^\tau(x)}(\lambda^u(\tau)\eta)=	\varphi^\tau_x(\eta)+\omega^{2\tau}_x(\eta).
\end{multline*}
But this is precisely the defining equation of $\varphi^{2\tau}_x$. Hence $\varphi^{2\tau}_x=\varphi^{\tau}_x$, $x\in M$.

	By identical arguments $\kappa^\tau_x=\kappa^{\tau'}_x$.
\end{proof}

We can conclude that the adapted charts defined by $\imath_x=\hat\jmath_x\circ u_x$, $x\in M$, do not depend on the choice of $\tau>0$.

Finally the twisted cocycle equations on the polynomials $P^s_x(\tau)$ follow easily by Taylor expanding all terms in~\eqref{eq_sum} and matching the main polynomial terms, with a similar argument for $P^u_x(\tau)$.
\end{proof}

\bibliographystyle{alpha}
\bibliography{biblio.bib}  

\end{document}